\documentclass[11pt]{amsart}
\usepackage{a4wide}
\usepackage{amssymb,amsthm}
\usepackage{fullpage}
\usepackage{color}
\usepackage{amsmath}

\newcommand{\beq}{\begin{equation}}
\newcommand{\eeq}{\end{equation}}

\newcommand{\R}{{\mathbb R}}

\newcommand{\N}{{\mathbb N}}

\newcommand{\pa}{\partial}

\def\eps{\varepsilon}

\Large

\newtheorem{prop}{Proposition}
\newtheorem{theoreme}[prop]{Theorem}
\newtheorem{lem}[prop]{Lemma}
\newtheorem{cor}[prop]{Corollary}

\newtheorem{rem}[prop]{Remark}

\numberwithin{equation}{section}
\numberwithin{prop}{section}
\begin{document}
\title{ Transverse instability of the line solitary  water-waves}
\author{Frederic Rousset}
\address{IRMAR, Universit\'e de Rennes 1, campus de Beaulieu,  35042  Rennes cedex, France}
\email{ frederic.rousset@univ-rennes1.fr  }
\author{Nikolay Tzvetkov}
\address{D\'epartement de Math\'ematiques, Universit\'e Lille I, 59 655 Villeneuve d'Ascq Cedex, France}
\email{nikolay.tzvetkov@math.univ-lille1.fr}
\date{}
\maketitle
\begin{abstract}
We prove  the linear and nonlinear  instability of the line solitary  water  waves
with respect to transverse 
perturbations. 
\end{abstract}

\tableofcontents

\section{Introduction}

The water waves problem  that is to say the study of fluid motions in the presence of a free surface
has been the object of many studies in the past thirty years. This problem which is highly nonlinear  is very interesting in many aspects.
The study of the well-posedness of the Cauchy problem  has been widely studied
recently  \cite{Wu, L1, Ambrose-Masmoudi, Ch-Lind, Lind, Coutand, Shatah}.
A lot of  progress has also been made in the rigorous justification of  important asymptotic models
like the KdV or KP equations \cite{Craig, Alvarez-Lannes, Schneider-Wayne}.
Finally, let us also mention that    the   study of   bifurcation of travelling waves
\cite{AK, Buffoni, Iooss, Groves-Mielke}  or  more complicated patterns \cite{ Strauss, Iooss2, Plotnikov-Toland}
has also attracted a lot of attention.
An important aspect  of the theory which is in some sense in the intersection of the above aspects 
is the study of the  dynamical stability  of  important  particular  patterns  like solitary waves.

Our goal here is to study the stability of the line solitary waves constructed
by  Amick and Kirchg\"assner in \cite{AK}. In the context of the water waves model,  in the  physical
situation,
the velocity of the fluid depends on three variables,  while
the free surface of the fluid which is also unknown is two-dimensional. We shall refer to this situation
as the two-dimensional case.
We shall  call  
one-dimensional waves or line waves   solutions of the water waves  equations  for which the surface
 of the fluid  is invariant by translation in one direction.

The stability of  the line solitary waves when submitted to one-dimensional perturbations
was  studied  by Mielke in  \cite{Mielke} where a  conditional orbital stability result was established. 
This means that,  as long as  the solution of the water
 waves equations issued from a small one-dimensional perturbation
  of a  line solitary wave    exists in the energy space,  it remains close to translates of the solitary wave
   in the energy norm.  The precise statement is  given below.

Here,  we shall prove the nonlinear Lyapounov  instability
of these waves  when  they are  submitted  to  
perturbations depending in a nontrivial way on the transverse variable.
For that purpose we  construct a family of  smooth solutions of the water 
waves equations which give  arbitrarily small perturbations  to a line solitary wave 
at the initial time and which after  
(long) times  separate from the solitary wave (an its spatial translates) at some fixed distance, 
the distance being measured in some natural norm for the problem.
More precisely, we prove an instability result in the $L^2$ norm which  thus implies instability
in the energy norm. 
Our result  contains the fact that the solution 
remains smooth on a sufficiently long time scale where 
the  instability can be observed.

The destabilization of one-dimensional stable patterns by transverse perturbations
arises  very often in dispersive equations. In the early 1970's, using the theory of integrable systems 
Zakharov \cite{Z2} obtained the transverse instability of the soliton of the  Korteweg -de Vries (KdV) equation  
considered as a one-dimensional  solution of the (two-dimensional) Kadomtsev-Petviashvili-I (KP-I) equation.
Since   the KP-I equation can be obtained as a long-wave asymptotic model from the
water waves system in the presence of enough surface tension (Bond number larger than
$1/3$), 
the situation considered by Zakharov can be thought as a strongly simplified  model 
for the problem that we consider here namely the full water waves system with  strong surface tension
(Bond number larger than $1/3$).
Let us point out that the surface tension seems to have a destabilizing effect. Indeed,  
when the surface tension is weaker (Bond number smaller
than $1/3$), the asymptotic model in the same long-wave regime  is the KP-II equation and for the KP-II equation,
  the  linearized equation about the solitary wave has no unstable spectrum as shown in  \cite{APS}.
An interesting open question is therefore  the study of the   transverse stability of the line solitary waves
 constructed in \cite{Iooss}
in the case of small surface tension (Bond number smaller than $1/3$). Note that even for the model
 case of the KP-II equation the nonlinear stability is an interesting  unsolved problem.     

The main drawback of Zakharov approach is that a lot of dispersive equations
like the water waves system that we want to study are not known to  be completely  integrable.

An important feature of  most of   the  important models is that they 
are endowed with an Hamiltonian structure. This structure in the water-waves setting 
was exhibited by Zakharov \cite{Zakharov}.
Nevertheless,  the general framework of Grillakis-Shatah-Strauss
\cite{GSS} which  has been developed in order to prove stability  or instability of constrained
minimizers of  the Hamiltonian, and has been successful  for studying
orbital stability of solitary waves in many dispersive models,
does not seem to apply in transverse stability problems.
The main reason is that the two-dimensional energy is infinite at the one-dimensional object.  

In our previous works \cite{RT1,RT2}, we developed an approach
to study the transverse instability of solitary waves  for Hamiltonian partial differential equations
which applies in the situation considered by 
Zakharov and also to many other dispersive, not necessarily integrable, equations. 
This method inspired by a  work by Grenier \cite{Grenier} in fluid mechanics
consists in reducing the problem to the proof of  linear instability for a family of  one-dimensional
problems by  proving that linear instability implies nonlinear instability.

The water waves system with surface tension does not enter in the general framework
of \cite{RT2}. Among the main difficulties  are the  high level of nonlinearity in the equations
which makes the study of the Cauchy problem for perturbations of the solitary wave non-trivial
and the presence of a non-local term which does not allow to reduce the study 
of the  equation linearized about the solitary wave  to the study of ordinary differential equations.
Nevertheless, the general philosophy
of our approach can be used,  more details on the description of  our approach
will be given in the end of this introduction.

The remaining part of the introduction is organized as follows.
We first present the water waves system.  
Next, we describe the solitary waves as a special solution of the water waves system.
Further, we give more details on  the result of Mielke \cite{Mielke} about  stability with respect to one dimensional perturbations.
We then state our instability result  with respect to two dimensional (transverse) perturbations.
We end the introduction by explaining the general strategy behind our proof.

\subsection{The water waves  system  with  surface tension}
We shall use the notation  $Y=(X,z)\in\R^3$ with $X=(x,y)\in\R^2$.  We consider the situation
 where the fluid domain which is  unknown  is defined by
$$
\Omega_t=\{(X,z)\in\R^3\,:\, -h<z<\eta(t,X)\},
$$
where $t$ is the time,  $h$  is a parameter defining the fixed  bottom $z=-h$  and
 $z= \eta(t,X)$  is the equation of the free surface at time $t$.
We denote by $u$ the speed of the fluid. We  consider the motion of an  irrotational,  incompressible
fluid with constant density. This means that the velocity $u$ of the fluid is given by 
$u=\nabla_{Y}\varphi= (\partial_{x}\phi,\partial_{y}\phi,\partial_{z}\phi)$ for some scalar
function $\phi$ and hence we find that 
inside the fluid domain $\Omega_{t}$, 
\beq
\label{inside}
\nabla_{Y}\cdot u=  \Delta_{Y} \phi = 
(\partial_{x}^2+\partial_{y}^2+\partial_{z}^2)\phi(t,x,y,z)=0,\quad {\rm in}\quad \Omega_t.
\eeq 
On the boundaries  of $\Omega_{t}$, we make the usual 
assumption that no fluid particles cross the boundary.
At  the bottom of the fluid this reads 
\beq
\label{fond}
\partial_{z}\phi(t,x,y,-h)=0
\eeq
and on the free surface,  this yields the kinematic condition 
\beq
\label{surface1}
\partial_{t}\eta(t,X)+\nabla_{X} \phi(t,X,\eta(t,X))\cdot\nabla_{X} \eta(t,X)-\partial_{z}\phi(t,X,\eta(t,X))=0,
\eeq
where we use the notation 
$\nabla_{X} \equiv(\partial_{x},\partial_{y})^t$.
Finally,  taking into account the surface tension to compute the pressure on the free surface, 
we find the Bernouilli law: 
\beq
\label{surface2}
\partial_{t}\phi(t,X,\eta(t,X))+\frac{1}{2}|\nabla_{Y}\phi(t,X,\eta(t,X))|^{2}+g\eta(t,X)=
b\nabla \cdot\frac{\nabla_{X} \eta(t,X)}{\sqrt{1+|\nabla_{X} \eta(t,X)|^{2}}}\,,
\eeq
where $\nabla_{Y}\equiv (\nabla_{X} ,\partial_{z})$.
The coefficient $b$ is the Bond number  which measures  the influence of the  surface tension and $g$ 
is the gravitational constant.
The term $g\eta(t,X)$ is the trace of the gravitational force $gz$ on the free surface.

It is classical to rewrite the system \eqref{inside}, \eqref{surface1}, \eqref{surface2}
as a system  where all functions are evaluated on  the free surface only.
Let us next define the following   Dirichlet-Neumann  operator:  for given $\eta(X)$ and
$\varphi(X)$, we define $\phi(X,z)$ as the (well-defined) solution of the elliptic boundary value problem
$$
(\partial_{x}^2+\partial_{y}^2+\partial_{z}^2)\phi=0,\quad{\rm in}\quad
\{(X,z)\,:\,-h<z<\eta(X)\},
\quad
\phi(X,\eta(X))=\varphi(X),\quad \partial_{z}\phi(X,-h)=0, 
$$
and we define the Dirichlet-Neumann operator as
$$
G[\eta]\varphi\equiv
\partial_{z}\phi(X,\eta(X))-\nabla_{X}\eta(X)\cdot\nabla_{X}\phi(X,\eta(X))
=\sqrt{1+|\nabla_{X}\eta|^{2}}
(\nabla_{Y}\phi(X,\eta(X))\cdot n(X)),
$$
where
$$
n(X)=\frac{1}{\sqrt{1+|\nabla_{X}\eta|^{2}}}(-\nabla_{X}\eta(X),1)
$$
is the unit outward normal of  the free surface at the point $z=\eta(X)$.

This allows to  rewrite the  system in terms of  the functions evaluated on the free surface only.
Set
$$
\varphi(t,X)\equiv \phi(t,X,\eta(t,X))\,.
$$
Then one directly checks that the water waves problem \eqref{inside}, \eqref{fond}, 
\eqref{surface1}, \eqref{surface2}  is reduced to the study of the following system 
\begin{eqnarray}
\label{eq1}
\partial_{t} \eta  & = & G[\eta]\varphi,
\\
\label{eq2}
\partial_{t } \varphi  & = & -\frac{1}{2}|\nabla_{X}\varphi|^2+\frac{1}{2}
\frac{(G[\eta]\varphi+\nabla_{X}\varphi\cdot\nabla_{X}\eta)^2}{1+|\nabla_{X}\eta|^2}-g\eta
+b\nabla_{X}\cdot \frac{\nabla_{X}\eta}{\sqrt{1+|\nabla_X \eta|^{2}}}\,.
\end{eqnarray}
As noticed by Zakharov \cite{Zakharov}, 
the  system (\ref{eq1})-(\ref{eq2})  has a canonical Hamiltonian structure
$$
\partial_{t }\eta =\frac{\delta \mathcal{H}}{\delta\varphi},\quad \partial_{t } \varphi=-\frac{\delta \mathcal{H}}{\delta\eta}
$$
where  the Hamiltonian $\mathcal{H}$  is the total energy given by
$$
\mathcal{H}(\eta,\varphi)=\frac{1}{2}\int_{\R^2}\Big[G[\eta]\varphi\,\varphi+g\eta^2+2b(\sqrt{1+|\nabla\eta|^2}-1)\Big].
$$
This is the  sum of the kinetic energy, the gravitational potential energy and a surface energy due to stretching
of the surface.
The expression of the  variational derivatives of  $\mathcal{H}$  can be  checked by easy calculation thanks to the following lemma 
(see \cite[Theorem~3.20]{L1} for example).
\begin{lem}\label{DN'}
For an integer $k\geq 2$, consider the map $\eta \mapsto G[\eta]\varphi$, acting between 
the Sobolev spaces $H^{k+1/2}(\R^2)$ and $H^{k-1/2}(\R^2)$. Then
$$
D_{\eta} G[\eta]\varphi \cdot \zeta  =
-G[\eta](\zeta Z)-\nabla\cdot\Big(\zeta(\nabla_{X}\varphi-Z\nabla_{X}\eta)\Big),
$$
where $Z$,   linear in $\varphi$  and real valued,  is defined by
$$
Z=Z(\eta,\varphi)\equiv \frac{G[\eta]\varphi+\nabla_{X}\eta\cdot\nabla_{X}\varphi}
{1+|\nabla_{X}\eta|^2}\,.
$$
\end{lem}
Note that because of  the translational invariance in the problem, the momentum
$$\mathcal{P}(\eta, \varphi)= \int_{\mathbb{R}^2} \eta \varphi_{x}$$
is also a  formally conserved quantity.

The Hamiltonian structure of (\ref{eq1})-(\ref{eq2}) will be of  crucial
importance for many aspects of our analysis   in particular for  the choice of multipliers when
performing energy estimates. 

\subsection{The line solitary wave solution of  the water waves system (\ref{eq1})-(\ref{eq2})}
For $c\geq 0$, since we shall study solitary waves with speed $c$, 
we make a change of frame $X=(x,y,z)\mapsto (x-ct,y,z)$ which  changes the dynamical equations
\eqref{surface1}, \eqref{surface2} into
$$
\partial_{t}\eta (t,X)-c \partial_{ x}\eta(t,X)
+\nabla_{X}\phi(t,X,\eta(t,X))\cdot\nabla_{X}\eta(t,X)-\partial_{z}\phi(t,X,\eta(t,X))=0
$$
and
$$
\partial_{t}\phi (t,X,\eta(t,X))
-
c \partial_{x} \phi (t,X,\eta(t,X))
+\frac{1}{2}|\nabla_{Y}\phi(t,X,\eta(t,X))|^{2}+g\eta(t,X)=
b\nabla_X\cdot\frac{\nabla_X\eta(t,X)}{\sqrt{1+|\nabla_X \eta(t,X)|^{2}}}\,.
$$
By using again the Dirichlet-Neumann operator, 
the equations (\ref{eq1}), (\ref{eq2}) become
\begin{eqnarray}\label{eq1c}
\partial_{t}\eta  & = & c\, \partial_{x} \eta+ G[\eta]\varphi,
\\
\label{eq2c}
\partial_{t}\varphi  & = & c\, \partial_{x} \varphi  -\frac{1}{2}|\nabla_{X}\varphi|^2+\frac{1}{2}
\frac{(G[\eta]\varphi+\nabla_{X}\varphi\cdot\nabla_{X}\eta)^2}{1+|\nabla_{X}\eta|^2}-g\eta
+b\nabla_{X}\cdot \frac{\nabla_{X}\eta}{\sqrt{1+|\nabla_X \eta|^{2}}}
\end{eqnarray}
where $\varphi$ is again defined as $\varphi(t, X)= \phi(t, X, \eta(t,X))$.
A solitary wave with speed $c$ becomes  a stationary solution (i.e. independent of $t$) of (\ref{eq1c}), (\ref{eq2c}).
To study the existence of such solitary waves, it is classical to introduce
a non-dimensional version of the equations.
Let us perform the  change of variable
$$
\eta(t,X)=h\, \tilde{\eta}\Big(\frac{c}{h}t,\frac{1}{h}X\Big),\quad
\phi(t,X,z)=c\,h\,\tilde{\phi}\Big(\frac{c}{h}t,\frac{1}{h}X,\frac{1}{h}z\Big)\,.
$$
Then the equations satisfied by $\tilde{\eta}$, $\tilde{\phi}$ which for the sake of  simplicity will still be denoted by
$\eta$, $\phi$ are
$$
\partial_{t}\eta (t,X)- \partial_{x } \eta (t,X)
+\nabla_{X}\phi(t,X,\eta(t,X))\cdot\nabla_{X}\eta(t,X)-\partial_{z}\phi(t,X,\eta(t,X))=0
$$
and
$$
\partial_{t } \phi(t,X,\eta(t,X))-\partial_{x}\phi (t,X,\eta(t,X))
+\frac{1}{2}|\nabla_{Y}\phi(t,X,\eta(t,X))|^{2}+\alpha\eta(t,X)=
\beta\nabla_X\cdot\frac{\nabla_X\eta(t,X)}{\sqrt{1+|\nabla_X \eta(t,X)|^{2}}}\,,
$$
where the fluid domain is now $\{(X,z):-1<z<\eta(t,X)\}$ and
$$
\alpha=\frac{gh}{c^2},\quad \beta=\frac{b}{hc^2}\,.
$$
Note that the elliptic equation \eqref{inside} is not changed.
The equations formulated  on the free surface thus  become
\begin{eqnarray}\label{eq11}
\partial_{t} \eta  & = & \partial_{x} \eta + G[\eta]\varphi,
\\
\label{eq22}
\partial_{t }\varphi  & = &\partial_{x} \varphi  -\frac{1}{2}|\nabla_{X}\varphi|^2+\frac{1}{2}
\frac{(G[\eta]\varphi+\nabla_{X}\varphi\cdot\nabla_{X}\eta)^2}{1+|\nabla_{X}\eta|^2}-\alpha\eta
+\beta\nabla_{X}\cdot \frac{\nabla_{X}\eta}{\sqrt{1+|\nabla_X \eta|^{2}}}.
\end{eqnarray}
The 
Hamiltonian is now given by 
$$
H(\eta,\varphi)=
\frac{1}{2}\int_{\R^2}\Big[G[\eta]\varphi\,\varphi+\alpha\eta^2+
2\beta(\sqrt{1+|\nabla\eta|^2}-1)-2\,\eta\,\partial_{x}\varphi
\Big].
$$
In terms of the parameters $\alpha$ and $\beta$, we have the following existence result (see \cite{AK}) 
concerning stationary solutions 
of (\ref{eq11})-(\ref{eq22})
(or equivalently solitary wave solutions of the original problem (\ref{eq1})-(\ref{eq2})).
\begin{theoreme}[Amick-Kirchg\"assner \cite{AK}]\label{theoOS}
Suppose that $\alpha=1+\varepsilon^2$ and $\beta>1/3$. Then there exists $\varepsilon_0$ such that for every
$\varepsilon\in (0,\varepsilon_0)$ there is a stationary solution 
$(\eta_{\varepsilon}(x),\varphi_{\varepsilon}(x))$ (i.e. also independent of $y$) 
of (\ref{eq11})-(\ref{eq22}) under  the form
$$
\eta_{\varepsilon}(x)=\varepsilon^2 \Theta(\varepsilon x,\varepsilon),
\quad
\varphi_{\varepsilon}(x)=\varepsilon \Phi(\varepsilon x,\varepsilon),
$$
where $\Theta$ and $\Phi$  satisfy:
$$
\exists\, d>0,\quad \forall\,\alpha\in\N ,\quad \exists\, C_{\alpha}>0,\quad 
\forall\, (x,\varepsilon)\in\R\times (0,\varepsilon_0),\quad
|(\partial^{\alpha}_{x}\Theta)(x,\varepsilon)|\leq C_{\alpha} e^{-d|x|}  
$$
and
$$
\exists\, d>0,\quad \forall\,\alpha\geq 1 ,\quad
\exists\, C_{\alpha}>0,\quad 
\forall\, (x,\varepsilon)\in\R\times (0,\varepsilon_0),\quad
|(\partial^{\alpha}_{x}\Phi)(x,\varepsilon)|\leq C_{\alpha} e^{-d|x|}  \,.
$$
\end{theoreme}

Observe that the speeds of  the  solitary waves of (\ref{eq1})-(\ref{eq2})  built  in  the
above result are  close to 
$\sqrt{gh}$ which is independent of $\varepsilon$.

The  profiles $\Theta(\xi, \eps)$ and $\Phi(\xi, \eps)$ have smooth expansions in $\eps$.
In particular for $\eps =0$, we find
\begin{equation} \label{Thetadef}
\Theta(\xi,0)=- \mbox{cosh}^{-2}\, \Big(  {\xi \over  2 (\beta - 1/3 )^{1/2}  } \Big)
\end{equation}
and hence we recover the KdV solitary wave.

\subsection{Stability with respect to one-dimensional perturbations}
A very natural question is to study the stability of the solitary wave
solutions obtained in Theorem~\ref{theoOS}. Because of the invariance of the
problem with respect to spatial translations, usual Lyapounov stability cannot hold and thus it is natural to study the stability of the solitary wave
modulo these translations (orbital stability). It turns out that  under one-dimensional 
perturbations  the solitary waves of Amick-Kirchg\"assner are (orbitally) stable.

Let us fix the functional setting. We define the space $Z(\mathbb{R})$ as
$Z(\mathbb{R})= H^1(\mathbb{R}) \times H^{1 \over 2}_{*}(\mathbb{R})$ where 
$$ H^{1 \over 2}_{*}(\mathbb{R})= \Big\{ \varphi, \quad  \|\varphi\|_{H^{1 \over 2}_{*}}^2
= \int_{\mathbb{R}} |\xi|\,   \mbox{tanh}\, |\xi|\,  |\hat{\varphi}(\xi)|^2 <+\infty, \Big\}_{/\mathbb{R}}$$
which means that we do not distinguish functions that just differ by a constant.
Note that the control given by the $\|\cdot\|_{H^{1 \over 2}_{*}}$  semi-norm in the low frequencies is worse
than the one given by the usual  homogeneous $\dot{H}^{1 \over 2}$ semi-norm.  This is the natural semi-norm
associated to the Dirichlet-Neumann operator  and thus $Z(\mathbb{R})$ is the natural space associated
to the Hamiltonian. Nevertheless, to make this statement rigorous, we need a little bit more
control on the regularity of the surface. 
We  thus  also introduce   for $R>1$   the subspace 
$$ Z_{R}(\mathbb{R})= \Big\{ U= (\eta, \varphi) \in Z(\mathbb{R}), \quad -1 + {1 \over R} \leq \eta(x)
\leq R, \quad \|\eta_{x}\|_{L^\infty} \leq R\Big\}.$$

The result of \cite{Mielke} reads as follows.

\begin{theoreme}[Mielke \cite{Mielke}, $1d$ stability]
\label{stab_1d}
Let $\alpha$, $\beta$ and $\eps_0$ be as in Theorem \ref{theoOS}.  Then there exists $\eps_{1}>0$
such that for every  $\eps \in (0, \eps_{1}]$ and $R>1$,  the solitary wave $(\eta_{\eps}, \varphi_{\eps})$
is conditionally stable in the following sense.

For every $\kappa>0$, 
there exists $\delta >0$ such that :
if  $U=(\eta, \varphi): \,$ $[0, T) \rightarrow Z_{R}(\mathbb{R})$ is a continuous solution of 
\eqref{eq1}, \eqref{eq2},   which preserves the Hamiltonian $\mathcal{H}$
and the momentum $\mathcal{P}$ and satisfies
$\|U(0) - (\eta_{\eps}, \varphi_{\eps}) \|_{Z} \leq \delta$
then it satisfies
$$ \inf_{x_{0} \in \mathbb{R}} \|U(t, \cdot -x_{0})- (\eta_{\eps}, \varphi_{\eps}) \|_{Z} \leq \kappa, 
 \quad \forall t \in [0, T).$$
\end{theoreme}

As stated by Mielke,  the assumption that $\eps$ is sufficiently small can be replaced by assuming
that a family of  solitary waves depending smoothly on the speed  exists and that the list of  spectral assumptions 
and the condition on the moment of instability  necessary 
in the framework of \cite{GSS} hold.

\subsection{Main result: transverse nonlinear instability}
The situation drastically  changes if one considers $2d$ (transverse) perturbations.
The main result of this paper is   that  the solitary wave
solutions obtained in Theorem~\ref{theoOS} are (orbitaly) unstable when
submitted to two-dimensional localized perturbations (transverse
instability). Here is the precise statement.
\begin{theoreme}[Transverse instability]\label{main}
Let $\alpha$, $\beta$ and $\eps$ as in Theorem \ref{theoOS}.
There exists $\varepsilon_1>0$ such that for every
$\varepsilon\in (0,\varepsilon_1]$ the following holds true.

For every $s\geq 0$, there exists $\kappa>0$ such that for every $\delta >0$, we can find an initial data 
$(\eta_{0}^\delta(x,y),\varphi_0^\delta(x,y)) $  and a time $T^\delta \sim |\log \delta | $  such that 
$$
\|(\eta_{0}^\delta(x,y),\varphi_0^\delta(x,y)) - (\eta_{\varepsilon}(x),\varphi_{\varepsilon}(x))
\|_{{H^s(\R^2)\times H^s(\R^2)}} \leq \delta 
$$
and there exists a solution $(\eta^\delta(t,x,y),\varphi^\delta(t,x,y))$ of the water waves equation
(\ref{eq11})-(\ref{eq22}) with data $(\eta_{0}^\delta,\varphi_0^\delta)$, defined on $[0, T^\delta]$ 
and satisfying
$$
\inf_{a\in \R}
\|(\eta^\delta(T^\delta,x,y),\varphi^\delta(T^\delta,x,y))-
(\eta_{\varepsilon}(x-a),\varphi_{\varepsilon}(x-a))\|_{L^2(\R^2)\times
L^2(\R^2)}>\kappa.
$$
\end{theoreme}
Let us give a few comments on this result.

The instability is stated in the $L^2$ norm. This thus implies an instability in  the energy norm
$H^1 \times H^{1 \over 2}$. We shall actually establish the stronger result  that  the $L^2$  distance  of 
the solution   to  all
functions depending on $x$ only is at time $T^\delta$ larger than $\kappa$.

As in Theorem \ref{stab_1d}, the assumption that $\eps$ is sufficiently small can be replaced
by  the same assumptions as in  \cite{Mielke}. Namely, we need that the solitary wave
exists and that the linearization of the one-dimensional equation about the solitary wave
verifies  some spectral assumptions. Note that   we
only need the existence of the solitary wave and some stability properties of the one-dimensional
problem without 
any additional assumption in order to get  the transverse instability i.e $\eps_{1}$  is 
the same in Theorems \ref{stab_1d} and \ref{main}.

Let us remark that our theorem is not conditional:  we establish  the existence of the solution on $[0, T^\delta]$  which is already
a non trivial part of the statement. 

\subsection{Outlines of the paper}

To prove Theorem \ref{main}, we shall  construct the solution $U^\delta= (\eta^\delta, \varphi^\delta)$
of \eqref{eq11}, \eqref{eq22}
under the form 
$$ U^\delta = U_{\eps} +  U^{a} + V, \quad U_{\eps}= (\eta_{\eps}, \varphi_{\eps})^t$$
where following the approach of \cite{Grenier},  $U^{a}$ is an exponentially growing solution driven by the linear
instability and $V$ is a corrector that we add in order to get an exact solution
of the nonlinear equation.
There are three main parts in the paper. In the first part, we study the linearized  water waves equations
about the solitary wave, the  aim is to construct the leading part of $U^{a}$ as
an exponentially growing solution of the linearized equation with the maximal growth rate.
The second  step is the  construction of   the remaining part of $U^{a}$ where we describe
how the linear instability interacts with the nonlinear term.
The last step is the construction of the correction term $V$ where we need to study
a nonlinear problem.
Here are more details: 

\begin{itemize}
\item As a preliminary step, in Section \ref{sectionprelim}, we study the structure of the water waves equations
linearized about the solitary waves $(\eta_{\eps}, \varphi_{\eps})^t$. By using the fact that
the solitary waves do not depend on the transverse variable, we can Fourier transform
the linearized equation in the transverse variable to reduce the problem to
the study of a family of linear  equations  indexed by the transverse frequency
parameter $k \in \mathbb{R}:$ 
\beq
\label{introeqlin}
\partial_{t} U= JL(k) U
\eeq
where $L(k)$ is a symmetric operator and $J=\left( \begin{array}{cc} 0 & 1 \\ -1 & 0 \end{array}\right)$.
In the expression of the operator $L(k)$ arises the "Fourier transform" of the Dirichlet-Neumann operator
$G_{\eps, k}$ defined as
$$ G[\eta_{\eps}]\big( f(x) e^{iky} \big) = e^{iky} G_{\eps,k}\big( f(x) \big).$$
In order to understand the main properties
of the linearized equation \eqref{introeqlin}, we first need to study carefully  $G_{\eps, k}$.
This is the aim of Section \ref{sectionGepsk}. The estimates that we establish are rather classical
when  $k$ is fixed, we refer for example to \cite{L1}, \cite{Alvarez-Lannes}, the main novelty
is that we need to track carefully  the dependence in $k$ (especially when $k$ is close to zero) in the estimates.
We also point out the elementary but very useful property that  $G_{\eps, k}$ as a symmetric operator 
 depends on $|k|$ in a monotonous way.

\item  In Section \ref{sectionLk}, we study the properties of $L(k)$. We establish that
it has a self-adjoint realization on $L^2\times L^2$ with domain $H^2 \times H^1$
and study its spectrum. We first get  (Proposition \ref{essL})  that its essential
spectrum  is contained in $[c_{k}, + \infty)$ where $c_{k} \geq 0$ and $c_{k}>0$ 
if $k \neq 0$. Next, in Proposition  \ref{Lk}, we prove that  for $ \eps$ sufficiently
 small, $L(k)$ has at most one negative eigenvalue for every $k$.
Note that in the case $k=0$ the spectrum  of $L(k)$ can be described by using  the
 spectrum of the KdV equation  linearized about the KdV solitary wave
  as shown  by Mielke in \cite{Mielke}.

\item    In Section \ref{sectionJLk}, we study the operator $JL(k)$. We prove  (Proposition \ref{essJL})
that its essential spectrum is included in $i \mathbb{R}$ and locate its possible unstable
 (i.e. with positive real parts)
 eigenvalues in Proposition \ref{brute0cor}.  Finally, in  Theorem \ref{mode}, we prove
  the linear instability: we show that for some $k \neq 0$,  $JL(k)$ has an unstable eigenvalue.
This last  result is  known,  it was  obtained in \cite{PS}, \cite{GHS}, \cite{Bridges} for
 example by using  different formulations of the water waves equation.
  The proof that we get here is  very simple, it  just relies on the monotonous dependence
   of $L(k)$ in $k$ and a bifurcation argument based on the Lyapounov-Schmidt method.
   An important consequence of this part is that we  get the existence
    of a most unstable eigenmode i.e an eigenvalue $\sigma(k_{0})$ of $JL(k_{0})$
   such that 
   $$ \mbox{Re } \sigma(k_{0})= \mbox{sup } \big\{ \mbox{Re } \sigma, \quad   \exists k, \, \sigma \in \sigma \big(JL(k)\big)  \big\}.$$

\item    Once these main properties are established, we are  able to construct the unstable
approximate solution $U^{a}= (\eta^a , \varphi^a)$. From the spectral properties of $JL(k)$, we
take the first part $U^0$ of $U^{a}$ (see Proposition \ref{U0}) under the form
$$ U^0= \int_{I} e^{\sigma(k) t } e^{iky} U(k) \, dk$$
 where $\sigma(k)$ is an analytic curve passing through $\sigma(k_{0})$.

\item  The next step is to construct $U^{a}$ (Proposition \ref{Uap}).
We look for $U^{a}$ under  the form
$$ U^{a}= \delta \sum_{j=0}^M \delta^j U^j$$
where each term $U^j$  must be   bounded from above by $\sim e^{ \sigma_{0}(j+1) t}$
 with $\sigma_{0}=  \mbox{Re } \sigma(k_{0})$.  They are  solutions
  of   linear equations with source terms.
   The crucial property that is  thus needed 
   is  an accurate $H^s$  estimate for the semi-group of $JL(k)$.
    Since $JL(k)$ is not sectorial some work is needed to establish it.
  Here,  we use the Laplace transform. To control the high time frequencies,  we  use
   energy estimates  based on  the Hamiltonian structure of the equation and the properties  of  $L(k)$. For the bounded
    frequencies, we  use abstract arguments  based on the knowledge
     of the spectrum of $JL(k)$.

\item The last step  is to construct the correction term $V$ which solves a nonlinear water-waves  equation.
This is the aim of Section~\ref{sectionnonlin}. \\
     The local well-posedness for the water waves equation has been much studied recently, we refer
for example to \cite{Wu}, \cite{L1}, \cite{Zhang},  \cite{Ch-Lind}, \cite{Lind}, \cite{Coutand}, \cite{Shatah}.
Here, we want to prove that  there exists a smooth solution of the water waves equation
in the vicinity of the approximate unstable solution which remains  smooth on a sufficiently
long interval of time. Moreover, we want a precise estimate between
the exact and the approximate solution in order to get the instability result.
 For this reason the approaches like \cite{L1} or \cite{Zhang} which are based on 
  the Nash-Moser's scheme are not suitable for our purpose.
 It was noticed in \cite{Iguchi} that when there is no surface tension, the water waves
  system  has a quasilinear structure once we have applied three space derivatives on it.
When there is surface tension, the main difficulty is that  the commutator between a space derivative
and the term coming from the surface tension contains too many derivatives to be considered
 as a remainder.
   This situation arises classically in the study of  high order wave equations for example
$$ 
\partial_{tt} u = -  |D|^{ 3 \over 2}\big( a(u) |D|^{3 \over 2 }u \big), \quad
a \geq a_{0}.
$$
Note that in 1-D the water wave problem in Lagrangian
coordinates is indeed very close to this situation (see \cite{Schneider-Wayne} for example).
For such high order wave equations,  a good candidate in order to get  $H^s$ type  estimates is to
apply powers of the operator $|D|^{ 3 \over 2}\big( a(u) |D|^{3 \over 2 }u \big)$
to the equation.  This is the approach chosen in the study of the water waves system
in \cite{Shatah}.

Here, to handle this difficulty we shall use a slightly different approach which is based on the use of
time derivatives: 
the energies that we use involve simultaneous   space and time derivatives of the unknown.
The basic block  in the construction of our energies  comes from the Hamiltonian structure
of the system, nevertheless, we also need to  add some lower order terms in order to cancel
some commutators.
This approach  yields  slightly simpler  commutators to compute and allows  to get a quasilinear form
of the system when there is surface tension.
Note that  our argument provides the well-posedness (without Nash-Moser's
scheme) of the water waves with
surface tension (a result already obtained  in \cite{Zhang} via
Nash-Moser's scheme).
A technical difficulty in this section is that we need $H^s$ estimates of terms like
$$ \big(G[\eta_{\eps}+ \eta^{a}+ \eta] - G[\eta_{\eps} + \eta^{a}] \Big)\cdot( \varphi_{\eps}+
\varphi^{a}).$$
This yields   because  of the solitary wave (since  $\eta_{\eps}$, $\varphi_{\eps}$
and their derivatives are not in $H^s(\mathbb{R}^2)$) that  we need to study the Dirichlet Neumann operator
in a non $H^s$ framework. The final argument to get the instability is the one of  \cite{Grenier}.

\end{itemize}

\section{The linearized water waves equation about the solitary wave 
\label{sectionprelim}
$(\eta_{\varepsilon},\varphi_{\varepsilon})$}
In this section, we  shall  study the structure of the linearized water waves equations about
the solitary wave.

In view of Lemma \ref{DN'}, in order to express the linear
equation arising  from the linearization of \eqref{eq11}, \eqref{eq22}
about the solitary wave $Q_{\eps}= (\eta_{\eps}, \varphi_{\eps})$,
it is convenient to use the notation 
$$
Z_{\varepsilon}\equiv Z[\eta_{\varepsilon},\varphi_{\varepsilon}],\quad
\nabla_{X}\varphi_{\varepsilon}-Z_{\varepsilon}\nabla_{X}\eta_{\varepsilon}\equiv
\left( \begin{array}{c} v_\varepsilon 
\\
0
\end{array}\right).
$$
Thus
$$
v_\varepsilon=\partial_{x}\varphi_{\varepsilon}-
\frac{G[\eta_{\varepsilon}]\varphi_{\varepsilon}+\partial_{x}\eta_{\varepsilon}
\partial_{x}\varphi_{\varepsilon}}{1+|\partial_{x}\eta_{\varepsilon}|^2}
\partial_{x}\eta_{\varepsilon}\,.
$$
We also introduce  the operator (of Laplace-Beltrami type) $P_{\varepsilon}$ defined by 
$$
P_{\varepsilon}\eta\equiv
\beta\nabla_{X}\cdot
\Big[
\frac{\nabla_{X}\eta}
{(1+|\partial_{x}\eta_{\varepsilon}|^2)^{\frac{1}{2}}}
-
\frac{(\nabla_{X}\eta_{\varepsilon}\cdot\nabla_{X}\eta)\nabla_{X}\eta_{\varepsilon}}
{(1+|\partial_{x}\eta_{\varepsilon}|^2)^{\frac{3}{2}}}
\Big].
$$
Since the solitary wave is one-dimensional, we observe that
$$
(\nabla_{X}\eta_{\varepsilon}\cdot\nabla_{X}\eta)\nabla_{X}\eta_{\varepsilon}
=\left( \begin{array}{c} 
(\partial_{x}\eta_{\varepsilon})^2\partial_{x}\eta
\\
0
\end{array}\right),
$$
therefore, 
the linearization of (\ref{eq11})-(\ref{eq22}) about
$(\eta_{\varepsilon},\varphi_{\varepsilon})$
reads
\begin{eqnarray*}
\partial_{t} \eta  & = & \partial_{x}\eta +G[\eta_{\varepsilon}]\varphi-G[\eta_{\varepsilon}](Z_{\varepsilon}\eta)-
\partial_{x}(v_\varepsilon\eta),
\\
\partial_{t } \varphi  & = & \partial_{x } \varphi +P_{\varepsilon}\eta
-v_{\varepsilon}\partial_{x}\varphi+Z_{\varepsilon}G[\eta_{\varepsilon}]\varphi
-Z_{\varepsilon}G[\eta_{\varepsilon}](Z_{\varepsilon}\eta)
-(\alpha+Z_{\varepsilon}\partial_{x}v_{\varepsilon})\eta\,.
\end{eqnarray*}
This linear equation has a canonical Hamiltonian structure and can be written as
\beq
\label{lin0}
\partial_t\left( \begin{array}{c} \eta 
\\
\varphi
\end{array}\right)
\\
=
J\Lambda\left( \begin{array}{c} \eta 
\\
\varphi
\end{array}\right),
\eeq
where
$$
J=\left( \begin{array}{cc} 0 & 1 
\\
-1 & 0
\end{array}\right)
$$
is skew-symmetric 
and
$$
\Lambda=\left( \begin{array}{cc} 
-P_{\varepsilon}+\alpha+Z_{\varepsilon}G[\eta_{\varepsilon}] \big(Z_{\varepsilon}\cdot\big)+Z_{\varepsilon}\partial_{x}v_\varepsilon
& (v_\varepsilon-1)\partial_{x}-Z_{\varepsilon}G[\eta_{\varepsilon}]
\\
-\partial_{x}((v_\varepsilon-1)\cdot)-G[\eta_{\varepsilon}]\big(Z_{\varepsilon}\cdot\big) & G[\eta_{\varepsilon}]
\end{array}\right)
$$
is a symmetric operator.
As noticed  by Lannes in \cite{L1}, we get a more tractable  expression of the linearized equation 
if  we introduce the change of unknowns
\beq
\label{changev}
V_1=\eta,\quad V_2=\varphi-Z_{\varepsilon}\eta.
\eeq
Indeed,  if   $(\eta,\varphi)$ solves the system  \eqref{lin0}, 
then $(V_1,V_2)$ solves the system 
\begin{eqnarray*}
\partial_{t}V_1 & = &
G[\eta_{\varepsilon}]V_2-\partial_{x}((v_\varepsilon-1)V_1),
\\
\partial_{t}V_2 & = &
P_{\varepsilon}V_1-(v_\varepsilon-1)\partial_{x}V_2-(\alpha+(v_\varepsilon-1)\partial_{x}Z_{\varepsilon})V_1.
\end{eqnarray*}
As noticed in \cite{Alazard-Metivier}, this change of unknown is linked with the  "good unknown"
of Alinhac \cite{Alinhac}.
The last system can be written in the canonical Hamiltonian form 
\beq
\label{linsource}
\partial_t\left( \begin{array}{c} V_1 
\\
V_2
\end{array}\right)
\\
=
JL\left( \begin{array}{c} V_1 
\\
V_2
\end{array}\right)
\,,
\eeq
where  the symmetric operator  $L$  is  defined as follows
$$
L=\left( \begin{array}{cc} 
-P_{\varepsilon}+\alpha+(v_\varepsilon-1)\partial_{x}Z_{\varepsilon} & (v_\varepsilon-1)\partial_{x} 
\\
-\partial_{x}((v_\varepsilon-1)\cdot) & G[\eta_{\varepsilon}]
\end{array}\right).
$$
Since $\eta_{\eps}$ does not depend on $y$, the study of \eqref{linsource} can
be simplified by using the Fourier transform in $y$.
Indeed, if for some $k\in\R$,
\begin{equation}\label{r1}
V_1(x,y)=e^{iky}W_1(x),\quad V_2(x,y)=e^{iky}W_2(x)
\end{equation}
then
$$
L\left( \begin{array}{c} V_1 
\\
V_2
\end{array}\right)
=
e^{iky}L(k)\left( \begin{array}{c} W_1 
\\
W_2
\end{array}\right),
$$
where the symmetric operator  $L(k)$ is defined as
$$
L(k)
=
\left( \begin{array}{cc} 
-P_{\varepsilon,k}+\alpha+(v_\varepsilon-1)\partial_{x}Z_{\varepsilon} & (v_\varepsilon-1)\partial_{x} 
\\
-\partial_{x}((v_\varepsilon-1)\cdot) & G_{\varepsilon,k}
\end{array}\right)
$$
with
$$
P_{\varepsilon,k}u= \beta\Big(\partial_{x}\big(
(1+(\partial_x\eta_{\varepsilon})^2)^{-\frac{3}{2}}\partial_{x}u\big)
-k^2(1+(\partial_x\eta_{\varepsilon})^2)^{-\frac{1}{2}}
u\Big)
$$
and $G_{\varepsilon,k}$ is such that
\beq
\label{Gepsk}
G[\eta_{\varepsilon}](f(x)\exp(iky))=\exp(iky)G_{\varepsilon,k}(f(x))\,.
\eeq
The fact that $G_{\varepsilon,k}=G_{\varepsilon,k}(x,D_x,k)$ is independent of $y$
follows directly from the definition of the Dirichlet-Neumann operator.
Note that  $-P_{\eps, k}+ \alpha$ is a positive operator: there exists $c>0$ independent of $k\in\R$ such that for every $u\in H^1(\R)$,
\beq
\label{posP}
\int_{\R}((-P_{\varepsilon,k}u+\alpha u)\bar{u}\geq c\big(|u|_{H^1(\R)}^2+(k^2+1)|u|_{L^2(\R)}^2\big).
\eeq
Note that, for $k\in\R$, we can also  define the
operator $\Lambda(k)$ associated to $\Lambda$  acting on functions depending on $x$ only as
$$
\Lambda \Big(e^{iky}\left( \begin{array}{c} V_1(x) 
\\
V_2(x)
\end{array}\right) \Big)
=
e^{iky}\Lambda(k)\left( \begin{array}{c} V_1(x) 
\\
V_2(x)
\end{array}
\right).
$$
We find   for  $\Lambda(k)$  the expression
$$
\Lambda(k)
=
\left( \begin{array}{cc} 
-P_{\varepsilon,k}+\alpha+Z_{\varepsilon}G_{\varepsilon,k}\big(Z_{\varepsilon}\cdot\big)+Z_{\varepsilon}
\partial_{x}v_{\varepsilon} & (v_\varepsilon-1)\partial_{x}- Z_{\varepsilon}G_{\varepsilon,k}
\\
-\partial_{x}((v_\varepsilon-1)\cdot) -G_{\varepsilon,k}\big(Z_{\varepsilon}\cdot\big)
& G_{\varepsilon,k}
\end{array}\right).
$$
Due to the change of unknown \eqref{changev}, 
we have the relation
$$ JL(k) = P^{-1} J\Lambda(k) P$$
where
$$
P=\left( \begin{array}{cc} 1 & 0 
\\
Z_{\varepsilon} & 1
\end{array}\right),\quad
Q=P^{-1}=\left( \begin{array}{cc} 1 & 0 
\\
-Z_{\varepsilon} & 1
\end{array}\right)\,.
$$
Since $P$ and $P^{-1}$ are just smooth matrices, 
$JL(k)$ and $J\Lambda(k)$ have thus  the same spectrum.
Moreover, we also have that 
$L(k)$
and $\Lambda(k)$ are linked through
\begin{equation}\label{vrazka}
L(k)=P^{\star}\Lambda(k)P,\quad \Lambda(k)=Q^{\star}L(k)Q
\end{equation}
therefore, it is also possible to relate  spectral properties
of $L(k)$ and $\Lambda(k)$ via the analysis of the corresponding quadratic forms.

In the next section,   we shall establish some useful properties on  the Dirichlet-Neumann operator
$G_{\eps, k}$ and on the spectrum of $L(k)$.


\section{Study  of the Dirichlet to  Neumann operator $G_{\eps, k }$ }
\label{sectionGepsk}
In this section, we shall study the basic properties of $G_{\eps,k}$. An elementary but very useful property
that we establish is   the monotonicity
property of  $G_{\eps, k}$ with respect to $k$. The proofs of most of the other properties are  inspired 
by similar considerations in \cite{Alvarez-Lannes,L1}, the point here being to track  the 
dependence with respect to $k$ in  the estimates.

Note that, because of the definition 
\eqref{Gepsk},  we  need to work with   complex valued functions.  For complex valued 
functions,  we shall denote the complex  $L^2$ scalar product as
\begin{equation}\label{reaaal}
(u, v) =  \int u(x) \overline{v(x)}\, dx.
\end{equation}
We shall use slightly abusively  the same notation for the  scalar product of $L^2 \times L^2$,  thus 
for $U=(U_1,U_2)$, $V=(V_1,V_2)$ in $L^2\times L^2$, we define 
$$ 
(U,V)= (U_{1}, V_{1}) + (U_{2}, V_{2}).
$$
Note  that we have 
\beq\label{J0}
{\rm Re}\,(JU, U)= 0, \quad \forall\, U\in L^2\times L^2.
\eeq
We shall first prove the following statement.
\begin{prop}
\label{DirNeum}
\begin{enumerate}
For every $\eps>0$, 
we have the following properties: 
\item[i)] $G_{\eps,k}$ is symmetric :
$$ 
\big( G_{\eps, k }u, v \big) = \big(u, G_{\eps, k}v\big),\quad 
\forall\, u, \, v \in H^{\frac{1}{2} }(\mathbb{R})\,.
$$
\item[ii)] If $|k_{1}|> |k_{2}|$, then $G_{\eps, k_{1}} -G_{\eps, k_{2}}$ is a positive definite operator :
$$ 
\big(G_{\eps, k_{1}}u, u \big) > \big( G_{\eps, k_{2}}u, u \big), 
\quad \forall\, u \in H^{\frac{1}{2}}(\mathbb{R}), \, u \neq 0.
$$
\item[iii)] There exist $c>0$ and $C>0$ such that for every $k \in \mathbb{R}$, we  have 
\begin{eqnarray}
\label{DNC} & & \big|  \big(G_{ \eps, k} u , v \big) \big|  \leq   C \Big| 
\frac{ \sqrt{D_{x}^2 + k^2 }} {\big(1 +\sqrt{D_{x}^2 + k^2 } \big)^{\frac{1}{2} }  }u \Big|_{L^2} \,
\Big|\frac{ \sqrt{D_{x}^2 + k^2 }} {\big(1 +\sqrt{D_{x}^2 + k^2 } \big)^{\frac{1}{2} }  }v \Big|_{L^2}
, \quad \forall\, u, \, v \in H^{ \frac{1}{2}}(\mathbb{R}),
\\\label{DNm} & & \big(G_{ \eps, k} u , u \big) \geq c  \Big| 
\frac{ \sqrt{ D_{x}^2 + k^2}}{
\big(  1 +  \sqrt{D_{x}^2 + k^2 }
\big)^{ \frac{1}{2} }   }  u \Big|_{L^2}^2\,, \quad \forall\,
u \in H^{ \frac{1}{2}} (\mathbb{R})\, .
\end{eqnarray}  \\  
\end{enumerate}
\end{prop}
Note that the estimates of the above proposition  have a sharp dependence in $k$.
In particular,  \eqref{DNC}, \eqref{DNm} are uniform in $k$, $k \in \mathbb{R}$.
We do not care on the dependence of these estimates in $\eps.$
\begin{proof}[Proof of Proposition~\ref{DirNeum}]
We first prove $i).$ We recall that 
by definition, we have 
$$
G_{\varepsilon,k}(u)(x)=
\sqrt{1+(\partial_x\eta_{\varepsilon}(x))^2}(\nabla_{x,z} \phi_{k}^u(x,\eta_{\varepsilon}(x) )\cdot n(x)),
$$
where $\phi_k^u(x,z)$ is the solution of the  elliptic problem
\beq
\label{H1}
(\partial_x^2-k^2+\partial_z^2)f =0,\quad -1<z<\eta_{\varepsilon}(x),\,\, x\in\R \quad  \partial_{z} f(x,-1)=0
\eeq
such that 
\beq
\label{Hb}
\quad  f (x,\eta_{\varepsilon}(x))=u(x),\quad x\in\R\,.
\eeq
The identity i) will be a simple consequence  of the Green formula.
Indeed,  let us set $D=\{(x,z):-1<z<\eta_{\varepsilon}(x)\}$ and 
$\Sigma=\{z=\eta_{\varepsilon}(x)\}\cup \{z=-1\}$
and consider  $\phi_{k}^u$, $\phi_{k}^v$  the solutions of \eqref{H1}, \eqref{Hb}
with respective traces $u$ and $v$ on the upper boundary.
Then, by definition,  we have
$$  (G_{\eps, k} u, v ) =  \int_{\mathbb{R}} G_{\eps, k } u(x) \, \overline{v(x)} \, dx= 
\int_{\Sigma} \frac{\partial \phi_{k}^u }{ \partial n }(\tau) \overline{\phi_{k}^v(\tau)} 
d \Sigma (\tau),
$$
where $d \Sigma (\tau) $ is the volume element of the surface $z= \eta_{\eps}(x)$.
Consequently, since $\partial_{z} \phi_{k}^u(x,-1)=\partial_{z} \phi_{k}^v(x,-1)=0$,
the Green formula and the equations satisfied by $\phi_{k}^u$, $\phi_{k}^v$  yield
\beq\label{Guv}
(G_{\eps, k} u, v ) = 
\int _{D}\Big( \nabla_{x,z} \phi_{k}^u \cdot \overline{\nabla_{x,z} \phi_{k}^v}+ k^2 \phi_{k}^u 
\overline{\phi_{k}^v} \Big) dxdz  = 
(u, G_{\eps, k} v).
\eeq
This proves~i).
\bigskip

Let us now prove ii).  We first 
observe  that if $u$ is real then $G_{\varepsilon,k}u$ is also real.
Therefore, if $u=u_1+iu_2$ with real valued $u_1$ and $u_2$, we have that
$$
(G_{\varepsilon,k_1}u, u ) =(G_{\varepsilon,k_1}u_1, u_1 )+(G_{\varepsilon,k_1}u_2, u_2 ).
$$ 
Consequently, we can assume that $u$ is real valued for the proof.
Thanks to  \eqref{Guv}, we have
\begin{eqnarray*}
(G_{\varepsilon,k_1}u, u ) 
&=  & \int_{D}\Big |\nabla_{x,z} \phi_{k_{1}}^u |^2 + k_{1}^2 |\phi_{k_{1}}^u|^2 \Big) dxdz \\
& = & \int_{D}\Big(  |\nabla_{x,z} \phi_{k_{1}}^u |^2 + k_{2}^2  |\phi_{k_{1}}^u|^2 \Big) dx dz
+ (k_{1}^2- k_{2}^2) \int_{D} |\phi_{k_{1}}^u|^2 \, dx dz \\
& > &  \int_{D}\Big(  |\nabla_{x,z} \phi_{k_{1}}^u |^2 + k_{2}^2  |\phi_{k_{1}}^u|^2 \Big) dx dz
\end{eqnarray*}
since $|k_{1}|>|k_{2}|$ and $u \neq 0$.
Next, since 
$\phi_{k_1}^u$ and $\phi_{k_2}^u$ verify  the same boundary conditions,
we have thanks to the variational characterization of $\phi_{k_{2}}^u$  that
$$
\int_{D}|\nabla_{x,z}\phi_{k_1}^u|^2+k_2^2\int_{D} |\phi_{k_1}^u|^2
\geq \int_{D}|\nabla_{x,z}\phi_{k_2}^u|^2+k_2^2\int_{D}|\phi_{k_2}^u|^2
\,.
$$
Consequently, by using again \eqref{Guv}, we get
$$ (G_{\varepsilon,k_1}u, u )  >  \int_{D}|\nabla_{x,z}\phi_{k_2}^u|^2+k_2^2\int_{D} |\phi_{k_2}^u|^2
= (G_{\eps, k_{2}}u, u ).$$
This proves ii).

\bigskip

We can now prove iii). 
Note that here, since $\eta_{\eps}$ is smooth and fixed, we do not care on the way  the estimates depend on
the regularity of $\eta_{\eps}.$ 

Next, to prove \eqref{DNm}, \eqref{DNC}, it is convenient to rewrite 
the elliptic problem \eqref{H1} in a flat domain.
We can define implicitly a function $\psi_{k}^u$ on the flat domain
$\mathcal{S}= \mathbb{R} \times (-1, 0 )$ by 
$$
\phi^u_{k}(x,z)= \psi_{k}^u
\Big(x,\frac{z-\eta_{\varepsilon}(x)}{1+\eta_{\varepsilon}(x)}\Big), 
\quad x\in\R,\quad -1<z<\eta_{\varepsilon}(x).
$$
Since we have  by the chain rule
\beq\label{grad}
\nabla \phi_{k} ^u(x,z) = M(x,z) \nabla \psi_{k}^u(x, m(x,z)), \eeq
where
$$ m(x,z)=   \frac{z-\eta_{\varepsilon}(x)}{1+\eta_{\varepsilon}(x)}, 
\quad M(x,z)= \left( \begin{array}{cc}  1 &  \partial_{x} m \\  0 & \partial_{z} m
\end{array} \right),$$
we also get by using that  the divergence is the  $L^2$ adjoint of the gradient that
for    a vector field  $u(x,z)$ on $D$ such that
$$  u(x,z) = v(\Phi(x,z)), \quad \Phi(x,z)= (x, m(x,z))$$
we have
$$ \nabla \cdot u (x,z) =  \mbox{det}\,( D\Phi(x,z) )\,
\nabla_{ Y} \cdot \Big(\mbox{det}( D \Phi^{-1}(Y) ) M\big( \Phi^{-1}(Y)\big)^* v(Y) \Big)_{/Y= \Phi(x,z)}.$$
This allows to get that
$$ \Delta \phi_{k}^u= \nabla \cdot \nabla \phi_{k}^u =  \Delta_{g} \psi_{k}^u$$
where the operator $\Delta_{g}$  defined as 
\begin{equation}\label{beltrami}
\Delta_{g}(\psi)=(\det(g))^{-1/2}{\rm div}\Big((\det(g))^{1/2}g^{-1}
\nabla \psi
\Big)
\end{equation}
is the Laplace Beltrami operator  associated to the metric $g$ which is defined
through its inverse $g^{-1}$  by
$$
g^{-1}(x,z)\equiv
\left( \begin{array}{cc} 1 &  -\frac{\partial_x\eta_{\eps}(x)(z+1)}{1+\eta_{\eps}(x)}
\\
-\frac{\partial_x\eta_{\eps}(x)(z+1)}{1+\eta_{\eps}(x)} &
\frac{1+(z+1)^2(\partial_x\eta_{\eps}(x))^2}{(1+\eta_{\eps}(x))^2}
\end{array}\right) = M \big(\Phi^{-1}(x,z)\big)^* M(\Phi^{-1}(x,z)\big), \quad (x, z) \in \mathcal{S}.
$$
Consequently, 
if $\phi_{k}^u$ solves
$$
(\partial_x^2-k^2+\partial_z^2)\phi=0,\quad x\in\R,\quad -1<z<\eta_{\varepsilon}(x),
$$
with boundary conditions
$\phi(x,\eta_{\varepsilon}(x))=u(x)$, $\partial_{z}\phi(x,-1)=0$ then $\psi_{k}^u$, solves
\begin{equation}
\label{metric}
(-\Delta_g+k^2)\psi=0, \quad (x, z)\in \mathcal{S}\qquad
\partial_{z}\psi(x,-1)=0,\quad \psi(x,0)=u(x)\,,
\end{equation}
where $\mathcal{S}$ is the strip $\mathcal{S}= \mathbb{R} \times (-1, 0).$
By using \eqref{grad}, the map $G_{\varepsilon,k}$ can  be expressed in terms of  $\psi_{k}^u$    as 
\begin{equation}\label{pechka}
G_{\varepsilon,k}(u)(x)=-\partial_x\eta_{\varepsilon}(x)\partial_x\psi_{k}^u(x,0)+
\frac{1+(\partial_x\eta_{\varepsilon}(x))^2}{1+\eta_{\varepsilon}(x)}\partial_z\psi_{k}^u(x,0)\,.
\end{equation}
Therefore, using the Green formula together with (\ref{beltrami}) and 
the equation solved by $\psi^u_{k}  $, we obtain that for $u,v\in H^{1\over 2 }(\R)$,
\beq\label{DNnew}
(G_{\varepsilon,k}(u),v)=
\int _{\mathcal{S}}  \Big( g^{-1} \nabla_{x,z} \psi^u_{k}  
\cdot \overline{\nabla_{x,z} {\bf v}} + k^2 \psi_{k}^u\overline{\bf{v}} \Big)
(\mbox{det }g)^{ \frac{1}{2}} dx dz,
\eeq 
where ${\bf v}$  can be any $H^1$ function on $\mathcal{S}$  such that ${\bf v}(x,0)=v(x)$.
\bigskip

To estimate the solution $\psi_{k}^u$ of  \eqref{metric}, we shall use the decomposition
\beq
\label{dec}
\psi_{k}^u= u^H_{k} +  u_{k}^r,
\eeq
where $u^H_{k}$ is the  solution of
\beq
\label{uH}\big( - \Delta_{x,z} + k^2 \big) u^H_{k} = 0, \quad (x,z) \in \mathcal{S}, \quad
\partial_{z} u^H_{k}(x,-1)=0, \quad u^H_{k}(x,0)= u(x),
\eeq
$\mathcal{S}$ being again 
the strip $\mathbb{R}\times (-1, 0),$ and thus the remainder  $u_{k}^r$ is the  solution
of the  elliptic problem with homogeneous boundary condition
\beq
\label{vH}\big( - \Delta_{g} + k^2 \big) u_{k}^r =\big( \Delta_{g} - k^2\big)u^H_{k}, \quad (x,z) \in \mathcal{S}, \quad
\partial_{z} u_{k}^r(x,-1)=0, \quad u_{k}^r(x,0)= 0.
\eeq

By solving an ODE,
one can write down explicitly the expression of the Fourier transform in $x$,   $\hat{u}_{k}^H$ 
of $u_{k}^H$. We have: 
\beq\label{uHhat}
\hat{u}^H_{k}(\xi, z)= \frac{ \cosh \big( \sqrt{\xi^2 + k^2} (z+1) \big) }{ \cosh \sqrt{\xi^2+ k^2} } 
\hat{u}(\xi), \quad \xi \in \mathbb{R}, \, z \in (-1, 0).
\eeq 
The estimate of  $\psi_{k}^u $  will be a consequence of the two following lemmas.
\begin{lem}\label{uHlem}
There exists $C>0$ such that for every $k\in\R$, every $s\in\R$, every $u\in H^{\infty}(\R)$,
\begin{eqnarray}
& &
\label{uH1}  \| \Lambda^s u^H_{k}  \|_{L^2 (\mathcal{S})} \leq C  | \Lambda^s 
\Lambda_{k}^{-\frac{1 }{ 2}} u |_{L^2},\\
\label{uH2} & & \|\Lambda^s\partial_{z} u^H_{k} \|_{L^2(\mathcal{S})} \leq C
| \Lambda^s{ \sqrt{D_{x}^2 + k^2}    \Lambda_{k}^{- \frac{ 1 }{ 2  } } }  u |_{L^2}\,,
\end{eqnarray} 
where $\Lambda$ and $\Lambda_{k}$ are the Fourier multipliers
$$ \Lambda = (1+  D_{x}^2)^{\frac{1}{2} }, \quad \Lambda_{k}= ( 1 + k^2 + D_{x}^2)^{\frac{1}{2}}. $$
\end{lem}
\begin{proof}[Proof of Lemma \ref{uHlem}]
First, we observe that it suffices to consider the case $s=0$. Next, we note that there exists $C>0$ such that
for every $\omega\geq 0$, we have the inequalities
\begin{equation}\label{uH1bis} 
\int_{-1}^0 \frac{ \cosh^2(\omega (z+1) )  }{ \cosh^2(\omega )  }  \, dz\leq \frac{C}{1+\omega},
\quad
\int_{-1}^0 \frac{ \sinh^2(\omega (z+1) )  }{ \cosh^2(\omega )  }  \, dz\leq 
\frac{C}{1+\omega}\,.
\end{equation}
Indeed, inequalities (\ref{uH1bis}) can be easily obtained for instance by performing the  change of variable
$z'=(1+z)\omega$. Now (\ref{uH1}) and (\ref{uH2}) follow from (\ref{uH1bis}) with $\omega=\sqrt{ \xi^2+ k^2 }$
via an application of the Parseval identity. This completes the proof of Lemma~\ref{uHlem}.
\end{proof}
Let us now give the needed estimates for $u^r_{k}$.
\begin{lem}\label{vHlem}
Let us fix an integer $s\geq 0$. 
There exists $C>0$ such that for every $k\in\R$,  every $u\in H^{\infty}(\R)$,
the solution of \eqref{vH} satisfies the estimate
\beq\label{estvH}
\|\Lambda^s \nabla_{x,z} u^r_{k}\|_{L^2(\mathcal{S})}^2 + k^2\| \Lambda^s u^r_{k} \|_{L^2(\mathcal{S})}^2  
\leq C \big| \Lambda^s
\frac{ \sqrt{D_{x}^2 + k^2 } }{ \big(1 +\sqrt{D_{x}^2 + k^2 } \big)^{\frac{1}{2} }  }   u |_{L^2}^2.
\eeq
\end{lem}
\begin{rem}
By a standard density argument, the statement of Lemma~\ref{uHlem} and  Lemma~\ref{vHlem} may be extended to
functional classes such that the right hand-side of the corresponding inequalities makes sense.
\end{rem}
\begin{proof}[Proof of Lemma~\ref{vHlem}]
We have the following  estimates
\begin{equation}\label{pechka2}
\|\Lambda^{s} \nabla_{x,z} u^r_{k} \|_{L^2(\mathcal{S})}^2 + 
k^2\| \Lambda^s u^r_{k}\|_{L^2(\mathcal{S})}^2  \leq C \big(  \|\Lambda^s \nabla_{x,z} u^H_{k}\|_{L^2}^2  + 
k^2 \|\Lambda^s u^H_{k} \|_{L^2 }^2  \big).
\end{equation}
Indeed,   (\ref{pechka2}) for $s=0$, is just   the standard energy estimate:  it suffices to take the $L^2$ scalar product of equation 
\eqref{vH} with $ u^r_{k}$ and to perform integration by parts by using that
$u^r$ satisfies homogeneous boundary conditions. For $s\geq 1$ one may apply the standard argument for propagation 
of higher regularity in linear elliptic equations.
Using (\ref{pechka2}) and Lemma~\ref{uHlem} yield (\ref{estvH}).
This completes the proof of Lemma~\ref{vHlem}. 
\end{proof}

We are  now  in position to get \eqref{DNC}. Thanks to \eqref{DNnew}, we have
\beq\label{Guu}
\big( G_{\eps, k } u, v \big) =  
\int _{\mathcal{S}}  \Big( g^{-1} \nabla_{x,z} \psi^u_{k}  \cdot \overline{\nabla_{x,z} \psi^v_{k}} + 
k^2 \psi_{k}^u \, \overline{ \psi_{k}^v }  \Big)(\mbox{det }g)^{ \frac{1}{2}} dx dz.
\eeq  
Consequently, \eqref{DNC} follows by using the Cauchy-Schwarz inequality and   \eqref{dec}, \eqref{uH1} (with $s=0, \, 1$), \eqref{uH2} (with $s=0$)
and \eqref{estvH} (with $s=0$).

As in \cite{Alvarez-Lannes}, \eqref{DNm} will be a consequence of the trace formula.
Let us choose $\chi(z)$ a smooth compactly supported  cut-off function such that $\chi(0)=1$ and
$\chi$ is supported in $(-1, 1)$.  We shall consider $\psi(x,z)= \chi(z) \psi_{k}^u(x,z)$. 
Note that since $\chi$ does not depend on $x$, we have $ \hat{\psi}(\xi,z)= \chi(z) \hat{\psi}_{k}^u(x,z).$
We can write
\begin{eqnarray*}
| \hat{u}(\xi) |^2 = 
|\hat{\psi}(\xi, 0 ) |^2 
& \leq &
2 \int_{-1}^0  |\hat{\psi}(\xi,z)| \, |\partial_{z} \hat{\psi}(\xi,z) |\, dz \\
& \leq & C  \int_{-1}^0  \Big( |\hat{\psi}_{k}^u(\xi,z)|^2 + | \partial_{z} \hat{\psi}_{k}^u(\xi,z) |
|\hat{\psi}_{k}^u(\xi,z)|\Big) dz.
\end{eqnarray*}
This yields
$$ 
\frac{\xi^2 + k^2 }{ 1 + \sqrt{\xi^2 + k^2} } | \hat{u}(\xi) |^2 
\leq C 
\int_{-1}^0  \Big( (\xi^2 + k^2 )   |\hat{\psi}_{k}^u(\xi,z)|^2 +  
|\partial_{z} \hat{\psi}_{k}^u(\xi,z)|^2 \Big) dz.
$$
Consequently, we can integrate in $\xi$, use the Parseval identity and \eqref{Guu} to get
$$
\Big| \frac{ \sqrt{ D_{x}^2 + k^2} }{
\big( 1 +  \sqrt{D_{x}^2 + k^2 }  \big)^{ \frac{1}{2} }  }  u \Big|_{L^2}^2
\leq C \int _{\mathcal{S}}  \Big( g^{-1} \nabla_{x,z} \psi_{k}^u  \cdot \overline{\nabla_{x,z} \psi_{k}^u} 
+ k^2 |\psi_{k}^u|^2 \Big)
(\mbox{det }g)^{ \frac{1}{2}} dx dz = \big( G_{\eps, k } u, u \big).$$
This ends the proof of \eqref{DNm}. 
The proof of  Proposition~\ref{DirNeum} is completed.
\end{proof}
We next establish some additional qualitative properties of $G_{\eps, k}$.   
\begin{prop}
\label{DirNeum2}
The operator $G_{\eps, k}$ verifies: 
\begin{itemize}
\item[i)]
For every $k\in \mathbb{R}$, $ G_{\eps, k} \in \mathcal{B}(H^s, H^{s-1})$ for every $s \in \mathbb{R}$.
\item[ii)] $G_{\eps,k}$ depends continuously on $k$ for $k \in \mathbb{R}$ and analytically on 
$k$ for $k \in \mathbb{R}\backslash\{0\}$ in the operator norm of $\mathcal{B}(H^1, L^2)$.
\item[iii)]
For every $k$,  we have the decomposition
\beq
\label{princip}
G_{\eps, k}= |D_{x}| + G_{\eps, k}^0(x, D_{x}) 
\eeq
where for every $k$,  $G_{\eps, k}^0$ is a bounded operator on $H^s$,  $ G_{\eps, k}^0
\in \mathcal{B}(H^s, H^s) $  for every $s$.
\end{itemize}  
\end{prop}
Note that in this lemma we state mostly qualitative properties  of $G_{\eps,k}$ which hold  locally  in $k$.
An immediate corollary of \eqref{princip} is that $G_{\eps,k}$  verifies an elliptic regularity criterion.
\begin{cor}\label{elliptic}
If $u\in H^s $ is such that $G_{\eps, k} u \in H^s$, then $u \in H^{s+1}$.
\end{cor}

\begin{proof}[Proof of Proposition~\ref{DirNeum2}]
We first prove i).  By using \eqref{DNnew}, we get
\begin{multline}\label{dirS}
\big( \Lambda^{s-{ \frac{1}{2} } } G_{\eps, k} (u), v\big)
= \big( G_{\eps,k} (u), \Lambda^{s- {\frac{1}{2}} } v \big) 
\\
=  \int _{\mathcal{S}}  \Big( g^{-1}\nabla_{x,z} \psi^u_{k}  \cdot 
\overline{\nabla_{x,z} (\Lambda^{s- {\frac{1}{2} }} v_{k}^H)} + 
k^2\psi_{k}^u\overline{\Lambda^{s - {\frac{1}{2} }} v_{k}^H} \Big)(\mbox{det }g)^{\frac{1}{2}} dx dz,
\end{multline}
where $v_{k}^H$ is defined by
$$
v_{k}^H(x,z)=
\frac{ \cosh \big( \sqrt{D_x^2 + k^2} (z+1) \big) }{ \cosh \sqrt{D_x^2+ k^2} }(v).
$$ 
Next, we write
\begin{multline*}
\big( \Lambda^{s-{ 1\over 2 } } G_{\eps, k} (u), v\big)=
\int _{\mathcal{S}}  \Big( g^{-1}\nabla_{x,z}\Lambda^s \psi^u_{k}  \cdot 
\overline{\nabla_{x,z} (\Lambda^{- {1\over 2 }} v_{k}^H)} + 
k^2\Lambda^s\psi_{k}^u\overline{\Lambda^{ - {1 \over 2 }} v_{k}^H} \Big)(\mbox{det }g)^{ 1 \over 2} dx dz
+
\\
\int _{\mathcal{S}}  \Big([\Lambda^s, (\mbox{det }g)^{ 1 \over 2}g^{-1}]
\nabla_{x,z} \psi^u_{k}  \cdot 
\overline{\nabla_{x,z} (\Lambda^{- {1\over 2 }} v_{k}^H)} + 
k^2[\Lambda^s,(\mbox{det }g)^{ 1 \over 2}]\psi_{k}^u\overline{\Lambda^{ - {1 \over 2 }} v_{k}^H} \Big) dx dz.
\end{multline*}
For $s\geq 0$, an integer, we can apply the Cauchy-Schwarz inequality, Lemma~\ref{uHlem} and 
Lemma~\ref{vHlem} to get the bound
$$  \big|\big( \Lambda^{s-{ 1\over 2 } } G_{\eps, k} u, v\big) \big|
\leq C(k)    |\Lambda^{ s + { 1 \over 2 } } u |_{L^2} \, |v|_{L^2}$$
and hence
$$ | G_{\eps, k }u |_{H^{ s- { 1 \over 2 } } }\leq C(k) |u|_{H^{ s + {1 \over 2 } } }$$
(note that for this estimate we do not need to express precisely the dependence of  $C(k)$ in $k$).
Therefore $ G_{\eps, k }$ is continuous from $H^{s+\frac{1}{2}}$ to  $H^{s-\frac{1}{2}}$ 
for $s\geq 0$ an integer. By interpolation $ G_{\eps, k }$ is 
continuous from $H^{s}$ to  $H^{s-1}$ for $s\geq 1/2$.
Next, since $ G_{\eps, k }$ is symmetric, by duality $ G_{\eps, k }$ is 
continuous from $H^{1-s}$ to  $H^{-s}$ for $s\geq 1/2$. Thus  $ G_{\eps, k }$ is 
continuous from $H^{s}$ to  $H^{s-1}$ for every $s\in\R$.
\bigskip

Let us turn to the proof of ii). We shall  first establish the continuity at zero.
To study $\big(G_{\eps,k} -G_{\eps, 0}\big) u$, we consider again  $\psi_{k}^u $ the solution of \eqref{metric}
and  we shall use the expression
\beq\label{defdir}
G_{\varepsilon,k}(u)(x)=-\partial_x\eta_{\varepsilon}(x)\partial_x\psi_{k}^u(x,0)+
\frac{1+(\partial_x\eta_{\varepsilon}(x))^2}{1+\eta_{\varepsilon}(x)}\partial_z\psi_{k}^u(x,0)\,.      
\eeq
We first notice that $\psi_{k}^u  - \psi_{0}^u$ solves the elliptic equation
\beq\label{diff1} 
\Delta_{g}( \psi_{k}^u- \psi_{0}^u  )  = k^2 \psi_{k}^u\eeq
with a homogeneous Dirichlet boundary condition on the upper boundary
$$
\big( \psi_{k}^u - \psi_{0}^u\big)(x,0) = 0.
$$
Note that this implies  by the Poincare inequality that
$$ 
\|\psi_{k}^u  - \psi_{0}^u\|_{L^2 (\mathcal{S})} \leq 
C \| \nabla ( \psi_{k }^u - \psi_{0}^u)\|_{L^2 (\mathcal{S)}}.
$$
Consequently, from the elliptic regularity, we get  from \eqref{diff1} that
$$ 
\|\psi_{k}^u - \psi_{0}^u \|_{H^2(\mathcal{S})}
\leq C k^2\|\psi_{k}^u\|_{L^2}.
$$
Therefore, the trace theorem and the definition \eqref{defdir} yield
$$ 
|\big( G_{\eps, k} - G_{\eps, 0}\big) u |_{L^2}\leq 
C\|\psi_{k}^u - \psi_{0}^u \|_{H^2(\mathcal{S})}
\leq
C k^2 \|\psi_k^u \|_{L^2 (\mathcal{S})}.
$$
By using the  Poincar\'e inequality and Lemma~\ref{uHlem} and 
Lemma~\ref{vHlem}, we get 
$$
\|\psi_{k}^u\|_{L^2(\mathcal{S})} \leq C  \|\nabla \psi_k^{u} \|_{L^2(\mathcal{S})}
\leq C |\Lambda_k^{-1}(D_x^2+k^2)^{\frac{1}{2}}u|_{L^2}
.
$$
Therefore, we obtain that for $|k|\leq 1$,
$$  
\|\psi_{k}^u\|_{L^2} \leq C |u|_{H^{ 1 \over 2}}.
$$
Consequently, we  get that there exists $C>0$ such that for every $|k|\leq 1$,
$$ 
|\big( G_{\eps, k} - G_{\eps, 0}\big) u |_{L^2}\leq  C k^2  |u|_{H^{ 1 \over 2}}$$
which proves the  continuity of $G_{\eps,k}$ at zero as an operator in $\mathcal{B}(H^{1\over 2}, L^2)$
which is even better than the claimed property.

To prove the analyticity, it suffices to  use again the decomposition
\eqref{dec} for $\psi_{k}^u$. From the explicit expression, of $u^H_{k}$, we get
that it depends analytically on $k$ for $k \neq 0.$
Then $u_{k}^r$ also depends analytically on $k$ since $u_{k}^r$ can be expressed as 
$$ u^k_{r} = R_{g}(k^2) F(k)\cdot( u^H_{k})$$
where $R_{g}(\lambda)=( \Delta_{g}- \lambda )^{-1}$ is the resolvent of the Laplace Beltrami (with mixed 
boundary conditions) operator $\Delta_{g}$ and $F(k)$ is a linear operator depending on $u^H_{k}$.
Since  $F$ and $R_{g}$ depend analytically on $k$ for $k \neq 0$, the result follows.
More precisely, from the above considerations and very crude estimates,  it follows immediately that $G_{\eps,k}$ depends
analytically on $k$ in the operator norm  $\mathcal{B}(H^{5\over 2 } , L^2)$. From the Cauchy formula and
the fact that $G_{\eps, k}$ belongs to $\mathcal{B}(H^1, L^2)$ this yields the analyticity
of $G_{\eps,k}$ as an operator in $\mathcal{B}(H^1, L^2)$.
\bigskip 

The proof of iii) which is for example detailed  in  \cite{L1}
where moreover  one tracks the  dependence of the estimates on the regularity of the surface
(see also   \cite{Taylor})  relies on the construction of a parametrix for the elliptic equation \eqref{metric}.
Let us just give the main steps in the argument.  We define the operator  $\Delta_g^{ap}=\Delta_g^{ap}(x,z,D_x,D_z)$ as
$$
\Delta_g^{ap}\equiv-a\Big(\partial_z+\frac{b\partial_x+(1+\eta_{\eps})^{-1}\langle
D_x\rangle}{a}
\Big)\Big(\partial_z+\frac{b\partial_x-(1+\eta_{\eps})^{-1}\langle D_x\rangle}{a}\Big),
$$
where
$$
a=a(x,z)\equiv\frac{1+(z+1)^2(\partial_x\eta_{\eps}(x))^2}{(1+\eta_{\eps}(x))^2},
\quad 
b=b(x,z)\equiv-\frac{\partial_x\eta_{\eps}(x)(z+1)}{1+\eta_{\eps}(x)} \,.
$$
Using some basic pseudo-differential calculus,   one can show  that
$\Delta_{g}^{ap}$ is a good approximation of $-\Delta_{g} + k^2$ in the sense that 
\begin{equation}\label{A12}
\|(-\Delta_g+k^2)-\Delta_g^{ap}\|_{H^{s+1}(\mathcal{S})\rightarrow H^{s}(\mathcal{S})}\leq C_{s,k}\,.
\end{equation}
If we set
$
\eta_{\pm}(x,z,D_x)\equiv a^{-1}(-b\partial_x\pm(1+\eta_{\eps})^{-1}\langle D_x\rangle)
$
then
$$
\Delta_g^{ap}=-a(\partial_{z}-\eta_{-}(x,z,D_x))(\partial_{z}-\eta_{+}(x,z,D_x))\,.
$$
We next  find a parametrix  $ \phi_{ap}=\phi_{ap}(x,z,D_x)$  for  $\Delta_{g}^{ap}$
such that $ (\phi_{ap})_{/z=0}= Id$. This is given by 
$$
\phi_{ap}=\exp\big(-\int_{z}^0  \eta_{+}(x,z',D_x)dz'\big),\quad z\in [-1,0]\,.
$$
This   linear operator  $\phi_{ap}$ enjoys heat-flow type smoothing effects.
Finally  we have on the one hand
$$
G_{\varepsilon,k}(u)(x)=-\partial_x\eta_{\varepsilon}(x)\partial_x\psi_{k}^u(x,0)+
\frac{1+(\partial_x\eta_{\varepsilon}(x))^2}{1+\eta_{\varepsilon}(x)}\partial_z\psi_{k}^u(x,0)\,
$$
and on the other hand  that
$$
\langle
D_x\rangle(u)= -\partial_x\eta_{\varepsilon}(x)\partial_x\phi_{ap}(u)(x,0)+
\frac{1+(\partial_x\eta_{\varepsilon}(x))^2}{1+\eta_{\varepsilon}(x)}\partial_z\phi_{ap}(u)(x,0)\,.
$$
The result thus follows by  proving  the bound
\begin{equation}\label{strike}
|\nabla_{x,z}(\psi_{k}^u-\phi_{ap}(u))(x,0)|_{H^s(\R)}\leq C_{k,s}|u|_{H^s(\R)}\,
\end{equation}
which is a consequence of the properties of $\phi_{ap}$ and elliptic regularity for the problem
solved by $\psi_{k}^u - \phi_{ap}(u)$.
This completes the proof of Proposition~\ref{DirNeum2}.
\end{proof}
\begin{rem}
{\rm Using the arguments of \cite[Chapter~7.12]{Taylor}, one may show that 
$$
G[\eta_{\varepsilon}]-\sqrt{|D_x|^2+|D_y|^2(1+(\partial_x\eta_{\varepsilon}(x))^2)}
$$
is a zero order pseudo-differential operator, independent of $y$, and thus its symbol $q(x,\xi_1,\xi_2)$ 
satisfies
\begin{equation}\label{kiko}
\big|\partial^{\alpha}_{x}\partial^{\beta_1}_{\xi_1}\partial^{\beta_2}_{\xi_2}q(x,\xi_1,\xi_2)
\big|\leq
C_{\alpha,\beta_1,\beta_2}\langle|\xi_1|+|\xi_2|\rangle^{-\beta_1-\beta_2}
\leq
C_{\alpha,\beta_1,\beta_2}\langle\xi_1\rangle^{-\beta_1-\beta_2}\,.
\end{equation}
Thus part  iii)  Proposition~\ref{DirNeum2} is also a consequence of
(\ref{kiko}) with $\beta_2=0$ and the $L^2$ boundedness criterion for zero
order pseudo-differential operators.
}
\end{rem}
In the next proposition, we give useful   commutator estimates.
\begin{prop}[Commutators]\label{com}
We have the following properties: 
\begin{itemize}
\item[i)] For every $s \geq 1$,  $K>0$, there exists $C_{s, K}>0$ such that for every 
$u \in H^{ s + { 1 \over 2 } }$, 
\beq\label{com1}
\big|  \big[ \partial_{x}^s, G_{\eps, k }\big]u \big|_{H^{1 \over 2 } }\leq C_{s,K}  
\Big( \Big| { |D_{x} | \over 1 +  |D_{x}|^{1 \over 2} } u \Big|_{H^s} +  |k| \, |u|_{H^s} \Big), 
\quad \forall k, \, |k| \leq K.
\eeq
\item[ii)] For every  $K>0$ and every smooth function  $f(x) \in \mathcal{S}(\mathbb{R})$, 
there exists $C_{K} >0$ such   that for every 
$u \in H^{ { 1 \over 2 } }$, 
\beq\label{com2}
\Big| { \rm Re } \Big(  f \partial_{x} u, G_{\eps, k } u  \Big) \Big| \leq C_{K}
\Big( \Big| { |D_{x} | \over 1 +  |D_{x}|^{1 \over 2} } u \Big|_{L^2} +  |k| \, |u|_{L^2} \Big), 
\quad \forall k, \, |k| \leq K.
\eeq
\end{itemize}
\end{prop}  
Note that in this proposition we do not pay attention to the  dependence of the estimates
in $k$ for large $k$,  since this will not be needed.    
\begin{proof}[Proof of  Proposition~\ref{com}]
To prove i), we shall estimate
$$  
I= \big(\Lambda^{1 \over 2 } \partial_{x}^s G_{\eps, k} u, v \big) - 
\big( \Lambda^{1 \over 2} G_{\eps, k} \partial_x^s u, v\big).
$$
By using again \eqref{DNnew}, we write
\begin{eqnarray*}
I & = & (-1)^s \int_{\mathcal{S}}  
\Big( g^{-1} \nabla_{x,z} \psi_{k}^u \cdot \overline{ \nabla_{x,z} \Lambda^{1 \over 2} 
\partial_{x}^sv_{k}^H} + k^2 \psi_{k}^u \,\overline{ \Lambda^{1 \over 2 } \partial_{x}^s v_{k}^H} \Big)
(\det g )^{1 \over 2}\, dxdz 
\\
& & \quad 
- \int_{\mathcal{S}} \Big( g^{-1 }  \nabla_{x,z} \psi_{k}^{\partial_{x}^s u} \cdot 
\overline{\nabla_{x,z}  \Lambda^{1 \over 2 } v_{k}^H} + k^2  \psi_{k}^{\partial_{x}^s u }\, 
\overline{\Lambda^{1\over 2 } v_{k}^H} \Big)
(\det g )^{1 \over 2} \, dxdz \\
& = &\int_{\mathcal{S}} \Big( g^{-1} \big( \nabla_{x,z} \partial_{x}^s \psi_{k}^u - 
\nabla_{x,z} \psi_{k}^{\partial_{x}^s u } \big)  \cdot\overline{ \nabla_{x,z}   \Lambda^{1 \over 2 } v_{k}^H} 
+ k^2  \big( \partial_{x}^s \psi_{k}^u - \psi_{k}^{\partial_{x}^s u }\big) 
\overline{\Lambda^{1\over 2} v_{k}^H}\Big)  (\det g )^{1 \over 2} \, dxdz  \\
& & \quad    + \int_{\mathcal{S} }\Big(  \big[ \partial_{x}^s,  
(\det g)^{1 \over 2 } g^{-1} \big] \nabla_{x,z} \psi_{k}^u
\cdot \overline{ \nabla_{x,z} \Lambda^{1 \over 2}v_{k}^H} + 
k^2 \big([\partial_{x}^s, (\det g)^{1 \over 2 }\big] 
\psi_{k}^u \,\overline{ \Lambda^{1\over 2 } v_{k}^H} \Big) \, dx dz \\
& & \equiv J_{1}+ J_{2}. 
\end{eqnarray*}
We  estimate the second integral above by 
\begin{multline*}
|J_{2}| \leq   C_{K} 
\Big( \| \Lambda^{1\over 2}  v_{k}^H\|_{L^2(\mathcal{S}) }
\, \big\|[ \partial_{x}^s,  (\det g)^{1 \over 2 } g^{-1} ] 
\nabla_{x,z} \psi_{k}^u \big\|_{H^1(\mathcal{S}) } 
\\
+  |k| 
\|\Lambda^{1\over 2} v_{k}^H \|_{L^2(\mathcal{S} ) } \, 
\big\|  
[\partial_{x}^s, (\det g)^{1 \over 2 }\big]\psi_{k}^u \big\|_{L^2(\mathcal{S}) } \Big)
\end{multline*}
and hence, by using again Lemma~\ref{uHlem}, Lemma~\ref{vHlem} and standard commutator estimates, we find
$$ 
|J_{2}| \leq C_{s,K}  |v|_{L^2}  
\Big(  \Big| {  |D_{x} | \over 1 +  |D_{x}|^{1 \over 2} } u \Big|_{H^s} +  |k| \, |u|_{H^s} \Big) .  
$$
In a similar way, we estimate $J_{1}$ as follows
$$ 
|J_{1}| \leq C_{s,K}|v|_{L^2} \Big(\big\| 
\nabla_{x,z} \partial_{x}^s \psi_{k}^u - \nabla_{x,z} \psi_{k}^{\partial_{x}^s u }
\big\|_{H^1(\mathcal{S})} +  |k| \big\| \partial_{x}^s \psi_{k}^u - \psi_{k}^{\partial_{x}^s u }
\big\|_{L^2(\mathcal{S})} \Big).
$$
To conclude, we notice that $\psi=   \partial_{x}^s \psi_{k}^u -  \psi_{k}^{\partial_{x}^s u }$
solves the elliptic equation
\beq\label{voda}
- \Delta_{g} \psi + k^2 \psi= [\partial_{x}^s, \Delta_{g}  ]\psi_{k}^u
\end{equation}
with the homogeneous boundary conditions
$$ 
\partial_{z} \psi(x, -1)= 0, \quad \psi(x,0) = 0.
$$
Consequently, 
the  $H^s$  elliptic regularity estimates  for (\ref{voda}), the Poincar\'e inequality 
and again Lemma~\ref{uHlem} and Lemma~\ref{vHlem} yield
$$ 
|J_{1}| \leq  C_{s,K}  |v|_{L^2}  
\Big(  \Big| {  |D_{x} | \over 1 +  |D_{x}|^{1 \over 2} } u \Big|_{H^s} +  |k| \, |u|_{H^s} \Big) .  
$$
This ends the proof of i).

\bigskip

Let us now prove ii). 
We use again \eqref{DNnew} to write
\begin{eqnarray*}
\big(G_{\eps, k} u, f\pa_{x} u  \big)
&=  & 
\int_{\mathcal{S}} \Big( g^{-1} \nabla \psi_{k}^u\cdot  \overline{\nabla(f \partial_{x}\psi_{k}^u)} 
+ k^2 \psi^u_{k}\overline{ f\partial_{x} \psi_{k}^u} \Big) (\det g)^{1 \over 2}\, dxdz \\
& = &  
\int_{\mathcal{S}} \Big( g^{-1} \nabla \psi_{k}^u\cdot \overline{\partial_{x} \nabla \psi_{k}^u}
\Big) \overline{f} (\det g)^{1 \over 2} \, dxdz
\\
& &  +   \int_{\mathcal{S}} \Big(  g^{-1}\nabla \psi_{k}^u
\cdot 
\overline{ \partial_{x} \psi_{k}^u \,\nabla f}  
+  k^2 \psi^u_{k} \overline{f\partial_{x} \psi_{k}^u} \Big) (\det g)^{1 \over 2}\, dx dz \\
& \equiv & I_{1}+ I_{2}.
\end{eqnarray*}
By using again Lemma~\ref{uHlem} and Lemma~\ref{vHlem}, we immediately get that
$$ |I_{2}| \leq C_{s,K} \Big(  \Big| {  |D_{x} | \over 1 +  |D_{x}|^{1 \over 2} } u \Big|_{L^2}^2
+ |k|^2  |u|_{L^2}^2 \Big).
$$
To estimate $I_{1}$, we first integrate by parts to obtain
$$2\, {\rm Re  }\, I_{1} = - \int_{\mathcal{S} } \partial_{x}\big( \overline{f}\,(\det g)^{1 \over 2 }
g^{-1} \big) \nabla \psi_{k}^u \cdot  \overline{\nabla \psi_{k}^u} \, dxdz$$
and then by Lemma ~\ref{uHlem} and Lemma~\ref{vHlem}, we also get
$$ 
| {\rm Re }\, I_{1}| \leq C_{s,K} \Big(  \Big| {  |D_{x} | \over 1 +  |D_{x}|^{1 \over 2} } u \Big|_{L^2}^2
+ |k|^2  |u|_{L^2}^2 \Big).
$$ 
This yields the desired estimate for  $ {\rm Re }\,\big(G_{\eps, k} u, f\pa_{x} u  \big)$.

This ends the proof of Proposition~\ref{com}.
\end{proof}
Let us set 
$$ 
G_{k}[\eta]u=  e^{-ik y } G[\eta]( u e^{iky})
$$
for functions $\eta(x)$, $u(x)$ which depends on $x$ only.
We are interested in estimates of  $D^j_{\eta}G_{k}[\eta_{\eps}] u  \cdot \big(h_{1}, \dots, h_{j}\big)$.
We shall use the notation
$$ 
D^j_{\eta} G_{\eps, k} u  \cdot \big(h_{1}, \dots, h_{j}\big) = 
D^j_{\eta}G_{k}[\eta_{\eps}] u  \cdot \big(h_{1}, \dots, h_{j}\big).
$$ 
\begin{prop}\label{lemderiveej}
For every $s>1/2$, we have the estimate
\beq\label{deriveej} 
\Big| D^j_{\eta} G_{\eps,k}  u  \cdot  \big(h_{1}, \dots, h_{j}\big) \Big|_{H^{s-{1 \over 2 } } }
\leq C_{s}\Big|  { \sqrt{ D_{x}^2 + k^2 } \over \big(1 +  \sqrt{ D_{x}^2 + k^2 } \big)^{1 \over 2 }  } 
u  \Big|_{H^s}  \prod_{i=1}^j  |h_{i}|_{H^{s+1}}.
\eeq
\end{prop}
\begin{proof}
It suffices to take the derivative of \eqref{dirS} with respect to $\eta$ and then
to use the standard  Sobolev-Gagliardo-Nirenberg-Moser estimates for products
in Sobolev spaces and again Lemma~\ref{uHlem} and Lemma~\ref{vHlem}.
This completes the proof of Proposition~\ref{lemderiveej}.
\end{proof}
\section{Study of the operator $L(k)$ arising in the linearization of the Hamiltonian }
\label{sectionLk}
As a preliminary, we first establish the following statement.
\begin{lem}\label{self}
$L(k)$ has a self-adjoint realization on $L^2(\mathbb{R}) \times L^2(\mathbb{R})$ with domain
$H^2(\mathbb{R} )\times H^1(\mathbb{R})$.
\end{lem}
\begin{proof}
We first notice that $L(k)$ enjoys an  elliptic regularity property, namely if $u=(u_1,u_2)\in L^2\times L^2$ 
is such that
$L(k)u\in L^2\times L^2$ then $u\in H^2\times H^1$. Indeed, using the first equation and the elliptic 
regularity for the second order operator  $P_{\eps,k}$, we obtain that $u_1\in H^1$. Then, using the elliptic regularity for
$G_{\eps,k}$ established in Corollary~\ref{elliptic} and the second equation, 
we obtain that $u_2\in H^1$. Finally, using again
the elliptic regularity for $P_{\eps,k}$, we obtain that $u_1\in H^2$.

Next, we also observe that $L(k)$ is symmetric in $H^{\infty}\times H^\infty$, namely
\begin{equation}\label{symetric}
(L(k)u,v)=(u,L(k)v),\quad \forall\, u,v\in H^{\infty}\times H^\infty\,.
\end{equation}
Moreover, let us  consider the closure $\overline{L(k)}$ of $L(k)$   defined on the  domain
$$
D(\overline{L(k)})=
\{
u\in L^2\times L^2\,:\, \exists\,u_n\in  H^{\infty}\times H^\infty,
u_n\rightarrow u\,\, {\rm in}\,\, L^2\times L^2,\,\,
L(k)u_n\,\,{\rm converges\,\,in}\,\,
L^2\times L^2
\}.
$$
We shall  show that $\overline{L(k)}$ is self adjoint and  that $D(\overline{L(k)})=H^2\times
H^1$. By definition,  the adjoint of  $\overline{L(k)}$, denoted by $\overline{L(k)}^\star$
has the domain
$$
D(\overline{L(k)}^\star)=
\{
u\in L^2\times L^2\,:\, \exists\, C>0,\,| (u,L(k)v ) |\leq
C\|v\|_{L^2\times L^2},\,\,\forall\,
v\in D(\overline{L(k)}) \}
$$
and  moreover,  the  following inclusions hold:
$$
H^2\times H^1\subset D(\overline{L(k)})\subset D(\overline{L(k)}^\star)\subset
H^2\times H^1\,.
$$
Indeed, the first inclusion follows from the density of $H^{\infty}\times
H^\infty$ in $H^2\times H^1$ and the fact that $L(k)$ is continuous from
$H^2\times H^1$ to $L^2\times L^2$. The second inclusion follows from the fact
that $L(k)$ is symmetric (see (\ref{symetric})). The third inclusion is the
most difficult to check. It follows from the elliptic regularity, since
$D(\overline{L(k)}^\star)$ can be also seen as the function in $ L^2\times
L^2$ such that $L(k)u$ (a priori defined in a weak sense) belongs to $ L^2\times
L^2$.
This completes the proof of Lemma~\ref{self}.
\end{proof}
\subsection{Essential spectrum}
Our aim is now to locate the essential spectrum of $L(k)$.
\begin{prop}\label{essL}
For every $\eps \in (0, \eps_{0}]$, and for every $k \in \mathbb{R}$,  there exists 
$c_{k} \geq 0$,  such that 
$$
\sigma_{ess}(L(k)) \subset [ c_{k}, +\infty) \subset [0, + \infty).
$$
Moreover, for $k \neq 0$, we have $c_{k}>0$.
\end{prop}
Note that for $k \neq 0$ the essential spectrum of $L(k)$ is included in $(0, + \infty).$
\begin{proof}[Proof of Proposition~\ref{essL}]
Since $L(k)$ is self-adjoint, its spectrum is real.
We thus  only have to prove that
$\gamma + L(k)$ is Fredholm with zero index for $\gamma \geq 0$, 
$k\neq 0$ and for $\gamma >0$, $k=0$.   Towards this, we shall prove
that $L(k)+\gamma$  can be written as
$$ 
L(k) + \gamma =  \mathcal{I}(\gamma, k ) + \mathcal{K}(\gamma, k),
$$
where $\mathcal{I}(\gamma, k)$ is  an invertible operator for $ \gamma>0$
with domain $H^2 \times H^1$ and $\mathcal{K}(\gamma, k)$ is a relatively compact perturbation.

Let us first  have a look to  the asymptotic behaviour of the coefficients of $L(k)$.
In Theorem \ref{theoOS}, we have already recalled that $\eta_{\eps}$, $\partial_{x} \varphi_{\eps}$
and their higher order derivatives 
have an exponential decay towards zero at infinity.
Moreover, since for  a  solitary wave, we have   $G[\eta_{\eps}] \varphi_{\eps}= - \partial_{x} \eta_{\eps}$,
we also get that $ G[\eta_{\eps}] \varphi_{\eps} $ and its derivatives tend to zero exponentially fast
at infinity. 
This yields  in particular  that $v_{\eps}$ and $\partial_{x} Z_{\eps}$
have an exponential decay towards $0$ when $x$ tends to $\pm \infty$.
Consequently, we can first  write the decomposition 
\begin{equation*} 
L(k) + \gamma = L_{0}(\gamma,k) + C(k), 
\end{equation*}
where
$$
L_{0}(\gamma, k)=  
\left( 
\begin{array}{cc} 
-\beta\zeta^{-3}\partial_x^2+\beta k^2+\alpha+\gamma & -  (1 - v_{\eps} ) \partial_{x}  \\ 
\partial_{x} &\gamma+ G_{\eps, k} \end{array} 
\right)
$$
and
$$
C(k)=
\left( 
\begin{array}{cc}
3\beta\zeta^{-4}\zeta'\partial_x - \beta k^2(1-\zeta^{-1})+ (v_{\eps} - 1) \partial_{x} Z_{\eps} &  0 \\   
- \partial_{x}(  v_{\eps} \cdot) & 0
\end{array} 
\right)
$$
where  $\zeta$ is defined as
$$
\zeta(x)=(1+(\partial_x\eta_{\eps}(x))^2)^{\frac{1}{2}}\,.
$$
Note that the function $\zeta(x)$ has an exponential decay towards $1$ when $x$ tends to $\pm \infty$, while its 
derivatives decay exponentially to zero.

The domain of $L_{0}(\gamma,k)$ is again $H^2 \times H^1$.
Consequently, thanks to the  exponential decay in its coefficients, we
get that $C(k)$ is a relatively compact perturbation i.e
$C(k)$  seen as an operator in $\mathcal{B}(H^2 \times H^1, L^2 \times L^2)$ is compact.

Next, by using that there exists $\eps_0>0$ such that for $\eps \leq \eps_{0}$ we have $1-v_{\eps} >0$, 
we write the factorization
\begin{equation*}
L_{0}(\gamma, k)=  A_{1}  L_{1}(\gamma, k ) +    C_{1}(\gamma,k ),
\end{equation*}
where 
\begin{eqnarray*}
& & 
L_{1}(\gamma,k) =
\left(  
\begin{array}{cc} 
-  \beta(1-v_{\eps})^{-1}\zeta^{-3}\partial_x^2+\beta k^2+\alpha+\gamma   & -   \partial_{x} 
\\ 
\partial_{x} & \gamma+  G_{\eps,k} \end{array}\right),  
\\
& &  
C_{1}(\gamma, k) = 
\left( 
\begin{array}{cc} 
(\beta k^2+\alpha+\gamma)v_{\eps} & 0 
\\  0  & 0 
\end{array}
\right), \\
& & 
A_{1} =  
\left( \begin{array}{cc}  1 - v_{\eps} & 0 \\ 
0 & 1
\end{array} 
\right) .
\end{eqnarray*}
Note that   $C_{1}(\gamma,k)$ is again a  relatively compact perturbation because of the decay 
of $v_{\eps}$ at infinity,  
while $A_{1}$ is just an invertible matrix. We can simplify  $L_{1}(\gamma, k)$
a little bit by  writing
$$ 
L_{1}(\gamma, k ) = L_{2}(\gamma, k) B_{1} + C_{2},
$$
where
\begin{eqnarray*}
& &  
L_{2}(\gamma, k) = 
\left( 
\begin{array}{cc} 
- \beta \partial_{x}^2 + \beta k^2 + \alpha + \gamma  &  - \partial_{x} \\
\partial_{x} & \gamma + G_{\eps, k} \end{array}\right), \\
& & 
B_{1}= \left( \begin{array}{ll} \zeta^{-3}(1-v_{\eps})^{-1} & 0 \\ 0 & 1 
\end{array}  
\right), 
\\
& & 
C_{2}  = 
\left( 
\begin{array}{cc}  - \beta \big[ ( 1 - v_{\eps})^{-1} \zeta^{-3}, \partial_{x}^2 \big] -   
(\beta  k^2+\alpha+\gamma)(1-\zeta^{-3}(1-v_{\eps})^{-1}) & 0  
\\
\,\partial_x \Big( \big( 1 - \zeta^{-3}(1-v_{\eps})^{-1}\big) \cdot \Big) & 0 
\end{array} 
\right).
\end{eqnarray*}
Again, we see that $C_{2}$ is a relatively compact perturbation since   because
of the decay of its coefficients it  is   compact   as an operator  in 
$ \mathcal{B}(H^2 \times H^1, L^2 \times L^2)$. Moreover,   $B_{1}$ is just an invertible matrix.

Next, to simplify the expression of $L_{2}$, we shall use a factorization inspired
by the work of Mielke~\cite{Mielke}.    Thanks to \eqref{DNm}, we note that
the operator $\gamma + G_{\eps, k}$ satisfies for some $c=c(k)>0$  the estimate:
$$
\big( (\gamma + G_{\eps, k }  ) u, u \big) \geq c \Big(  \Big| {|D_{x} | \over 1 + |D_{x} |^{ 1 \over 2 } }  u\Big|_{L^2}^2
+  (\gamma + k^2) |u |_{L^2}^2 \Big).
$$
Consequently, for $ \gamma \geq 0$ and $k>0$ or for $\gamma= 0  $ and $ k>0$
we have that for some $c>0$ (depending on $\gamma$ and $k$),
\beq
\label{DNm2}
\big( (\gamma + G_{\eps, k }  ) u, u \big) \geq c  |u|_{H^{1 \over 2 } }^2.
\eeq
Moreover, thanks to  \eqref{DNC}, we also have that
$$\big( (\gamma + G_{\eps, k }  ) u, v \big)   \leq C |u|_{H^{1 \over 2}} \, |v  |_{H^{1 \over 2 }}$$
thus the quadratic form
$ \big( (\gamma + G_{\eps, k }  ) \cdot , \cdot  \big)$ is continuous and coercive on 
$ H^{1 \over 2}$.
By the Lax-Milgram lemma and    Corollary~\ref{elliptic}, we thus get that
the operator $ \gamma + G_{\eps, k}$  defined on $L^2$ with domain $H^1$
is invertible for every $(\gamma, k)$ such that $\gamma \geq 0$ and $k >0$
or $\gamma >0 $ and $k= 0 $. The existence of $ (\gamma + G_{\eps, k}) ^{-1}$
allows to introduce the factorization
$$ L_{2}(\gamma, k) = A_{2}(\gamma, k) L_{3}(\gamma, k) B_{2}(\gamma, k)$$
where
\begin{eqnarray*}
& &L_{3}(\gamma, k) =  \left( \begin{array}{cc}   - \beta \partial_{xx} +\beta k^2 + \alpha + \gamma
+  \partial_{x} (\gamma + G_{\eps, k} )^{-1} \partial_{x} & 0 \\ 0 &   \gamma +  G_{\eps, k}
\end{array}\right), \\
& &  A_{2}(\gamma, k) =  \left( \begin{array}{lc} 1  & -\partial_{x} (\gamma + G_{\eps, k })^{-1} \\
0 & 1  \end{array}\right),\\
& & B_{2}(\gamma, k)  = \left( \begin{array}{ll}  1 & 0 \\  (\gamma + G_{\eps, k } )^{-1} \partial_{x}
& 1 \end{array} \right)= A_{2}(\gamma, k)^*.
\end{eqnarray*}
Note that $A_{2}(\gamma,k)$  and  $B_{2}(\gamma,k)$ are bounded invertible operators on $L^2 \times L^2$.
To get our last simplification, we shall prove that the operator
\begin{equation}\label{novagodina}
\partial_{x}(\gamma + G_{\eps, k  } )^{-1} \partial_{x} - \partial_{x} (\gamma + G_{0,k})^{-1}
\partial_{x} 
\end{equation}
is a compact operator in $ \mathcal{B}(H^2 \times H^1, L^2 \times L^2)$, where the operator
$G_{0, k}$ is the Dirichlet Neumann  for the flat surface $\eta = 0$,
$$  
G_{0, k} u= e^{-ik y } G[0](u e^{iky}).
$$
Coming back to (\ref{uHhat}), we obtain that $G_{0,k}$ is a Fourier multiplier, namely
$$
G_{0,k}=\mbox{tanh}(D_x^2+k^2)\sqrt{D_x^2+k^2}\,.
$$
In order to study the compactness properties of (\ref{novagodina}), we shall use the following lemma.
\begin{lem}\label{pert}
The operator $ R_{\eps} = G_{\eps, k } -G_{0, k}$ is a bounded operator in $\mathcal{B}(H^1, L^2)$
and a compact operator from $H^s$ to $L^2$ for every $s>1$.
\end{lem}
Let us postpone the proof of Lemma~\ref{pert}. We now show how  we can end the proof of Proposition~\ref{essL}
by using  Lemma~\ref{pert}.
We write
$$ \gamma +  G_{\eps, k }=  ({\rm Id} +  R_{\eps}^{(1)} )( \gamma + G_{0, k }), \quad
 R_{\eps}^{(1)}  =   R_{\eps} (\gamma + G_{0, k })^{-1}. 
$$
We have that ${\rm Id} + R_{\eps}^{(1)}$ is a bounded invertible operator on $L^2$ and
$$  ({\rm Id} + R_{\eps}^{(1)})^{-1} = (\gamma + G_{0,k}) (\gamma + G_{\eps, k })^{-1}.$$
Moreover, thanks to Lemma~\ref{pert} and since
$(\gamma + G_{0, k })^{-1}$  is a bounded operator from $H^s$ to $H^{s+1}$, 
we get that $R_{\eps}^{(1)}$  is a compact operator from  $  H^s \mbox{ to } L^2$ for every $s>0$.       
Next, since
$$ ({\rm Id} + R_{\eps }^{(1)})^{-1} = {\rm Id}  -  ({\rm Id} + R_{\eps}^{(1)})^{-1} R_{\eps}^{(1)}, $$ 
we get that
$  ({\rm Id} + R_{\eps }^{(1)})^{-1} = {\rm Id} + R_{\eps}^{(2)}$ 
where $ R_{\eps}^{(2)} $ is a compact operator  from  $  H^s \mbox{ to } L^2$ for every $s>0$.
To conclude, we write that
$$  
(\gamma +  G_{\eps, k })^{-1} = (\gamma + G_{0,k})^{-1} ( {\rm Id} + R_{\eps}^{(1)} )^{-1}
= (\gamma + G_{0,k})^{-1} + (\gamma + G_{0,k})^{-1} R_{\eps}^{(2)}
$$
which yields
$$ \partial_{x}  (\gamma +  G_{\eps, k })^{-1} \partial_{x} = 
\partial_{x} (\gamma + G_{0,k})^{-1} \partial_{x} + \partial_{x}(\gamma + G_{0,k})^{-1} R_{\eps}^{(2)}
 \partial_{x}$$ 
  and we observe that $\partial_{x} (\gamma + G_{0,k})^{-1}$ is a bounded operator on $L^2$, 
   that $\partial_{x}$ is a bounded operator from $H^2$ to $H^1$ and
    that $R_{\eps}^{(2)}$ is a compact operator from $H^1$ to $L^2$. Consequently,
     we have obtained that
$$  
\partial_{x}  (\gamma +  G_{\eps, k })^{-1} \partial_{x} = 
\partial_{x} (\gamma + G_{0,k})^{-1} \partial_{x}  + R_{\eps}^{(3)}, 
$$
where  $R_{\eps}^{(3)} $ is a compact operator from $H^2$ to $L^2$.
  This finally allows to write that
  $$ L_{3}(\gamma, k) = L_{4}(\gamma, k ) 
   + C_{3}(\gamma,k)$$
    where 
  $$   L_{4}(\gamma, k ) =   \left( \begin{array}{cc}   - \beta \partial_{x}^2 + \beta
k^2 + \alpha + \gamma
+  \partial_{x} (\gamma + G_{0, k} )^{-1} \partial_{x} & 0 \\ 0 &   \gamma +  G_{\eps, k}
\end{array}\right)$$
and $C_{3}(\gamma, k)$
is a relatively compact perturbation.

Gathering all our transformations, we  find that
\begin{equation}\label{decL}
\gamma + L(k) = A_{1} A_{2}(\gamma, k)  L_{4}(\gamma, k) B_{2}(\gamma,k) B_{1} + \mathcal{K}
\end{equation}
where $\mathcal{K}$ is a relatively compact perturbation for $\gamma\geq 0$, $k\neq 0$ and
 for $\gamma >0$, $k=0$. 
 Consequently to get that $\gamma + L(k)$ is Fredholm with index zero, it suffices  
  to prove that  $A_{1} A_{2}(\gamma, k)  L_{4}(\gamma, k) B_{2}(\gamma,k) B_{1}$
   is invertible. Since $A_{1}$, $A_{2}$, $B_{1}$, $B_{2}$ are bounded invertible operators,
    we only have to  prove that
     $ L_{4}(\gamma, k)$ is invertible. Moreover, we see that  $L_{4}$ is a diagonal operator
      and we have already seen that $\gamma + G_{\eps, k}$ is invertible for
       $\gamma \geq 0, $  $k \neq 0 $ or $\gamma >0$, $k=0$.
        Therefore, it only remains to study the invertibility 
    of 
    $$ - \beta \partial_{x}^2 + \beta k^2 + \alpha + \gamma
+  \partial_{x} (\gamma + G_{0, k} )^{-1} \partial_{x} .$$
This  operator is just the Fourier  multiplier by 
\begin{eqnarray*}
m(\xi) & = &  \beta  \xi^2  + \beta k^2 + \alpha +  \gamma   - \frac{ \xi^2  
}{ \gamma+\mbox{tanh}(\sqrt{\xi^2+k^2})\,\sqrt{\xi^2+k^2}}
\\
& \geq &
\beta  \xi^2  + \beta k^2 + \alpha +  \gamma   - \frac{ |\xi|
}{ \mbox{tanh} |\xi|}. 
\end{eqnarray*}
As  observed  in \cite{Mielke}, we have the following statement.
\begin{lem}\label{calculus}
For $\beta> 1/3$, we have the inequality
$$
\beta x^2- \frac{x}{{\rm tanh}\, (x)}\geq -1,\quad \forall\, x\geq 0. 
$$
\end{lem}
\begin{proof}
Let us  set 
$f(x)=\beta x^2- \frac{x}{{\rm tanh}\, (x)}, $ then $f'(x)=g(x)/(e^x-e^{-x})^2$, where
$$
g(x)=2\beta x(e^{2x}+e^{-2x})-(e^{2x}-e^{-2x})+(4-4\beta)x\,.
$$
We have that near zero $g(x)=8(\beta-1/3)x^3+{\mathcal O}(x^4)$ and  that
$$
g^{(4)}(x)=32\beta x(e^{2x}+e^{-2x})+16(4\beta-1)(e^{2x}-e^{-2x})>0,\quad \forall\, x\geq 0. 
$$
Hence $g^{(3)}(x)$ is an increasing function
and since for $\beta >1/3$, $g^{(3)}(0)>0$, we obtain that $g^{(3)}(x)> 0$ for every $x\geq 0$.
Next, in a similar way, we obtain successively  that $g^{''}(x)$, $g'(x)$ and $g(x)$ are non-negative. 
Therefore $f(x)$ is
an increasing function which implies that for every $x\geq 0$, $f(x)\geq f(0)=-1$.
This ends  the proof of Lemma~\ref{calculus}. 
\end{proof}
Using Lemma~\ref{calculus}, we get
$$ 
m(\xi) \geq  \beta k^2 + \gamma + \alpha - 1 .
$$
Since 
$\alpha = 1 + \eps^2$, $m(\xi)$ is uniformly  bounded from below by a  positive
number for $\gamma \geq 0$. Consequently, the operator $L_4$ is invertible.
To end the proof of Proposition~\ref{essL},  it only remains to  give the proof of Lemma \ref{pert}\,.
\begin{proof}[Proof of Lemma \ref{pert}]
Thanks to the fundamental theorem of calculus, we can write
\begin{equation*}
\big(G_{\eps, k} -G_{0, k}\big) \varphi=  e^{-iky }\int_{0}^1  
\big(D_{\eta}G[s \eta_{\eps}](e^{iky}\varphi)\cdot \eta_{\eps}\big)\, ds.
\end{equation*}
Using Lemma~\ref{DN'}, we have 
\beq
\label{mielkee}
\big( G_{\eps, k} - G_{0, k}\big)\varphi
= 
\int_{0}^1 \Big( -G_{k}[s \eta_{\eps}] (\eta_{\eps} Z_{k}( s \eta_{\eps}, \varphi))
- \partial_{x} \big(  \eta_{\eps} ( \partial_{x} \varphi - sZ_{k}(s\eta_\eps,\varphi)\partial_{x} \eta_{\eps}) 
\big) + k^2 \eta_{\eps} \varphi\Big) ds,
\eeq
where we use the notations 
$$ 
G_{k}[\eta]\varphi=  e^{-ik y } G[\eta] (e^{iky} \varphi), \quad
Z_{k}( \eta , \varphi) = { G_{k}[\eta]\varphi + \partial_{x} \eta \partial_{x } \varphi \over
1 + |\partial_{x} \eta |^2}.
$$
In this formula, it seems at first sight that  $\big( G_{\eps, k} - G_{0, k}\big)$
is a second order operator. Nevertheless, by using iii) of
Proposition~\ref{DirNeum2}, basic commutator estimates   and the decay of 
$\eta_{\eps}$ and its derivatives, we get that
\begin{multline*}
\big( G_{\eps, k} - G_{0, k}\big)\varphi 
=  \int_{0}^1 \Big( 
\partial_x^2\varphi
\big({\eta_{\eps} \over  { 1 + s^2 (\partial_{x} \eta_{\eps} )^2} }-\eta_{\eps}+
{\eta_{\eps} s^2 (\partial_{x} \eta_{\eps} )^2\over  { 1 + s^2 (\partial_{x} \eta_{\eps} )^2} }
\big)
\\
+
\partial_x |D_x| \varphi
\big(
{s\eta_{\eps}  \partial_{x} \eta_{\eps}   \over  { 1 + s^2 (\partial_{x} \eta_{\eps} )^2} }
-
{s\eta_{\eps} \partial_{x} \eta_{\eps}   \over  { 1 + s^2 (\partial_{x} \eta_{\eps} )^2} }
\big)\Big) ds + \mathcal{R}  = \mathcal{R},
\end{multline*}
where the operator $\mathcal{R}$ is a sum of terms which are  all made of the product of  a first 
order (pseudo differential) operator  i.e. belonging to $\mathcal{B}(H^s , H^{s-1})$  and of   
a rapidly decreasing function. This yields that  $\mathcal{R}$ is compact as an operator 
from $H^s$  to $L^2$ for every $s>1$ and ends the proof of Lemma~\ref{pert}.
\end{proof}
This also  ends  the proof of Proposition~\ref{essL}.
\end{proof}         
\subsection{Negative eigenvalues of $L(k)$}
The next step  is the study of the eigenvalues of $L(k)$  outside the essential spectrum.

As in \cite{Mielke}, it is convenient to introduce a reduced operator in order 
to  study the eigenvalues of $L(k)$. We first  define  the operator
$$ 
Mu= - \partial_{x}^{-1} G_{\eps, 0} \partial_{x}^{-1} u
$$
where $\partial_{x}^{-1}$ is defined by the division by $i\xi$ in the Fourier space. Note that 
$M$ is well-defined for smooth functions whose support of their Fourier transform  does not meet zero. 
To study $M$, it is convenient to  introduce the bilinear symmetric  form 
$$ 
Q(u, v) = (Mu, v),
$$
for $u,v\in H^{\infty}(\R)$ such that $\hat{u}$, $\hat{v}$ have  supports which do not meet zero.

We have the following statement: 
\begin{lem}\label{M}

$Q$ extends to a continuous  and coercitive  bilinear form  on $ H^{-{1 \over 2} } \times H^{- {1 \over 2 } } $.

\end{lem}
As a consequence of this statement, we get thanks to the continuity of $Q$  that $M$ is well-defined as an operator
in $\mathcal{B}(H^{-{1\over 2}}, H^{ { 1 \over 2 } } )$.
Moreover,  thanks to the Lax-Milgram lemma, we can  thus define  the inverse $M^{-1}$
as an operator in  $\mathcal{B}( H^{ 1 \over 2}, H^{- { 1 \over 2} } ).$  By using \eqref{princip} in Proposition
\ref{DirNeum2}, we can then get that $M$ is a  continuous bijection  from $H^s$ to  $H^{s+1}$
 for every $s\in \mathbb{R}$.
\begin{proof}[Proof of Lemma~\ref{M}]
We notice that
$$ 
Q(u,v)=  \big( G_{\eps, 0} \partial_{x}^{-1}u, \partial_{x}^{-1}v \big)
$$
consequently, thanks to \eqref{DNC} in Proposition~\ref{DirNeum}, we get that
$$ 
| Q(u, v )| \leq C |u|_{H^{-{1 \over 2 } } } \, |v|_{H^{-{ 1 \over 2 }}}.
$$
In a similar way, we get that
$$ 
Q(u,u) \geq c |u|_{H^{ - { 1 \over 2 }}}^2
$$ 
thanks to \eqref{DNm}. This ends  the proof of Lemma~\ref{M}.       
\end{proof}
Next, as in \cite{Mielke}, we can use  the operators $M$ and $M^{-1}$ to notice that 
\begin{eqnarray}\label{L0dec}
\big( L(0) U, U \big) & = &  \Big(\big( - P_{\eps, 0 } + \alpha  - \gamma_{\eps}  \partial_{x} Z_{\eps} \big)
U_{1}, U_{1} \Big) -  \big( M^{-1}(\gamma_{\eps} U_{1}), \gamma_{\eps} U_{1}\big) \\
\nonumber   & &  + \Big( M \big( \partial_{x} U_{2} - M^{-1} (\gamma_{\eps} U_{1}) \Big), 
    \partial_{x}U_{2} - M^{-1}(\gamma_{\eps} U_{1}) \Big)
  \end{eqnarray}
  where we have set 
  $\gamma_{\eps}\equiv  1 - v_{\eps}.$ 
        Since $M$ is nonnegative, we  obtain that     
$$ \big( L(0) U, U \big) \geq \Big(\big( - P_{\eps, 0 } + \alpha  - \gamma_{\eps}  \partial_{x} Z_{\eps} \big)
  U_{1}, U_{1} \Big) 
     -  \big( M^{-1}(\gamma_{\eps} U_{1}), \gamma_{\eps} U_{1}\big)\equiv ( A_{\eps} U_{1}, U_{1}).$$
 By using that  for the  solitary wave $\eta_{\eps}$, we have
 $  G[\eta_{\eps}] \varphi_{\eps} =- \partial_{x} \eta_{\eps}, $
  we get
  that   
$$ 
\gamma_{\eps}=  { 1 -\partial_{x} \varphi_{\eps} \over  1 + (\partial_{x} \eta_{\eps})^2},\quad
Z_{\eps}=-\gamma_{\eps}\partial_{x}\eta_{\eps}\,.
$$
Therefore, we obtain the expression 
\beq\label{Aeps}
A_{\eps}\eta  = - P_{\eps, 0 }\eta + \alpha\eta +\gamma_{\eps}   \partial_{x} ( \gamma_{\eps} \partial_{x}
\eta_{\eps}) \eta - \gamma_\eps M^{-1}( \gamma_{\eps} \eta) .
\eeq
The spectrum of the operator $A_{\eps}$ is studied in \cite{Mielke}.
Note that our notations are slightly different from the one of  Mielke in \cite{Mielke}, in particular,
here $\alpha$ is the rescaled coefficient coming from the term taking into account the gravity in the equation.
After a suitable rescaling $A_{\eps}$ tends (in a rather weak sense) to the operator obtained by 
linearizing the KdV equation about the KdV solitary wave. This allows to prove the following statement.
\begin{prop}[Mielke~\cite{Mielke}]\label{Mielke}
There exists $\eps_{0}>0$ such that for every $\eps$, $\eps \in (0, \eps_{0})$, the spectrum of $A_{\eps}$ 
which is a self adjoint operator on $L^2$ with domain $H^2$  consists of a negative 
simple eigenvalue $\lambda_{\eps}^-$, the simple eigenvalue zero and the remaining of the spectrum is included
in $[\lambda_{\eps}, + \infty[ $ for some $\lambda_{\eps}>0$.
\end{prop}  
\begin{proof}
By using  similar arguments as in the proof of Proposition \ref{essL}, we can first locate the essential
spectrum  of $A_{\eps}$. At first, we can write that $A_{\eps}$ is a  relatively compact perturbation
 of  the operator $ (1+ | \pa_{x} \eta_{\eps}|^2 )^{-  {3 \over 2 } }B_{\eps}$ where
 $$ B_{\eps } \varphi= - \beta  \,\partial_{x}^2 \varphi + \alpha \varphi  - M^{-1} \varphi $$
 and thus, it suffices to study the essential spectrum of $B_{\eps}$.
 The next step is to prove that $M^{-1}$ is a relatively compact perturbation of 
  $ M_{0}^{-1}= -  \partial_{x} G[0]^{-1} \pa_{x}$
  which is also well-defined thanks to Lemma \ref {M}.
     As in the end of the proof of Proposition \ref{essL},  it suffices to prove that
     \beq
     \label{mielkeee}
      M = ( {\rm Id } + \mathcal{C}) M_{0}
      \eeq with $\mathcal{C}$ is  compact as an operator in $\mathcal{B}(H^2, H^1)$.
      Indeed,  if this fact is proven,  as  in the proof of  Proposition  \ref{essL}, we immediately get from this that  
      $$M^{-1}= M_{0}^{-1} \big( {\rm Id} -  ( {\rm Id} + \mathcal{C})^{-1} \mathcal{C} \big) ={M}_{0}^{-1}
       + \mathcal{K}$$
        where $\mathcal{K}$ is a compact operator in  $\mathcal{B}(H^2, L^2)$.
         Consequently,  $B_{\eps}$ is a relatively compact perturbation of
          $$  A_{0}=  - \beta\, \partial_{x}^2 + \alpha\, {\rm Id } +M_{0}^{-1}
           =\beta|D_x|^2+ \alpha\, {\rm Id }-\frac{|D_x|}{{\rm tanh}(|D_x|)} $$
           and hence  we get from Lemma \ref{calculus} that its essential spectrum  is contained
            in $[\eps^2, +\infty)$.

           To prove \eqref{mielkeee}, we can use 
       \eqref{mielkee},  to get
          $$
 \pa_{x}^{-1}\big( G[\eta_{\eps}] - G[0]\big) \pa_{x}^{-1}\varphi
= 
\int_{0}^1 \Big( - \pa_{x}^{-1} G[s \eta_{\eps}] (\eta_{\eps} Z( s \eta_{\eps},  \pa_{x}^{-1}\varphi))
-\big(  \eta_{\eps} (  \varphi - sZ(s\eta_\eps, \pa_{x}^{-1}\varphi)\partial_{x} \eta_{\eps}) 
\big) \Big) ds. $$
Note that this formula is meanigful since $ \partial_{x}^{-1} G[s\eta_{\eps}], $ $ G[s\eta_{\eps}] \pa_{x}^{-1}$ and
thus  $Z( s \eta_{\eps},  \pa_{x}^{-1}\cdot )$ are well defined bounded operators   on $L^2$ 
 thanks to  \eqref{DNC} in Proposition \ref{DirNeum}. By using \eqref{princip} in Proposition
  \ref{DirNeum2} we can write
  $$ \partial_{x}^{-1}G[s\eta_{\eps}]= \partial_{x}^{-1} |D_{x}| + \mathcal{R}_{1}, 
   \quad  G[s\eta_{\eps}] \partial_{x}^{-1}= \partial_{x}^{-1} |D_{x}|  + \mathcal{R}_{2}$$
   where $\mathcal{R}_{1}$, $\mathcal{R}_{2}$ are   bounded operators  from $H^s$ to $H^{s+1}$. 
   Consequently, as in the proof of Proposition \ref{essL}, we obtain that
   $ \pa_{x}^{-1}\big( G[\eta_{\eps}] - G[0]\big) \pa_{x}^{-1} = M - M_{0}$ is a compact operator in  $\mathcal{B}(H^1,  H^{2- \eps })$ and hence in   $\mathcal{B}(H^1,  H^{ 1 })$.
    Since $M_{0}^{-1} \in \mathcal{B}(H^2, H^1)$ is invertible we obtain \eqref{mielkeee}.

For the study of the eigenvalues, we shall just give  a sketch of the argument of Mielke \cite{Mielke}
by explaining how the problem can be  reduced to the study of the KdV  problem.
  We first 
get   that zero is an eigenvalue   from the  differentiation of the equation satisfied by the solitary
wave. 
Let us denote by $S_{\eps}$ the scaling map $S_{\eps}(\eta)(x)=\eta(\eps x)$. Then $S_{\eps}^{-1}(\eta)(x)=
\eta(x/\eps)$. It turns out that in the limit $\eps\rightarrow 0$, 
the operator $\eps^{-2}S_{\eps}^{-1}A_{\eps}S_{\eps}$ is a zero order perturbation of the operator
$$
\eps^{-2}S_{\eps}^{-1}A_{0}S_{\eps}=
\beta|D_x|^2+\eps^{-2}(1+\eps^2)-\eps^{-2}\frac{|\eps D_x|}{{\rm tanh}(|\eps D_x|)}\,.
$$
For fixed $\xi\in \R$, we have (see Lemma~\ref{calculus}) the expansion for $\eps$ near zero :
$$
\beta\xi^2+\eps^{-2}(1+\eps^2)-\eps^{-2}\frac{|\eps \xi|}{{\rm tanh}(|\eps \xi|)}
=
(\beta-1/3)\xi^2+1+{\mathcal O}(\eps)
$$
which allows to prove that
\beq
\label{anatol0}
\lim_{\eps\rightarrow 0}|\eps^{-2}S_{\eps}^{-1}A_{0}S_{\eps}(u)-
\big(-(\beta-1/3)\partial_x^2+1\big)(u)|_{L^2}=0,
\quad \forall\, u\in H^2(\R).
\eeq
Next, we  can write
$$
\eps^{-2}S_{\eps}^{-1}A_{\eps}S_{\eps}=\eps^{-2}S_{\eps}^{-1}(A_{\eps}-A_0)S_{\eps}+
\eps^{-2}S_{\eps}^{-1}A_{0}S_{\eps}\,.
$$
and the  main point in the proof is to show that
\begin{equation}\label{anatol}
\lim_{\eps\rightarrow 0}
\Big\|\eps^{-2}S_{\eps}^{-1}(A_{\eps}-A_0)S_{\eps}
+3\cosh^{-2}\Big(\frac{x}{2(\beta-1/3)^{1/2}}\Big)\Big\|_{H^2\rightarrow L^2}=0.
\end{equation}
Indeed despite  the rather weak link  between the operators, one can then deduce 
that  for small $\eps$, the  eigenvalues  of $\eps^{-2}S_{\eps}^{-1}A_{\eps}S_{\eps}$ are small
perturbations of the ones    of 
$$
-(\beta-1/3)\partial_x^2+1-3\cosh^{-2}\Big(\frac{x}{2(\beta-1/3)^{1/2}}\Big)
$$
which is the operator that arises  when linearizing the KdV equation about a solitary wave
and whose spectrum is very well-known. We refer to \cite{Mielke}  p. 2348 for the details.
We shall just give the proof of \eqref{anatol} which can be easily obtained
 from  some of our  previous arguments.
Coming back to the definition of $A_{\eps}$, (\ref{anatol})  reduces directly to 
\begin{equation}\label{anatol_bis}
\lim_{\eps\rightarrow 0}
\Big\|\eps^{-2}S_{\eps}^{-1}
\Big(
-\gamma_{\eps}M^{-1}(\gamma_\eps\cdot)+\frac{| D_x|}{{\rm tanh}(| D_x|)}
\Big)S_{\eps}
+3\cosh^{-2}\Big(\frac{x}{2(\beta-1/3)^{1/2}}\Big)\Big\|_{H^2\rightarrow L^2}=0
\end{equation}
(the other terms involved in the definition of $A_{\eps}$ tend easily to zero). 
For that purpose,  we write  the  Taylor expansion
$$
G[\eta_\eps](\varphi)=G[0](\varphi)+D_{\eta}G[0](\varphi)\cdot\eta_\eps+r_{\eps}(\varphi)= A\varphi+ B\varphi,
$$
where
$$ A= G[0], \quad  r_{\eps}(\varphi)  = \int_{0}^1 (1- t ) D_{\eta}^2 G[t \eta_{\eps}] \varphi \cdot ( \eta_{\eps}, \eta_{\eps}) \, dt.$$
To compute $M^{-1}= -  \partial_{x}\big( A+ B\big)^{-1} \partial_{x}$, we  use a Neumann series expansion
to get
$$\partial_{x}\big( A+ B\big)^{-1} \partial_{x} =  \pa_{x}A^{-1}\pa_{x} + \sum_{j=1}^{+\infty}
 (-1)^j  \pa_{x} \big( A^{-1} B \big)^j A^{-1} \pa_{x}.$$ 
 Let us observe that  we can write
 $$  \pa_{x} \big( A^{-1} B \big)^j A^{-1} \pa_{x}= \pa_{x} A^{- {1 \over 2 } } \big( A^{- {1 \over 2 } }
  B A^{ - {1 \over 2 }} \big) \cdots  \big( A^{- {1 \over 2 } }
  B A^{ - {1 \over 2 }} \big) A^{- {1 \over 2 } }  \pa_{x}$$
   where  $\big( A^{- {1 \over 2 } }
  B A^{ - {1 \over 2 }} \big)$ is repeated $j$ times.  Note that  
   $ \pa_{x} A^{- {1 \over 2 } }   $ and $  A^{- {1 \over 2 } } \pa_{x}$  are  bounded  Fourier multipliers 
    in $\mathcal{B}(H^s, H^{s- {1\over 2}} )$  and that by setting
     $ B= B_{1}+ r_{\eps}$, we  first have
     $$  A^{- {1 \over 2 } }B_{1} A^{ - {1 \over 2 }} \varphi= - G[0]^{{1  \over 2 }} \big( \eta_{\eps} G[0]^{1\over 2 }
      \varphi\big) - \pa_{x } G[0]^{ - { 1\over 2 }  }\big(  \eta_{\eps}  \pa_{x} G[0]^{- {1\over 2}  }
      \varphi\big)$$
thanks to Lemma~\ref{DN'}.   By using classical  commutator estimates, we have
$$  \big| A^{- {1 \over 2 } }B_{1} A^{ - {1 \over 2 }}\big|_{\mathcal{B}(H^s, H^s)}
 \leq C_{s} \eps^2.$$
 Using estimates in the spirit of \eqref{DNC} for the Frechet derivative of $G[\eta]$, we also have
 $$  \big| A^{- {1 \over 2 } } r_{\eps } A^{ - {1 \over 2 }}\big|_{\mathcal{B}(H^s, H^s)}
 \leq C_{s} \eps^4.$$
 Therefore the Neumann series is well-defined and converges and
  for $\eps\ll 1$ we have the expansion
$$
-M^{-1}(\varphi)=-\frac{| D_x|}{{\rm tanh}(| D_x|)}(\varphi)-
\partial_xG[0]^{-1}D_{\eta}G[0](G[0]^{-1}\partial_x\varphi)\cdot\eta_{\eps}+R_{\eps}(\varphi),
$$
with $\|S_{\eps}^{-1}R_{\eps}S_{\eps}\|_{H^2\rightarrow L^2}={\mathcal O}(\eps^3)$.
Next, using Lemma~\ref{DN'}, we may write
\begin{equation}\label{venn}
\partial_xG[0]^{-1}D_{\eta}G[0](G[0]^{-1}\partial_x\varphi)\cdot\eta_{\eps}=
-\partial_{x}(\eta_{\eps}\partial_x\varphi)-
\frac{| D_x|}{{\rm tanh}(| D_x|)}
\Big(
\eta_{\eps}
\frac{| D_x|}{{\rm tanh}(| D_x|)}
(\varphi)\Big).
\end{equation}
After the conjugation with $\eps^{-1}S_{\eps}$ the operator defining the 
first term in the right hand-side of (\ref{venn})
tends to zero as $\eps\rightarrow 0$ 
in $\mathcal{B}(H^2,L^2)$.

Thanks to \cite{AK}, the 
solitary wave $(\eta_{\eps},\varphi_{\eps})$ may be written as
\begin{eqnarray*}
\eta_\eps(x)& = &-\eps^2\cosh^{-2}\Big(\frac{\eps x}{2(\beta-1/3)^{1/2}}\Big)+{\mathcal O}(\eps^4),
\\
\varphi_{\eps}(x) & = & -2(\beta-1/3)^{1/2}\eps\,{\rm tanh}\Big(\frac{\eps x}{2(\beta-1/3)^{1/2}}\Big)+
{\mathcal O}(\eps^3)\,.
\end{eqnarray*}
Hence the second term in the right hand-side of (\ref{venn}), conjugated by $\eps^{-1}S_{\eps}$ 
has the same limit as
\begin{equation*}
\cosh^{-2}\Big(\frac{x}{2(\beta-1/3)^{1/2}}\Big)
\Big(\frac{|\eps D_x|}{{\rm tanh}(|\eps D_x|)}\Big)^2,
\end{equation*}
the commutator tending to zero in  $\mathcal{B}(H^2,L^2)$.
Next, we have that
$$
\lim_{\eps\rightarrow 0}\Big\|\cosh^{-2}\Big(\frac{x}{2(\beta-1/3)^{1/2}}\Big)
\Big(\frac{|\eps D_x|}{{\rm tanh}(|\eps D_x|)}\Big)^2-\cosh^{-2}\Big(\frac{x}{2(\beta-1/3)^{1/2}}\Big)
\Big\|_{H^2\rightarrow L^2}=0.
$$
Therefore, we obtain that 
$$
\lim_{\eps\rightarrow 0}
\Big\|\eps^{-2}S_{\eps}^{-1}
\Big(
-\gamma_{\eps}M^{-1}(\gamma_\eps\cdot)+\gamma_{\eps}\frac{| D_x|(\gamma_{\eps}\cdot)}{{\rm tanh}(| D_x|)}
\Big)S_{\eps}+\cosh^{-2}\Big(\frac{x}{2(\beta-1/3)^{1/2}}\Big)
\Big\|_{H^2\rightarrow L^2}=0
$$
and thus (\ref{anatol_bis}) would be a consequence of
\begin{equation}\label{anatol_tris}
\lim_{\eps\rightarrow 0}
\Big\|\eps^{-2}S_{\eps}^{-1}
\Big(-
\gamma_{\eps}\frac{| D_x|(\gamma_{\eps}\cdot)}{{\rm tanh}(| D_x|)}
+\frac{| D_x|}{{\rm tanh}(| D_x|)}
\Big)S_{\eps}
+2\cosh^{-2}\Big(\frac{x}{2(\beta-1/3)^{1/2}}\Big)\Big\|_{H^2\rightarrow L^2}=0.
\end{equation}
Coming back to the definition of $\gamma_{\eps}$, we obtain that
\begin{equation}\label{abatol_4}
\gamma_{\eps}=1+\eps^2\cosh^{-2}\Big(\frac{\eps x}{2(\beta-1/3)^{1/2}}\Big)+{\mathcal O}(\eps^4), 
\end{equation}
in $W^{s,\infty}(\R)$, $s\geq 0$.
Now, (\ref{anatol_tris}) follows from (\ref{abatol_4}).
\end{proof}
\begin{rem}[Fixing the value of $\eps$]
From now on,  $\eps$ will be fixed in the range of validity of Proposition~\ref{Mielke},  
Proposition~\ref{essL} and Theorem~\ref{theoOS}. More precisely, from now on we fix an $\eps$ such that
$\eps \in (0, \eps_{0}]$, where $\eps_0$ is determined by Proposition~\ref{Mielke}, 
Proposition~\ref{essL} and Theorem~\ref{theoOS}.
\end{rem}     
Let us define $\eta_{\eps}^-$ and $\eta_{\eps}^0$ as the $L^2$ normalized eigenvalues of 
$A_{\eps}$, i.e. such that
\beq\label{Avp}
A_{\eps} \eta_{\eps}^- = \lambda_{\eps}^-\eta_{\eps}^-  ,   \quad    A_{\eps} \eta_{\eps}^0= 0.
\eeq
Since $A^\eps$ is a self adjoint operator, we get from Proposition~\ref{Mielke} that there exists
$c_{\eps}>0$ such that
\beq\label{ortho1}
\big(A_{\eps} \eta, \eta\big) \geq c_{\eps} |\eta|_{H^1}^2, \quad
\forall\, \eta \in H^1(\mathbb{R}), \quad
(\eta, \eta_{\eps}^-)=0, \, (\eta, \eta_{\eps}^0) = 0.
\eeq
Indeed, Proposition~\ref{Mielke} yields the weaker bound
\begin{equation}\label{Avp_weak}
\big(A_{\eps} \eta, \eta\big) \geq c_{\eps} |\eta|_{L^2}^2, \quad
\forall\, \eta \in H^1(\mathbb{R}), \quad
(\eta, \eta_{\eps}^-)=0, \, (\eta, \eta_{\eps}^0) = 0.
\end{equation}
But using that 
$$
|(M^{-1}\eta,\eta)|\leq |\eta|_{H^{\frac{1}{2}}}|M^{-1}\eta|_{H^{-\frac{1}{2}}}
\leq C |\eta|_{H^{\frac{1}{2}}}^2\leq C  |\eta|_{L^2}|\eta|_{H^1}\,,
$$
we obtain that
\begin{equation}\label{Avp_crude}
\big(A_{\eps} \eta, \eta\big) \geq \tilde{c}_{\eps}\Big( |\eta|_{H^1}^2-C|\eta|_{L^2}^2\Big),\quad
\forall\, \eta \in H^1(\mathbb{R}), \quad (\eta, \eta_{\eps}^-)=0, \, (\eta, \eta_{\eps}^0) = 0.
\end{equation}
A combination of (\ref{Avp_weak}) and (\ref{Avp_crude}) gives (\ref{ortho1}).
Thanks to Proposition~\ref{Mielke}, we get the following crucial property of $L(k)$.
\begin{prop}\label{Lk}
The operator $L(0)$ has a unique simple negative eigenvalue and
for every $k$, $L(k)$ has at most  one simple   negative eigenvalue.
Moreover, 
for every $\eps \in (0, \eps_{0})$, there exists $c>0$ such that for every $k$, we have that
\beq
\label{Lkm}
\big(L(k)U, U \big) \geq c \Big( |U_{1}|_{H^1}^2 +  \Big| { |D_{x}| \over  1 + |D_{x}|^{ 1 \over 2 } } U_{2}
\Big|_{L^2}^2  + { |k|^2  \over 1 +  |k| } |U_{2}|^2_{L^2} \Big), \quad
\eeq
for every $U$ such that
\beq
\label{Unorm}
(U_{1},  \eta_{\eps}^-)=0, \, (U_{1}, \eta_{\eps}^0) = 0 .
\eeq
\end{prop}
Note that the estimate \eqref{Lkm} is  uniform for $k \in [0, + \infty[$. In some part
of the paper, we  shall only need a weaker estimate  uniform for
$k \in [0,K]$ for some $K>0$ fixed.  In this case, we can deduce from  \eqref{Lkm} that
for every $U$ which satisfies  \eqref{Unorm} we have
\beq
\label{Lkmweak}
\big(L(k)U, U \big) \geq c_{1} \Big( |U_{1}|_{H^1}^2 +  \Big| { |D_{x}| \over  1 + |D_{x}|^{ 1 \over 2 } } U_{2}
\Big|_{L^2}^2  + { |k|^2 } |U_{2}|^2_{L^2} \Big), \quad \forall k, \, |k| \leq K,
\eeq
where $c_{1}>0$ depends on $K$.
\begin{proof}[Proof of Proposition~\ref{Lk}]
We shall first prove the estimate \eqref{Lkm}. 
At first, we notice that it suffices to prove  that the weaker estimate
\beq\label{Lkm1}
\big(L(k)U, U \big) \geq c \Big( |U_{1}|_{H^1}^2 +  \Big| { |D_{x}| \over  1 + |D_{x}|^{ 1 \over 2 } } U_{2}
\Big|_{L^2}^2 \Big), \quad  \forall\, U,  \,  (U_{1},  \eta_{\eps}^-)=0, \, (U_{1}, \eta_{\eps}^0) = 0
\eeq
holds and to combine it with a crude estimate to get the
result. Indeed, let us use that 
\beq
\label{IT1}
\big(  L(k)U, U \big)
\geq m    \big(G_{\eps, k } U_{2}, U_{2} \big) -  M   | U_{1}|_{H^1}^2
-  2\big |\big( \partial_{x} \big( (v_\eps -1)U_{1}) , U_{2} \big) \big|\Big)
\eeq
where $m>0$ and $M>0$ are harmless numbers independent of $k$ which will
change from line to line.
To estimate the last term, we use that 
\begin{eqnarray*}  
\big |\big( \partial_{x} \big( (v_\eps -1)U_{1}) , U_{2} \big) \big|
& \leq &  |\partial_{x} U_{2} |_{H^{ - {1 \over 2}  } }\, | (v_\eps - 1) U_{1}|_{H^{ 1 \over 2} }
\leq M  \big|   { | D_{x}| \over  1 + |D_{x}|^{1 \over 2 }  }U_{2} \big|_{L^2}\,   | (v_\eps - 1) U_{1}|_{H^{ 1} } \\
& \leq & M  \big|   { | D_{x}| \over  1 + |D_{x}|^{1 \over 2 }  }U_{2} \big|_{L^2}^2
+ |v_\eps - 1|_{W^{1,\infty}}^2    | U_{1}|_{H^{ 1} }^2\\ 
& \leq  & M \Big(  \big|   { | D_{x}| \over  1 + |D_{x}|^{1 \over 2 }  }U_{2} \big|_{L^2}^2
+   | U_{1}|_{H^{ 1} }^2 \Big).
\end{eqnarray*}
Consequently, we get from \eqref{IT1} that
$$\big(  L(k)U, U \big)
\geq m    \big(G_{\eps, k } U_{2}, U_{2} \big) -  M  \Big(   | U_{1}|_{H^1}^2  +   \big|   { | D_{x}| \over  1 + |D_{x}|^{1 \over 2 }  }U_{2} \big|_{L^2}^2\Big).$$
Next, we can add \eqref{Lkm1} times a large constant and the last estimate to get
$$\big(  L(k)U, U \big) \geq 
m \Big( |U_{1}|_{H^1}^2 +  \Big| { |D_{x}| \over  1 + |D_{x}|^{ 1 \over 2 } } U_{2}
\Big|_{L^2}^2 +  \big(G_{\eps, k } U_{2}, U_{2} \big) \Big).$$
Finally, we can use \eqref{DNm},  which gives in particular that
$$  \big(G_{\eps, k } U_{2}, U_{2} \big) \geq   { |k|^2  \over 1 +  |k| } |U_{2}|^2_{L^2} $$
to get \eqref{Lkm}, assuming \eqref{Lkm1}.

It remains to prove that  \eqref{Lkm1} holds.
Thanks to the monotonicity  of $G_{\eps, k}$ with respect to $k$ proven in Proposition \ref{DirNeum} ii)
and since we also obviously have
$$ \big(- P_{\eps, k_{1}} U_{1},  U_{1} \big) \geq \big ( -P_{\eps, k_{2}} U_{1}, U_{1} \big),
\quad |k_{1}| \geq |k_{2}|,$$
we get in particular that
\beq
\label{Lk1}
\big(  L(k) U, U \big) \geq (L(0) U, U), \quad \forall\, U \in H^1 \times H^{ 1 \over 2 }.
\eeq
Next, thanks to \eqref{L0dec}, we have
\beq
\label{Lk2}(L(0) U, U) = (A_{\eps} U_{1}, U_{1}) +  
\Big( M \big( \partial_{x} U_{2} - M^{-1} (\gamma_{\eps} U_{1}) \Big), 
    \partial_{x}U_{2} - M^{-1}(\gamma_{\eps} U_{1}) \Big).
\eeq 
Consequently,  we can use the assumption that
$$ (U_{1},  \eta_{\eps}^-)=0, \, (U_{1}, \eta_{\eps}^0) = 0$$
and hence \eqref{ortho1} and   the coercivity of $Q$ in  Lemma \ref{M} to get
$$ \big(  L(k) U, U \big) \geq c \Big(  |U_{1}|_{H^1}^2 +   \Big| \partial_{x} U_{2} - M^{-1}( \gamma_{\eps}
U_{1}) \Big|_{H^{- { 1 \over 2 } } } ^2\Big)$$
for some $c>0$ ($c$ actually depends on $\eps$ but we do not care on this dependence here 
and thus we omit it in our notations).
The expansion of the second term and the Cauchy-Schwarz inequality  give
$$  \Big| \partial_{x} U_{2} - M^{-1}( \gamma_{\eps}
U_{1}) \Big|_{H^{- { 1 \over 2 } } } ^2 \geq |\partial_{x} U_{2} |_{H^{-{ 1 \over 2 } } }^2 +
\big| M^{-1}  ( \gamma_{\eps}
U_{1}) \big|_{H^{- { 1 \over 2 } } }^2  - 2   |\partial_{x} U_{2} |_{H^{-{ 1 \over 2 } } }
\big| M^{-1}  ( \gamma_{\eps}
U_{1}) \big|_{H^{- { 1 \over 2 } } }.$$
Consequently, by using the inequality
\beq
\label{Young}
2 ab \leq \delta a^2 + { 1 \over \delta b^2}, \quad \forall\, a, \, b \geq 0, \, \forall\, \delta >0,
\eeq
we get
$$ \Big| \partial_{x} U_{2} - M^{-1}( \gamma_{\eps}
U_{1}) \Big|_{H^{- { 1 \over 2 } } } ^2 \geq (1 - \delta )  |\partial_{x} U_{2}|^2_{ H^{- { 1 \over 2 } } }
- \big( { 1 \over \delta }- 1  \big)    \big| M^{-1}  ( \gamma_{\eps}
U_{1}) \big|_{H^{- { 1 \over 2 } } }^2$$
where $\delta <1$ will be chosen carefully later.
Next, since $M^{-1} \in \mathcal{B}(H^{1 \over 2}, H^{-{1 \over 2 }})$ and
$\gamma_{\eps} \in W^{ 1, \infty}$, we can use the crude estimate
$$  \big| M^{-1}  ( \gamma_{\eps}
U_{1}) \big|_{H^{- { 1 \over 2 } } } \leq C |\gamma_{\eps} U_{1}|_{H^{ 1 \over 2 } }
\leq C  |\gamma_{\eps} U_{1}|_{H^ 1  } \leq C  |U_{1}|_{H^1}$$
for some $C>0$.
This yields
$$ \Big| \partial_{x} U_{2} - M^{-1}( \gamma_{\eps}
U_{1}) \Big|_{H^{- { 1 \over 2 } } } ^2 \geq (1 - \delta )  |\partial_{x} U_{2}|^2_{ H^{- { 1 \over 2 } } }
- C \big( { 1 \over \delta } -1 \big ) |U_{1}|_{H^1}^2$$
and hence we find that 
$$ \big(  L(k) U, U \big) \geq c \Big(
\big(  1- C \big( { 1 \over \delta } - 1 \big)  \big) |U_{1} |_{H^1}^2
+  ( 1 - \delta ) |\partial_{x} U_{2}|_{H^{ - { 1 \over 2 } } }^2 \Big).
$$
To conclude, it suffices to choose $\delta$ in the non-empty interval
$ \big( { C  \over 1 + { C} }, 1 )$ thus for example
$$ \delta = { 1 \over 2 }\Big( 1 +   { C  \over 1 + { C} } \Big).$$
This proves \eqref{Lkm1} which in turn implies \eqref{Lkm}.

Next, we shall 
prove that for every $k$, $L(k)$ has at most one  simple negative eigenvalue. 
Thanks to Proposition~\ref{Mielke}, we notice that the
 quadratic form $(\cdot, A_{\eps} \cdot)$ is nonnegative
  on $(\eta_{\eps}^-)^\perp.$
Thanks to  \eqref{Lk1},   \eqref{Lk2} and Lemma \ref{M}, this yields 
 that the quadratic form $(L(k)\cdot,  \cdot)$ is  nonnegative
  on $(\eta_{\eps}^-, 0)^\perp.$  By contradiction, we obtain
   that $L(k)$ has at most one simple  negative eigenvalue. Indeed, otherwise
    $L(k)$ would have an invariant subspace at least two-dimensional on which
    the quadratic form $(L(k)\cdot,  \cdot)$ is strictly negative. Since this subspace
     must meet $(\eta_{\eps}^-, 0)^\perp$ in a non-trivial way, this yields a contradiction.

  Now, we shall prove that a negative eigenvalue of $L(0)$ indeed exists.
   Since $L(0)$ is self-adjoint with essential spectrum included in $[0,  + \infty[$
    thanks to  Proposition~\ref{essL}, it suffices to prove that there exists $U$ such that
     $(L(0) U, U) <0$. 
   In view of \eqref{Lk2}, a good candidate would be $U= (U_{1}, U_{2})$ such that 
    $$ U_{1}= \eta_{\eps}^-, \quad \partial_{x} U_{2}= M^{-1}(\gamma_{\eps} U_{1}).$$
   Nevertheless, the second equation above does not necessarily have a solution
    $U_{2}$ in $L^2$.  To fix this difficulty we shall define a sequence $U^n=(U^n_1,U^n_2)$ in $L^2$
     and prove that  $(L(0)U^n, U^n)$  becomes negative for $n$ sufficiently large.
    We set 
   $$ U^n_{1}= \eta_{\eps}^-, \quad  U_{2}^n= \chi(n D_{x})  \partial_{x}^{-1} M^{-1}(\gamma_{\eps}  \eta_{\eps}^- )$$
   where $\chi$ is a smooth bounded function such that $\chi(\xi)=0$ on $(-1/2, 1/2)$ and
    $\chi(\xi)=1$  for $|\xi|\geq 1$.  Next, thanks to Lemma \ref{M}, we notice that 
 \begin{eqnarray*}
& &   \Big|  \Big( M \big( \partial_{x} U_{2}^n - M^{-1} (\gamma_{\eps} U_{1}^n) \Big), 
    \partial_{x}U_{2}^n - M^{-1}(\gamma_{\eps} U_{1}^n) \Big) \Big| \\
   &   \leq C  &   \big| \partial_{x } U_{2}^n - M^{-1} (\gamma_{\eps} U_{1}^n ) \big|_{H^{-{1\over 2 } }}^2
    \\
    & \leq C &  \big| \big(\chi(nD_{x}) - 1 \big) M^{-1}(\gamma_{\eps} \eta_{\eps}^-) \big|_{H^{ - { 1 \over 2 } }  }
    ^2
\end{eqnarray*}
and hence, since $M^{-1}(\gamma_{\eps} \eta_{\eps}^-) \in H^{- { 1 \over 2} }$, 
we get from  the dominated convergence theorem  that
$$ 
\lim_{n\rightarrow\infty }   
\Big( M \big( \partial_{x} U_{2}^n - M^{-1} (\gamma_{\eps} U_{1}^n) \Big), 
\partial_{x}U_{2}^n - M^{-1}(\gamma_{\eps} U_{1}^n) \Big)  = 0.
$$
Consequently, thanks to \eqref{Lk2}, we get that $(L(0) U^n, U^n)$ is negative for $n$ sufficiently large.    
This ends the proof of Proposition~\ref{Lk}.
\end{proof}
\section{ Study of the linearized about the solitary wave operator $JL(k)$. Transverse linear instability.}
\label{sectionJLk}
As in \cite{RT2}, we shall say that the  linearized equation \eqref{lin0} or equivalently \eqref{linsource}
has an unstable eigenmode with transverse frequency $k$ and amplification parameter
$\sigma $ with $\mbox{Re }\sigma >0$  if  there is a non-trivial solution of
\beq\label{lininst}
\partial_{t} V = J L V
\eeq
under the form
\beq
\label{instdef}
V(t,x,y)= e^{\sigma t } e^{ik y} U(x)
\eeq
with $U \in H^2 \times H^1$. This provides a non-trivial solution of \eqref{lin0} via (\ref{vrazka}).
By substitution of the ansatz \eqref{instdef} in the equation \eqref{lininst}, we
get the resolvent equation
\beq
\label{resolvent}
\sigma U= JL(k) U.
\eeq
Note that if there is a solution $U\in H^{k+ 1}  \times H^k $, we find
from the first equation of the system 
that
$ G_{\eps, k}  U_{2} \in H^k$ thus since $G_{\eps,k}$ is  a first order elliptic
operator (see Corollary~\ref{elliptic}), we  get
that  $U_{2} \in H^{k+1}$.  Next, the second equation gives
that $P_{\eps, k} U_{1} \in  H^{k}$ and hence since $P_{\eps, k}$
 is a second order elliptic operator, we get that
  $U_{1}\in H^{k+2}$. Consequently, one can get by induction
   that an unstable eigenmode $U$ is necessarily smooth, 
    $U \in H^\infty\times H^\infty$.

\subsection{Location of unstable eigenmodes}         
We start with a Lemma  which gives a  crucial preliminary information on the possible
solutions of (\ref{resolvent}).
\begin{lem}\label{GSS}
For every $\eps \in (0, \eps_{0}), $ 
$\sigma(JL(0) )\subset i \mathbb{R}$.
Moreover, for every $k\neq 0$,  $JL(k)$ has at most one unstable eigenmode
which is necessarily simple. Finally, if (\ref{resolvent}) has an unstable
mode with amplification parameter
$\sigma $ then $\sigma\in \mathbb{R}$.
\end{lem}
\begin{proof}
The first part  is a direct consequence of the  one-dimensional stability result of Mielke \cite{Mielke}.
For the second part, we follow Pego-Weinstein \cite{PW}. 
Suppose that there exist linearly independent $u_1$ and $u_2$ such that 
$JL(k)u_j=\sigma_j u_j$, ${\rm Re}(\sigma_j)>0$, $j=1,2$. Set
$v_j(t)\equiv e^{\sigma_j t}u_j$. 
Thus $\pa_{t}v_j=JL(k)v_j$. Next we observe that thanks to the  symmetry of $L(k)$
and the skew-symmetry of $J$, we have
$
\pa_{t}(L(k)v_1(t),v_2(t))=0$.
This implies that for every real $t$,
$
e^{t(\sigma_1+\overline{\sigma_2})}(L(k)u_1,u_2)=(L(k)u_1,u_2).
$
Using that ${\rm Re}(\sigma_j)>0$, we obtain that $(L(k)u_1,u_2)=0$.
Since  we know that $L(k)$ has at most one negative direction, we obtain that
there exists a complex number $\gamma$ such that $(L(k)(u_1+\gamma
u_2),u_1+\gamma u_2)\geq 0$. Therefore
$
(L(k)u_1,u_1)+|\gamma|^2(L(k)u_2,u_2)\geq 0\,.
$
By taking the scalar product of $JL(k)u_j=\sigma_j u_j$ by $L(k)u_j$ and
taking the real part, we obtain that $(L(k)u_j,u_j)=0$. Therefore
$u_1+\gamma u_2$ is in the kernel of $L(k)$. This in turn implies that
$\sigma_1 u_1+\gamma\sigma_2 u_2=0$ which is a contradiction. Next we can show
similarly that an unstable eigenvalue can not be of multiplicity higher than
$1$. Indeed, if we suppose that $u_1$ and $u_2$ are  such that
$
JL(k)u_1=\sigma u_1
$
and
$ 
JL(k)u_2=\sigma u_2+u_1
$
then we may consider $v_1$ and $v_2$ defined as
$
v_1(t)=e^{\sigma t}u_1,
$
$
v_2(t)=e^{\sigma t}(u_2+t u_1)
$
and obtain a contradiction as above. Finally, we observe that if
(\ref{resolvent}) has a nontrivial solution for some $\sigma$ and $k$, then by taking
the complex conjugate, we obtain that (\ref{resolvent}) has a nontrivial
solution with the same $k$ and $\sigma$ replaced by $\bar{\sigma}$. If
$\sigma$ is not real this contradicts the previous analysis which showed that
for each $k$ there is at most one $\sigma$ such that (\ref{resolvent}) has a nontrivial
solution. This completes the proof of Lemma~\ref{GSS}.
\end{proof}
In the next lemma, we give a further localization where  unstable eigenmodes
must be sought. 
We have the following statement giving further information on the location of
the possible unstable eigenmodes.
\begin{prop}[Location of unstable eigenmodes]\label{brute0cor}
We have the following information on the location of unstable eigenmodes:     
\begin{itemize}
\item[i)]There exists $K>0$ such that if $|k|>K$ then there is no  
unstable eigenmode with transverse frequency $k$ and amplification parameter $\sigma$ satisfying 
$\mbox{Re}(\sigma)  >0$.
\item [ii)] There exists $M>0$ such that for every $k$ , $ |k|\leq K$,  there is no unstable
eigenmode with transverse frequency $k$ and with amplification parameter $\sigma$ satisfying 
$\mbox{Re}(\sigma) \geq M$.
\end{itemize}
\end{prop}
\begin{proof}[Proof of Proposition~\ref{brute0cor}]
By taking the  scalar product of \eqref{resolvent} by $L(k)U$ and then taking the real part,  we get that
$$ \mbox{Re}(\sigma) (L(k)U, U ) =0$$ and hence if $\mbox{Re } \sigma >0$, this yields 
\beq
\label{large1}
(L(k) U, U )=0.
\eeq
Next, we get  by a very crude estimate that 
$$(L(k) U, U) \geq c|U_{1}|_{H^1}^2 + k^2 |U_{1}|_{L^2}^2 + (G_{\eps, k } U_{2}, U_{2})
- C \big(  | U_{1}|_{H^1} \,  |U_{2}|_{L^2} +  |U_{1}|_{L^2}^2\big)$$
where  $c>0$, $C>0$ are  independent of $k$.
Thanks to \eqref{DNm}, for $k$  large (actually $k\geq 1$ is sufficient), we have
$$ (G_{\eps, k } u, u ) \geq c \Big(   \Big| {  |D_{x} | \over 1 +  |D_{x}|^{1 \over 2} } u \Big|_{L^2}^2
+  |k|
|u|_{L^2}^2\Big).$$
Consequently, we obtain
\beq
\label{Lbelow}
(L(k) U, U) \geq c\Big( |U_{1}|_{H^1}^2 + k^2 |U_{1}|_{L^2}^2 +  \Big| {  |D_{x} | \over 1 +  |D_{x}|^{1 \over 2} } u \Big|_{L^2}^2
+ |k| |U_{2}|_{L^2}^2\Big)
- C \big(  | U_{1}|_{H^1} \,  | U_{2}|_{L^2} +  |U_{1}|_{L^2}^2\big)
\eeq
and hence, thanks to a new use of \eqref{Young}, we easily get that  for $k$ sufficiently large
\beq
\label{klarge}
(L(k) U, U)\geq c\Big( |U_{1} |_{H^1}^2 + |U_{2}|_{H^{1 \over 2 } }^2  \Big).
\eeq
In particular, we get from \eqref{large1} that $U=0$. This proves i).

We turn to the proof of ii).
We use a decomposition of $L(k)$ under the form
\beq\label{Ldec1}
L(k) = L_{0}(k) + L_{1}, 
\eeq
where
\beq
\label{Ldec2}
L_0(k) = \left( \begin{array}{cc}
- P_{\eps, k } + \alpha  & 0 \\ 0 & G_{\eps, k } \end{array}
\right), \quad  L_{1} = \left(\begin{array}{cc}
(v_{\eps} - 1) \partial_{x} Z_{\eps} & (v_{\eps} - 1 )\partial_{x}  \\
- \partial_{x} ( (v_{\eps} - 1 ) \cdot) &  0 \end{array} \right).
\eeq
Note that $L_{0}$ is a real-symmetric operator. 
By taking the scalar product of \eqref{resolvent} with $L_{0}(k) U$, we find
\beq\label{large1_bis}
\mbox{Re}(\sigma) (L_{0}(k) U, U )  = {\rm Re } \, \big(J L_{1} U, L_{0}(k)U \big).
\eeq
By using an integration by parts and
\eqref{DNm}, we get that
\beq
\label{L0m} 
(L_{0}(k)U, U )  \geq c\Big( |U_{1}|_{H^1}^2 +  \Big| { |D_{x}|\over  1 + |D_{x}|^{1 \over 2 }   } U_{2}
\Big|_{L^2}^2 +  |k |^2 \, |U_{2}  |_{L^2}^2 \Big)
\eeq
for some $c>0$.
To estimate the right-hand side of \eqref{large1_bis}, we  need  to estimate the following
quantities:
\begin{eqnarray*}
& &  I=\big| {\rm Re  }\, \big(  (v_{\eps} - 1 ) (\partial_{x} Z_{\eps})  U_{1}, G_{\eps, k } U_{2}
\big)  \big|
,  \\
& & II=  \big|{\rm Re }\, \big( v_{\eps} -1) \partial_{x} U_{2}, G_{\eps, k} U_{2} \big) \big|
,  \\
& & III=  \big| {\rm Re }\, \big( - \partial_{x} ( (v_{\eps} - 1) U_{1}) , (P_{\eps, k} + \alpha) U_{1}
\big)\big|.
\end{eqnarray*}

The term $I$ is easy to bound, it suffices to use \eqref{DNC}
to get
$$ I \leq C  |U_{1}|_{H^1} \, \Big(  \Big| {| D_{x}|\over  1 + |D_{x}|^{1 \over 2 }   } U_{2}
\Big|_{L^2} +  |k | \, |U_{2}  |_{L^2} \Big).
$$
The estimate of  $II$ follows from the commutator estimate of Proposition \ref{com}    which yields
$$ 
II \leq C    \Big(  \Big| {| D_{x}|\over  1 + |D_{x}|^{1 \over 2 }   } U_{2}
\Big|_{L^2}^2 +  |k |^2 \, |U_{2}  |_{L^2}^2\Big).
$$
Finally, by using integration by  parts, we also easily get that
$$ III \leq C |U_{1} |_{H^1}^2.$$
We have thus proven that
\beq\label{L0L1}
\big| \big(J L_{1} U, L_{0}(k)U \big) \big| 
\leq C\Big( |U_{1}|_{H^1}^2 +  \Big| { |D_{x}|\over  1 + |D_{x}|^{1 \over 2 }   } U_{2}
\Big|_{L^2}^2 +  |k |^2 \, |U_{2}  |_{L^2}^2\Big).  
\eeq
Consequently,  we obtain from \eqref{large1_bis}, \eqref{L0m}
that
$$  c \, \mbox{Re}(\sigma) \Big( |U_{1}|_{H^1}^2 +  \Big| { |D_{x}|\over  1 + |D_{x}|^{1 \over 2 }   } U_{2}
\Big|_{L^2}^2 +  |k |^2 \, |U_{2}  |_{L^2}^2 \Big)
\leq C \Big( |U_{1}|_{H^1}^2 +  \Big| { |D_{x}|\over  1 + |D_{x}|^{1 \over 2 }   } U_{2}
\Big|_{L^2}^2 +  |k |^2 \, |U_{2}  |_{L^2}^2 \Big)$$
for some constant $C>0$, depending on $K$ but independent of $\sigma$. 
This yields that $U=0$  if $\mbox{Re }(\sigma)$ is sufficiently large. This proves ii).
This completes the proof of Proposition~\ref{brute0cor}.
\end{proof}
\subsection{Existence of an unstable eigenmode}
\begin{theoreme}[Linear instability]\label{mode}
There exists $\sigma>0$ and $k\neq 0$  and a nontrivial $U\in H^2 \times H^1$ such that
$$ JL(k) U= \sigma U.$$
\end{theoreme}
To prove the existence of an unstable eigenmode, we shall follow the
general method presented in \cite{RT2}. Note  that this result was proven in \cite{GHS} 
by using a different formulation of the water waves equations.
\begin{proof}[Proof of Theorem~\ref{mode}]
Let us  set
$
M(k)\equiv J L(k)J.
$ 
Note that
\beq
\label{simetria}
(M(k) u, u ) = - (L(k) Ju, Ju), \quad \forall\, u\in H^1(\mathbb{R}) \times H^{1\over 2 }(\mathbb{R}).
\eeq
Moreover,  since  $J$  is  invertible matrix,  we get from Proposition~\ref{Lk}
that there exists $\lambda >0$  and $v, w \,  \in H^1 \times H^{ 1 \over 2}$ such that
$ M(0) v= \lambda v$,  $M(0) w = 0$ and that for some $c>0$
\beq
\label{estM(0)}
(M(0)z,z)\leq -c\Big(|z_1|_{H^1}^2+\Big| {  |D_x| \over 1 +   |D_{x}|^{1 \over 2 } } z_2\Big|_{L^2}^2\Big)\,,
\quad \forall\, z=(z_1,z_2), \quad (z, v)=0, \, (z,w)=0.
\eeq
Next we set
$$
f(k)\equiv \sup_{
z\in H^1\times H^{1/2}, |z|_{L^2\times L^2}= 1}(M(k)z,z)\,.
$$
Since $(M(0)v, v)>0$, we already know that $f(0)>0$.
Moreover, from Proposition~\ref{DirNeum} ii) and the obvious monotonicity  of $ P_{\eps, k}$,
we get that $M(k)$  is strictly decreasing in $k$ on $\mathbb{R}_{+}$ as a symmetric operator.
In particular, this yields  that $f(k)$ is decreasing on $\mathbb{R}_{+}.$
By using (\ref{simetria}) and \eqref{klarge}, we obtain that for $k\gg 1$ one
has $f(k)<0$. Therefore, since $f$ is continuous,  there exists a smallest $k_{0}>0$ such that $f(k_{0})=0$. The
following lemma is the key element in the proof of the existence of an
unstable mode.
\begin{lem}\label{ker}
There exists $k_1\in (0,k_0]$ such that $\dim({\rm Ker}(M(k_1))=1$. Moreover, for every $F\in
L^2\times L^2$ orthogonal to ${\rm Ker}(M(k_1))$ there exists a unique   $v\in
H^2\times H^1$ orthogonal to ${\rm Ker}(M(k_1))$ such that $M(k_1)v=F$.
\end{lem}
\begin{proof}
Thanks to  Proposition~\ref{essL}, we already know
that the essential spectrum of $M(k)$ is in $\mathbb{R}_{-}$ and for $k\neq 0$ it
is even included in $(-\infty, - m)$ for some $m>0$. Indeed, we have
$$ M(k) = -J L(k) J^{-1}.$$

Thanks to the classical variational characterization of the largest eigenvalue,
and  since $f$ is strictly decreasing, we  get that for $k \in (0, k_{0})$, the largest eigenvalue
of $M(k)$ is positive.  Moreover,  thanks to Proposition~\ref{Lk}, 
we further 
obtain that there is exactly one positive eigenvalue of $M(k)$ for $k \in (0, k_{0})$.
Indeed, if there were at least  two positive eigenvalues of $M(k)$ for some
positive $k$, we would get  by the monotonicity of $M(k)$ that $M(0)$ is
positive on a subspace  of dimension at least two and this contradicts the fact that
$L(0)$ has a unique simple negative eigenvalue. 

If for some $k\in (0, k_{0})$ the kernel of $M(k)$ is non-trivial then,  the value $k_{1}$ that 
we are looking for is this $k$. If  for every $k\in (0, k_{0})$, the kernel of $M(k)$
is trivial  then we have $k_{1}= k_{0}$. Indeed, by definition of $k_{0}$, the kernel of 
$M(k_{0})$ is non-trivial and by the classical variational characterization of eigenvalues, 
$0$ is the largest eigenvalue of $M(k_{0})$. Moreover, this eigenvalue is simple. Indeed,  for
$0<k<k_{0}$, we have that $M(k)$ is strictly negative
on a subspace of codimension $1$ (this is a consequence of
the fact that the kernel of  $M(k) $  is assumed to be  trivial  for $k<k_{0}$  and of the fact
that $M(k)$ has a unique positive eigenvalue for $k, $ $0<k<k_{0}$). Consequently,  by monotonicity, we
get that $M(k_{0})$ is also strictly negative on a subspace 
of codimension $1$. Consequently, it cannot vanish on  a subspace of
dimension at least two.    

The second part  in the statement of the lemma is a consequence of the fact that $M(k_{1})$ is
symmetric and  Fredholm index 
zero. Indeed,  $0$ is not in the essential spectrum of $M(k_{1})$
thanks to Theorem \ref{essL} since $k_{1}\neq 0$.
This completes the proof of Lemma~\ref{ker}.
\end{proof}
We next finish the proof of Theorem \ref{mode}, i.e. we prove  the existence of an unstable mode.
Since $J$ is invertible, it is equivalent to prove that
there exist $k\neq 0$, $\sigma>0$ and $u\in H^2\times H^1$ different from zero such that
$$
M(k)u=\sigma Ju\,.
$$
Let $k_1$ be the number defined in Lemma~\ref{ker} with corresponding kernel
spanned by $u$.
We need to solve $F(v,k,\sigma)=0$, with $\sigma>0$, where $F(v,k,\sigma)\equiv M(k)v-\sigma Jv$. We have that 
$F(u,k_1,0)=0$. We look for $v$ as $v=u+w$, where 
$$
w\in {u}^{\perp}\equiv \{v\in H^2\times H^1\,:\, (v,u)=0\}.
$$
Therefore we need to solve $G(w,k,\sigma)=0$ with $\sigma>0$, where
$$
G(w,k,\sigma)=M(k)u+M(k)w-\sigma Ju-\sigma Jw,\quad w\in {u}^{\perp}\,.
$$
We have that
$$
D_{v,k}G(0,k_1,0)[w,\mu]=\mu\Big[\frac{d}{dk}M(k)\Big]_{k=k_1}u+M(k_1)w\,.
$$
By, using Lemma~\ref{ker}, we shall  obtain that $D_{v,k}G(0,k_1,0)$ is a
bijection form $ {u}^{\perp} \times \mathbb{R}$ to $L^2\times L^2$ if we establish that
$$
\Big(\Big[\frac{d}{dk}M(k)\Big]_{k=k_1}u,u\Big)< 0\,.
$$
By explicit computation, we have
$$ 
\Big[\frac{d}{dk}M(k)\Big]_{k=k_1}=
J \left( \begin{array}{cc}  2 k_{1} \big(1+  ( \partial_{x} \eta_{\eps})^2 \big)^{  - {1 \over 2 }  }  & 0 \\
0 & \Big[\frac{d}{dk} G_{\eps, k }\Big]_{k=k_1} \end{array} \right)  J.
$$
From  Proposition \ref{DirNeum} ii), we have that
$ \Big[\frac{d}{dk} G_{\eps, k }\Big]_{k=k_1}$ is a nonnegative operator. Therefore, we obtain
$$  \Big(\Big[\frac{d}{dk}M(k)\Big]_{k=k_1}u,u\Big) \leq - 2 k_{1} \int_{\mathbb{R}} { |u_{2}|^2
\over  \big(1+  ( \partial_{x} \eta_{\eps})^2 \big)^{   {1 \over 2 }  }  } \, dx <0$$
Indeed, we have from the structure of $L(k)$ that    $u_{2}$ does not vanish identically: assume
that $u_{2}$ vanishes identically, then  $M(k_{1}) u = 0$
gives that $G_{\eps,k_{1}}u_{1} = 0$  and hence from Proposition \ref{DirNeum} iii), we get
$u_{1}=0$ which is impossible.

Consequently, we have shown 
that $D_{v,k}G(0,k_1,0)$ is a
bijection form $ {u}^{\perp} \times \mathbb{R}$ to $L^2\times L^2$ and 
we can apply the implicit function theorem,
in order to  complete  the proof of Theorem~\ref{mode}.
\end{proof}
\subsection{Essential spectrum of $JL(k)$}
\begin{prop}\label{essJL}
For $\eps \in (0, \eps_{0})$, 
the essential spectrum  of $JL(k)$ is included in  $i \R$,  for every $k$.
\end{prop}
\begin{proof}
We can use the most restrictive definition for the essential spectrum i.e.
following \cite{Henry}, we say that $\lambda$ is not in the essential spectrum if $\lambda$
is an isolated eigenvalue of finite multiplicity. 
Note that since  we already know by Lemma~\ref{GSS} that $JL(k)$ has at most one unstable eigenvalue, 
it still  suffices following \cite{Henry}  to prove that $\lambda - JL(k)$ is Fredholm with zero index
for $\mbox{Re}(\lambda) \neq 0$. Moreover, the case $k=0$ is  already given by Lemma~\ref{GSS},
consequently, it suffices to consider the case $k \neq 0$ only.
We shall proceed in a similar way  as in the proof of Proposition~\ref{essL}.
Since we are in the case $k \neq 0, $  we can use the decomposition \eqref{decL}
for $\gamma = 0$. This yields
$$ L(k) = A_{1} A_{2}( k ) L_{4}(k) B_{2}(k) B_{1} + \mathcal{K}$$
where, using the notation 
$$ A_{2}( k ) = A_{2}(0, k), \quad B_{2}(k) = B_{2}(0, k ) , \quad L_{4}(k) = L_{4}(0, k)$$
we have that  $A_{1}$, $A_{2}$, $B_{1}$, $B_{2}$ are bounded invertible operators
and that $\mathcal{K}$ is a relatively compact perturbation. We thus have
$$ \lambda - JL(k) = \lambda - J  A_{1} A_{2}( k ) L_{4}(k) B_{2}(k) B_{1}  + \tilde{K}$$
where $\tilde{K}$ is a relatively compact perturbation. Next, we can write that
\begin{eqnarray*}
\lambda - JL(k)
=  & &J A_{1} J^{-1}\big( \lambda - J  A_{2}  L_{4} B_{2} \big)  B_{1} \\
& & + \lambda  J A_{1}\big(    A_{1}^{-1} - {\rm Id} \big)J^{-1} 
+ \lambda  J A_{1} J^{-1}\big(  {\rm Id}  -   B_{1}  \big) + \tilde{K}.
\end{eqnarray*}
Since the matrices  $ A_{1}^{-1} - {\rm Id}$ and ${\rm Id} - B_{1}$  have exponentially decreasing
coefficients, we  find again that 
$$ \lambda  J A_{1}\big(    A_{1}^{-1} - {\rm Id} \big)J^{-1}  + 
\lambda  J A_{1} J^{-1}\big(  {\rm Id}  -   B_{1}  \big)$$
is a relatively compact perturbation. Consequently, to prove Proposition~\ref{essJL}, 
it suffices to prove that   $\lambda - J  A_{2}  L_{4} B_{2}$ is invertible
for $\mbox{Re } \lambda \neq  0$.

The operator $L_{4}$ has been studied in the proof of Proposition~\ref{essL}.
We have proven that its spectrum is included in $( 0, + \infty )$. Since  it  is moreover a symmetric operator, we 
 obtain that
\beq
\label{L31}
(L_{4} U, U ) \geq c  |U|_{L^2}^2, \quad \forall\, U \in H^2 \times H^1
\eeq
for some $c>0$.
 Since by an integration by parts and  \eqref{DNm}, we have
\beq
\label{L32}
(L_{4} U, U ) \geq c|\partial_{x} U_{1}|_{L^2}^2  + c |U_{2}|_{H^{1 \over2}}^2   - C |U_{1}|_{L^2}^2,
\eeq
we can combine the two estimates \eqref{L31}, \eqref{L32} to get
\beq
\label{L33}
(L_{4} U, U ) \geq c_{0} \Big(|U_{1}|_{H^1}^2 + |U_{2}|_{H^{1 \over 2 }}^2 \Big), 
\quad \forall\, U \in H^1 \times H^{1\over 2}
\eeq
 for some $c_{0}>0$.
Finally, let us notice that  $A_{2}$ and $B_{2}$ are bounded  invertible operators 
on $H^s$ for every $s$ and that  $B_{2}= A_{2}^*$.  This implies that
the operator $ L_{5}(k) = A_{2} L_{4} B_{2}$ is a symmetric operator which satisfies
 thanks to \eqref{L33}
\beq
\label{L4}
(L_{5} U, U ) \geq c \Big(|U_{1}|_{H^1}^2 + |U_{2}|_{H^{1 \over 2 }}^2 \Big), \quad 
\forall\, U \in H^1 \times H^{1\over 2}
\eeq
for some $c>0$. 
We shall use this property of $L_{5}$ to prove that the operator
 $\lambda - J L_{5}$ is invertible  if $\mbox{Re }\lambda \neq 0$.
For $F \in L^2 \times L^2$, we want to prove that  the equation
\beq
\label{reseq}
\big(  \lambda -JL_{5}(k) \big)U = F \eeq
has a unique solution $U \in H^2 \times H^1.$
The estimate \eqref{L4} immediately gives that there is at most one solution.
Indeed, if 
$$\big( \lambda - JL_{5}\big) U=0, $$
by taking the scalar product  and the real part with $L_{5}U$, we  find
 that
 $$ \mbox{Re}\,(\lambda) \, (L_{5}U, U) = 0$$
  and hence since $ \mbox{Re}\,\lambda \neq 0$, we get from \eqref{L4}
   that $U=0$.

To prove that there exists a solution to \eqref{reseq}, we shall use  a classical
approach  based on a duality argument combined with an a priori bound
which is  typically  used  in the context of evolution equations.
To solve \eqref{reseq} for $F\in L^2 \times L^2$, we look $U$ under the form
$U=JV$ and thus we need to solve
$$
AV=J^{-1}F,\quad A=\lambda{\rm Id}-L_{5}J\,.
$$
Since $A^*=\bar{\lambda}{\rm Id}+JL_{5}$, we get
\begin{equation}\label{bound-hb}
|A^{\star}V|_{H^1\times H^{1/2}}\geq c|V|_{H^1\times H^{1/2}}\,.
\end{equation}
Indeed, it suffices to consider $(A^\star V,L_5V)$ and apply \eqref{L4}.
Next, we define $\mathcal{F}$ as
$$
\mathcal{F}=\{U\in H^1\times H^{1/2}\,:\, \exists\, V\in  H^1\times H^{1/2}  ,\, A^\star(V)=U  \}\,.
$$
Thanks to (\ref{bound-hb}), $\mathcal{F}$ is a closed
set of $H^1\times H^{1/2}$. Indeed, let $U_n\in \mathcal{F}$ that  converges to some
limit $U$ in $H^1\times H^{1/2}$. Then there exists $V_n\in H^1\times H^{1/2} $
such that $U_n=A^\star(V_n)$ and thanks to (\ref{bound-hb}) $V_n$ converges to
some limit $V$ in  $H^1\times H^{1/2}$. In particular $A^{\star}(V_n)$
converges in $H^{-2}\times H^{-2}$ to $A^\star V$ which allows to identify $U$ and
$A^\star V$, i.e. $U=A^\star V$ and thus get that  $U\in \mathcal{F}$.  We have thus  proven
that $\mathcal{F}$ is a closed
set of $H^1\times H^{1/2}$. Now, we define the linear form $l: H^1\times H^{1/2}\rightarrow {\mathbb C}$ as
$$
l(U)=
\left\{
\begin{array}{l}
(J^{-1}F,V),\quad{\rm if}\quad U\in  \mathcal{F}\quad {\rm with}\quad U=A^{\star}V,
\\
0,\quad{\rm if}\quad U\in \mathcal{F}^{\perp}
\end{array}
\right.
$$
Using again (\ref{bound-hb}) and the fact that $\mathcal{F}$ is closed in $H^1\times H^{1/2}$, we
obtain that $l$ is continuous on $H^1 \times H^{1 \over 2 } $  and therefore there exists $V\in  H^{-1}\times
H^{-1/2}$ such that
$$
l(U)=(V,U),\quad \forall\, U\in H^1\times H^{1/2}\,.
$$
If $U=A^{\star}W$ with $ W\in H^{3}\times H^{2}$ then $U\in \mathcal{F}$.
Therefore
$$
(AV,W)=
(V,A^{\star}W)=(J^{-1}F,W),\quad \forall\, W\in H^{3}\times H^{2}   \,.
$$
Hence $AV=J^{-1}F$ and thus $U=JV$ is a solution of \eqref{reseq}. Moreover thanks to the elliptic
regularity $U \in H^2 \times H^1$.
This ends the proof of Proposition~\ref{essJL}.
\end{proof}
As a  consequence of the Lyapounov-Schmidt method, Proposition~\ref{essJL} and Lemma~\ref{GSS}, 
we have the following statement  important  for  future use: 
\begin{cor}
\label{evans}
For every $(\sigma_{0}, k_{0})$, $k_{0} \neq 0$, ${\rm Re}\,\sigma_{0}>0$, $ \sigma_{0}\in \sigma (JL(k_{0})),$
the set 
$$\{ (\sigma, k), \, \sigma \in \sigma(J(L(k)) \}$$
in  a vicinity of  $(\sigma_{0}, k_{0})$
is the graph of  an analytic curve $k \mapsto \sigma(k)$ and
$\sigma(k)$ is   
an eigenvalue of $JL(k)$.
\end{cor}
\section{Construction of an approximate  unstable solution}
\label{sectionUapp}
Let us write the system \eqref{eq11}, \eqref{eq22} under the abstract form
\beq
\label{ww}
\partial_{t} U = \mathcal{F}(U)
\eeq
where
$$ 
U = \left( \begin{array}{ll} \eta \\ \varphi  \end{array}\right) , \quad
\mathcal{F}(U)= \left(\begin{array}{ll} \eta_{x}+ G[\eta]\varphi
\\
\varphi_{x}  - {  1 \over 2 } |\nabla \varphi|^2 + {1\over 2}
{ (G[\eta] \varphi + \nabla \varphi \cdot \nabla \eta )^2 \over 1 + |\nabla \eta |^2}
- \alpha \eta + \beta \nabla \cdot \big(  { \nabla \eta \over  \sqrt{ 1 +  |\nabla \eta |^2} } \big)
\end{array} \right).
$$
We shall also use the notation  $Q= (\eta_{\eps}, \varphi_{\eps})$ for the
solitary wave. Following the method  of Grenier \cite{Grenier} used in our previous 
works \cite{RT1}, \cite{RT2}, the main ingredient in the proof
of  Theorem~\ref{main}  is the construction of  an approximate unstable solution of \eqref{ww} under the form
\beq\label{Vapp}
U= Q+ \delta  U^{a}, \quad U^{a}= \sum_{j=0}^M \delta^j U^j.
\eeq
To measure the regularity of the approximate solution, we introduce for $U=(\eta, \varphi)$ the  ``norm''
$$ 
\|U(t) \|_{E^s}^2 = \sum_{  0 \leq \alpha + \beta+\gamma  \leq s}
\|  \partial_{t }^\alpha \partial^\beta_{x}\partial^{\gamma}_{y} U(t, \cdot) \|_{ L^2(\mathbb{R}^2)}^2.$$
Note that since we shall work in this section with linear problems and   very smooth ($H^\infty$) solutions we do not need for the
moment to emphasize some differences in the regularity of each components of $U$.
\subsection{Construction of $U^{0}$}
In the next proposition we   first construct the leading term  $U^0$ of
the approximate solution \eqref{Vapp} with a maximal growth rate.
\begin{prop}\label{U0}
There exists $U^0 (t,x,y)\in  \cap_{s \geq 0 } E^s$   such that
\beq\label{eqL}
\partial_{t} U^0 = J\Lambda U^0
\eeq
and such that there exist an integer $ m \geq 1$ and $\sigma_0>0$ such that for every $s\geq 0$
\beq
\label{U0est} 
{ 1 \over c_{s}} { e^{\sigma_{0} t }  \over { (1 + t)^{ 1 \over 2 m }}}
\leq \| U^0 (t)\|_{ E^s} \leq c_{s}    { e^{\sigma_{0} t }  \over { (1 + t)^{ 1 \over 2 m }}}, 
\quad \forall\, t \geq 0.
\eeq
Moreover $\sigma_{0}$  is such that  the real part of the amplification
parameter of every unstable eigenmode of \eqref{eqL} is non bigger  than $\sigma_{0}$.
\end{prop}
\begin{rem}
As we shall see in the proof, we can choose $U^0$ under the form
\begin{equation}\label{U0form} 
U^0(t,x,y) = \int_{I} e^{\sigma(k) t } e^{iky } U(k)(x)\, dk, \quad I = I_{0} \cup -I_{0}\,,
\end{equation}
where  $I_{0}\subset (0,\infty)$ is a small interval with left extremity
$k_0\neq 0$ such that $\sigma_{0}= \sigma (k_{0})$ and $U(k)$ is an unstable eigenmode
with transverse frequency $k$.
\end{rem}
\begin{proof}[Proof of Proposition~\ref{U0}]
We first  recall that $U^0$ solves \eqref{eqL} if and only if 
$$ V^0= P^{-1}U^0, \quad P= \left( \begin{array}{cc} 1 & 0  \\ - Z_{\eps} & 1 \end{array} \right) $$
solves
\beq
\label{eqLbis}
\partial_{t} V^0 = JL V^0.
\eeq
Since the matrix $P$ is invertible and does not depend on $t$, it suffices to construct  
a solution $V^0$ of \eqref{eqLbis} which satisfies the estimate
\eqref{U0est}. The first step is to  find the most unstable eigenmode which solves
\beq\label{eqvp2}\sigma U = J L(k)U\eeq
i.e., we are looking for the largest  $\sigma$  such that 
$\sigma$ is an eigenvalue of $JL(k)$. Thanks to Theorem~\ref{mode}, we already
know that there exists $k_{0} \neq 0$ such that $JL(k_{0})$ has a nontrivial  unstable eigenvalue.
Thanks to  Proposition~\ref{brute0cor},  we also  know that unstable
eigenmodes must be sought only for transverse frequencies $k$ such that 
$k \in [0,K]$.  Moreover,  for $k \in [0,K]$, we have that 
the amplification parameter $\sigma$ of the possible unstable eigenmodes
should  be real and satisfy $\sigma \leq  M$.

Let us  assume that  the  unstable eigenmode given by Theorem \ref{mode}
is such that $\sigma = \delta$. Thanks to the previous remarks, 
the most unstable eigenmode (i.e. with the largest $\sigma$) has to be sought
in the compact set $\mathcal{R}$ of $\mathbb{R}\times\R$ defined by
$$
\mathcal{R}\equiv\big\{(\sigma, k)\,:\,
\delta/2 \leq\sigma \leq M, \,\, |k |\leq K\big\}.
$$
Moreover,   thanks to Corollary \ref{evans},   the set 
$\{(\sigma, k), \, \sigma >0, \, k \neq 0, \quad \sigma \in \sigma(JL(k)\}$
is locally  the graph of an analytic curve. 
If we define $\Omega = \big\{k, \, \exists\,\sigma, \, \sigma>\delta/2,\, 
\sigma \in \sigma(JL(k))  \big\}$,  we thus  get  that $\Omega$ is a bounded
(and non empty)  open  set  of 
$\mathbb{R}$. One can decompose $\Omega$ as $\Omega= \cup_{m} I_{m}$ where
$I_{m}$ are disjoint, open and bounded  
intervals which are the connected components of $\Omega$. On each $I_{m}$ the above considerations
prove that there exists an analytic function $k \mapsto \sigma(k)$ such
that $\sigma(k)$ is the only  eigenvalue of $JL(k)$  in $\sigma >0.$ 
We shall prove next that $k \mapsto  \sigma(k)$ has a continuous
extension to $\overline{I_{m}}$. Indeed, if  $k_{n}$ is a sequence converging to an extremity $\kappa$ of  $I_{m}$, 
since $\sigma(k_n)$ is bounded ($\sigma (k_{n}) \in \mathcal{R}$), then we can
extract a subsequence not relabelled such that $\sigma(k_{n})$ tends to some $\sigma$. Moreover, we also 
have $\sigma\geq \delta /2$,  and  $\sigma \in \sigma(JL(\kappa))$ since
$JL(k)$ depends continuously on $k$. Thanks to Proposition~\ref{essJL}, $\sigma$ is actually an eigenvalue
of $JL(\kappa)$ and hence is the only unstable eigenvalue of  $JL(\kappa)$,
thanks to Lemma~\ref{GSS}.  By uniqueness of the limit, we get that  
$\lim_{k \rightarrow \kappa, k \in I_{m}} \sigma(k)= \sigma$ and hence, we can define  a continuous function on $\overline{I_{m}}$.
Finally, we also notice that if $\partial I_{m} \cap \partial I_{m'}
\neq\emptyset$, then the continuations must coincide again thanks to the fact
that there is at most one unstable eigenmode. Consequently, we have actually a
well-defined continuous function $k \rightarrow \sigma(k)$ on  $\overline{\Omega}$
which is a compact set. This allows to define $k_0$ and $\sigma_0$ by 
$$ 
\sigma_{0} \equiv \sigma(k_{0})= \sup\{\sigma(k), \quad k \in \overline{\Omega}\}>0
$$ 
($k_0$ is not necessarily unique). Note that  $k_{0}\neq 0$ thanks to Lemma \ref{GSS}.
Moreover, $ \sigma(k)$ is an analytic function in the vicinity of $
k_{0}$  and hence,  there exists $m\geq 2$ so that
\begin{equation}\label{nondeg}
\sigma'(k_0)=\cdots=\sigma^{(m-1)}(k_0)=0,\quad \sigma^{(m)}(k_0)\neq 0.
\end{equation}
Let $I_0\subset\Omega$ be an interval containing $k_0$ which does meet zero.
For $k\in I_0$, let us denote by $U(k)$ the unstable mode corresponding to transverse
frequency $k$ and amplification parameter $\sigma(k)$.

Taking $I_0$  sufficiently small, one can take a  smooth curve $k \mapsto U(k)
\in H^\infty$ which is continuous from $I_0$ to $H^s$ for every $s$. Indeed,
by continuity of $k \mapsto \sigma(k)$, we can choose a disk $B(\sigma(k_{0}), r) \subset \{ \mbox{Re } \sigma>0 \}$
such that for every $k \in I_{0}$,  $\sigma(k)$ belongs to the interior of the
disk. In particular  on $\partial B(\sigma(k_{0}), r)$, there is no eigenvalue of 
$JL(k)$ for $k \in I_{0}$ and hence thanks to Proposition~\ref{essJL}, we get
that the resolvent $(JL(k)-\sigma)^{-1}$  of $JL(k)$ is well defined for $(\sigma, k) \in  \partial B(\sigma(k_{0}), r)
\times \overline{I_{0}}$. Consequently, the eigenprojection on the only
unstable eigenmode with transverse frequency $k$ can be written under the form 
$$
P(k) = \frac{1}{2\pi i}
\int_{\partial B(\sigma(k_{0}), r)} (\sigma - JL(k))^{-1} \, d\sigma.
$$
This allows to choose $U(k)$ under the form
\beq\label{defU} 
U(k) =\frac{1}{2\pi i}  
\int_{\partial B(\sigma(k_{0}), r)} (\sigma - JL(k))^{-1}U(k_{0}) \, d\sigma.
\eeq
With this definition, $U(k)$ is non trivial for $k$ in a vicinity of $k_{0}$
and depends smoothly on $k$ since $JL(k)$ depends analytically on $k$ for $k \neq 0$.  
We have that $\sigma(k)=\sigma(-k)$. By the definition \eqref{defU}, we also have $\overline{U(k)}= U(-k)$.
Then we set $I= I_{0} \cup -I_0$ and 
$$
V^0(t,x,y)\equiv \int_{I}\, e^{\sigma(k) t } e^{ i k y }\,U(k) \, dk,
$$
where the dependence in  $x$ of $V^0$ is in $U(k)$. Note that $V^0$  is real-valued by the choice of $I$.
By the Bessel-Parseval identity, we get for every $s, \, \alpha  \in \mathbb{N}$ that 
$$ \|\partial_{t}^\alpha V^0(t, \cdot)\|_{H^s(\R^2)}^2
= C\int_{I} e^{ 2 \sigma(k)t} \sum_{s_{1}+ s_{2}\leq  s}
|\sigma(k)|^{2\alpha}  k^{2s_{2}} | \partial_{x}^{s_{1}}U(k)|_{L^2(\mathbb{R})}^2 \, dk,
$$
where $C$ is an harmless number. Recall that  $\sigma_0\equiv \sigma(k_0)$.
Thanks to  (\ref{nondeg}), we can apply the Laplace method (see e.g. \cite{JD,F}) 
and obtain that  for every $s, \, \alpha \geq 0$ there exists
$c_{s,\alpha}\geq 1 $ such that for every $t\geq 0$
\begin{equation*}
\frac{1}{c_{s, \alpha }}\frac{1}{(1+t)^{\frac{1}{2m}}}e^{\sigma_0 t}
\leq\|\partial_{t}^\alpha V^0(t,\cdot)\|_{H^{s}(\R^2)}\leq \frac{c_{s, \alpha }}{(1+t)^{\frac{1}{2m}}}e^{\sigma_0 t}\,.
\end{equation*}
This completes the proof of Proposition~\ref{U0}.
\end{proof}
\subsection{Construction of $U^{a} $ }
The aim of this section is to  prove the following statement.
\begin{prop}\label{Uap}
For every $M \geq 0$, there exists an expansion
\beq\label{expUap} 
U^{a} =  U^0 + \sum_{j=1}^{M+1} \delta^j U^j, \quad U^j \in
\mathcal{C}^\infty(\mathbb{R}_{+},  H^\infty(\R^2)),\quad \delta\in \R
\eeq
such that for every $j$, $U^{j}(0)= 0$   and  for some $C_{s, j}$ we have  the estimates
\beq\label{Ujest}
\|U^j(t)  \|_{ E^s} \leq { C_{s, j } \over (1+t )^{j+ 1 \over 2 m } } e^{(j+ 1) \sigma_{0}t}, \quad\forall\, t \geq 0.
\eeq
Moreover, $ Q+ \delta U^{a}$ is an approximate solution of \eqref{ww} in the sense
that
\beq
\label{Uapeq}
\partial_{t}\big( Q+ \delta U^{a}\big) - \mathcal{F}(Q + \delta U^{a}) = R^{ap}
\eeq
and there exists $\delta_0>0$ such that for every $\delta \in (0, \delta_0]$, the estimate,
\beq \label{Rap}
\| R^{ap} (t) \|_{ E^s} \leq { C_{M, s} \delta^{M+ 3 } \over ( 1 + t  )^{M+3 \over 2 m } }
e^{(M+3) \sigma_{0} t }, 
\eeq
holds for $t\in [0,T^\delta]$, where $T^\delta$ is such that
$$
\frac{e^{\sigma_0 T^\delta}}{(1+T^\delta)^{\frac{1}{2m}}}=\frac{1}{\delta}\,.
$$
\end{prop}
\begin{proof}[Proof of Proposition~\ref{Uap}]
By using the Taylor expansion of $\mathcal{F}$
$$ \mathcal{F}(Q+ \delta U) = \mathcal{F}( Q)
+ \sum_{ k=1}^{M+2}  { \delta^k \over k ! }  D^k \mathcal{F}[Q]\big(U, \dots, U\big)
+ \delta^{M+3} R_{M, \delta}(U),$$
we can plug  the expansion \eqref{expUap} into  the equation \eqref{ww} and
identify the terms in front of each power of $\delta$ to get for every $j \geq 1$
\beq\label{eqUk}
\partial_{t} U^j - J \Lambda U^j = \sum_{p=2}^{j+1} \sum_{
\stackrel{ 0 \leq l_{1}, \dots, l_{p}\leq M}{l_{1}+ \cdots  + l_{p} = j+1 - p} } 
{ 1 \over p!} D^p \mathcal{F}[Q]\big( U^{l_{1}}, \dots, U^{l_{p}} \big).
\eeq
Note that the right hand-side of \eqref{eqUk} involves only the $U^l$ for
$l\leq j-1$. This will allow to solve these equations by induction.
Moreover, thanks to  \eqref{U0form}, the Fourier transform in $y$  of $U^0$
is compactly supported. Consequently, it will be possible to solve the equations
\eqref{eqUk}  with the Fourier transform of  $U^j$ compactly supported in $y$
( in $B(0, R(|j|+1) )$ for example). 
This remark yields the introduction of the following ``norms'' for functions of
$x$ only :
$$ 
|U(t) |_{X^s_{k}}^2 = \sum_{0\leq  \alpha + \beta  \leq s} 
\Big(\big|\partial_{t }^\alpha \partial_{x}^\beta U_{1} (t, \cdot)\big|_{ H^1(\mathbb{R} ) }^2 + 
\big| \partial_{t }^\alpha \partial_{x}^\beta U_{2}\big|_{ \dot{H}^{1\over 2}_{k}(\R)}^2\Big),
$$
with the definition 
$$   
|\varphi|_{ \dot{H}^{1\over 2}_{k}(\R)}^2 \equiv
\big| { |D_{x}| \over 1 +  |D_{x}|^{1 \over 2} } \varphi\big|_{L^2(\mathbb{R})}^2 +
|k|^2   |   \varphi |_{  L^2(\mathbb{R}) }^2    .
$$
Note that the  ``semi-norm''  $ H^1 \times  \dot{H}^{1\over 2 }_{k}$ is the  "energy norm" which is   
naturally  associated to the operator $L(k)$. Note also that it does not give any control on
the $L^2$ norm of $U_{2}$ for $k=0$. In the sequel $k$ will range in a compact
set containing the origin. Thus we will only pay attention to the uniformness of the bounds near
$k=0$ and for that purpose $\dot{H}^{1\over 2}_{k}$ turns out to be quite natural.
The main ingredient towards the proof of Proposition~\ref{Uap} will be the
following result.
\begin{prop}[Semi-group bound for $J \Lambda(k)$]\label{semi-group}
Let us fix $\gamma>\sigma_0$ (where $\sigma_0$ is defined in the proof of
Proposition~\ref{U0}), $\rho>0$  and $s\in \N$. 
For every
every $F(t,x, k)$, $F(\cdot, \cdot, k) \in
C^{\infty}(\R;H^{\infty}(\mathbb{R})) $ satisfying uniformly 
for $|k  | \leq K$ the estimates
\beq\label{Fhyp}
\sum_{ \alpha + \beta \leq s}\|\partial_{t}^\alpha \partial_{x}^\beta
F(t,\cdot, k )\|_{L^2}
\leq \Lambda_{s} \frac{e^{\gamma t}}{(1+ t )^\rho } ,\quad\forall\, t\geq 0,
\eeq
if  $U$ solves
\beq
\label{semieq}
\partial_{t}U=J\Lambda(k)U+F,\quad U(0)=0, 
\eeq
then there exists $C_{s}$ depending only on  $\Lambda_{s+s_{0}}$ for some 
$s_{0}\geq 0$ such that for every $k,$  $|k|  \leq K$,
\beq\label{semiest}
|U(t,\cdot) |_{L^2 } + |U(t,\cdot)|_{ X^s_{k}  } \leq 
C_{s} \frac{e^{\gamma t}}{(1+ t )^\rho },\quad \forall\, t\geq 0\,.
\eeq
\end{prop}
\begin{rem}
{\rm Notice that in particular (\ref{semiest}) provides bounds of $U$ and its time
derivatives in usual Sobolev spaces. These bounds will be used in the
application of Proposition~\ref{semi-group} to the proof of  Proposition~\ref{Uap}.}
\end{rem}
\begin{proof}[Proof of Proposition~\ref{semi-group}]
We shall focus on the proof of the estimate \eqref{semiest} assuming that $U$ is a smooth
solution of \eqref{semieq}. We shall not detail the proof of the existence of  the solution
which can be obtained in a classical way (for example by using the vanishing
viscosity method as in \cite{L1})  once  the a priori estimates \eqref{semiest} are established.
By using again the change of unknown $V=P^{-1}U$, it is equivalent to study the equation
\beq
\label{semieq2}
\partial_{t} V  = JL(k) V +  F, \quad V(0) = 0
\eeq
with  $F$ satisfying the estimates \eqref{Fhyp} and to prove that $V$ verifies the estimate \eqref{semiest}.
We shall first ignore the estimate of the $L^2$ norm of $V_{2}$ and  prove the estimate
\beq\label{semiest1}
|V(t,\cdot)|_{ X^s_{k}  } \leq C_{s} {e^{\gamma t} \over ( 1 + t
)^\rho},\quad \forall\, t\geq 0, \quad \forall\, k,  \, |k| \leq K
\eeq
by induction on $s$. We start with the proof of   the estimate
for $s=0$.
\subsubsection{Proof of \eqref{semiest1} for $s=0$}
By using the Laplace transform, we shall first reduce the proof of the estimate to a resolvent estimate.
Let us fix $T>0$ and introduce $G(t,x,k) $ such that
$$ 
G=0, \, t<0, \quad G=0, \, t>T, \quad G(t,x,k) =F(t,x,k),  \, t \in [0, T].
$$ 
We notice that the solution $\tilde{V}$ of
\beq
\label{tildeV}
\partial_{t} \tilde{V} = JL(k)\tilde{V} + G, \quad  \tilde{V}(0)=0, \quad \forall\, t \geq 0\eeq
verifies
\beq
\label{V=V}
V(\tau,x,k) = \tilde{V}(\tau,x,k), \quad \forall\, \tau  \in 
[0, T].
\eeq
Indeed, $W= V - \tilde{V}$ is a solution of
\begin{equation}\label{eqW_0}
\partial_{t} W = JL(k) W, \quad W(0) = 0,  \quad t \in [0, T].
\end{equation}
By using again  the decomposition \eqref{Ldec1} of  $L(k)$, we get
the energy estimate
\begin{equation}\label{eqW}
{ 1 \over 2 } { d \over dt } (L_{0}(k) W, W )  = {\rm Re }\,  \big( JL_{1} W, L_{0}(k)
W\big), \quad t \in [0, T].
\end{equation}
The right hand side  was already estimated in the proof of  Proposition~\ref{brute0cor} (see \eqref{L0L1}). 
We have proven that
\beq\label{L0L12}
| \big( JL_{1} W, L_{0}(k) W\big) | \leq C  |W|_{X^0_{k}}^2,
\eeq
where $C$ is a constant independent of $t$, $k, \,  |k| \leq K$ and $W\in X^0_{k}$ .
By using the estimate \eqref{DNm} we have also seen in \eqref{L0m}
that for some $c>0$
\beq\label{L0m2}
(L_{0}W, W ) \geq c | W  |_{X_{k}^0}^2, \quad \forall\, k, \,  |k | \leq K.
\eeq
Next, we can integrate \eqref{eqW} in time and use \eqref{L0L12}, \eqref{L0m2} to get
$$ 
| W(t)  |_{X_{k}^0}^2 \leq C  \int_{0}^t  | W(s)  |_{X_{k}^0}^2\, ds, \quad\forall\, t \in [0, T].
$$
By the Gronwall inequality, we get that $ | W(t)  |_{X_{k}^0}^2$
vanishes on $[0, T]$.  This implies that $W_{1}$=0 on $[0, T]$  and
then that $W_{2}=0$ on $[0, T]$ by using the second equation of \eqref{eqW_0}.
Consequently, we  shall  study \eqref{tildeV}.  For some $\gamma_{0}$ such that
\beq\label{gamma0}
\sigma_{0} <\gamma_{0} < \gamma, 
\eeq  
let us set
$$ 
W(\tau,x)= \mathcal{L} \tilde{V}(\gamma_{0}+ i \tau), \quad 
H(\tau, x) = \mathcal{L}G ( \gamma_{0} + i \tau ), \quad (\tau, x) \in \mathbb{R}^2\,,
$$
where $\mathcal{L}$ stands for the Laplace transform in time :
$$ 
\mathcal{L}f(\gamma_{0}+ i  \tau) = \int_{0}^{\infty} e^{- \gamma_{0} t - i \tau\, t} f(t)\, dt.
$$
Since $\tilde{V}(0)=0$, $W$ solves the resolvent equation
\beq
\label{Wres}
(\gamma_{0}+ i \tau) W- JL(k) W  = H(\tau, \cdot).
\eeq
By the choice of $\gamma_{0}$ in \eqref{gamma0}, $\gamma_{0} + i \tau$ is not
in the spectrum of $JL(k)$ for every $k$. Consequently,  $W$ is given by 
$$ 
W= \big((\gamma_{0} + i \tau){\rm Id} - JL(k)  \big)^{-1} H. 
$$
The next step is to obtain an estimate of $W$ uniform in $\tau$. 
We first provide the bound for large values of $\tau$ (note that here we do
not use that $\gamma_0>\sigma_0$). Here is the precise statement.
\begin{lem}\label{res1}
Fix $\gamma_0>0$ and $K>0$.
There exist $M>0$ and $C>0$ such
that for every $|\tau|\geq M$, every $\gamma\geq \gamma_0$ every $f=(f_1,f_2)\in H^{1}\times H^{\frac{1}{2}}$, 
every $|k|\leq K$ if $U=(U_1,U_2)$ solves
\begin{equation}\label{rugbi}
(\gamma+i\tau)U=JL(k)U+f
\end{equation}
then
\begin{equation}\label{sad0}
|U_{1}|_{H^1} +  \Big| { |D_{x}| \over  1 + |D_{x}|^{ 1 \over 2 } } U_{2}\Big|_{L^2} + |k | |U_{2}|_{L^2}
\leq C|f|_{H^1\times H^{1 \over 2 }}\,.
\end{equation}
\end{lem}
\begin{proof}[Proof of Lemma~\ref{res1}]
We shall use Proposition~\ref{Lk}. Let us set $\Phi_{-}= ( \eta_{\eps}^{-}, 0)^t$, $\Phi_{0}= (\eta_{\eps}^0
, 0)^t$, we can moreover assume that 
$\Phi_{-}$ and $ \Phi_0$  are normalized in $L^2\times L^2$.  Then,  
every $U=(U_1,U_2)^{t}\in H^2\times H^1$ can be written as
\begin{equation}\label{her1}
U=\alpha \Phi_{-}+\beta \Phi_0+U^{\perp}, \quad (U^\perp, \Phi_{-})=0, \, (U^\perp, \Phi_{0})=0
\end{equation}
and thanks to \eqref{Lkmweak}, we have for some $c>0$
\begin{equation}\label{her2}
(L(k)U^{\perp},U^{\perp})\geq c\Big(
|U^{\perp}_{1}|_{H^1}^2 +  \Big| { |D_{x}| \over  1 + |D_{x}|^{ 1 \over 2 } } U^{\perp}_{2}\Big|_{L^2}^2
+  |k|^2 |U_{2}^{\perp}|_{L^2}^2 \Big),\quad \forall\, k,  \,\,  |k |\leq K.
\end{equation}
Next, if $U=(U_1,U_2)^{t}\in H^2\times H^1$ is a solution of (\ref{rugbi}) then
\begin{equation}\label{her3}
\gamma (L(k)U,U)={\rm Re} \, (f,L(k)U).
\end{equation}
To estimate the right-hand side, we use that
$$
|(f,L(k)U) | \leq C \Big( \big(|f_{1}|_{H^1} +  |f_{2}|_{L^2}\big) \, |U_{1}|_{H^1} +  
|(f_{2},  G_{\eps, k } U_{2}) | + |\big(f_{1},  (v_{\eps}- 1 ) \partial_{x}U_{2}\big)  |\Big).
$$
Next, by using  \eqref{DNC},  and  that
$$\big| (v_{\eps}- 1)f_{1},  \partial_{x} U_{2}\big)\big| \leq | (v_{\eps }-1 ) f_{1}|_{H^{1\over 2 } } | \partial_{x}U_{2}
 |_{H^{ - {1 \over 2 } } } $$
we get that
$$ |(f_{2},  G_{\eps, k } U_{2}) | + |\big(f_{1}, 
(v_{\eps}- 1 ) \partial_{x}U_{2}\big)  |\leq C \big( |f_{1}|_{H^{1 }} +  |f_{2}|_{H^{1\over2}}
\big)\Big(  \Big| { |D_{x}| \over  1 + |D_{x}|^{ 1 \over 2 } } U_{2}\Big|_{L^2}
+  |k|  |U_{2}|_{L^2} \Big).$$
Consequently, we have shown that
\beq
\label{ressource}
|(f,L(k)U) | \leq C \big(|f_{1}|_{H^1} +  |f_{2}|_{H^{1 \over 2 } }\big)\big( |U_{1}|_{H^1}
+ \Big| { |D_{x}| \over  1 + |D_{x}|^{ 1 \over 2 } } U_{2}\Big|_{L^2}
+  |k|  |U_{2}|_{L^2} \big).
\eeq
Furthermore, using integrations by part and  some crude estimates, 
we can estimate the left hand-side of  \eqref{her3} as follows
\begin{equation}\label{her3bis}
(L(k)U,U)\geq  (L(k) U^\perp, U^\perp)  - C \Big(|\alpha|^2 + |\beta|^2 + (|\alpha|+|\beta|)
\big(|U^\perp_1|_{L^2}+\Big| { |D_{x}| \over  1 + |D_{x}|^{ 1 \over 2 } } U_{2}^\perp \Big|_{L^2}\big)\Big)
\end{equation}
for some $C>0$. Consequently, we can combine  \eqref{her2}, \eqref{her1} with  \eqref{her3}, 
\eqref{her3bis}, \eqref{ressource} to get that
\begin{eqnarray*}
& &|U^{\perp}_{1}|_{H^1} +  \Big| { |D_{x}| \over  1 + |D_{x}|^{ 1 \over 2 } } U^{\perp}_{2}\Big|_{L^2}^2
+ k^2 |U_{2}^\perp|_{L^2}^2  \\
& & \leq  C \Big(|f_{1}|_{H^1} +  |f_{2}|_{H^{1 \over 2 } } +  |\alpha | + |\beta | \Big)
\Big( |U_{1}^\perp|_{H^1}
+ \Big| { |D_{x}| \over  1 + |D_{x}|^{ 1 \over 2 } } U_{2}^\perp \Big|_{L^2}
+  |k|  |U_{2}^\perp|_{L^2} \Big)  + C (|\alpha |^2 + |\beta |^2).
\end{eqnarray*} 
A use of the inequality \eqref{Young} yields  
\beq\label{sad1}
|U^{\perp}_{1}|_{H^1}^2 +  \Big| { |D_{x}| \over  1 + |D_{x}|^{ 1 \over 2 } } U^{\perp}_{2}\Big|_{L^2}^2
+ k^2 |U_{2}^\perp|^2 
\leq C\big(|f|_{H^1\times H^{1 \over 2 }}^2 +|\alpha|^2+|\beta|^2\big).
\eeq
We now take the $ L^2\times L^2$ scalar product of
$$
(\gamma+i\tau)U=JL(k)U+f
$$
with $\Phi_{-}$ and $\Phi_0$ to arrive at
$$
(\gamma+i\tau)\alpha=
-(U,L(k)J\Phi_{-})+(f,\Phi_{-})
$$
and
$$
(\gamma+i\tau)\beta=
-(U,L(k)J\Phi_{0})+(f,\Phi_{0})\,.
$$
By using again   \eqref{DNC} and the fact that $\Phi_{0}$, $\Phi_{-}$
are smooth and fixed,  we have  for $i=0, \,-$, 
$$ |(U,L(k)J\Phi_{i})| \leq C \Big(  |U_{1}|_{L^2} + 
\Big| { |D_{x}| \over  1 + |D_{x}|^{ 1 \over 2 } } U_{2}\Big|_{L^2}
+ |k|  |U_{2}|_{L^2} \Big).$$ 
Therefore, we obtain that
\begin{eqnarray}
\label{sad2}
(\gamma+|\tau|)|\alpha| & \leq & C\big(|U_1|_{L^2}+\Big| { |D_{x}| \over  1 + |D_{x}|^{ 1 \over 2 } } U_{2}\Big|_{L^2}
+|k| \, |U_{2}|_{L^2}+|f|_{L^2\times L^{2}}\big),
\\
\label{sad3}
(\gamma+|\tau|)|\beta| & \leq & C\big(|U_1|_{L^2}+ \Big| { |D_{x}| \over  1 + |D_{x}|^{ 1 \over 2 } } U_{2}\Big|_{L^2}
+ |k |\, |U_{2}|_{L^2}+|f|_{L^2\times L^{2}}\big).
\end{eqnarray}
Combining (\ref{sad1}), (\ref{sad2}) and (\ref{sad3}), we obtain that for $|\tau|$ sufficiently large,
depending on  $K$ and $\gamma_0$, we arrive at the (\ref{sad0}). 
This completes the proof of Lemma~\ref{res1}.
\end{proof}
Using Lemma~\ref{res1}, we get
\beq\label{large}
|W(\tau, \cdot)|_{X^0_{k}} \leq C  |H(\tau, \cdot) |_{ H^1 \times H^{1 \over 2} }, \quad 
\forall\, \tau, \, k,  \, |\tau | \geq M, \, |k | \leq K.
\eeq
Next, we give the argument for $|\tau|\leq M$.
Since, on the compact set $\{\lambda =\gamma_{0} + i \tau, \, |\tau  |\leq M\}$,  there is no spectrum
of $JL(k)$ by  the choice of $\gamma_{0}$, we get  by the continuous dependence of 
$L(k)$ in $k$ that the resolvent
$$ 
\mathcal{R}(\tau, k ) = \big( (\gamma_{0} + i \tau){\rm Id} - JL(k) \big)^{-1}
$$
is uniformly bounded on  $[-M, M]\times [-K, K]$ i.e.
$$ 
|W(\tau)|_{H^2 \times H^1 } \leq C |H(\tau) |_{L^2\times L^2}, \quad \forall\, \tau, \, k ,  \, |\tau| \leq M, \, |k| \leq K .
$$
Consequently, we have in particular proven  the uniform estimate
\beq\label{resolvent_bis}
|W(\tau, \cdot)|_{X^0_{k}} \leq C  |H(\tau, \cdot) |_{ H^1 \times H^{1 \over 2 } }
, \quad \forall\, \tau, \, k,\quad \, |k | \leq K.
\eeq 
By the Bessel-Parseval identity, \eqref{V=V} and \eqref{resolvent_bis}, we get
\begin{eqnarray*}
& & \int_{0}^T e^{-2 \gamma_{0} t } |V(t)|_{X_{k}^0}^2\, dt
\leq \int_{0}^{+ \infty} e^{-2 \gamma_{0} t } |\tilde{V}(t)|_{X_{k}^{0}}^2\, dt
= C \int_{\mathbb{R}} |W(\tau)|_{X_{k}^0}^2 \, d\tau \\
& & \leq C \int_{\mathbb{R}} |H(\tau)|_{H^1 \times H^{\frac{1}{2}}}^2 \, d\tau
= \int_{0}^{T} e^{-2 \gamma_{0} t } |F(t, k )|_{H^1 \times L^2 }^2\, dt
\end{eqnarray*}
and finally thanks to \eqref{Fhyp}, we get that there exists $C>0$ such that for every $T >0$, 
\beq\label{estsource}
\int_{0}^T e^{-2 \gamma_{0} t } |V(t)|_{X_{k}^0}^2\, dt\leq C \int_{0}^T{ e^{2( \gamma - \gamma_{0})t }
\over  ( 1 + t )^{2 \rho}}\, dt
\leq C { e^{ 2( \gamma - \gamma_{0})T } \over  ( 1 + T )^{2 \rho } }
\eeq
since $\gamma_{0}$ was fixed  such that $\gamma> \gamma_{0}$.

To finish the proof, we can use an energy estimate for \eqref{semieq2}.
By using again the decomposition \eqref{Ldec1}, we  get the energy estimate
$$ { 1 \over 2 } { d \over dt } e^{- 2 \gamma_{0} t }\big( L_{0}(k) V, V\big)
=  e^{-2 \gamma_{0} t }  {\rm Re }\,\big( JL_{1} V, L_{0}(k)V  \big) - 2 \gamma_{0} e^{-2 \gamma_{0}t}
\big(L_{0}(k)U, U \big) +  {\rm Re }\, \big( F, L_{0}(k) V\big).$$
Since, by using an integration by parts  and  \eqref{DNC}
we have
\beq
\label{semi1}
\big|  ( F, L_{0}(k) V \big ) \big| \leq C |F|_{X^0_{k}} \, |V|_{X^0_{k}}, \quad
\big(L_{0}(k)U, U \big) \leq C |U|_{X^k_{0}}^2, 
\eeq
a new use of \eqref{L0L12},  \eqref{L0m2}  and \eqref{Fhyp} gives
$$ e^{- 2 \gamma_{0} t }  |V(t) |_{X^0_{k}}^2 \leq C \int_{0}^t e^{- 2 \gamma_{0} t } | V(s) |_{X^k_{0}}^2 \, ds
+  C \int_{0}^t  { e^{ 2  ( \gamma - \gamma_{0}  ) s } \over ( 1 + s )^{2 \rho}  }\, ds. 
$$
Consequently,  we can use \eqref{estsource} to get
$$
e^{- 2 \gamma_{0} t }  |V(t) |_{X^0_{k}}^2 \leq C { e^{ 2( \gamma - \gamma_{0})t } \over  ( 1 + t )^{2 \rho } }.
$$
This ends the proof of \eqref{semiest1} for $s=0$.
\begin{rem}
{\rm The argument for $|\tau|\leq M$ given above is different compared to a similar
analysis in our previous works \cite{RT1,RT2}. In \cite{RT1,RT2}, we use an
ODE argument since the linearized about a solitary wave equation may be easily
reduced to an ODE. For the water waves problem such a reduction is not
clear. On the other  hand, it is not clear to us how to adapt the approach presented
here to the case of the KP-I type equations, the problem being that the
analogue of $J$ for the KP-I type equations is $\partial_x$ which makes the
counterpart of Proposition~\ref{essJL} more difficult to establish.}
\end{rem} 
\subsubsection{Proof of \eqref{semiest1} for $s \geq 1$}
We shall use the following estimate
\beq\label{young2}
|\partial_{x} f |_{L^2 } \leq   C\Big(  \Big| { |D_{x} | \over 1 +  |D_{x}|^{1 \over 2 } } \partial_{x }f \Big|_{L^2}
+ \Big| { |D_{x} | \over 1 +  |D_{x}|^{1 \over 2 } } f \Big|_{L^2}
\Big)
\,.
\eeq
One may obtain (\ref{young2}) by analysing separately the low and the high frequencies.
For the proof of  \eqref{semiest1} for $s \geq 1$,
we proceed by induction. Let us assume that \eqref{semiest1} is proven for $s' \leq s-1$ i.e.
\beq\label{ind1}
|V(t)|_{X_{k}^{s'} } \leq C {e^{\gamma t } \over (1 + t )^\rho }, \quad \forall\, t \geq 0, \, \forall\, k, \,
|k| \leq K, \, \forall\, s', \, s' \leq s-1.
\eeq
We have to estimate $|\partial_{t}^{s-i} \partial_{x}^i V |_{X_{k}^0}$
for $i \leq s$.  We shall   now use  an induction on $i$.
For $i=0$, since the coefficients of  $L(k)$ do not depend on time, we get
that $\partial_{t}^s V $ solves
$$ 
(\partial_{t} - JL(k) \big)( \partial_{t}^s V ) = \partial_{t}^s F.
$$
Moreover, by using the equation \eqref{semieq2} and \eqref{Fhyp}, we get that
at $t=0$
\beq\label{t=0} 
| \partial_{t}^s V(0)  |_{H^l} \leq C_{s,l},
\eeq
where $C_{s,l}$ depends only on norms of $F$ at $t=0$. Thus, we get in particular that
$$ W= \partial_{t}^s V(t) - \partial_{t}^sV(0)$$
solves the equation 
$$ \partial_{t} W - JL(k) W = \tilde{F}, \quad W(0)=0$$
with a source term $\tilde{F}$ satisfying \eqref{Fhyp}.
By using the result of the previous subsection, we  get 
$$ 
|W(t) |_{X_{k}^0} \leq C { e^{\gamma t } \over ( 1 + t )^\rho}, \quad
\forall\, t \geq 0
$$
and hence 
\beq\label{i=1}
|\partial_{t}^sV(t) |_{X_{k}^0} \leq C { e^{\gamma t } \over ( 1 + t )^\rho}, \quad \forall\, t \geq 0.
\eeq  
Now, for $j\geq 1$, let us assume that  
\beq\label{indhyp}
|\partial_{t}^{s-i} \partial_{x}^i V |_{X_{k}^0} \leq C{ e^{\gamma t} \over (1+ t)^\rho }, \quad \forall\, i \leq j-1,
\quad \forall\, t \geq 0.
\eeq
By applying $\partial_{t}^{s-j} \partial_{x}^j$ to equation \eqref{semieq2}, we get the equation
\begin{equation}\label{semieq3}
\partial_{t} ( \partial_{t}^{s- j } \partial_{x}^j V) = J \Big(L(k) \partial_{t}^{s-j} \partial_{x}^j  V
+ [\partial_{x}^j, L(k)] \partial_{t}^{s-j} V \Big) + \partial_{t}^{s-j} \partial_{x}^j F.
\end{equation}
Thanks to Proposition~\ref{com} and (\ref{young2}),  we easily get the estimate
\beq
\nonumber
|[\partial_x^j,L(k)] \partial_{t}^{s-j} V|_{L^2}\leq C \sum_{i \leq j}|\partial_{t}^{s-j} \partial_{x}^iV |_{ X_{k}^0 }.
\end{equation}
Consequently,  thanks to the induction  assumption \eqref{ind1}, we get that
\beq
\label{komutat}
|[\partial_x^j,L(k)] \partial_{t}^{s-j} V|_{L^2}\leq C\Big( |\partial_{t}^{s-j} \partial_{x}^jV |_{ X_{k}^0 }
+ { e^{\gamma t} \over  ( 1 + t )^\rho}\Big).
\end{equation}
By taking the scalar product  and the real part of \eqref{semieq3},  against
$ L(k) \partial_{t}^{s-j} \partial_{x}^j V+   [\partial_{x}^j, L(k) ] \partial_{t}^{s-j} V$, 
we get thanks to \eqref{Fhyp} and \eqref{komutat} that 
\begin{multline}\label{semind1}
{1\over 2}{d \over dt} \big( \partial_{t}^{s-j} \partial_{x}^j V,  L(k) \partial_{t}^{s-j} \partial_{x}^j V \big)
+  {\rm Re }\, \big( \partial_{t }  \partial_{t}^{s-j} \partial_{x}^j V,  [\partial_{x}^j,
L(k) ] \partial_{t}^{s-j} V \big) 
\\
\leq C\Big( 
{e^{ 2 \gamma t } \over (1+ t )^{2\rho} } +
| \partial_{t}^{s-j}\partial_{x}^j F |_{L^2}
|\partial_{t}^{s-j} \partial_{x}^jV |_{ X_{k}^0 } +
{\rm Re }\, \big( \partial_{t}^{s-j} \partial_{x}^j F,
L(k) \partial_{t}^{s-j} \partial_{x}^j V \big)
\Big).
\end{multline}
Thanks to  \eqref{ressource} and \eqref{Fhyp}, we have
\beq\label{semind2} 
|\big( \partial_{t}^{s-j} \partial_{x}^j F, L(k) \partial_{t}^{s-j} \partial_{x}^j V \big) |
\leq C {e^{\gamma t } \over (1+ t )^{\rho}} | \partial_{t}^{s-j} \partial_{x}^j V |_{X_{k}^0}\,.
\eeq
Moreover,   by using the expression of $L(k)$, we can write
\beq \label{semi2}
 {\rm Re }\,\big( \partial_{t }  \partial_{t}^{s-j} \partial_{x}^j V,  [\partial_{x}^j, L(k) ] \partial_{t}^{s-j} V \big) \\
=   {\rm Re }\, \big( \partial_{t}  \partial_{t}^{s-j} \partial_{x}^j V_{2}, [\partial_{x}^j , G_{\eps, k}]\partial_{t}^{s-j}V_{2}
\big) + R,
\eeq
where $R$ can  be estimated   by 
\begin{multline}\label{Rsemi}
|R |   \leq C 
\Big(|\partial_{t}^{s-j + 1 } \partial_{x}^{j- 1 }  V_{1}|_{H^1  } \, 
\big(   
| \partial_{t }^{s-j}  V_{1 }|_{H^{j+ 1 }}  
+|\partial_{t}^{s-j} \partial_{x}V_{2} |_{ H^{j-1} } \big) 
\\ 
+ \Big|\partial_{t}^{s-j + 1 } { |D_{x}| \over 1 + |D_{x}|^{1 \over 2 }} \partial_{x}^{j- 1 }  V_{2}\Big|_{ L^2  }\,  
|V_{1}|_{H^{j+1}}
\Big).
\end{multline} 
Note that to get the last term above, we have used    that
$$
\Big| \Big(  \partial_{t}^{s- j + 1} \partial_{x}^j V_{2}, [\partial_{x}^j,  \partial_{x}\big(( v_{\eps} - 1 ) \cdot \big) ]V_{1}
\Big) \Big| 
\leq C \Big|   { |D_{x}| \over 1 +  |D_{x}|^{1 \over 2 } } \partial_{t}^{s- j + 1} \partial_{x}^{j- 1 }  V_{2} \Big|
_{L^2}\,  \Big | [\partial_{x}^j,  \partial_{x}\big( (v_{\eps} - 1 ) \cdot \big) ]V_{1} \Big|_{H^1},
$$
while for the first term we used a direct commutator of differential operators estimate.
By using the   induction assumptions \eqref{ind1}, \eqref{indhyp} and the inequality (\ref{young2}), we get
$$  |\partial_{t}^{s-j} \partial_{x}V_{2} |_{ H^{j-1} }
\leq C\Big(   \Big| { |D_{x} | \over 1 +  |D_{x}|^{1 \over 2 } } \partial_{t}^{s-j}\partial_{x }^j V_{2}\Big|_{L^2}
+  { e^{\gamma t} \over (1+ t)^{\rho} }\Big),
$$
and thus coming back to \eqref{Rsemi}, we obtain
\beq\label{Rsemi2}
|R| \leq C\Big( { e^{\gamma t } \over (1+ t )^{\rho} } | \partial_{t}^{s-j}\partial_{x}^j V|_{X^0_{k}}
+ {e^{ 2 \gamma t } \over (1+ t)^{ 2 \rho} }\Big).
\eeq
To estimate the first term in the right-hand side of \eqref{semi2}, we   write
$$
\big|  \big( \partial_{t }  \partial_{t}^{s-j} \partial_{x}^j V_{2},  [\partial_{x}^j,  G_{\eps, k} ] \partial_{t}^{s-j} V_{2} \big)
\big| \leq C
\Big| \partial_{t}^{s-j+1} { |D_x| \over {1 +  |D_{x}|^{ 1\over 2 } } } \partial_{x}^{j-1} V_{2}\Big|_{L^2}
\, \big|  [\partial_{x}^j,  G_{\eps, k} | \partial_{t}^{s-j} V_{2} \big|_{H^{1\over 2} }
$$
and hence the induction assumption \eqref{ind1}, \eqref{indhyp} and
the commutator estimate \eqref{com1}  yield
\beq\label{semi3}
\big|  \big( \partial_{t }  \partial_{t}^{s-j} \partial_{x}^j V_{2},  [\partial_{x}^j,  G_{\eps, k} ] \partial_{t}^{s-j} V_{2} \big)
\big|\leq C\Big(   { e^{\gamma t } \over (1+ t )^{\rho} } | \partial_{t}^{s-j}\partial_{x}^j V|_{X^0_{k}}
+ {e^{ 2 \gamma t } \over (1+ t)^{ 2 \rho} }\Big).
\eeq
Consequently, we can integrate \eqref{semind1} in time and use \eqref{t=0}, 
\eqref{semind2}, \eqref{semi2}, \eqref{Rsemi2}, 
\eqref{semi3} to obtain
\beq\label{semi4}
\big( \partial_{t}^{s-j}\partial_{x}^j V, L(k) \partial_{t}^{s-j}\partial_{x}^j V\big)(t)
\leq C\Big(
{e^{ 2 \gamma t } \over (1+ t)^{ 2 \rho} } + \int_{0}^t  { e^{\gamma \tau} \over (1+ \tau )^{\rho} } | \partial_{t}^{s-j}\partial_{x}^j V(\tau)|_{X^0_{k}}\, d\tau\Big).
\eeq
A crude bound from below on $L(k)$ gives for some $c>0$, $C>0$, 
\begin{eqnarray*}
& &  \big( \partial_{t}^{s-j}\partial_{x}^j V, L(k) \partial_{t}^{s-j}\partial_{x}^j V\big)(t)
\\
& &\geq  c \big| \partial_{t}^{s-j} \partial_{x}^j V |_{X_{k}^0}^2 -  C\Big(   \big|\partial_{t}^{s-j}\partial_{x}^j V_{1}|_{L^2}^2
+   \big|\partial_{t}^{s-j}\partial_{x}^j V_{1}|_{H^{1\over 2 } } 
\Big| { |D_{x} | \over 1 +  |D_{x}|^{1 \over 2 } } \partial_{t}^{s-j}\partial_{x }^j V_{2}\Big|_{L^2} \Big),
\quad \forall\, k, \, |k| \leq K 
\end{eqnarray*}
and hence by the interpolation inequality
$$   
\forall\,\delta>0,\,\,\exists\,C(\delta)>0\,:\,\quad
\big|\partial_{t}^{s-j}\partial_{x}^j V_{1}|_{H^{1\over 2 } } \leq  \delta  \big|\partial_{t}^{s-j}\partial_{x}^j V_{1}|_{H^{1} } + C(\delta) \big|\partial_{t}^{s-j}\partial_{x}^j V_{1}|_{L^2 }$$
we get by choosing $\delta$ sufficiently small and the induction assumption \eqref{ind1} that 
$$  
\big( \partial_{t}^{s-j}\partial_{x}^j V, L(k) \partial_{t}^{s-j}\partial_{x}^j V\big)(t)
\geq {c \over 2 } \big| \partial_{t}^{s-j} \partial_{x}^j V(t) |_{X_{k}^0}^2 - C { e^{2 \gamma t } \over (1+ t)^{2\rho} }.$$
Consequently, we can plug this last estimate into \eqref{semi4}
to get
\begin{eqnarray*}
\big| \partial_{t}^{s-j} \partial_{x}^j V(t) |_{X_{k}^0}^2&  \leq &
C\Big(  {e^{ 2 \gamma t } \over (1+ t)^{ 2 \rho} } +\int_{0}^t { e^{\gamma \tau} \over (1+ \tau )^{\rho} } | \partial_{t}^{s-j}\partial_{x}^j V(\tau)|_{X^0_{k}}\, d\tau\Big) \\
& \leq &   C(\delta){e^{ 2 \gamma t } \over (1+ t)^{ 2 \rho} }  + \delta    \int_{0}^t
| \partial_{t}^{s-j}\partial_{x}^j V(\tau)|_{X^0_{k}}^2 \, d\tau
\end{eqnarray*}
for every $\delta>0$. Note that we have used the inequality \eqref{Young}
to get the last estimate.
By  the choice $\delta <2 \gamma$, we get from the Gronwall inequality
that
$$ \big| \partial_{t}^{s-j} \partial_{x}^j V(t) |_{X_{k}^0}^2  \leq C
{e^{ 2 \gamma t } \over (1+ t)^{ 2 \rho} } .
$$
This ends the proof of \eqref{semiest1}.
\subsubsection{ $L^2$ estimate }
To finish the proof of  \eqref{semiest}, it remains to estimate the $L^2$ norm
of $V_{2}$ which is not given by the estimate \eqref{semiest1} for small $k$.
It suffices to use the equation for $V_{2}$ in \eqref{semieq2} which gives that
$$ 
|V_{2}(t) |_{L^2} \leq C \int_{0}^t \big(  | \partial_x V_{2}(\tau)|_{L^2 } +
|V_{1}(\tau)|_{H^2}
+|F_2(\tau)|_{L^2}\big)d\tau
$$
and then to use  \eqref{semiest1} (for $s=1$) together with (\ref{young2})
and (\ref{Fhyp}), to get
$$   |V_{2}(t) |_{L^2} \leq C  {e^{  \gamma t } \over (1+ t)^{  \rho} } .$$
This ends the proof of Proposition~\ref{semi-group}.
\end{proof}
\subsubsection{End of the proof of  Proposition~\ref{Uap}}
We proceed by induction. We have already built $U^0$ in
Proposition~\ref{U0}. Fix $j\geq 1$ and assume that  the $U^{l}$ are built for
$ l \leq j-1 $. 
We shall estimate the solution of \eqref{eqUk} by using Proposition~\ref{semi-group}.
Towards this, it suffices to check assumption \eqref{Fhyp}, where the source term is
defined by the right hand-side of (\ref{eqUk}).
Let us denote by $\mathcal{S}^j(t, x, y)$
the right hand side of \eqref{eqUk} and by $\hat{ \mathcal{S} }^j(t,x,k)$ its
Fourier transform with respect to $y$. 
From Proposition~\ref{lemderiveej} and the standard product estimates in Sobolev spaces, we get
$$
| \hat {\mathcal{S} }^j (t,\cdot,k ) |_{F^s}\leq C\, \sum_{p=2}^{j+ 1 }
\sum_{ \stackrel{ 0 \leq l_{1}, \dots, l_{p}\leq M}{l_{1}+ \cdots  + l_{p} = j+1 - p}}
\Big(|\hat{U}^{l_{1} } |_{F^{s+ s_{0} } }* \cdots * |\hat{U}^{l_{p}} |_{F^{s+ s_{0} } }\Big)(t,k)\,,
$$
where $*$ stands for the convolution with respect to the $k$ variable and $|\cdot|_{F^s}$ is naturally defined as
$$ 
|V(t, \cdot) |_{F^s} = \sum_{\alpha + \beta \leq s}
|\partial_{t}^\alpha \partial_{x}^\beta v(t, \cdot)|_{L^{2}(\mathbb{R}) }.
$$
Consequently, by using  repeatedly the Cauchy-Schwarz inequality in the
integrations defining the convolution,  the fact
that the $\hat{U}^i$ are compactly supported  in $k$, and the Bessel-Plancherel identity, we get
\begin{equation}\label{k_k}
| \hat {\mathcal{S} }^j (t,\cdot,k ) |_{F^s} \leq C(R,s,  j) \sum_{p=2}^{j+ 1 }
\sum_{ \stackrel{ 0 \leq l_{1}, \dots, l_{p}\leq M}
{l_{1}+ \cdots  + l_{p} = j+1 - p}}\|U^{l_{1}}\|_{E^s}\cdots \|U^{l_{p}}\|_{E^s}.
\end{equation}
By the induction assumption, again the Bessel-Plancherel identity and the fact that $\mathcal{S}^j(t,x,k)$ is 
compactly supported in $k$, after suitable integrations in $k$ starting from (\ref{k_k}), we finally get
$$   
\| \mathcal{S}^j (t) \|_{E^s}\leq C(R,s, j)  {e^{( j+ 1) \sigma_{0} t } 
\over  ( 1 + t )^{ {j+1 \over 2 m } } }.
$$
Consequently,  since $( j+ 1) \sigma_{0} >\sigma_{0}$, the
estimate of $\|U^j(t)\|_{E^s}$    follows thanks to
Proposition~\ref{semi-group}.
Finally, \eqref{Rap} follows from \eqref{Ujest} and crude estimates in
Sobolev spaces applied to $\|R_{M,\delta}(U^{a})\|_{E^s}$ and the other terms
involving $\delta^p$ with $p$ at least $M+3$. This ends the proof of Proposition~\ref{Uap}.
\end{proof}

\section{Proof of Theorem~\ref{main}  (the nonlinear analysis)}
\label{sectionnonlin}
Let us set  $V^{a}= Q+\delta U^{a}$ where $U^{a}$ is the  approximate solution
given by Proposition~\ref{Uap}.  To prove  our instability result,  we shall prove that we
can construct a true  solution $U^\delta$  of  \eqref{eq11}, \eqref{eq22},  that we can still 
consider in its abstract form \eqref{ww}, up to time $T^{\delta}\sim \log(1/\delta)$,
under the form
\beq\label{Uinstabl}
U^\delta = V^{a}+ U,\quad U^\delta(0)=V^{a}(0)=Q+\delta U^0(0).
\eeq
We therefore need to solve the equation
\beq\label{Ueqinstab}
\partial_{t} U = \mathcal{F}(V^{a} + U) -  \mathcal{F}(V^{a})   -  R^{ap}, \quad t>0, \quad U(0)= 0
\eeq
and obtain  estimates for $U$.
More precisely, we need to prove that the solution of \eqref{Ueqinstab}
is defined on a sufficiently large interval (of size $\log(1/\delta)$) of time in order to see 
the linear instability and also to prove that $U$ remains negligible in front of $V^{a}$.

The aim of the following is to  prove  a priori estimates for $U$ suitable for that purpose. 
These estimates rely on the transformation of the system into a quasilinear form.
Once these estimates are established, the result will follow by a continuation
argument as in \cite{Grenier} provided the number  $M$ of terms in the approximate
 solution is chosen sufficiently large.
  The  proof is organized as follows:

  - In the next subsection, we introduce
   useful notations and functional spaces. Then we state the key energy estimate.

   - Then in  Subsection \ref{DirNeumf},
    we study the Dirichlet-Neumann operator $G[\eta^{a}+ \eta] \varphi$.
    We need to track carefully the dependence of the estimates  with respect to the regularity
      of the surface. Here we need   the case that $\varphi$ is in  the  Sobolev scale, 
        for this part the analysis will be very close to the one of \cite{L1},
       but also the case that $\varphi= \varphi_{\eps}$ is the line solitary wave  and thus $\varphi$ is very smooth
        but not in the Sobolev scale.  The technically most subtle estimate is the estimate
         of 
         $ D_{\eta} G[\eta^{a}+ \eta] \varphi_{\eps} \cdot h$ in $H^1$ when $h$ is in $H^2$ 
         which is given in Proposition~\ref{pak_bis}.

  - Next, in  Subsection~\ref{quasin},
    we derive  a  quasilinear form of the system  by applying three space-time
     derivatives to the equation. We isolate a principal part of the equation and
      a lower order part that  mostly arises from commutators and can be considered as  made of semilinear terms.
     Subsection~\ref{semin} is devoted to the estimates of these semi-linear terms.

   - The energy estimates (which rely on the Hamiltonian structure of the linearized water-waves system)
   are given in Subsection \ref{energien}.

   - Subsection \ref{finn} is  devoted to the conclusion that is the proof of the nonlinear instability.
   
   - Finally, in Section \ref{sketch}, we  briefly explain how we can use   our a priori estimates
    in order to rigorously get the  local existence of a smooth solution for \eqref{Ueqinstab}
     by using the vanishing viscosity method.

  \subsection{Notations and functional spaces}
  Let us first introduce several notations. 
For $\sigma\in\R$, we denote by
$\Lambda^\sigma$ the Fourier multiplier on ${\mathcal S}'(\R^2)$ with symbol
$(1+|\xi|^2)^{\sigma/2}$, $\xi\in\R^2$. We shall also denote by $| \nabla |$ the Fourier 
multiplier  by  $|\xi|$.
For $\alpha = (\alpha_{0}, \alpha_{1}, \alpha_{2})\in\N^3$, we shall use the notation
$$ \partial^\alpha= \partial_{t}^{\alpha_{0}}\partial_{x}^{\alpha_{1}} \partial_{y}^{\alpha_{2}}.$$
Next, for $k\in\N$, we set
$$ 
\langle\partial\rangle^{k}u= (\partial^{\alpha } u  )_{|\alpha| \leq   k },\quad \langle\nabla\rangle^{k}
u= (\partial^{\alpha_1}_x\partial^{\alpha_2}_y u)_{\alpha_1+\alpha_2 \leq  k}\,.
$$
If $\|\cdot\|$ is a norm, by $\|\langle\partial\rangle^ku\|$ we denote the sum
of $\|\cdot\|$ norms of all the components of $\langle\partial\rangle^ku$ (if $u$ is
a tensor an additional summation over the components of $u$ should be
added). A similar convention shall be used for  $\langle\nabla\rangle^{k}u$.
For $k\in  \mathbb{N}$ and $U(t)=(U_{1}(t), U_{2}(t))$ we define $X^k$ by 
$$ 
\| U(t) \|^2_{X^k}= \sum_{| \alpha | \leq k} \Big(\|\partial^{\alpha } U_{1}(t) \|^2_{H^1}
+ \|\partial^{\alpha } U_{2}(t) \|^2_{H^{1\over 2 } }\Big)
\approx\|\langle\partial\rangle^k U_1(t)\|_{H^1}^2+\|\langle\partial\rangle^k U_2(t)\|_{H^\frac{1}{2}}^2.
$$
For $t>0$, we define the space $X^k_{t}$ of functions defined on
$[0,t]\times\R^2$ equipped with the norm
$$
\|U\|_{X^k_t}=\sup_{0\leq \tau\leq t}\|U(\tau)\|_{X^k}\,.
$$
Next, we define ${\mathcal W}^{k}$ by
$$
\|u(t)\|_{{\mathcal W}^{k}}=\sum_{|\alpha|\leq k}\|\partial^\alpha u(t)\|_{L^\infty(\R^2)}
=\|\langle\partial\rangle^ku\|_{L^\infty(\R^2)}
$$
and for $t>0$, we use the notation  ${\mathcal W}^{k}_t$  for the space of  functions defined on
$[0,t]\times\R^2$ equipped with the norm
$$
\|U\|_{{\mathcal W}^{k}_t}=\sup_{0\leq \tau\leq t}\|U(\tau)\|_{{\mathcal W}^{k}}\,.
$$
We shall denote by $\omega(x)$ a generic continuous, positive non decreasing
function on $\R^+$  with $\omega \geq 1$. 
This function may change from line to line and in
fact may be chosen under the form  $C(1+|x|)^N$, where $N$ may change in each appearance of $\omega$.

Since we want to construct a solution $U^\delta$ of \eqref{ww} under the form
$U^\delta = U+ V^{a}$ with $V^{a}= (\eta^{a}, \varphi^{a})$, we shall use the following
convention throughout the section: for  a function  or an operator $g(U)$, we set:
\beq
\label{convention} 
g^\delta = g(U+ V^a), \quad g^a= g(V^a)
\eeq
and thus
$$
g^\delta - g^{a} = g(U + V^{a}) - g(V^{a}).$$
For example, 
we shall use the notation 
$$G^\delta\varphi - G^{a} \varphi= G[\eta^{a}+ \eta]\varphi - G[\eta^{a}]\varphi, \quad
 Z^\delta - Z^{a}= Z[V^{a} + U]- Z[V^{a}]$$
 where $Z$ is defined in Lemma \ref{DN'}
and  the abstract equation \eqref{Ueqinstab} becomes
\beq
\label{Ueqinstab2}
\partial_{t} U =  \mathcal{F}^\delta - \mathcal{F}^{a} + R^{ap}.
\eeq
\subsection{Statement of the energy estimate}
The aim of this section is to establish an a priori  energy estimate for a smooth  enough solution  $U$ of
 \eqref{Ueqinstab2} defined on $[0, T]$ and satisfying the constraint
 \beq
 \label{hautbas}
  1 -  \|\eta^a(t) \|_{L^\infty} - \|\eta(t)\|_{L^\infty} >0, \quad \forall t\in [0, T]
 \eeq
 where $\eta$ stands for the first component of $U$.
\begin{theoreme}\label{theoenergie}
Let $U(t)$ a smooth solution of \eqref{Ueqinstab2} on $[0, T]$ satisfying \eqref{hautbas}.  Then 
for  $m \geq  2 $, $S\geq  5 $ and $t \in [0, T]$
we have the estimate:
\begin{multline*}
\| U(t)\|^2_{X^{m+3}} \leq 
\omega\Big(\|R^{ap}\|_{{X}^{m+3}_t}+\|V^{a}\|_{{\mathcal W}^{m+S}_t}+\|U\|_{X^{m+3}_t}\Big)
\\ 
\times\Big(\| R^{ap}\|^2_{{X}^{m+3}_t}+
\int_{0}^t\big(\|U(\tau)\|^2_{X^{m+3}}+\|R^{ap}(\tau)\|_{X^{m+3}}^2
\big)d\tau\Big).
\end{multline*}
\end{theoreme}

Of course we can replace $m+3$ by $m$ for $m$ larger than $5$  but we decided to keep this form
of the energy estimate  in order to emphasize the fact that we have differentiated three times
 the system to quasilinearize it  before performing the energy estimate.

 This estimate is far from  being  the  best one to use in terms of regularity  to get  well-posedness of
  the Cauchy problem, we  are not interested here in this issue since it is not relevant for
   the proof of Theorem \ref{main}. Indeed, the smoother the involved norms are, the  better
    the instability result is.
    
In the following,  we shall  always assume that $U$ verifies the constraint \eqref{hautbas}
 without making  explicit reference to it.
 In a similar way, as soon as a  Dirichlet Neumann operator $G[\zeta]$ is  involved, 
 we always assume that $\zeta$ satisfies $1- \| \zeta \|_{L^\infty}>0$ without
 recalling it.
 
\subsection{Preliminary estimates on the Dirichlet-Neumann operator}\label{DirNeumf}
In this section, we recall some useful properties of the Dirichlet-Neumann operator
$G[\eta]\varphi$. The new points with respect to  similar estimates  in
\cite{L1}, \cite{Alvarez-Lannes} is the introduction of time derivatives in
our estimates and the use of Schauder elliptic regularity estimates. We need
to use this elliptic theory and in some cases combine it  with the Sobolev
regularity theory since the solitary waves do not belong to the usual Sobolev spaces on $\R^2$.
In this section, we absolutely do not aim at giving  optimal regularity estimates,
we just give the one which are sufficient  for the proof of Theorem \ref{theoenergie}.

As in  Section \ref{sectionGepsk}, the problem will be reduced  to elliptic estimates
in a flat strip, consequently, for this section, it is useful to introduce the following notations.

For a   function $u(t,X, z)$ defined on  $[0, T]\times \mathcal{S}$ where $\mathcal{S}$
 is the strip $ \mathbb{R}^2 \times (0, 1)$, we set $$D^\alpha u = \partial_{t}^{\alpha_{0}} \partial_{x}^{\alpha_{1}} \partial_{y}^{\alpha_{2}}
\partial_{z}^{\alpha_{3}}u, \quad \,  \alpha \in \mathbb{N}^4, \quad 
 \nabla_{X,z}^\alpha u =   \partial_{x}^{\alpha_{1}} \partial_{y}^{\alpha_{2}}
\partial_{z}^{\alpha_{3}} u , \quad \alpha \in \mathbb{N}^3.$$
Moreover, as previously $\| \langle \nabla_{X,z}  \rangle^m u  \|_{L^2(\mathcal{S})}$ will stand for the sum
 of the $L^2(\mathcal{S})$ norm  of $\nabla_{X,z}^\alpha u $ for $|\alpha | \leq m$ (and thus this  norm
  is equivalent to the standard Sobolev norm $H^m(\mathcal{\mathcal{S}})$ of the strip)
   while 
   $ \| \langle D  \rangle^m u (t)  \|_{L^2(\mathcal{S})}$ will stand for the sum
 of the $L^2(\mathcal{S})$  norms of $ D^\alpha u(t) $ for $|\alpha | \leq m $. 
  With these definitions, we have in particular that
\beq
\label{relXH}
 \| \langle D \rangle^m u (t) \|_{L^2(\mathcal{S}) }
= \sum_{l=0}^m  \|  \langle \nabla_{X,z} \rangle^{m-l}  \partial_{t}^l u (t) \|_{L^2(\mathcal{S})}
\approx  \sum_{l=0}^m \| \partial_{t}^l u(t) \|_{H^{m- l} (\mathcal{S})}.
\eeq 
Finally, we also set $ \|u(t) \|_{\mathcal{W}^{m}(\mathcal{S})}= \|\langle D \rangle^m  u(t) \|_{L^\infty(\mathcal{S})}.$

In this whole  subsection the time variable $t$ is only a parameter, we shall therefore
omit  to write down  explicitly the  dependence  on this parameter.

Let us  establish some product estimates which will be of constant use throughout this
section.

\begin{lem} 
\label{prodS}
For $m \geq 2$,   $|\alpha |+ |\beta | \leq  m$,  and $k=0, \, 1, \, 2$, we have
\begin{eqnarray*}
 \| D^\alpha u \,  D^\beta v  \|_{H^k(\mathcal{S})}
 &  \leq &  C  \|  \langle D \rangle^m  u \|_{H^k(\mathcal{S})} \,
    \|  \langle D \rangle^m    v \|_{H^k(\mathcal{S})}, \\ 
\| D^\alpha u  \, D^\beta v   \|_{H^k(\mathcal{S})} & \leq &   
  C  \|u \|_{\mathcal{W}^{m+ k}(\mathcal{S})} \, \|  \langle D \rangle^m v \|_{H^k(\mathcal{S})}. 
\end{eqnarray*}
\end{lem}
\begin{proof}[Proof of Lemma \ref{prodS}]
The second estimate is   obvious. Let us prove the first one. We start with the case $k=0$.
From the symmetry of the expression, it suffices to estimate
 $ \| D^\alpha u  D^\beta  v\|_{L^2(\mathcal{S})}$ for $| \alpha | \leq |\beta|$.
 When $\alpha \neq 0$, we can use the Sobolev embedding $H^1(\mathcal{S})\subset  L^4
 (\mathcal{S})$ to
  get
  $$  \| D^\alpha u  D ^\beta v\|_{L^2(\mathcal{S})}
   \leq \| D^\alpha u \|_{H^1(\mathcal{S})} \| D^\beta v \|_{H^1(\mathcal{S})}
    \leq    \|  \langle D \rangle^m   u \|_{L^2(\mathcal{S})} \,
    \|   \langle D \rangle^m  v \|_{L^2(\mathcal{S})},$$
     since $ |\alpha | \leq | \beta | \leq m-1$.
     When $\alpha=0$, we just write  
  $$  \| u D ^\beta v\|_{L^2(\mathcal{S})} \leq \|u \|_{L^\infty(\mathcal{S})} \|   \langle D \rangle^m v\|_{L^2(\mathcal{S})}$$
   and the result follows from the Sobolev embedding 
  $  \|u \|_{L^\infty(\mathcal{S})} \leq C \| \langle D \rangle^m u\|_{L^2(\mathcal{S})}  $ when $m \geq 2$.
   For $k=1$, it suffices to use the previous estimate with $u$ and $v$ replaced  by $\nabla u $ and
    $v$ or $u$ and $\nabla v$. The case $k=2$ is very simple since $H^2(\mathcal{S})$
     is an algebra. This ends the proof of Lemma \ref{prodS}.

\end{proof}
We are now  able to state our first set of estimates on the Dirichlet-Neumann operator 
which will be intensively used in the proof of Theorem~\ref{main}.  We start with
the estimates in the Sobolev framework.
\begin{prop}\label{pak}
Let us set 
$$ \underline{\omega}=
\omega \big(  \|\langle \pa \rangle^m \eta \|_{H^{ 5 \over 2 }(\mathbb{R}^2)} + \|\eta_{0} \|_{\mathcal{W}^{m+3}
(\mathbb{R}^2)}\big)$$
and let  $m \geq 2$.

Then we have the following estimates:
\begin{itemize}
\item 
for 
 $\sigma=-1/2,1/2,1$,
 \begin{equation}\label{rennes}
\|\langle \partial \rangle^m G[\eta+\eta_0]u\|_{H^\sigma}\leq
\underline{\omega}\, \|\langle\partial\rangle^m u\|_{H^{\sigma+1}}.
\end{equation}
\item  For $n\geq l\geq 1$, $\sigma=-1/2, 1/2,1$, 
\begin{multline}\label{rennes1,5}
\|\langle \partial \rangle^m \big(D_{\eta}^n G[\eta+\eta_0]u\cdot(h_1,\cdots,h_n)\big)\|_{H^\sigma}
\\
\leq \underline{\omega}\|\langle\partial\rangle^m u\|_{H^{\sigma+1}}
\big(\prod_{j=1}^l\| h_j\|_{{\mathcal W}^{m+3}}\big)\big(\prod_{j=l+1}^n\|\langle \partial\rangle^{m+1} h_j\|_{H^{1}}\big)
\end{multline}
(the first product is defined as $1$ if $l=0$ and the second product is defined as $1$ if $l=n$).
\\
\item 
Finally, we have the following commutator estimate
\begin{equation}\label{rennes2}
\|[\partial^{\alpha},G[\eta+ \eta_{0}]](u)\|_{H^{-\frac{1}{2}}}\leq\underline{\omega}\,
\|\langle\partial\rangle^{m-1} u\|_{H^{\frac{1}{2}}}\,,\quad \forall\,|\alpha|\leq m.
\end{equation}
\end{itemize}
\end{prop}

\begin{rem}
\label{Sanremark}
The estimates that we have stated are the  useful ones for the proof of Theorem \ref{theoenergie}.
In particular, the important thing is that the dependence in $\eta$   in $\underline{\omega}$ 
 is controlled by $ \| \langle \pa \rangle^{m+3} \eta \|_{H^1}$.
We shall actually prove some more precise estimates. For example, we shall get that for
 $m \geq 2$, $| \alpha | \leq m $,  we have
\beq
\label{rennesrefined}
\| \partial^\alpha \big( G[\eta + \eta_{0}] u \big) \|_{H^{\sigma }(\mathbb{R}^2)}
   \leq  \omega\big(   \| \langle \pa \rangle^m  \eta \|_{{H}^{ \sigma + 1 }(\R^2)} +
\| \eta_{0} \|_{\mathcal{W}^{m+  \sigma + {3 \over 2 } }(\R^2)} \big) \| \langle \pa \rangle^m u \|_{{H}^{
\sigma + 1  }(\R^2)}
\eeq
 for $\sigma = -1/2, \, 1/2, \,3/2$ and also 
 \beq
 \label{San}
\| \partial^\alpha D_{\eta}G[\eta+\eta_0](u)\cdot h \|_{H^{\sigma}(\R^2)}
\leq 
\underline{\omega}\|  \langle \pa \rangle^m 
u\|_{{H}^{\sigma + 1 }(\R^2)} \,\|  \langle \pa \rangle^m  h\|_{{H}^{\sigma + 1 }(\mathbb{R}^2)}
\eeq
for  $\sigma = -1/2, \, 1/2.$
\end{rem}

\begin{proof}[Proof  of Proposition \ref{pak} ]
We shall split the proof in various Lemmas.

As in the work by Lannes \cite{L1} an important point  is to
choose in the optimal way (by using a harmonic extension) with respect to the Sobolev regularity the map which
flatters the domain.
We shall denote by $\mathcal{S}$ the flat strip $\mathcal{S}= \mathbb{R}^2 \times (-1, 0 )$. 
\begin{lem}\label{harmonic}
Consider  $H$
which can be written as
$H =\eta+\eta_0$  and satisfies $ 1 - \|\eta \|_{L^\infty} - \|\eta_{0}\|_{L^\infty} >\tilde{\kappa}$ 
with $\tilde{\kappa}>0$   and 
$  \pa^\alpha \eta \in {H}^s(\R^2)$ for  $ | \alpha | \leq m$, $m \geq 2$,  $s \geq 1/2$,  $\eta_0 \in \mathcal{W}^{k}(\R^2)$.

Then,  there exists a map  
$
\theta \,:\mathcal{S}\rightarrow \R
$
such that $\theta(X,-1)=-1$, $\theta(X,0)=H(X)$ which can be decomposed as
\beq
\label{dectheta}
\theta=\theta_1(\eta)+\theta_2(\eta_0)
\eeq  
with  the estimates
$$
\| \langle D \rangle^m  \theta_{1}\|_{{H}^{k}(\mathcal{S})}\leq
C_{m,k}\| \langle \pa \rangle^m \eta \|_{{H}^{k - { 1 \over 2 }}(\R^2)},\,
\quad 
\| \theta_{2}\|_{\mathcal{W}^{m}(\mathcal{S})}
\leq C_{m}(1+\| \eta_0\|_{\mathcal{W}^{m}(\R^2)}),$$
moreover there exists $\kappa >0$ such that
\begin{equation}\label{nondegen}
\pa_{z}\theta \geq \kappa,\quad \forall\, X\in\R^2,\quad \forall
z\in [0,1]\,. 
\end{equation}
In particular the map
$
(X,z)\mapsto (X,\theta(X,z))
$
is a diffeomorphism from the strip
$\mathcal{S}= \mathbb{R}^2 \times (-1, 0 )$
to
$
\{(X,z)\in\R^2\times\R\,:\, -1<z< H(X)\}.
$
\end{lem}

\begin{rem}
\label{remtheta}
As we shall see in the proof, $\theta_{1}$  is linear in $\eta$ and $\theta_{2}$ affine in $\eta_{0}$, 
consequently, 
 because of the decomposition \eqref{dectheta}
we also have the property that 
$D \theta(H)\cdot (h_{1}+ h_{2})=\theta_{1}(h_{1}) +\big( \theta_{2}(h_{2})-z\big)$
if $ h_{1}$ is in some Sobolev space  and $h_{2}\in \mathcal{W}^k$.
 Moreover, we also deduce that 
 for $n\geq 2$, $D ^n\theta=0$. 
\end{rem}

As we shall see below the same idea as in \cite{L1} can be used.  Our
situation is slightly different since the surface is made of a Sobolev part  and a smooth
non-decaying part while the bottom is flat. Moreover, we have  taken into account
 the presence of time derivatives.

\begin{proof}[Proof of Lemma \ref{harmonic}] 
Let $\tilde{\theta}_1$ be defined on $\mathcal{S}$ as the (well-defined) solution of the elliptic problem
$$
\Delta\tilde{\theta}_1=0,\quad \tilde{\theta}_1(X,-1)=0,\quad \tilde{\theta}_1(X,0)=\eta(X).
$$
Then by  standard elliptic regularity $\tilde{\theta}_1\in H^{s+\frac{1}{2}}(\R^2)$ if $\eta \in H^s(\mathbb{R}^2)$
and  hence since the time is only a parameter in the problem, $ \langle D \rangle^m \tilde{ \theta}_{1} \in {H}^{
k }(\mathcal{S})$  if  $ \langle \pa \rangle^m  \eta  \in {H}^{k- {1 \over 2 }}(\mathbb{R}^2)$. 
Observe that the dependence of $\tilde{\theta}_1$ with respect to $\eta$ is linear.
Next we consider  the function $\theta_1$ defined   on $\mathcal{S}$ by
$\theta_1(X,z)=(1+z)\tilde{\theta}_{1}(X,\epsilon z)$, where $\epsilon\in (0,1)$ is a
small number to be fixed later.  We also consider the function $\theta_2$ defined on $\mathcal{S}$ by
$
\theta_2(X,z)=\eta_0(X)+(1+\eta_0(X))z.
$
Then the map $\theta\equiv \theta_1+\theta_2$ satisfies the required
properties, provided $\epsilon$  is  small enough. Indeed 
\begin{equation*}
\pa_z\theta(X,z)=1+ H(X)+\epsilon \,z\,\pa_z\tilde{\theta}_{1}(X,\epsilon z)+
\epsilon\int_{0}^{z}\pa_z\tilde{\theta}_1(X,\epsilon \zeta)d\zeta.
\end{equation*}
Therefore for  $\epsilon\ll 1$, we can achieve
(\ref{nondegen}) since $1- | H |_{L^\infty}\geq \tilde{\kappa}>0$.

The claimed expression for the Frechet derivatives of $\theta(H)$ in Remark \ref{remtheta} follows
directly from the construction.
This completes the proof of Lemma~\ref{harmonic}.
\end{proof}
\begin{rem}\label{IPP}
Let us observe that the map $\theta$ satisfies
$\pa_x\theta(X,-1)=\pa_{y}\theta(X,-1)=0$, a fact which is useful in
integration by parts arguments over $\mathcal{S}$.
\end{rem}

We next express the Dirichlet-Neumann operator  in terms of a solution of a PDE defined
on $\R^2\times(0,1)$ with a domain flattened by the map constructed in  Lemma~\ref{harmonic}.
For $u(X)$ a given function on $\R^2$, if $\phi^u$ is defined on the domain 
$$
\{(X,z)\in\R^2\times\R\,:\, -1<z< H(X)= \eta(X)+ \eta_{0}(X)\}
$$
and is such that $\phi^u(X,\eta(X))=u(X)$ and $\partial_z\phi^u(X,-1)=0$ then we
can define a function $\psi^u$ on the flat domain $\mathcal{S}$ by 
$$
\psi^{u}(X,z)=\phi^u(X,\theta(X,z)),\quad (X,z)\in \R^2\times [-1,0]
$$
and we have that $\psi^u(X,0)=u(X)$, $\partial_z\psi^u(X,-1)=0$.
Next, if $\phi$ solves the problem
$$
(\partial_x^2+\partial_y^2+\partial_z^2)\phi=F
$$
on $\{-1<z< H(X)\}$ then $\psi(X,z)=\phi(X,\theta(X,z))$ solves
\begin{equation}\label{psiu}
{\rm div}_{X,z}(g(X,z)\nabla_{X,z}\psi(X,z))=\pa_z\theta(X,z)F(X,\theta(X,z))
\end{equation}
on $\mathcal{S}$, where $g$ is defined by
\begin{equation}\label{g_def}
g(X,z)\equiv
\left( \begin{array}{ccc} \pa_z\theta(X,z) & 0 &  -\pa_x\theta(X,z)
\\
0 & \pa_z\theta(X,z) & -\pa_y\theta(X,z)
\\
-\pa_x\theta(X,z) &  -\pa_y\theta(X,z) &
\frac{1+(\pa_x\theta(X,z))^2+(\pa_y\theta(X,z))^2}{\pa_z\theta(X,z)}
\end{array}\right),  \quad (X, z) \in \mathcal{S}.
\end{equation}
Note that our notation is slightly different from the one of Section \ref{sectionGepsk}.

Consequently, 
if $\phi^u$ solves
$$
(\partial_x^2+\partial_y^2+\partial_z^2)\phi=0,\quad X\in\R^2,\quad
-1<z< H(X),\quad
\phi(X, H(X))=u(X),\,\, \partial_{z}\phi(X,-1)=0
$$
then $\psi^u(X,z)=\phi^u(X,\theta(X,z))$ solves
\begin{equation}\label{metric_bis}
{\rm div}_{X,z}(g(X,z)\nabla_{X,z}\psi(X,z))=0, \quad (x, z)\in \mathcal{S},\qquad
\partial_{z}\psi(X,-1)=0,\quad \psi(X,0)=u(X)\,.
\end{equation}
We observe (see Remark~\ref{IPP}) that  if $\psi$ and $\phi$ are smooth enough, decaying at infinity in $X$ and are
such that $\phi(X,0)=0$ and $\pa_{z}\psi(X,-1)=0$ then
$$
\int_{{\mathcal{S}}}{\rm div}_{X,z}(g(X,z)\nabla_{X,z}\psi(X,z))\phi(X,z)dXdz= - \int_{{\mathcal{S}}}
g(X,z)\nabla_{X,z}\psi(X,z)\cdot\nabla_{X,z}\phi(X,z)dXdz\,.
$$
Coming back to the definition of the Dirichlet-Neumann operator, 
using the Green formula and a change of variable  justified  by Lemma~\ref{harmonic}, we can infer the identity
\begin{equation}\label{vajno}
(G[\eta+ \eta_{0}](u),v)=\int_{-1}^0\int_{\R^2}g(X,z)\nabla_{X,z}\psi^u(X,z)\cdot\nabla_{X,z}{\bf v}(X,z)dXdz,
\end{equation}
where ${\bf v}(X,z)$ is such that ${\bf v}(X,0)=v(X)$.
The identity (\ref{vajno}) will be used frequently in the sequel.

Let us denote by $P$ the elliptic operator defined by
$$P\psi\equiv{\rm div}_{X,z}(g(X,z)\nabla_{X,z}\psi(X,z)),$$ where $\theta$ is defined by Lemma~\ref{harmonic}. 
As  before, the   proof of Proposition \ref{pak} will follow from the study of
the elliptic operator $P$.

At first,  in view of Lemma \ref{harmonic}, we shall establish a useful decomposition
 of $g$ with a part  which has sharp Sobolev regularity and a smooth part.

\begin{lem}
\label{metricest}
There exists a decomposition
$ g= g_{1}+ g_{2}$ such that we have
\begin{eqnarray}
\label{dkg1}
& & \| \langle  D\rangle^m  g_{1}\|_{{H}^k(\mathcal{S})}
\leq 
\omega\big( \| \langle \pa \rangle^m \eta\|_{{H}^{ k + {1 \over 2 }}(\R^2)}+
\| \eta_0\|_{\mathcal{W}^{m+ k + 1}(\R^2)}\big), \quad m \geq 2, \, k=0, \, 1, \,  2,
\\
\label{dkg2}& &  \|  g_{2} \|_{\mathcal{W}^k(\mathcal{S})} \leq \omega \big(\|\eta_0\|_{\mathcal{W}^{k+1}(\R^2)}\big), \quad \forall k.
\end{eqnarray}

\end{lem}


\begin{proof}[Proof of Lemma \ref{metricest}]
We set
$$
g_{2}(X,z)\equiv
\left( \begin{array}{ccc} \pa_z\theta_{2}(X,z) & 0 &  -\pa_x\theta_{2}(X,z)
\\
0 & \pa_z\theta_{2}(X,z) & -\pa_y\theta_{2}(X,z)
\\
-\pa_x\theta_{2}(X,z) &  -\pa_y\theta_{2}(X,z) &
\frac{1+(\pa_x\theta_{2}(X,z))^2+(\pa_y\theta_{2}(X,z))^2}{\pa_z\theta_{2}(X,z)}
\end{array}\right),  \quad (X, z) \in \mathcal{S}$$
and $g_{1}= g-g_{2}$.
The estimate of $g_{2}$ is an easy consequence of Lemma \ref{harmonic}. Indeed, 
note that   since $\partial_{z} \theta_{2}= 1 + \eta_{0}$,  we have that
\beq
\label{dztheta2}
 | \partial_{z} \theta_{2} | \geq \kappa >0.
 \end{equation}
 
Next, most of the terms arising in $g_{1}$  can also  be estimated
 by using Lemma \ref{harmonic}. The nonlinear terms can be estimated by using 
 Lemma \ref{prodS}. For example, let us estimate
  $$   {1 \over \partial_{z} \theta_{2}  } -   {1 \over \partial_{z} \theta }  = 
{ \partial_{z} \theta_{1} \over \partial_{z} \theta_{2} \partial_{z} \theta }.$$
   At first, by using the second estimate of Lemma \ref{prodS} and \eqref{nondegen}, \eqref{dztheta2}, we get
$$ \big\| \langle  D \rangle^m    \big( { \partial_{z} \theta_{1} \over \partial_{z} \theta_{2} \partial_{z} \theta} 
 \big)\big\|_{{H}^k(\mathcal{S})}
  \leq \omega\big(  \|  \theta_{2}\|_{\mathcal{W}^{m+k+ 1}(\mathcal{S})}
   \big) \Big( \| \langle D \rangle^m \theta_{1} \|_{H^{k+1}(\mathcal{S})} +  \sum_{|\alpha| + |\beta| = m, 
   \beta \neq 0 } \|D^\alpha   \partial_{z} \theta_{1} D^\beta { 1 \over  \partial_{z} \theta_{1}}
     \big\|_{H^k(\mathcal{S})} \Big).$$ 
    Next, we can use the first estimate of Lemma \ref{prodS}  to get 
      $$  
     \big\|  \langle D \rangle^m  \big( { \partial_{z} \theta_{1} \over \partial_{z} \theta_{2} \partial_{z} \theta} 
 \big)\big\|_{{H}^k(\mathcal{S})}
  \leq \omega\big( \|  \theta_{2}\|_{\mathcal{W}^{m+k+ 1}(\mathcal{S})} +  \| \langle D \rangle^m \theta_{1} \|_{H^{k+1}(\mathcal{S})}
   \big) \|\langle  D \rangle^m  \theta_{1} \|_{{H}^{k+ 1 }(\mathcal{S})} $$
   and the result follows by using 
         Lemma   \ref{harmonic}.
    The other terms can be handled in a similar way.

 This ends the proof of Lemma \ref{metricest}.
\end{proof}
The next step will be to study the elliptic equation $Pu= \nabla_{X,z}\cdot F$.
We shall make use of the following elliptic regularity result.
\begin{lem}\label{ell-reg}
For $m\geq 2$,  and $F$ such that $(F_{3})_{/Z=-1}=0$, then the solution  of
\begin{equation}\label{div}
Pu= \nabla_{X,z}\cdot F,\quad (X,z )\in \mathcal{S}, \quad u(X,0)=\partial_{z}u(X,-1)=0,
\end{equation}
satisfies the estimate
\beq
\label{divest}
\|\langle  D \rangle^m u\|_{{H}^k(\mathcal{S})}
\leq\omega\big(\|\langle D  \rangle^m g_{1}\|_{{H}^{k-1}(\mathcal{S})} + \| 
g_{2} \|_{\mathcal{W}^{m+k-1}(\mathcal{S})} \big)
\| \langle D \rangle^m  F\|_{{H}^{k-1}(\mathcal{S})}, \quad k=1, \, 2, 3.
\eeq
\end{lem}
Before giving the proof  Lemma~\ref{ell-reg}, we state a corollary  which is our basic tool in the proof of
Proposition~\ref{pak}.
\begin{cor}\label{usful}
For $m \geq 2$,  and $F$ such that $(F_{3})_{/Z=-1}=0$, then   the solution of
$$
Pu=\nabla_{X,z}\cdot F, \quad (X,z) \in \mathcal{S}, \quad u(X,0)=\partial_{z}u(X,-1)=0
$$
satisfies the estimate 
$$
\|\langle D \rangle^m u\|_{{H}^k(\mathcal{S})}
\leq\omega\big( \| \langle \pa \rangle^m  \eta \|_{{H}^{k- {1 \over 2 }}(\R^2)} + \| 
\eta_{0} \|_{\mathcal{W}^{m+k }(\R^2)} \big)
\| \langle D\rangle^m F\|_{{H}^{k-1}(\mathcal{S})}, \, k=1, \, 2, 3.
$$
\end{cor}
\begin{proof}[Proof of Corollary~\ref{usful}] It  suffices to combine  Lemma~\ref{ell-reg} 
and  Lemma \ref{metricest}.
\end{proof}
\begin{proof}[Proof of Lemma~\ref{ell-reg}]
We need to 
estimate $\|\partial_{t}^l u \|_{H^{m+k - l}}$ for $l \in [0, m]$ where $H^k$  stands for  the standard
Sobolev space in the strip. We shall reason by induction on $l$.

When $l=0$, the estimate of $ \|u \|_{H^{m+k}}$ is  the usual elliptic regularity estimate
(note that thanks to Lemma \ref{harmonic}, the matrix $g$ is positive definite thus $P$ is indeed
an elliptic operator).  The needed  estimate was actually  established in \cite{L1}  Theorem 2.9.
 We thus  already  have  that
\beq
\label{lannesest} \|u\|_{H^{m+k}} \leq \omega\big(  \|g_{1}\|_{H^{m+k-1}} +  \|g_{2}\|_{W^{m+k-1, \infty}} \big)
 \| F \|_{H^{m+ k-1}}.
 \eeq
 Now, let us assume that $\|\partial_{t}^j u \|_{H^{m+k - j}}$ is estimated for $j \leq l-1$. Then,
  we can apply $\partial_{t}^l$ to \eqref{div} to get the equation
  $$  P \partial_{t}^l u = \partial_{t}^l \nabla_{X,z}\cdot F -  \nabla_{X, z} \cdot \big(  [\partial_{t}^l, g ] \nabla_{X,z} u \big), \quad
   (X,z) \in \mathcal{S}, \quad  \partial_{t}^l u(X,0)= 0, \, \partial_{z}\partial_{t}^l u(X,-1)=0.$$
   Consequently, we can use again \eqref{lannesest} to get that 
  \beq
  \label{modif0} \| \partial_{t}^lu\|_{H^{m+k - l }} \leq \omega\big(  \|g_{1}\|_{H^{m+k-1}} +  \|g_{2}\|_{W^{m+k-1, \infty}} \big)
\big( \|  \partial_{t}^l F \|_{H^{m-l + k- 1}} +  \big\|  [\partial_{t}^l, g ] \nabla_{X,z} u\big\|
_{H^{m-l + k-1}}
\big).\eeq
 To estimate the last term in the  right hand side,  the only
  difficulty is to estimate the terms involving the commutator with  $g_{1}$. In this case,  we need to estimate
    $ \| \nabla_{X,z}^{\gamma_{1}} \partial_{t}^{l-j} g_{1}  \, \nabla_{X,z}^{\gamma_{2}}  \partial_{t}^{j} 
     \nabla_{X,z }u \|_{H^{k-1}(\mathcal{S})}$
     for $ j \leq l-1$ and $ | \gamma_{1} | + |\gamma_{2} | \leq m-l$.
             As in the proof of Lemma \ref{prodS} we find
    \beq
    \label{modif}  \| \nabla_{X,z}^{\gamma_{1}} \partial_{t}^{l-j} g_{1}  \, \nabla_{X,z}^{\gamma_{2}}  \partial_{t}^{j} 
     \nabla_{X,z }u \|_{H^{k-1}(\mathcal{S})} \leq C  \| \langle D \rangle^{m} g_{1} \|_{{H}^{k-1}(\mathcal{S})}\,
      \|  \partial_{t}^j u \|_{H^{ k+ m- j }(\mathcal{S})}.\eeq
      Indeed, when $k= 3$, this estimate is straightforward since $H^2$ is an algebra.
       Let us explain the proof when $k=1$. As in the proof of Lemma \ref{prodS}, we can write
    $$  \| \nabla_{X,z}^{\gamma_{1}} \partial_{t}^{l-j} g_{1}  \, \nabla_{X,z}^{\gamma_{2}}  \partial_{t}^{j} 
     \nabla_{X,z }u \|_{L^2 (\mathcal{S}) }
      \leq   C   \| \nabla_{X,z}^{\gamma_{1}} \partial_{t}^{l-j} g_{1} \|_{H^{1}(\mathcal{S})} \,
      \| \nabla_{X,z}^{\gamma_{2}}  \partial_{t}^{j} 
     \nabla_{X,z }u \|_{H^{1}(\mathcal{S})}$$
     and therefore  \eqref{modif} follows except when  $j=0$ and $\gamma_{2}=0$. In this case,
     we just write
     $$    \| \nabla_{X,z}^{\gamma_{1}} \partial_{t}^{l} g_{1}  \,  
     \nabla_{X,z }u \|_{L^2 (\mathcal{S}) } \leq C
      \|  \langle D \rangle^m g_{1} \|_{L^2(\mathcal{S})} \, \| \nabla_{X,z} u \|_{L^\infty(\mathcal{S})}$$
       and the result follows by Sobolev embedding since $m \geq 2$. The proof of \eqref{modif}
        when $k=2$ follows the same lines.
        
                 We can use  the induction assumption and \eqref{modif0}, \eqref{modif} to conclude 
                 since  $j \leq l-1$.
        This ends the proof of Lemma \ref{ell-reg}.
\end{proof}

Let $\psi^u$ be defined as a solution of (\ref{metric_bis}). We have the following bounds for $\psi^u$.
\begin{lem}\label{lema_psi}
For $m \geq 2$,  we have the estimate,
\begin{equation*}
\|\langle D \rangle^m \psi^u\|_{{H}^k(\mathcal{S})}
\leq \omega( \| \langle \pa \rangle^m  \eta \|_{{H}^{k- {1 \over 2 }}(\mathbb{R}^2)}
+ \| \eta_{0} \|_{\mathcal{W}^{m+k}(\mathbb{R}^2 )}\big)\| \langle \pa \rangle^m  u\|_{{H}^{ k - {1 \over 2 } }(\R^2)}, 
\quad k=1, \, 2, \,3.
\end{equation*}
\end{lem}
Note that we state here a regularity result for all the derivatives of $\psi^u$  whereas in Lemma
 \ref{vHlem} we  could  consider  only tangential derivatives and  one normal derivative.
  This difference is important in order to get an optimal estimate in terms of the regularity of the surface.

\begin{proof}
We split $\psi^u$ as $\psi^u=u^H+u^r$, where $u^H$ is defined via its Fourier transform by
\begin{equation}\label{fourier}
\hat{u}^H(\xi, z)= \frac{ \cosh \big( |\xi| (z+1) \big) }{ \cosh |\xi|  } 
\hat{u}(\xi), \quad \xi \in \mathbb{R}^2, \, z \in (-1, 0).
\end{equation}
As in the proof of Lemma~\ref{uHlem}, we get the bound
\begin{equation}\label{flat}
\|\langle  D \rangle^m  u^H\|_{{H}^k(\mathcal{S})}\leq 
C\| \langle \pa \rangle^m  u\|_{{H}^{k - {1 \over 2 }}(\R^2)}.
\end{equation}
Since $u^r$ solves $P(u^r)=-P(u^H)$ with homogeneous boundary conditions, by  using Corollary~\ref{usful}, we get
$$ 
\|\langle D \rangle^m u^r\|_{{H}^k(\mathcal{S})}
\leq  \omega( \| \langle \pa \rangle^m  \eta \|_{{H}^{ k - {1 \over 2 } }(\mathbb{R}^2)}
+ \| \eta_{0} \|_{\mathcal{W}^{m+k}(\mathbb{R}^2 )}\big)
\|  \langle  D  \rangle^m   \big(g \nabla_{X,z}u^H\big)\|_{{H}^{k- 1 }(\mathcal{S})}.
$$
Next,  by using Lemma \ref{metricest}, we can  write
$$ \|  \langle  D  \rangle^m   \big(g \nabla_{X,z}u^H\big)\|_{{H}^{k-1}(\mathcal{S})}
\leq \| \langle D  \rangle^m  \big(g_{1} \nabla_{X,z}u^H\big)\|_{{H}^{k-1}(\mathcal{S})}
+ \omega( \|\eta_{0}\|_{\mathcal{W}^{m+k}})  \| \langle D  \rangle^m u^H \|_{{H}^{k}
(\mathcal{S})}.$$
 From Lemma  \ref{prodS}  we  infer
$$    \| \langle D  \rangle^m  \big(g_{1} \nabla_{X,z}u^H\big)\|_{{H}^{k-1}(\mathcal{S})}
\lesssim 
\|  \langle D \rangle^m g_{1} \|_{{H}^{k-1} (\mathcal{S})} 
  \| \langle D  \rangle^m u^H \|_{{H}^k
(\mathcal{S})}.$$
This yields by using \eqref{metricest} 
 $$  \|\langle D  \rangle^m u^r\|_{{H}^k(\mathcal{S})}
\leq   \omega( \|  \langle \pa \rangle^m  \eta \|_{{H}^{ k - {1 \over 2 }}(\mathbb{R}^2)}
+ \| \eta_{0} \|_{\mathcal{W}^{m+k}(\mathbb{R}^2 )}\big)  \| \langle  D  \rangle^m u^H \|_{{H}^k
(\mathcal{S})}$$
and hence   Lemma~\ref{lema_psi}   follows by  combining this estimate with \eqref{flat}.

\end{proof}
\bigskip

We are now in position to  give the proof of (\ref{rennes}).  
We first prove it for $\sigma = -1/2, \, 1/2$  and $3/2$.
Let us write
$$
\partial^\alpha=\partial_t^j\partial_X^\beta,\quad j+|\beta|=|\alpha|.
$$
We need to evaluate the quantity
$$
\|\partial_t^j\partial_X^\beta \big(G[\eta+\eta_0]u \big)\|_{H^{\sigma }}=
\|\partial_t^j\partial_X^\beta\Lambda^\sigma \big(
G[\eta+\eta_0]u \big)\|_{L^2}.
$$
By duality, we write for $v\in \mathcal{S}(\mathbb{R}^2)$,
\begin{eqnarray*}
(\partial_t^j\partial_X^\beta\Lambda^\sigma  \big(G[\eta+\eta_0]u\big),v) & =
&(-1)^{|\beta|}\partial_t^j(G[\eta+\eta_0]u, \partial_X^\beta\Lambda^\sigma   v)
\\
& = &
(-1)^{|\beta|}\partial_t^j\int_{\mathcal{S}}g(X,z)\nabla_{X,z}(\psi^u)\cdot
\nabla_{X,z}(\partial_X^\beta\Lambda^\sigma {\bf  v})\,dXdz
\\
& = &
\int_{\mathcal{S}}\Lambda^{\sigma + {1\over 2 }} \pa^\alpha(g(X,z)\nabla_{X,z}\psi^u)\cdot\nabla_{X,z}
\Lambda^{- {1 \over 2 } }{\bf v}\,dXdz,
\end{eqnarray*}
where ${\bf v}$ is defined by
\begin{equation}\label{bf_v}
{\bf v}(x,y,z)=\frac{ \cosh \big( \sqrt{D_x^2 + D_y^2} (z+1) \big) }{ \cosh \sqrt{D_x^2+ D_y^2} }(v).
\end{equation}
Since we have by using (\ref{flat}) that 
$$
\|\nabla_{X,z} \Lambda^{ - {1 \over 2 } }{\bf v}\|_{L^2({\mathcal {S}})}\leq C\|v\|_{L^2(\R^2)}\,,
$$
we  get from  Cauchy-Schwarz that
$$ \Big| \int_{\mathcal{S}}\Lambda^{\sigma + {1\over 2 }} \pa^\alpha(g(X,z)\nabla_{X,z}\psi^u)\cdot\nabla_{X,z}
 \Lambda^{- {1 \over 2 }}{\bf v}\,dXdz
\Big| \leq   C  \| \langle D \rangle^m  \big( g(X,z)\nabla_{X,z}\psi^u \big) \|_{{H}^{ \sigma + {1\over 2 } }(\mathcal{S})}
\|v\|_{L^2(\R^2)}
$$
and thus, we can use  again Lemma \ref{prodS} and  Lemma \ref{metricest}
 (note that $\sigma + 1/2 \in \mathbb{N}$)   to get
  that 
\begin{eqnarray*}
& &  \Big| \int_{\mathcal{S}}\Lambda^{ \sigma + {1\over 2 }} \pa^\alpha(g(X,z)\nabla_{X,z}\psi^u)\cdot\nabla_{X,z} \Lambda^{- {1 \over 2 } }{\bf v}\,dXdz
\Big|  \\  & &\leq   C\big( \| \langle D \rangle ^m g_{1}\|_{{H}^{\sigma+ {1\over 2 } }(\mathcal{S})}+ \| g_{2}\|_{\mathcal{W}^{m+\sigma+ {1\over 2 } }(\mathcal{S})}
\big) \| \langle  D  \rangle^m \psi^u \|_{{H}^{\sigma + {3 \over 2 } }(\mathcal{S})}  \|v\|_{L^2(\R^2)}\\
& &  \leq    \omega\big( \|  \langle \pa \rangle^m  \eta \|_{{H}^{\sigma+ 1 }(\R^2)} +
\| \eta_{0} \|_{\mathcal{W}^{m+\sigma + {3 \over 2 }}(\R^2)} \big) \| \langle D  \rangle^m \psi^u \|_{{H}^{\sigma + {3 \over 2 }}(\mathcal{S})} \|v\|_{L^2(\R^2)}.
\end{eqnarray*}
Consequently, we get by using Lemma \ref{lema_psi} with $k= \sigma + 3/2$ to get  that
\begin{multline*}
\Big| \int_{\mathcal{S}}\Lambda^{\sigma + {1\over 2  }}   \pa^\alpha(g(X,z)\nabla_{X,z}\psi^u)\cdot\nabla_{X,z}
 \Lambda^{- {1 \over 2 } }{\bf v}\,dXdz
\Big|  \\
\leq    \omega\big( \|  \langle \pa \rangle^m  \eta \|_{{H}^{\sigma+ 1 }(\R^2)} +
\| \eta_{0} \|_{\mathcal{W}^{m+\sigma + {3 \over 2 }}(\R^2)} \big) \| \langle \pa \rangle^m  u \|_{{H}^{\sigma + 1}(\R^2)}
  \|v\|_{ L^2 (\R^2)}
  \end{multline*}
  and hence, we find that
  $$ \| \partial^\alpha \big( G[\eta + \eta_{0}] u \big) \|_{H^{ \sigma }(\mathbb{R}^2)}
   \leq   \omega\big( \|  \langle \pa \rangle^m  \eta \|_{{H}^{\sigma+ 1 }(\R^2)} +
\| \eta_{0} \|_{\mathcal{W}^{m+\sigma + {3 \over 2 }}(\R^2)} \big) \|
 \langle \pa \rangle^m  u \|_{{H}^{\sigma + 1 }(\R^2)}.$$
This proves \eqref{rennes} for $\sigma= -1/2, \,  1/2 $ and  $3/2$ and actually the refined version \eqref{rennesrefined}.

     To get the $H^1$ estimate,  it suffices to interpolate between the  $H^{1\over 2}$ 
      estimate and the $H^{3\over 2}$ estimate. More precisely, 
       we define the linear  operator $A$ acting on the tensor  $(\partial^\alpha u)_{|\alpha | \leq m}$ as
        $A  (\partial^\alpha u)_{|\alpha | \leq m} = \big(\partial^\alpha (G[\eta + \eta_{0}] \cdot u)\big)_{|\alpha | \leq m}.$
         Since this operator  maps continuously $H^{3 \over 2} $ in $H^{1 \over 2}$
          and $H^{5\over 2}$ into $H^{  {3 \over 2 } }$, it also maps  continuously $H^2$ in $H^1$.
        This ends the proof of \eqref{rennes}

Let us now turn to the proof of the commutator estimate (\ref{rennes2}). For $v\in H^{\frac{1}{2}}(\R^2)$, we can write
\begin{multline}
\label{comg}
([\partial^\alpha, G[\eta + \eta_{0}]](u),v)=\int_{\mathcal{S}}g\nabla_{X,z}(\partial^\alpha\psi^u-\psi^{\partial^\alpha
u})\cdot \nabla_{X,z}{\bf v}\,dXdz \\ +\int_{\mathcal{S}}[\partial^\alpha,g]\nabla_{X,z}\psi^u\cdot\nabla_{X,z}{\bf v}\,dXdz,
\end{multline}
where ${\bf v}$ is again defined by (\ref{bf_v}). We have that
\begin{equation}\label{iut}
P(\partial^\alpha\psi^u-\psi^{\partial^\alpha u})=
[P,\partial^\alpha]\psi^u={\rm div}_{X,z}([g,\pa^\alpha]\nabla_{X,z}\psi^u)
\end{equation}
and moreover  $\partial^\alpha\psi^u-\psi^{\partial^\alpha u}$ satisfies homogeneous
boundary conditions. Multiplying (\ref{iut}) by
$\partial^\alpha\psi^u-\psi^{\partial^\alpha u}$ and integrating over
$\mathcal{S}$ yields
$$
\|\nabla_{X,z}(\partial^\alpha\psi^u-\psi^{\partial^\alpha u})\|_{L^2(\mathcal{S})}\leq 
C\|[g,\pa^\alpha]\nabla_{X,z}\psi^u\|_{L^2(\mathcal{S})}\,.
$$
To estimate the commutator, we need
to estimate $\|\pa^\beta g_{1} \pa^\gamma \nabla_{X,z} \psi^u \|_{L^2(\mathcal{S})}$
for $|\beta|+ |\gamma| \leq m, $ $|\gamma | \neq m$.
 Again, when $|\gamma | \leq m-2$, we can use  the Sobolev embedding $
  H^1(\mathcal{S})  \subset L^4(\mathcal{S})$ while when $|\gamma|=m-1$,
   we put  $\nabla^\beta g $ in $L^\infty$. This yields (since $m \geq 2$)
$$ \|\pa^\beta g \pa^\gamma \nabla_{X,z} \psi^u \|_{L^2(\mathcal{S})}
\leq \omega \big(   \| \langle  D \rangle^m  g_{1}\|_{{H}^1} + \|g_{2}\|_{\mathcal{W}^m} \big)
 \| \langle D \rangle^{m-1} \psi^u \|_{{H}^{1}(\mathcal{S})}$$
and hence,  we obtain from Lemma \ref{metricest} and Lemma \ref{lema_psi} that
$$
\|[g,\pa^\alpha]\nabla_{X,z}\psi^u\|_{L^2(\mathcal{S})}
\leq \underline{\omega}
\| \langle D \rangle^{m-1}\psi^u\|_{{H}^{1}}\leq
\underline{\omega}
\| \langle \pa  \rangle^{m-1} u\|_{{H}^{1 \over 2 }(\R^2)}.
$$
Consequently, we can use \eqref{comg} and the above estimates to get  from Cauchy-Schwarz that
$$
|([\partial^\alpha, G[\eta+ \eta_{0}]](u),v)|\leq \underline{\omega}\|\langle \pa \rangle^{m-1} u\|_{{H}^{{1\over 2}}(\R^2)}
\|v\|_{H^{\frac{1}{2}}(\R^2)}\,.
$$
This proves (\ref{rennes2}) by duality.

Let us now turn to the proof of the bounds on the Frechet derivative of $G[\eta+\eta_0]u$.
We only consider the case $n=1$, the case $n>1$ can be handled by applying a
straightforward induction argument (see \cite{Alvarez-Lannes} for similar analysis).
Moreover, we focus on the case  $l=0$ which is the most interesting one since we need
a sharp estimate with respect to the regularity of $h$ in this case.
By duality, we write for $v\in  \mathcal{S}(\mathbb{R}^2)$ and $k=0,1$ (we shall take $k = \sigma + 1/2$)
\begin{eqnarray*}
(\partial^\alpha\Lambda^{k}
D_{\eta}G[\eta+\eta_0](u)\cdot h,v) 
& = &
\int_{\mathcal{S}}\Lambda^{k}\partial^\alpha
(
g\,\nabla_{X,z}
(D_{\eta}\psi^u\cdot h)
)
\cdot\nabla_{X,z}{\bf v}\,dXdz
\\
& & +
\int_{\mathcal{S}}
\Lambda^{k}\partial^\alpha
((D_{\eta} g\cdot h)\nabla_{X,z}\psi^u)\cdot\nabla_{X,z}{\bf v}\,dXdz
\\
& \equiv &
J_1+J_2,
\end{eqnarray*}
where ${\bf v}$ is defined by (\ref{bf_v}). 
Using the Cauchy-Schwarz inequality, we obtain the bound
$$
J_1\leq C\| \langle D \rangle^m 
(g(X,z)\nabla_{X,z}(D_{\eta}\psi^u\cdot h))\|_{{H}^k(\mathcal{S})}\|v\|_{H^{\frac{1}{2}}}\,.
$$
For $k=0, 1$ and $m\geq 2$,  by using Lemma \ref{prodS} and Lemma \ref{metricest} we find
\begin{eqnarray}
\nonumber
\| \langle D \rangle^m 
(g(X,z)\nabla_{X,z}(D_{\eta}\psi^u\cdot h))\|_{{H}^k(\mathcal{S})}
& \leq &\omega \big( \| \langle D \rangle^m   g_{1}\|_{{H}^{k}}
+ \| g_{2} \|_{\mathcal{W}^{m+k}} \big)\,\,
\|\langle D \rangle^m (D_{\eta}\psi^u\cdot h)\|_{{H}^{k+1}(\mathcal{S})}\
\\
\label{zvezda}
& \leq &
\underline{\omega}\, \| \langle D \rangle^m (D_{\eta}\psi^u\cdot
h)\|_{{H}^{k+ 1 }(\mathcal{S})}\,.
\end{eqnarray}
Next, we need to  evaluate
$\| \langle D \rangle^m (D_{\eta}\psi^u\cdot h)\|_{{H}^{k+ 1 }(\mathcal{S})}$. 
For that purpose, we observe that $D_{\eta}\psi^u\cdot h$ solves the problem
$$
P(D_{\eta}\psi^u\cdot h)=-{\rm div}(D_{\eta}g\cdot h\,\,\nabla_{X,z}\psi^u)
$$
on $\mathcal{S}$ with homogeneous boundary conditions.
Thus, using Corollary~\ref{usful}, we get in particular  that 
\begin{equation}\label{deux_cas}
\| \langle D \rangle^m 
D_{\eta}\psi^u\cdot h\|_{{H}^{k+1 }(\mathcal{S})}\leq 
\underline{\omega}\,\| \langle D \rangle^m \, 
(D_{\eta}g\cdot h\,\,\nabla_{X,z}
\psi^u) \|_{{H}^k (\mathcal{S})}.
\end{equation}
To estimate  the right hand side of (\ref{deux_cas}), we use again Lemma \ref{prodS}, we find
\begin{eqnarray}\label{kolia1}
& &\| \langle D  \rangle^m\, 
(D_{\eta}g\cdot h\,\,\nabla_{X,z}
\psi^u) \|_{{H}^k (\mathcal{S})}  \lesssim \ \| \langle D \rangle^m   \big( D_{\eta}g \cdot h \big)   \|_{{H}^k(\mathcal{S})}
\|\langle D  \rangle^m \nabla_{X,z}\psi^u\|_{{H}^k(\mathcal{S}) }.
\end{eqnarray}
Using Lemma~\ref{lema_psi}, we get
\begin{equation}\label{kolia2}
\|\langle D  \rangle^m \nabla_{X,z}\psi^u\|_{{H}^k(\mathcal{S}) }
\leq  \underline{\omega}\| \,  \langle \pa  \rangle^m  u\|_{{H}^{ k + {1 \over 2 } }(\R^2)}\,.
\end{equation}
In order to  finish  the estimate for $J_1$, it remains to evaluate 
$    \| \langle  D \rangle^m   D_{\eta}g \cdot h  \|_{{H}^k(\mathcal{S})}$
(we recall that in this situation $D_{\eta}g_{2}\cdot h=0$).
Coming back to the definition of $g$ in terms of $\theta$ an by using
Remark~\ref{remtheta}, we get the bound
\begin{eqnarray}\label{kolia3}
& &\| \langle  D \rangle^m  \big( D_{\eta}g \cdot h \big)  \|_{{H}^k(\mathcal{S})}
\leq
\underline{\omega} \, \| \langle \pa \rangle^m h\|_{{H}^{ k+ {1 \over 2 }}(\mathcal{S})}\,.
\end{eqnarray}
Combining the above estimates yields the bound
$$
J_1\leq\underline{\omega} \, \|  
\langle \pa \rangle^m u\|_{{H}^{k + {1 \over 2 } }(\R^2)} \,\|
 \langle \pa  \rangle^m  h\|_{{H}^{ k + { 1 \over 2 }}(\mathbb{R}^2)} \, \| v \|_{H^{1 \over 2 }(\mathbb{R}^2)}\,.
$$
This ends the analysis of the contribution of $J_1$. Let us now turn to the analysis of $J_2$.
 By using again  the Cauchy-Schwarz inequality, we obtain
$$
J_2\leq C\| \langle D \rangle^m 
((D_{\eta}g(X,z)\cdot h)\nabla_{X,z}\psi^u)\|_{{H}^k(\mathcal{S})}\
\|v\|_{H^{\frac{1}{2}}(\mathbb{R}^2) }\,.
$$
Next,  by using  again Lemma \ref{prodS}, we get 
$$
\|\langle D \rangle^m 
((D_{\eta}g(X,z)\cdot h)\nabla_{X,z}\psi^u)\|_{{H}^k(\mathcal{S})}\leq
C\|\langle D \rangle^m  D_{\eta}g\cdot h  \|_{{H}^k(\mathcal{S})}\,\,
\|\langle D \rangle^m  \psi^u\|_{{H}^{k+ 1 }(\mathcal{S})}
$$
and then we can conclude the bound of $J_2$ as we did in the analysis of
$J_1$, see (\ref{kolia1}), (\ref{kolia2}), (\ref{kolia3}). 
We have thus proven that
$$| (\partial^\alpha\Lambda^{k}
D_{\eta}G[\eta+\eta_0](u)\cdot h,v)  | \leq 
\underline{\omega}\|  \langle \pa  \rangle^m 
u\|_{{H}^{k + {1 \over 2 }}(\R)^2} \,\|  \langle \pa \rangle^m  h\|_{{H}^{ k + {1\over 2 }}(\mathbb{R}^2)}
 \,\|v \|_{H^{1 \over 2 }(\mathbb{R}^2) }.$$
From this, we  deduce that
$$  \| \partial^\alpha D_{\eta}G[\eta+\eta_0](u)\cdot h \|_{H^{\sigma}}
\leq 
\underline{\omega}\|  \langle \pa \rangle^m 
u\|_{{H}^{\sigma + 1 }(\R)^2} \,\|  \langle \pa \rangle^m  h\|_{{H}^{\sigma + 1 }(\mathbb{R}^2)}.$$
for $\sigma= -1/2,\, 1/2$ with  $\sigma = k - 1/2$.
This yields the desired estimate 
 for  $\sigma= -1/2,\, 1/2$
and actually the refined version \eqref{San} stated in Remark
 \ref{Sanremark}. 

It remains to study  the case $\sigma=1$. Again, we start from  

\begin{eqnarray*}
(\partial^\alpha\Lambda \big(
D_{\eta}G[\eta+\eta_0](u)\cdot h \big),v) 
& = &
\int_{\mathcal{S}}\Lambda \partial^\alpha
(
g\,\nabla_{X,z}
(D_{\eta}\psi^u\cdot h)
)
\cdot\nabla_{X,z}{\bf v}\,dXdz
\\
& & +
\int_{\mathcal{S}}
\Lambda\partial^\alpha
((D_{\eta} g\cdot h)\nabla_{X,z}\psi^u)\cdot\nabla_{X,z}{\bf v}\,dXdz
\\
& \equiv &
\tilde{J}_1+\tilde{J}_2.
\end{eqnarray*}
 We first   obtain  for $\tilde{J}_{1}$ that 
\begin{eqnarray}
\label{frech1}
\tilde{J}_1 &\leq  & C\| \Lambda \partial^\alpha \big
(g(X,z)\nabla_{X,z}(D_{\eta}\psi^u\cdot h)\big)\|_{H^{1\over 2 }(\mathcal{S})}
\| \nabla_{X,z}  \Lambda^{- {1 \over 2 } }{\bf v} \|_{L^{{2}}(\mathcal{S})} \\
\nonumber & \leq &   C\|  \partial^\alpha 
(g(X,z)\nabla_{X,z}(D_{\eta}\psi^u\cdot h))\|_{H^{3\over 2 }(\mathcal{S})}
\|v\|_{
L^2(\R^2)}\,.
\end{eqnarray}
To estimate this term, we shall use the following classical lemma about products
in Sobolev spaces
\begin{lem}
\label{LP}
We have the estimates
\beq
\label{LP1} \| u v \|_{H^{3 \over 2}(\mathcal{S})}   \leq  C_{\sigma} \|u \|_{\mathcal{C}^\sigma(\mathcal{S})}\,
\|v \|_{H^{ 3 \over 2}(\mathcal{S})} \eeq
 for every $\sigma \in (3/2, 2)$,
\beq
\label{LP2}
 \| D^\alpha u  D^\beta v \|_{H^{3\over 2}(\mathcal{S})} 
  \leq C_{m}\| \langle D \rangle^m u \|_{H^{3 \over 2}(\mathcal{S})} \| \langle D \rangle^m v \|_{H^{3\over2}
  (\mathcal{S})}
\eeq
 for $ |\alpha | + |\beta | \leq m$, $m \geq 2$.
 \end{lem}
\begin{proof}
 The first  estimate is  an easy consequence of the fact that
 \beq
 \label{sobfrac} \| f \|_{H^{1 \over 2}(\mathcal{S})}^2 =  \| f \|_{L^2 (\mathcal{S})}^2 + \int_{\mathcal{S}} \int_{\mathcal{S}}
  { |f(Y)- f(Y') |^2 \over | Y- Y' |^{ 4 } } \, dY dY'.\eeq
  Indeed, let us first prove that 
 \beq
 \label{prodprelim}
  \| u v \|_{H^{1 \over 2}(\mathcal{S})}   \leq  C_{\beta} \|u \|_{\mathcal{C}^\beta(\mathcal{S})}\,
\|v \|_{H^{ 1 \over 2}(\mathcal{S})}
\eeq
as soon as $\beta \in (1/2, 1)$.
From \eqref{sobfrac},  the term involving the $L^2$ norm can be easily estimated, 
 for the other term, we write
\begin{eqnarray*}
  \| u v \|_{H^{1 \over 2}(\mathcal{S})}^2  &  \lesssim &  
\| u \|_{L^\infty}^2 \|v \|_{L^2}^2  + \int_{\mathcal{S}} \int_{\mathcal{S}}
  { |u(Y)- u(Y') |^2\,  |v(Y') |^2 \over | Y- Y' |^{ 4 } } \, dY dY' \\
   & &  + \int_{\mathcal{S}} \int_{\mathcal{S}}
  { |v(Y)- v(Y') |^2\,  |u(Y)| ^2 \over | Y- Y' |^{ 4 } } \, dY dY'.
  \end{eqnarray*}
   The second integral is obviously bounded by $ \|u\|_{L^\infty}^2 \|v\|_{H^{1 \over 2}}^2$.
    To estimate the first one, we use that
  \begin{eqnarray*}  \int_{\mathcal{S}}
  { |u(Y)- u(Y') |^2\,  \over | Y- Y' |^{ 4 } } \, dY
   & = &  \int_{|Y- Y'|\geq 1 }
  { |u(Y)- u(Y') |^2\,  \over | Y- Y' |^{ 4 } } \, dY + \int_{|Y- Y'|\leq 1 }
  { |u(Y)- u(Y') |^2\,  \over | Y- Y' |^{ 4 } }\, dY   \\
  &\lesssim  & \|u\|_{L^\infty}^2 + \|u \|_{\mathcal{C}^\beta}^2
  \end{eqnarray*}
   and hence we find 
$$ \int_{\mathcal{S}} \int_{\mathcal{S}}
  { |u(Y)- u(Y') |^2\,  |v(Y') |^2 \over | Y- Y' |^{ 4 } } \, dY dY' \lesssim  \|v \|_{L^2}^2\,  \|u \|_{\mathcal{C}^\beta}^2.$$
   This proves \eqref{prodprelim}.
 To get \eqref{LP1}, it suffices to use that
$$ \| u v \|_{H^{3 \over 2} } \leq  \|u v \|_{H^1 }+  \|u \, \nabla_{X,z}v \|_{H^{1 \over 2}} + \|\nabla_{X,z} u \,v \|_{H^{1 \over 2}} $$
 and to apply \eqref{prodprelim} to the second and  the third term.

 To prove the second estimate, it suffices to consider the case $| \alpha | \leq | \beta|$.
  When $| \beta | \leq m-1$, we write since $H^2({\mathcal{S})}$ is an algebra   the crude estimate 
 \begin{eqnarray*}
 \| D^\alpha u  D^\beta v \|_{H^{3\over 2}} & \lesssim &  \| D^\alpha u  D^\beta v \|_{H^2}
  \lesssim  \| D^\alpha u \|_{H^2}\, \|  D^\beta v \|_{H^2} \\ &  \lesssim &   \| \langle D \rangle^m u \|_{H^1} \,
   \| \langle D \rangle^m v \|_{H^1}
    \lesssim   \| \langle D \rangle^m u \|_{H^{3 \over 2 }} \,
   \| \langle D \rangle^m v \|_{H^{3\over 2}} .
   \end{eqnarray*}
   When   $|\beta| =m$ and thus $|\alpha |=0$, we use \eqref{LP1} to get
   $$ \| u D^\beta v \|_{H^{3 \over 2}} \lesssim \|u \|_{\mathcal{C}^\sigma} \, \| \langle D \rangle^m v \|_{H^{3\over2}}
    \lesssim \| \langle D \rangle^m u \|_{H^{3 \over 2}} \| \langle D \rangle^m v \|_{H^{3\over2}}.$$
    Note that the last estimate is a consequence of the Sobolev embedding and the fact that
     $m \geq 2$.
   This ends the proof of Lemma \ref{LP}.

  \end{proof}

  Let us come back to the estimate \eqref{frech1} of $\tilde{J}_{1}$.
   By using \eqref{LP2}, we get 
 \begin{multline}\label{gH32}
 \|  \partial^\alpha 
(g(X,z)\nabla_{X,z}(D_{\eta}\psi^u\cdot h))\|_{H^{3\over 2 }(\mathcal{S})} 
\\
\leq C\big( \|\langle D \rangle^m g_{1} \|_{H^{3 \over 2 } (\mathcal{S}) } +
\| g_{2} \|_{\mathcal{W}^{m+2 } (\mathcal{S}) } \big)  \, \| \langle D \rangle^m 
\nabla_{X,z}(D_{\eta}\psi^u\cdot h))\|_{H^{3\over 2 }(\mathcal{S})}
 \end{multline} 
and thus by using Lemma \ref{metricest}, we find
$$  
\tilde{J}_1 \leq \underline{\omega}\,  \| \langle D \rangle^m 
\nabla_{X,z}(D_{\eta}\psi^u\cdot h))\|_{H^{3\over 2 }(\mathcal{S})}
\|v\|_{L^2(\R^2)}\,.$$
Next, as already observed, we have
$$ P  (D_{\eta}\psi^u\cdot h) = - \nabla_{X,z}\cdot \big( D_{\eta}g \cdot h \, \nabla_{X,z} \psi^u\big):=
 \nabla_{X,z}\cdot  H. $$
From Corollary \ref{usful}, we have in particular
$$  \|\langle D \rangle^m \big( D_{\eta}\psi^u\cdot h\big) \|_{H^2(\mathcal{S})} \leq
\omega\big( \| \langle D \rangle^m  \eta \|_{{H}^{3\over 2 }(\R^2)} + \|\eta_{0}\|_{\mathcal{W}^{m+2}(\R^2)}
\big) \| \langle D \rangle^m H\|_{H^1(\mathcal{S}) }$$
and  $$   \|\langle D \rangle^m \big( D_{\eta}\psi^u\cdot h\big) \|_{H^3(\mathcal{S})} \leq
\omega\big( \|\langle D \rangle^m  \eta \|_{{H}^{5 \over 2 }(\R^2)} + \|\eta_{0}\|_{\mathcal{W}^{m+3}(\R^2)}
\big) \| \langle D \rangle^m H\|_{H^2(\mathcal{S}) }.$$
Consequently, we can interpolate between the two estimates to get
\beq
\label{interpolest}  \|\langle D \rangle^m \big( D_{\eta}\psi^u\cdot h\big) \|_{H^{5 \over 2}(\mathcal{S})} \leq
\underline{\omega} \, \| \langle D \rangle^m H\|_{H^{3\over 2} (\mathcal{S}) }.\eeq
Therefore, we infer
$$ \tilde{J}_1 \leq \underline{\omega}\,  \| \langle D \rangle^m 
(D_{\eta}\psi^u\cdot h))\|_{H^{5\over 2 }(\mathcal{S})}
\|v\|_{L^{2}(\R^2)} \leq \underline{\omega} \, \| \langle D \rangle^m \big( D_{\eta}g \cdot h \, \nabla_{X,z} \psi^u  \big)\|_{H^{3 \over 2}
(\mathcal{S})} \,\|v\|_{L^2(\R^2)}.$$
By using \eqref{LP2} in Lemma \ref{LP},  we obtain
$$ \| \langle D \rangle^m \big( D_{\eta}g \cdot h \, \nabla_{X,z} \psi^u  \big)\|_{H^{3 \over 2}
(\mathcal{S})} \leq \| \langle D \rangle^m \big(   D_{\eta}g \cdot h \big) \|_{H^{3\over 2}(\mathcal{S})}  \,  \| \langle D \rangle^m  \nabla_{X,z} \psi^u \|_{H^{3\over 2}(\mathcal{S})}.$$
Since $D_{\eta }g \cdot h $ has roughly the regularity of $ \nabla_{X,z} \theta_{1}(h)$
(by using the definition of $g$ and Remark~\ref{remtheta}), we find
$$ \| \langle D \rangle^m \big(   D_{\eta}g \cdot h \big) \|_{H^{3\over 2}(\mathcal{S})} \leq \underline{\omega} \,
\| \langle D \rangle^m \theta_{1}(h) \|_{H^{ 5\over 2}(\mathcal{S})} \leq \underline{\omega}\,
 \| \langle \partial \rangle^m  h  \|_{H^2(\mathbb{R}^2)} 
    \leq \underline{\omega} \,
 \| \langle \partial \rangle^{m+ 1 }  h  \|_{H^1(\mathbb{R}^2)} 
$$
 where  the intermediate estimate comes from a  new application of Lemma~\ref{harmonic}.
 Finally, from Lemma \ref{lema_psi},  we get
 $$  \| \langle D \rangle^m \nabla_{X,z} \psi^u \|_{H^1(\mathcal{S})} \leq \underline{\omega} \,  \| \langle \pa \rangle^m
  u \|_{H^{3\over 2 }(\R^2)}, \quad  \| \langle D \rangle^m \nabla_{X,z} \psi^u \|_{H^2
  (\mathcal{S})} \leq \underline{\omega} \,  \| \langle \pa \rangle^m
  u \|_{H^{5\over 2}(\R^2)}.
  $$
  Consequently, we can interpolate between the two estimates to get
\beq
 \| \langle D \rangle^m \nabla_{X,z} \psi^u \|_{H^{3\over 2 }(\mathcal{S})} \leq \underline{\omega}  \, \| \langle \pa \rangle^m
  u \|_{H^{ 2 }(\R^2)}.\eeq
  We thus obtain that
$$ \tilde{J}_{1} \leq \underline{\omega}  \,  \| \langle \partial \rangle^{m+ 1 }  h  \|_{H^1(\mathbb{R}^2)} 
\| \langle \pa \rangle^m
  u \|_{H^{ 2 }(\R^2)} \, \|v \|_{L^2(\R^2) }.$$
  By the same kind of argument, we obtain a similar estimate for $\tilde{J}_{2}$ and therefore, 
   we get
$$ 
\| \partial^\alpha \big( G[\eta + \eta_{0}] u \big) \|_{H^1(\R^2)}
    \leq     \underline{\omega}   \,  \| \langle \partial \rangle^{m+ 1 }  h  \|_{H^1(\mathbb{R}^2)} 
\| \langle \pa \rangle^m
  u \|_{H^{ 2 }(\R^2)}.
$$

This completes the proof of Proposition~\ref{pak}.
\end{proof}
We shall also need to estimate the Dirichlet-Neumann operator in the
case when it acts on a  smooth function  that does not belong to a Sobolev space. 
\begin{prop}\label{pak_bis}
For every $m\geq 0$ and $\mu \in (0, 1)$, we have the estimates  
\begin{equation}\label{maximum}
\|\langle \partial \rangle^m  G[\eta+\eta_0]u\|_{\mathcal{C}^{1+ \mu}(\R^2) }\leq 
\omega\big(\|\langle\partial\rangle^{m+3} \eta\|_{H^1}+\| \langle \pa \rangle^m \eta_0\|_{{\mathcal C}^{2+\mu}}\big)
\| \langle \pa \rangle^m u\|_{{\mathcal C}^{2+\mu }}\,
\end{equation}
and
\begin{multline}\label{maximum_pak}
\|\langle \partial \rangle^m  \big(D_{\eta}^nG[ \eta + \eta_0]u\cdot \big(h_{1}, \cdots, h_{n}\big) \big) \|_{\mathcal{C}^{1+ \mu } }\leq 
\omega\big(  \|\langle\partial\rangle^{m+3} \eta\|_{H^1} +      \| \langle \pa \rangle^m \eta_0\|_{{\mathcal C}^{2+\mu }}\big) \\
\Big(
\| \langle \pa \rangle^m h_{1}\|_{{\mathcal C}^{2 + \mu}}\cdots \| \langle \pa \rangle^m h_{n}\|_{{\mathcal C}^{2 + \mu}} \Big)
\| \langle \pa \rangle^m u \|_{{\mathcal C}^{2+\mu}}\,.
\end{multline}
Moreover, for $n>l\geq 0$,
\begin{multline}\label{maximum_tris}
\|\langle \partial \rangle^m 
\big(
D_{\eta}^n G[\eta+\eta_0]u\cdot(h_1,\cdots,h_n)\big)\|_{H^1}\leq \omega\big(\|\langle\partial\rangle^{m+3} \eta\|_{H^1}+
\|\eta_0\|_{{\mathcal W}^{m+ 3}}\big)
\\
\times\Big(\prod_{j=1}^l\| h_j\|_{{\mathcal W}^{m+3}}\Big)
\Big(\prod_{j=l+1}^n\|\langle \partial\rangle^{m+1} h_j\|_{H^{1}}\Big)\|u\|_{{\mathcal W}^{m+3}}
\end{multline}
(the first product is considered $1$ in the case $l=0$).
\end{prop}
This proposition will be used when basically $(\eta_{0}, u)$ is the solitary wave $ (\eta_{\eps}, \varphi_{\eps})$
and thus the way the estimates depend on the regularity of these functions is not very important
 for our purpose.  The important fact   is again that the estimates involve at
  most  the norm $\| \langle \pa \rangle^{m+ 3 } \eta  \|_{H^1}$.
   The estimate \eqref{maximum_tris} will be very useful to estimate terms like
    $$ \big(G[\eta+ \eta_{\eps}] - G [\eta_{\eps}]\big) \varphi_{\eps}
     = \int_{0}^1 D_{\eta} G [\eta + s\eta_{\eps}] \varphi_{\eps} \cdot \eta\, ds$$
     since it gives an $H^1$ estimate of this term as soon as $\eta$ is in some Sobolev space. 
\begin{proof}[Proof of Proposition \ref{pak_bis}]
For $u\in L^\infty$, we define by $\psi^u$ the (well-defined)
solution of the elliptic boundary value problem
\begin{equation}\label{metric_tris}
{\rm div}_{X,z}(g(X,z)\nabla_{X,z}\psi(X,z))=0, \quad (X, z)\in \mathcal{S},\qquad
\partial_{z}\psi(X,-1)=0,\quad \psi(X,0)=u(X)\,.
\end{equation}
The existence of (weak) solutions of (\ref{metric_tris}) can be obtained by
using the $L^\infty$ a-priori bound coming from the maximum  principle (see
\cite{GT}, Chapters~2-6). Observe that thanks to the homogeneous boundary
condition on $z=-1$ the maximum of $\psi$  is necessarily  reached on the boundary $z=0$.
One may also obtain the well-posedness of (\ref{metric_tris}), by
Sobolev type arguments. Namely, one may approach the problem on   $\mathcal{S}$ by 
problems on compact 
domains where the  maximum principle holds, get solutions on these domains by the Sobolev theory  
and then pass to the limit by using the uniform $L^\infty$ estimate. 

The next step is to obtain  regularity estimates  for
$\psi^u$. This will be  a consequence of the following elliptic regularity result:
\begin{lem}
\label{lemschaud}
For $m \geq 0$ and $\mu \in (0, 1)$, the $L^\infty$ solution of 
$$P u = F, \quad (X,z) \in \mathcal{S}, \quad u(X,0)=0, \,  \partial_{z}u (X,-1)= 0$$
 satisfies the estimate
$$ \| \langle D \rangle^m  u \|_{\mathcal{C}^{2 + \mu  }(\mathcal{S})}
 \leq \omega \big(  \|\langle\partial\rangle^{m+3} \eta\|_{H^1(\mathbb{R}^2)}
  + \|\langle \pa \rangle^m \eta_0\|_{{\mathcal C}^{ 2 +
\mu }(\R^2) }\big)
\| \langle D  \rangle^m F\|_{{\mathcal C}^{\mu}(\mathcal{S})}\,.$$
\end{lem}
\begin{proof}
When no spatial derivatives are involved, we have the following classical Schauder
 elliptic regularity result for $u$ (we refer to \cite{GT} for example):
$$
\|   u \|_{\mathcal{C}^{2 + k+  \mu  }(\mathcal{S})}
 \leq \omega \big(  \|g \|_{{\mathcal C}^{  1 + k  +
\mu }(\mathcal{S})) }\big)
\|   F\|_{{\mathcal C}^{ k + \mu }(\mathcal{S})}\,
$$
for every integer $k$.  By an induction on the number of time  derivatives involved,
 we easily deduce from this estimate that 
\beq
\label{schaudspace}
\|    \langle D \rangle^m u \|_{\mathcal{C}^{2 +  \mu  }(\mathcal{S})}
 \leq \omega \big(  \| \langle D \rangle^m g \|_{{\mathcal C}^{  1  +
\mu }(\mathcal{S})) }\big)
\|   \langle D \rangle^m  F\|_{{\mathcal C}^{  \mu }(\mathcal{S})}.
\eeq
To conclude, we first notice from the definition of $g$ that
\beq
\label{ginfty}   \| \langle D \rangle^m g \|_{{\mathcal C}^{  1  +
\mu }(\mathcal{S})) }  \leq \omega\big( \| \langle D \rangle^m \theta_{1} \|_{{\mathcal C}^{  2  +
\mu }(\mathcal{S})) } +  \| \langle D \rangle^m \theta_{2} \|_{{\mathcal C}^{  2  +
\mu }(\mathcal{S})) }\big). \eeq
From the explicit expression of $\theta_{2}$, we obviously have
\beq
\label{theta1infty}  \| \langle D \rangle^m \theta_{2} \|_{{\mathcal C}^{  2  +
\mu }(\mathcal{S}) } \leq  C
\|\langle \pa \rangle^m \eta_0\|_{{\mathcal C}^{ 2 +
\mu }(\R^2) }\eeq
 and moreover,  by Sobolev embedding and Lemma \ref{harmonic}, we have for every  $s>3/2$
\beq
\label{theta2infty}   \| \langle D \rangle^m \theta_{1} \|_{{\mathcal C}^{  2  +
\mu }(\mathcal{S}) }  \leq  C  \| \langle D \rangle^m \theta_{1} \|_{{H}^{  2  + s + 
\mu }(\mathcal{S}) }  \leq  C   \| \langle \pa  \rangle^m \eta  \|_{{H}^{  2  + s + 
\mu - {1 \over 2 } }(\R^2) } \leq   C   \| \langle \pa  \rangle^m \eta  \|_{{H}^{   4  }(\R^2) }\eeq
 since one  can always choose $s$ sufficiently close to $3/2$
   to have $ 2 + \mu +  s - 1/2 <4$. This ends the proof of Lemma \ref{lemschaud}.

\end{proof}

We can now  estimate $\psi^u$.
\begin{lem}\label{lema_psi_bis}
For every $m \geq 0$ and $\mu \in (0, 1)$, we have the estimate 
\begin{equation*}
\|\langle D \rangle^m  \psi^u\|_{\mathcal{C}^{2 + \mu }(\mathcal{S})}
\leq
\omega\big(\|\langle\partial\rangle^{m+3} \eta\|_{H^1(\mathbb{R}^2)}+\|\langle \pa \rangle^m \eta_0\|_{{\mathcal C}^{ 2 +
\mu }(\R^2) }\big)
\| \langle \pa \rangle^m u\|_{{\mathcal C}^{2+\mu}(\R^2)}\,.
\end{equation*}
\end{lem}
\begin{proof}
Again we consider the splitting $\psi^u=u^H+u^r$, where $u^H$ is defined by
(\ref{fourier}).  By standard properties of Fourier multipliers in H\"older spaces, we get
\beq
\label{uHinfty}
\|\langle D \rangle^m   u^H \|_{\mathcal{C}^s(\mathcal{S})}\leq
\| \langle \pa \rangle^m u\|_{{\mathcal C}^{s}(\R^2)}\,
\eeq
for every $s \geq 0 $ which is not an integer. 
Next,  since  $u^r$ solves the elliptic equation $Pu^r = -Pu^H$ with homogeneous
boundary conditions, we get by using Lemma \ref{lemschaud} that 
$$
\|\langle D \rangle^{m}   u^r \|_{\mathcal{C}^{ 2 + \mu }(\mathcal{S})} \leq 
\omega\big(\|\langle\partial\rangle^{m+3} \eta\|_{H^1(\mathbb{R}^2)}+\|\langle \pa \rangle^m \eta_0\|_{{\mathcal C}^{ 2 +
\mu }(\R^2) }\big)
\| \langle D \rangle^m P u^H \|_{{\mathcal C}^{\mu}(\mathcal{S})}.
$$
Furthermore,  since we have
$$  \| \langle D \rangle^m P u^H \|_{{\mathcal C}^{\mu}(\mathcal{S})}
\leq \omega \big( \| \langle D \rangle^m g \|_{\mathcal{C}^{1+ \mu } (\mathcal{S})} \big)
 \| \langle D \rangle^m u^H\|_{{\mathcal C}^{ 2 + \mu}(\mathcal{S})}, $$
  we get the  claimed estimate  by using \eqref{ginfty}, \eqref{theta1infty}, \eqref{theta2infty}
   and \eqref{uHinfty}.
\end{proof}
After these preliminaries, we can  get  (\ref{maximum}). Indeed, 
observe that in terms of $\psi^u$, the  Dirichlet-Neumann operator   reads
$$
(G[\eta+ \eta_{0}]u)(X)=\frac{1 + (\pa_x \theta(X,0) )^2 + (\pa_y\theta(X,0))^2}{\pa_z\theta(X,0)}\pa_z\psi^u(X,0)
-\nabla_{X}\theta(X,0)\cdot\nabla_{X}\psi^u(X,0)\,
$$
with 
$ \theta(X,0)= \eta + \eta_{0}.$
Consequently, we get
$$ \|  \langle \pa \rangle^m G[\eta+ \eta_{0} ] u \|_{\mathcal{C}^{1 + \mu }(\mathbb{R}^2)}  \leq \omega\big( \|
\langle \pa \rangle^m  \theta \|_{ \mathcal{C}^{2 + \mu }(\mathcal{S})} \big) \| \langle \pa \rangle^m  \psi^u \|_{\mathcal{C}^{2 + \mu}(\mathcal{S})}$$
 and hence  (\ref{maximum}) is a consequence of Lemma~\ref{lema_psi_bis}
and \eqref{theta1infty}, \eqref{theta2infty}.

The proof of (\ref{maximum_pak}) can be obtained   in the same way. 
This is left to the reader.

Let us finally give the proof of
(\ref{maximum_tris}). We shall only give the proof for $n=1$ since the
argument for $n>1$ follows by a direct induction argument and we focus on the case
that $h$ is in a Sobolev space since it is the one for which we  really need optimal regularity.
 The case that $h$ is in an H\"older space is covered by  \eqref{maximum_pak}.

Coming back to (\ref{vajno}), we obtain
\begin{eqnarray*}
(\partial^\alpha\Lambda^{\frac{3}{2}}
D_{\eta}G[\eta+\eta_0](u)\cdot h,v) 
& = & 
\int_{\mathcal{S}}\partial^\alpha
\Lambda^{\frac{3}{2}}
(g\,\nabla_{X,z}(D_{\eta}\psi^u\cdot h))\cdot\nabla_{X,z}{\bf v}dXdz
\\
& &+
\int_{\mathcal{S}}\partial^\alpha\Lambda^{\frac{3}{2}}
((D_{\eta}g\cdot h)\nabla_{X,z}\psi^u)\cdot\nabla_{X,z}{\bf v}\,dXdz
\\ 
& \equiv & J_1+J_2,
\end{eqnarray*}
where ${\bf v}$ is defined by (\ref{bf_v}) with $v\in H^{1/2}$.
Using the Cauchy-Schwarz inequality, we get
$$
J_1\leq C\|\langle\partial\rangle^m\Lambda^{\frac{3}{2}}(g\,\nabla_{X,z}(D_{\eta}\psi^u\cdot h))\|_{L^2}\|v\|_{H^{\frac{1}{2}}}\,.
$$
Next (see (\ref{gH32})), we can write
$$
\|\langle\partial\rangle^m\Lambda^{\frac{3}{2}}(g\,\nabla_{X,z}(D_{\eta}\psi^u\cdot
h))\|_{L^2}\leq\omega\big( \|\langle \pa \rangle^m \eta \|_{H^2} + \| \eta_{0} \|_{\mathcal{W}^{m+3 }}  \big)
\|\langle\partial\rangle^m\Lambda^{\frac{3}{2}}(\nabla_{X,z}(D_{\eta}\psi^u\cdot h))\|_{L^2}\,.
$$
Recall that $D_{\eta}\psi^u\cdot h$ solves the problem
$$
P(D_{\eta}\psi^u\cdot h)=-{\rm div}(D_{\eta}g\cdot h\,\,\nabla_{X,z}\psi^u)
$$
on $\mathcal{S}$ with homogeneous boundary conditions. By using  \eqref{interpolest}, we
infer 
$$
\|\langle\partial\rangle^m\Lambda^{\frac{3}{2}}\nabla_{X,z}(D_{\eta}\psi^u\cdot
h)\|_{L^2(\mathcal{S})}\leq  \omega(  \big( \|\langle \pa \rangle^m \eta \|_{H^{5\over 2 }} + \| \eta_{0} \|_{\mathcal{W}^{m+3 }}  \big)
\|\langle D \rangle^m\Lambda^{\frac{3}{2}} \big( D_{\eta}g\cdot h\, \nabla_{X,z}\psi^u \big)\|_{L^2(\mathcal{S})}\,.
$$
Using Lemma~\ref{LP}, we obtain 
$$
\|\langle D \rangle^m\Lambda^{\frac{3}{2}}\big((D_{\eta}g\cdot
h \,\nabla_{X,z}\psi^u \big)\|_{L^2(\mathcal{S})}
\leq
C\|\langle D \rangle^m \nabla_{X,z}\psi^u\|_{C^\sigma(\mathcal{S})}
\| \langle D \rangle^m\Lambda^{\frac{3}{2}}(D_{\eta}g\cdot
h)\|_{L^2(\mathcal{S})},
$$
provided $\sigma>3/2$. 
Coming back to the definition of $g$, thanks to  Lemma~\ref{harmonic} and Remark~\ref{remtheta}, we get
$$
\||\langle D \rangle^m\Lambda^{\frac{3}{2}}(D_{\eta}g\cdot
h)\|_{L^2(\mathcal{S})}
\leq
  \omega(  \big( \|\langle \pa \rangle^m \eta \|_{H^{5\over 2 }} + \| \eta_{0} \|_{\mathcal{W}^{m+3 }}  \big)\,\|\langle \partial\rangle^{m+1} h\|_{H^{1}(\mathbb{R}^2)}.
$$
Next, using Lemma~\ref{lema_psi_bis} with $\mu = \sigma -1$, we get
$$
\|\langle D \rangle^m \nabla_{X,z}\psi^u\|_{C^\sigma(\mathcal{S})}\leq
\omega \big( \| \langle \pa \rangle^{m+3} \eta \|_{H^1} + \|\langle \pa \rangle^m
 \eta_{0}\|_{\mathcal{W}^{m+3}}\big)  \| \langle  \pa \rangle^m u\|_{{\mathcal C}^{2 + \mu }}\,.
$$
Collecting the above bounds, we arrive at
$$
J_1\leq   \omega \big( \| \langle \pa \rangle^{m+3} \eta \|_{H^1} + \|\langle \pa \rangle^m
 \eta_{0}\|_{\mathcal{W}^{m+3}}\big)
\|\langle \partial\rangle^{m+1} h\|_{H^{1}}\| \langle \pa \rangle^m u\|_{{\mathcal C}^{2+\mu}}\|v\|_{H^{\frac{1}{2}}}\,.
$$
The estimate for $J_2$ is very similar and thus will be omitted.
This completes the proof of Proposition~\ref{pak_bis}. Note that we get a slightly better
result than stated.
\end{proof}
In our energy estimates, we shall also  use the following lemma.
\begin{lem}[see Proposition~3.4 of \cite{Alvarez-Lannes}]
\label{G}
There exists  $c>0$ such that for every $\eta\in W^{1,\infty}(\R^2)$  with 
 $ 1 - \|\eta\|_{L^\infty}\geq\delta$ for some $\delta>0$ we have 
$$
(G[\eta]v,v)\geq c(1+\|\eta\|_{W^{1,\infty}(\R^2)})^{-2}
\Big\|\frac{|\nabla|}{(1+|\nabla|)^{\frac{1}{2}}}v\Big\|_{L^2(\R^2)}^2 \, \quad \forall  v\in H^{1 \over 2}(\R^2)
$$
 and
 $$ \big(G[\eta]v, w\big) \leq \omega(\|\eta\|_{W^{1,\infty}(\R^2)}) \|v \|_{H^{1 \over 2 }(\mathbb{R}^2 )  } \, \| w \|_{H^{1 \over 2 }
 (\mathbb{R}^2)}, \, \quad \forall v, \, w \in H^{1 \over 2 }(\mathbb{R}^2).$$
\end{lem}
Note that we have given previously  the proof  of \eqref{DNm}  \eqref{DNC} which are  very  close estimates.



\subsection{Derivation of the quasilinear form}\label{quasin}
The aim of this subsection is to  isolate a principal part  which behaves as a quasilinear
symmetrizable hyperbolic  like system and a remainder which behaves  as a  semi-linear  term  after applying
a sufficient amount of derivatives to the equation.
Let us explain more precisely the strategy. We can consider the water waves system
under an abstract form 
$$ \partial_{t} U= \mathcal{F}(U).$$
When  applying the operator $\partial^\alpha$,   for $|\alpha|$ to be chosen, to the system, we find
\beq
\label{expl1} \partial_{t} \partial^\alpha  U = J\Lambda[ U] \cdot \partial^\alpha U + \mathcal{R}(U)
\eeq
where $ \partial_{t} - J\Lambda[U]$ is the linearized equation  about $U$ and
$\mathcal{R}(U)$  involves some lower order  commutators.  Let us set $U_{\alpha}= \partial^\alpha U$.
If we consider only the principal part in the equation for $U_{\alpha}$, 
in view of the skew symmetry of $J$ and the symmetry of $\Lambda[U]$, one expects 
to get energy estimates by taking the scalar product of the equation with  $\Lambda[U] U_{\alpha}$
and then  by reiterating the same process for higher order derivatives of $U_{\alpha}$.
The energy norm associated to $\Lambda[U]$ will be the $X^0$ norm and 
thus, we expect  to control the norm $\|U_{\alpha}(t) \|_{X^k}$.

A  good  "quasilinear structure"  for  $\partial^\alpha U$ which easily yields an energy estimate arises
for \eqref{expl1} if 
 :
\begin{enumerate}
\item[i)]  the norm  of the remainder  $\| \langle \partial \rangle^k  \mathcal{R}(U)  \|_{X^0}$
can be estimated in terms of   $\|U \|_{X^{k+|\alpha|}}$ when $k$ is sufficiently large. In this case, we shall say that
this term behaves as a semilinear term;
\item[ii)] 
the  estimate of the commutator $[\langle \partial \rangle^k, J\Lambda[U] ] V$
 in the energy space $X^0$ involves   at most the norm $\|U\|_{X^{k+|\alpha|}}$
  when $k$ is sufficiently large.
     \end{enumerate}

If  $\Lambda$ were  a first order operator (this arises classically for the usual quasilinear wave 
 equation rewritten as a first order system), by using the above second property,   we  expect 
   the commutator estimate
$$ \| [\langle \partial \rangle^k, J\Lambda[U] ] V\|_{X^0} \lesssim \|U\|_{X^{k+|\alpha|}} \|V\|_{X^k} $$
which allows to get an a priori  estimate for $U$ under the form
$$ \|U(t) \|_{X^{k+|\alpha|}} \lesssim   \|U(0) \|_{X^{k+|\alpha|}}  + \int_{0}^t \|U(\tau)\|_{X^{k+|\alpha|}}^2\, d\tau.$$
This  is a good without loss  estimate which can be easily combined with an approximation
argument (for example the vanishing viscosity method) in order to get  a local existence result
and to  prove Theorem \ref{main}
when considering a perturbation of  $V^{a}$  by using the Gronwall lemma.
Note that in this situation, there is no need to use time and space derivatives simultaneously.

As  already pointed out in the introduction, in our situation (which is formally close to the one  of higher order wave equations),  in order to close the energy estimate, 
the commutator  $[\langle \partial \rangle^k, J\Lambda[U ]  ]V$
cannot be considered  as harmless (or semilinear) since  its $X^0$ norm involves  an  $X^m$ norm of 
 $V$ with $m>k.$ Moreover, for the same reason, the term $\mathcal{R}$ in \eqref{expl1}
  cannot contain only semi-linear terms.
   Nevertheless, the above considerations can be  generalized to our framework. 
             The equation for $\partial^\alpha U$
    can be written under  the following more precise  form:
\beq
\label{expl2}
\partial_{t} \partial^\alpha  U = J \Big( \Lambda[ U] \cdot \partial^\alpha U  + \mathcal{Q}[U]\cdot(\partial^\beta U)_{ |\beta | \leq |\alpha|} \Big) + \mathcal{R}(U) 
\eeq  
where $\mathcal{Q}(U)\cdot$ is a linear operator acting on the tensor $(\partial^\beta U)_{ |\beta | \leq |\alpha|}$
 which is of lower order than $\Lambda$ but of too high order to be incorporated in  the  semilinear terms.
We shall prove that  $\Lambda$ and $\mathcal{R}$ match the above properties i) and ii)  for $|\alpha |=3$.
The energy estimate for \eqref{expl2} will  then be obtained by proving that the subprincipal term  $\mathcal{Q}$, 
 and the higher order part of the commutators $\big[ \partial^\beta, \Lambda[U] \big]$
  can be incorporated as harmless lower order terms in the energy. For this argument, it is important
   to use space and time derivatives simultaneously.

\subsubsection{Analysis of  \eqref{ww}}
Let us first denote by $U=(\eta,\varphi)$ a solution of \eqref{ww}.
We first focus on the first equation of \eqref{ww} which reads
\beq\label{first}
\pa_{t} \eta = \pa_{x} \eta +  G[\eta]\varphi.
\eeq  
Let $\mathcal{I}= \{t, x, y \}$.   We first  notice that
for $k \in \mathcal{I}$,  $\partial_{k} \eta$ solves the equation
\beq
\label{0der}
\partial_{t} \partial_{k} \eta =   \pa_{x} \pa_{k} \eta+ G[\eta] \partial_{k} \varphi + DG[\eta] \varphi \cdot \partial_{k}
\eta
\eeq
and thus by using  Lemma \ref{DN'}, we find
\beq\label{1der}
\partial_{t} \partial_{k} \eta = \pa_{x} \pa_{k} \eta+G(\pa_k\varphi-Z\pa_{k}\eta)-\nabla\cdot (\pa_k\eta\, v),
\eeq
where we shall  use for short hands throughout this section the notation
$$ 
G= G[\eta] , \quad v =  v[\eta, \varphi]= \nabla \varphi - Z \nabla \eta, \quad 
Z= Z[\eta, \varphi]=  { G[\eta]\varphi + \nabla \eta \cdot \nabla \varphi \over 1 +  |\nabla \eta |^2 }
$$
(the notation for $Z$ was already introduced in Lemma~\ref{DN'}).
Next, as in \cite{Iguchi}, we can derive an equation
for $\partial_{ijk} \eta $ for $i, \, j, \, k \in \mathcal{ I}$ by applying two more derivatives
to \eqref{0der}.  We find 
\begin{equation}\label{eqeta}
\pa_{t} \pa_{ijk} \eta = \pa_{x} \pa_{ijk}  \eta  + G \pa_{ijk}\varphi 
-G(Z\partial_{ijk} \eta)-\nabla\cdot(v\partial_{ijk} \eta)+
\mathcal{Q}_{1}^{ijk}[\eta, \varphi]+\mathcal{R}_{1}^{ijk}[\eta, \varphi],
\end{equation}
where
\beq
\label{Q1eq}
\mathcal{Q}_{1}^{ijk}[\eta, \varphi]= \sum_{\mathcal{\sigma} } 
D_{\eta}G[\eta] \partial_{\mathcal{\sigma}(i) \mathcal{\sigma}(j) }\varphi \cdot \pa_{\mathcal{\sigma}(k)} \eta
\eeq
the sum being taken on the circular permutations $\sigma$ of  the set $\{i,j,k\}$
is the subprincipal part of the equation which must be handled with some care
and $\mathcal{R}_{1}^{ijk}$  is under the form
\beq
\label{R1ijk} \mathcal{R}_{1}^{ijk}[\eta, \varphi] = \sum D^n_\eta G[\eta] \partial^\gamma \varphi \cdot\big( \partial^{\beta_{1}} \eta,
\cdots, \partial^{\beta_{n}} \eta \big)\eeq
where  the sum is taken on indices $n\in \mathbb{N}^*$, $\beta_{i}\in \mathbb{N}^3$, $\gamma \in \mathbb{N}^3$  which verify
\beq
\label{indices}  1 \leq n \leq 3,  \quad | \beta_{1}|+ \cdots + |\beta_{n}| + |\gamma| =  3, \quad
|\gamma | \leq 1, \quad  1 \leq | \beta_{i}| < 3 , \, \forall\, i.
\eeq
We shall prove below  that this term behaves as a semi-linear term and thus it  is not necessary
to write down a more precise formula for it.

Let us now study  the second equation of \eqref{ww}. If we apply the operator 
$\partial_{k}$ to the second equation of \eqref{ww}, we get by using the previous notation
\beq
\label{equation2}
\pa_{t } \pa_{k}\varphi = \partial_{x} \partial_{k} \varphi - v \cdot \nabla \partial_{k} \varphi + ZG\pa_{k} \varphi+
ZD_{\eta}G\varphi\cdot\pa_k\eta+Zv\cdot\nabla\pa_k\eta
+ \beta \nabla \cdot \big( A(\nabla \eta) \nabla \pa_{k} \eta \big) -\alpha \pa_{k} \eta, 
\end{equation}
  where the matrix $A(V)$ is given by
$$ 
A(V)=  { {\rm Id} \over   (1 +  |V|^2 )^{\frac{1}{2}}} - { V \otimes V \over { (1+ |V|^2 )^{3 \over 2 } } }.
$$

We now find that $\partial_{ijk} \varphi$ solves
\begin{multline}\label{eqphi}
\pa_{t} \partial_{ijk} \varphi =  \partial_{x} \partial_{ijk} \varphi - v \cdot \nabla \pa_{ijk}\varphi
+  Z G\big( \partial_{ijk} \varphi - Z \partial_{ijk} \eta) - Z (\nabla \cdot v) \pa_{ijk} \eta 
\\
+
\beta \nabla \cdot \big( A(\nabla \eta) \nabla \pa_{ijk} \eta \big)
- \alpha \pa_{ijk} \eta 
+
\mathcal{Q}_{2}^{ijk}[\eta, \varphi]+ \mathcal{R}^{ijk}_{2}[\eta, \varphi],
\end{multline}
where 
\begin{equation}
\label{Q2eq}
\mathcal{Q}_{2}^{ijk}[\eta, \varphi] =  \sum_{\sigma}
\beta\,\nabla\cdot\Big(  DA(\nabla \eta) \cdot(\pa_{\sigma(i)}\nabla \eta, \pa_{\sigma(j) \sigma(k)}\nabla \eta)  \Big)
\end{equation}
the sum being taken on the circular permutations of $\{i,j,k\}$ and  
\begin{eqnarray}
\label{R2ijk}
\mathcal{R}_{2}^{ijk}[\eta, \varphi] 
& = & -[\partial_{ij}, v] \cdot \nabla \pa_{k} \varphi + [\partial_{ij},ZG]\partial_{k} \varphi  
+[\pa_{ij},ZD_{\eta}G\varphi]\cdot\pa_{k}\eta
\\
\nonumber
& &
+ [\partial_{ij},Z v]\cdot\nabla \partial_{k} \eta  + \beta \nabla \cdot  \big( D^2A( \nabla \eta) \cdot
( \nabla \pa_{k} \eta, \nabla \pa_{j} \eta, \nabla \pa_{i}\eta\big) \big).
 \end{eqnarray}
Again, as we shall see below,  $\mathcal{R}_{2}^{ijk}$ can be considered as a semi-linear term
while $\mathcal{Q}_{2}^{ijk}$ must be handled with care.

Now, let us  set  $U_{ijk}= (\partial_{ijk}\eta, \partial_{ijk}\varphi)^t$,
then \eqref{eqeta}, \eqref{eqphi} can be written under the abstract  form
\beq\label{Uquasi}
\pa_{t} U_{ijk} =  J  
\big( 
\Lambda[\eta, \varphi] U_{ijk}+\mathcal{Q}^{ijk}[\eta, \varphi]\big)+ \mathcal{R}^{ijk}[\eta, \varphi]
\eeq
where $\mathcal{R}^{ijk}[\eta, \varphi]=(\mathcal{R}^{ijk}_1[\eta,
\varphi],\mathcal{R}^{ijk}_2[\eta, \varphi])^t$, $\mathcal{Q}^{ijk}[\eta, \varphi]=(-\mathcal{Q}^{ijk}_2[\eta,
\varphi],\mathcal{Q}^{ijk}_1[\eta, \varphi])^t$
and $\Lambda[\eta,\varphi]$ is the linearized about $(\eta,\varphi)$ operator. Namely
\begin{equation*}
\Lambda[\eta, \varphi]  =   
\left( 
\begin{array}{cc} 
- \beta\nabla \cdot( A(\nabla \eta)\nabla\,\cdot\,) + \alpha  + Z G \big(Z\cdot\big) + Z  \nabla \cdot v & 
v \cdot  \nabla - \partial_{x} - Z G \\
  - \nabla \cdot( v \,\cdot\,) - G\big( Z\cdot\big)+\partial_x  & G 
  \end{array} \right),
\end{equation*}
where $Z= Z[\eta, \varphi]$, $G=G[\eta]$, $v= v[\eta, \varphi]$.

By using Proposition~\ref{pak} (with $\eta^{a}=\varphi^{a}= 0$ for the
moment), 
a lengthy but straightforward computation
shows that for $k$ sufficiently large, we have the estimate
$$ \|\langle\pa\rangle^{k} \mathcal{R}_{1}^{ijk} \|_{H^1}
+ \| \langle\pa\rangle^{k} \mathcal{R}_{2}^{ijk} \|_{H^{1 \over 2 } } \leq 
\omega\big( \|\langle\pa\rangle^{k+ 3 } \eta \|_{H^1} +
\| \langle\pa\rangle^{k+ 3 } \varphi \|_{H^{1 \over 2 } } \big)  \big(      
 \|\langle\pa\rangle^{k+ 3 } \eta \|_{H^1} +
\|\langle\pa\rangle^{k+ 3 } \varphi \|_{H^{1 \over 2 } } \big)$$
which  indicates that we can consider $\mathcal{R}^{ijk}$ 
as a semi-linear term. 
We shall not prove this estimate now  since  we need to work in a more general framework
where $U$ is  taken as a perturbation of
$V^{a}$ which  is not in $H^s$.  In the next section we
describe the structure of this problem. The more general estimate that we shall prove below implies
the above claimed  estimate just  by taking $V^{a}=0$ in the following estimates.
\subsubsection{Analysis of  \eqref{Ueqinstab}}
\label{analysisof}
Recall that
$V^{a}=(\eta^{a},\varphi^{a})=Q+\delta U^{a}$, where $U^{a}$ is the approximate solution given by
 Proposition \ref{Uap}.  Coming back to the analysis of \eqref{Ueqinstab},
we have that thanks to  Proposition~\ref{Uap}, with the notations introduced
in the previous section, 
$U_{ijk}= (\partial_{ijk}\eta, \partial_{ijk}\varphi)^t$
solves
\begin{equation}\label{Ueq}
\pa_{t} U_{ijk} =  J  
\Big( \Lambda^\delta  U_{ijk} 
+(\mathcal{Q}^{ijk})^\delta -(\mathcal{Q}^{ijk})^{a}\Big)
+\mathcal{G}^{ijk}[\eta, \varphi] - \pa_{ijk} R^{ap},\quad U(0)=0,
\end{equation}
where the claimed  semi-linear term is now given by 
\begin{equation}
\label{defG}
\mathcal{G}^{ijk}[\eta, \varphi]=\mathcal{G}^{ijk}_1[\eta, \varphi]+\mathcal{G}^{ijk}_2[\eta, \varphi],
\end{equation}
where
$$
\mathcal{G}^{ijk}_1[\eta, \varphi]=
\mathcal{R}^{ijk}[\eta+ \eta^{a}, \varphi+ \varphi^{a}] - \mathcal{R}^{ijk}[\eta^{a}, \varphi^{a}]
= (\mathcal{R}^{ijk})^\delta - (\mathcal{R}^{ijk})^{a}
$$
and
$$
\mathcal{G}^{ijk}_2[\eta, \varphi]=
J\big( \Lambda [\eta+ \eta^{a}, \varphi+ \varphi^{a}]  -  \Lambda [ \eta^{a},  \varphi^{a}]
 \big) \partial_{ijk}V^{a} = J\big( \Lambda^\delta   -  \Lambda^{a}
 \big) \partial_{ijk}V^{a} \,.
$$
The terms $\pa_{ijk} R^{ap}$ enjoys the bounds provided by Proposition~\ref{Uap}.
It will also  be useful to  use the shorter notation
\begin{equation}\label{lionia}
\pa_{t} U_{ijk} =  J  ( \Lambda^\delta  U_{ijk}+(\mathcal{Q}^{ijk})^\delta -(\mathcal{Q}^{ijk})^{a}
)+F_{ijk}\,,
\end{equation}
for \eqref{Ueq}
where
\beq
\label{lioniaF}
F_{ijk}=\mathcal{G}^{ijk}[\eta, \varphi] - \pa_{ijk} R^{ap}\,.
\eeq

In order to perform our energy estimates, we shall use the canonical form of \eqref{lionia}
identified in \cite{L1} (in the absence of surface tension).
We set $W_{ijk}=P U_{ijk}$, where
$$
P\equiv\left( \begin{array}{cc} 1 & 0\\-Z[U+V^{a}] & 1\end{array} \right) = 
\left( \begin{array}{cc} 1 & 0\\-Z^\delta & 1\end{array} \right)
\,.
$$
We find  that $W_{ijk}$ solves the problem
\begin{equation}\label{W}
\pa_{t} W_{ijk} =  J  \big( L^\delta  W_{ijk}-JPJ\big((\mathcal{Q}^{ijk})^\delta -(\mathcal{Q}^{ijk})^{a}\big)\big)
+PF_{ijk},
\end{equation}
where
\begin{eqnarray*}
L[\eta, \varphi]  =
\left( \begin{array}{cc} - {\mathcal P}(\nabla\eta)+(v\cdot\nabla Z)+(\partial_{t}-\partial_x)Z &
  v \cdot  \nabla - \partial_{x} \\
  - \nabla \cdot( v \cdot) +\partial_x  & G 
  \end{array} \right),
\end{eqnarray*}
with $Z=Z[\eta,\varphi]$, $v=v[\eta, \varphi]$, $G=G[\eta]$ and
${\mathcal P}(\nabla\eta)$  defined by
$
{\mathcal P}(\nabla\eta)\equiv \beta\nabla\cdot(A(\nabla\eta)\nabla\cdot)-\alpha.
$
\subsection{Estimates on the semi-linear terms}\label{semin}
In this section we estimate the semilinear term $\mathcal{G}^{ijk}[\eta,
\varphi]$ arising in (\ref{Ueq}). In the estimate below $V^{a}$ will be
evaluated in ${\mathcal W}^{m+S}$  with $S$ sufficiently large  since we will make use of
Propositions~\ref{pak} and \ref{pak_bis}. Here is the main result of this section.
\begin{prop}\label{main_sect}
For $m\geq 2$, and $S\geq 5$, we have the estimate
$$
\|\mathcal{G}^{ijk}[\eta, \varphi]\|_{X^m}
\leq \omega\big(\|V^{a}\|_{{\mathcal W}^{m+S}}+\|U\|_{X^{m+3}}\big)\|U\|_{X^{m+3}}.
$$
\end{prop}
Since there are many terms that we need to estimate, we shall split the proof
of Proposition \ref{main_sect} in many Propositions and Lemmas.
To save place, we shall  use the notation
$$ \overline{\omega}_{m,S}= \omega\big(\|V^{a}\|_{{\mathcal W}^{m+S}}+\|U\|_{X^{m+3}}\big).$$
Our first result  towards the proof of Proposition \ref{main_sect} is:
\begin{prop}\label{R_1}
For $m \geq 2$ and $S \geq 5$, we have
\beq
\label{estR1}
\big\|\langle\partial\rangle^m
\big(({\mathcal R}^{ijk}_1)^\delta-({\mathcal R}^{ijk}_1)^{a}\big)
\big\|_{H^1}\leq \overline{\omega}_{m,S}
\|U\|_{X^{m+3}}.
\eeq
\end{prop}
\begin{proof}
From the definition of $\mathcal{R}_{1}^{ijk},$  we have to estimate the $H^1$ norm
of  terms like
$$ \langle \partial \rangle^m \Big( D^n_\eta G^\delta\partial^{\gamma}(\varphi^{a}+
\varphi) \cdot \big( \partial^{\beta_{1}}(\eta^{a}+ \eta), \cdots, \partial^{\beta_n}(\eta^{a}+ \eta)
\big) -  D^n_\eta  G^{a}\partial^{\gamma}\varphi^{a} \cdot \big( \partial^{\beta_{1}}\eta^{a}, \cdots,\partial^{\beta_n}\eta^{a}
\big) \Big)$$
where $n$, $\gamma$ and $\beta_{i}$ satisfies the constraints \eqref{indices}.
By multilinearity and symmetry, we need to estimate three types of terms:
$$ I_{1}= 
D^n_\eta G^\delta\partial^{\gamma}\varphi^{a} \cdot \big( \partial^{\beta_{1}}\eta^{a}, \cdots, \partial^{\beta_n}\eta^{a}
\big) -  D^n_\eta G^{a}\partial^{\gamma}\varphi^{a} \cdot \big( \partial^{\beta_{1}}\eta^{a}, \cdots,\partial^{\beta_n}\eta^{a}
\big), $$
$$ I_{2}= D^n_\eta G^\delta\partial^{\gamma}\varphi\cdot \big( \partial^{\overline{\beta}_{1}}\eta^{a}, \cdots,
\partial^{\overline{\beta}_{l} }\eta^{a}, \partial^{\overline{\beta}_{l+1} } \eta, \cdots,  \partial^{\overline{\beta}_n}\eta \big)$$
and 
$$ I_{3} =  D^n_\eta G^\delta\partial^{\gamma}\varphi^{a}\cdot \big( \partial^{\overline{\beta}_{1}}\eta^{a}, \cdots,
\partial^{\overline{\beta}_{l}} \eta^{a}, \partial^{\overline{\beta}_{l+1} } \eta, \cdots,  \partial^{\overline{\beta}_n}\eta \big)$$
with $0 \leq l  \leq  n$ for $I_{2}$, $ 0 \leq  l \leq n-1$ for $I_{3}$ and $\overline{\beta}_{i}= \beta_{\sigma(i)}$ for some permutation $\sigma$
of $\{1\cdots, n\}.$
In particular, from \eqref{indices}, we have  that $|\overline{\beta}_{i}|\leq 2$ and $|\gamma | \leq 1$.
From  the second estimate of Proposition \ref{pak}, we immediately get that 
$\|\langle \partial \rangle^m I_{2} \|_{H^1}$ satisfies the claimed estimate \eqref{estR1} that
 is
$$  \|\langle \partial \rangle^m I_{2} \|_{H^1} \leq 
\overline{\omega}_{m,S}
\|U\|_{X^{m+3}}.$$
In a similar way, the estimate
$$   \|\langle \partial \rangle^m I_{3} \|_{H^1} \leq 
\overline{\omega}_{m,S}
\|U\|_{X^{m+3}}$$
follows from  Proposition \ref{pak_bis} estimate \eqref{maximum_tris}.

For the estimate of $I_{1}$, we first write
\begin{eqnarray}
\label{Itaylor}
I_1 &=& \int_0^1\frac{d}{ds}
\Big(
D^n_{\eta}G[\eta^{a}+s\eta]\partial^\gamma \varphi^{a}\cdot\big(\pa^{\beta_{1}}\eta^{a}, \cdots
, \pa^{\beta_{n}} \eta^{a}\big)
\Big)ds \\
\nonumber&=&
\int_0^1
D^{n+1}_{\eta}G[\eta^{a}+s\eta]\partial^\gamma \varphi^{a}\cdot(\pa^{\beta_{1}}\eta^{a},\cdots, \pa^{\beta_{n}}\eta^{a}, \eta\big) ds\,.
\end{eqnarray}
From Proposition~\ref{pak_bis} estimate \eqref{maximum_tris} 
(with $\eta$ being the only function in $H^s$) and \eqref{indices}, we thus also obtain that 
$$
\|\langle\partial\rangle^m I_1\|_{H^1}\leq \overline{\omega}_{m,S} \|\langle\partial\rangle^{m+1}\eta\|_{H^{1}}.
$$
This ends  the proof of Proposition~\ref{R_1}.
\end{proof}
\begin{rem}
Note that the important point in the proof of Proposition~\ref{R_1} is that
$\eta$ does not appear with more than two derivatives in the directions of the
Frechet derivatives of $G$ and that $\varphi$ does not appear with more than
two derivatives, i.e. we are in the scope of applicability of
Propositions~\ref{pak} and Proposition~\ref{pak_bis}.
\end{rem}
We shall now turn to the estimate of  the terms involving $\mathcal{R}_{2}^{ijk}$. 
Towards this, we shall use  very often the following product estimates:
\begin{prop}\label{anr}
For $m\geq 2$, $\sigma = 1/2$ or $\sigma=1$, we have
\begin{eqnarray}
\label{anr1}
\|\langle \partial \rangle^m(uv)\|_{H^{\sigma}(\R^2)}
& \leq & C \|\langle \partial\rangle^m u\|_{H^{\sigma}(\R^2)}\|\langle \partial\rangle^m v\|_{H^{\sigma}(\R^2)}\,,
\\
\label{anr2}
\|\langle \partial \rangle ^m(uv)\|_{H^{\sigma}(\R^2)}
& \leq & 
C \|\langle \partial\rangle^{m} u\|_{W^{1,\infty}(\R^2)}\|\langle \partial \rangle^m v\|_{H^{\sigma}(\R^2)}\,.
\end{eqnarray}
\end{prop}
Note that in the second estimate $u$ and $v$ do not play symmetric parts.
As in Lemma \ref{LP}, a sharper but not  needed estimate holds by replacing
$  \|\langle \partial\rangle^{m} u\|_{W^{1, \infty}(\R^2)}$
with 
$ \|\langle \partial\rangle^{m} u\|_{\mathcal{C}^\alpha} $   for $\alpha >1/2$ when $\sigma= 1 /2 $.

The proof of Proposition \ref{anr} is very similar to the ones of Lemma \ref{prodS} and Lemma \ref{LP}
 and hence will be omitted.
In the next proposition, we evaluate the contribution of ${\mathcal R}^{ijk}_2$.
\begin{prop}\label{R_2}
For $m\geq 2$, $S \geq 5 $, we have the estimate
$$
\big\|\langle\partial\rangle^{m}\big(({\mathcal R}^{ijk}_2)^\delta-({\mathcal R}^{ijk}_2)^{a}\big)
\big\|_{H^{\frac{1}{2}}}\leq\overline{\omega}_{m,S}\|U\|_{X^{m+3}}
$$
\end{prop}
\subsubsection*{Proof of Proposition \ref{R_2}}
We shall  first establish  the following lemma.
\begin{lem}\label{annr}
For $m\geq 2$, $S\geq 5$,  we have
\begin{equation}
\label{annr1}
\|\langle\partial\rangle^{m+2}(v^\delta-v^{a})\|_{H^{\frac{1}{2}}}+
\|\langle\partial\rangle^{m+2}(Z^\delta-Z^{a})\|_{H^{\frac{1}{2}}}
\leq
\overline{\omega}_{m,S} \|U\|_{X^{m+3}}
\end{equation}
and for $|\gamma | \leq 2$, we also have
\begin{equation}
\label{annr2}
\big\|\langle\partial\rangle^{m}\big(  \partial^\gamma v^\delta\,  \psi\big)\big\|_{H^{\frac{1}{2}}}+
\big\|\langle\partial\rangle^{m}\big( \partial^\gamma Z^\delta \,\psi \big) \big\|_{H^{\frac{1}{2}}}
\leq\overline{\omega}_{m,S}\|\langle\partial\rangle^{m}\psi\|_{H^{\frac{1}{2}}}.
\end{equation}
\end{lem}
\begin{proof}
We start with the estimates involving $Z$.
Write $Z=Z_1+Z_2$, where $Z_1=(1+|\nabla\eta|^2)^{-1}G[\eta]\varphi$. 
A straightforward application of Proposition~\ref{anr} implies that
\begin{equation}\label{resol}
\|\langle\partial\rangle^{m+2}(Z_2^\delta-Z_2^{a})\|_{H^{\frac{1}{2}}}
\leq\overline{\omega}_{m,S} \|U\|_{X^{m+3}}\,.
\end{equation}
We shall not detail the proof of the inequality (\ref{resol}). Instead we shall
estimate in detail the contribution of $Z_1$ whose  proof is very similar and contains
 additional difficulties  coming from the presence of the Dirichlet-Neumann operator. 
Set $w[\eta]\equiv (1+|\nabla\eta|^2)^{-1}$. Then we can write
\begin{eqnarray}
\label{Z1dec}
Z_1^\delta -Z_1^{a}& = & 
(w^\delta-w^{a})G^{a}\varphi^{a}
+w^\delta(G^\delta-G^{a})\varphi^{a}
+w^\delta G^\delta \varphi\,.
\end{eqnarray}
To estimate the first term in the decomposition \eqref{Z1dec}, we use  Proposition~\ref{anr} to obtain
\beq
\|\langle\partial\rangle^{m+2}\big((w^\delta-w^{a})G^{a}\varphi^{a}\big)\|_{H^{\frac{1}{2}}}
\leq
C\|\langle\partial\rangle^{m+2}(w^\delta-w^{a})\|_{H^{\frac{1}{2}}}
\|\langle\partial\rangle^{m+2}(G^{a}\varphi^{a})\|_{W^{1, \infty}}\,.
\eeq
Next, since
$$
w^\delta -w^{a}=-\frac{|\nabla\eta+\nabla\eta^{a}|^2-|\nabla \eta^{a}|^2}
{(1+|\nabla \eta^{a}|^2)(1+|\nabla\eta+\nabla\eta^{a}|^2)}
$$
we obtain  after several  applications of Proposition~\ref{anr} that:
\begin{equation}\label{kore_0}
\|\langle\partial\rangle^{m+2}(w^\delta -w^{a})\|_{H^{\frac{1}{2}}}
\leq \overline{\omega}_{m,S} \|U\|_{X^{m+3}}\,.
\end{equation}
Since, by  using the estimate \eqref{maximum} of
Proposition~\ref{pak_bis}, we also have that for $S\geq 5$
$$
\|\langle\partial\rangle^{m+2}(G[\eta^{a}]\varphi^{a})\|_{W^{1, \infty}}
\leq 
\omega\big(\|V^{a}\|_{{\mathcal
W}^{m+S}}\big)
$$
we get  that 
\beq
\label{Z11}
\|\langle\partial\rangle^{m+2}\big((w^\delta -w^{a})G^{a}\varphi^{a}\big)\|_{H^{\frac{1}{2}}}
\leq\overline{\omega}_{m,S}\|U\|_{X^{m+3}}\,.
\eeq
Next, to estimate the second  term in \eqref{Z1dec}, we first  observe that for every $\psi$ smooth enough,   the bound
\begin{equation}\label{kore}
\|\langle\partial\rangle^{m+2}\big(w^\delta \psi\big)\|_{H^{\frac{1}{2}}}
\leq
\overline{\omega}_{m,S}\|\langle\partial\rangle^{m+2}\psi\|_{H^{\frac{1}{2}}}\,
\end{equation}
 holds. 
Indeed, it suffices to write
$w^\delta=(w^\delta -w^{a})+w^{a}$ and to use
Proposition~\ref{anr} and (\ref{kore_0}).
By  using (\ref{kore}), we thus obtain 
\begin{equation}\label{kore_1}
\|
\langle\partial\rangle^{m+2}(
w^\delta (G^\delta -G^{a})\varphi^{a}
)\|_{H^{\frac{1}{2}}}
\leq \overline{\omega}_{m,S}\|\langle\partial\rangle^{m+2}
((G^\delta -G^{a})\varphi^{a})\|_{H^{\frac{1}{2}}}\,.
\end{equation}
To estimate the last above term, we use that for $|\alpha |=2$
$$
\pa^\alpha \big( (G^\delta -G^{a})\varphi^{a} \big)=
\pa^\alpha\Big(\int_0^1 D_{\eta}G[s\eta+\eta^{a}]\varphi^{a}\cdot\eta ds\Big)\,
$$
 and we notice that 
 $ \partial^\alpha\big( D_{\eta}G[s\eta+\eta^{a}]\varphi^{a}\cdot\eta \big)$
  can be expanded as a sum of terms under the form
  $$ \int_{0}^1  D_{\eta}^n G[\eta^a + s \eta] \partial^\beta \varphi^a\cdot \big(\partial^{\gamma_{1}} \eta,  \partial^{\gamma_{2}}h_{1},
   \cdots, \partial^{\gamma_{n}} h_{n-1}\big)\, ds$$
    where $n \geq 1$,   the   $h_{i}$ may be $\eta$ or $\eta^a$ and $| \gamma_{i}| \leq 2.$
Since $\partial^{\gamma_{1}} \eta$ belongs to the Sobolev scale and 
 there is never  more than two
derivatives of  $\eta$ involved  in the above expression, we can  again use the estimate \eqref{maximum_tris} of  Proposition~\ref{pak_bis} to
infer that
$$
\|\langle\partial\rangle^{m+2}
((G^\delta -G^{a})\varphi^{a})\|_{H^{\frac{1}{2}}}
\leq
\overline{\omega}_{m,S} \|U\|_{X^{m+3}}\,.
$$
Coming back to (\ref{kore_1}), we have   thus proven that 
\beq
\label{Z12}
\|\langle\partial\rangle^{m+2}(w^\delta (G^\delta -G^{a})\varphi^{a}
)\|_{H^{\frac{1}{2}}}\leq\overline{\omega}_{m,S} \|U\|_{X^{m+3}}\,.
\eeq
Finally, to estimate the last  term in \eqref{Z1dec}, we first  write from 
another use of (\ref{kore}) that  
\begin{equation*}
\|\langle\partial\rangle^{m+2}(w^\delta G^\delta \varphi)\|_{H^{\frac{1}{2}}}
\\
\leq \overline{\omega}_{m,S} \|\langle\partial\rangle^{m+2}
(G^\delta\varphi)\|_{H^{\frac{1}{2}}}\,.
\end{equation*}
Since for $|\alpha |= 2$, we can decompose
$ \partial^\alpha \big( G[\eta + \eta^{a}]\varphi\big)$ as a sum of terms of the form
$$ D^n_{\eta}G[\eta+ \eta^{a}] \partial^\gamma \varphi \cdot(h_{1}, \cdots h_{n})$$
with $h_{i}= \partial^{\beta_i} \eta$ or $ \partial^{\beta_i} \eta^{a}$ and
$|\gamma| \leq 2$, $|\beta_{i}| \leq 2$, we can use Proposition~\ref{pak} to obtain 
$$
\|\langle\partial\rangle^{m+2}(G^\delta \varphi)\|_{H^{\frac{1}{2}}}
\leq
\overline{\omega}_{m,S}\|\langle \partial \rangle^{m+ 3} \varphi\|_{H^{1 \over 2}}
$$
and therefore, we obtain  that 
\beq
\label{Z13}
\|\langle\partial\rangle^{m+2}(w^\delta G^\delta \varphi)\|_{H^{\frac{1}{2}}} \leq 
\overline{\omega}_{m,S} \|U\|_{X^{m+3}}.
\eeq
By using   the decomposition  \eqref{Z1dec}  and the estimates \eqref{Z11}, \eqref{Z12}, \eqref{Z13},
we find  that
$$
\|\langle\partial\rangle^{m+2}(Z_1^\delta-Z_1^{a})\|_{H^{\frac{1}{2}}}
\leq
\overline{\omega}_{m,S} \,\|U\|_{X^{m+3}}
$$
which in turn taking into account (\ref{resol}) gives
\begin{equation}\label{resol_0}
\|\langle\partial\rangle^{m+2}(Z^\delta - Z^{a})\|_{H^{\frac{1}{2}}}
\leq
\overline{\omega}_{m,S}\,\|U\|_{X^{m+3}}\,.
\end{equation}
This proves the first estimate for $Z$ in Lemma \ref{annr}. 
Next, since 
can write for $| \gamma | \leq 2$ that 
$$
\|\langle\partial\rangle^{m}( \partial^\gamma Z^\delta \, \psi )\|_{H^{\frac{1}{2}}}
\leq
\|\langle\partial\rangle^{m}\big( \partial^\gamma (Z^\delta -Z^{a})\psi \big)\|_{H^{\frac{1}{2}}}
+
\|\langle\partial\rangle^{m}(\partial^\gamma Z^{a}\,\psi)\|_{H^{\frac{1}{2}}},
$$
by  using (\ref{resol_0}) together with Proposition~\ref{anr} and
Proposition~\ref{pak_bis} (to bound $Z^a$) we also find
\begin{equation}\label{resol_1}
\|\langle\partial\rangle^{m}(\partial^\gamma Z^\delta \,  \psi)\|_{H^{\frac{1}{2}}}
\leq
\overline{\omega}_{m,S}\,
\|\langle\partial\rangle^{m}\psi\|_{H^{\frac{1}{2}}}\,
\end{equation}
and hence the second estimate for $Z$ is proven.

It remains to prove the claimed estimates for $v$. Let us 
recall that $v[\eta,\varphi]=\nabla\varphi-Z[\eta,\varphi]\nabla\eta$.
Therefore we can write
$$
v^\delta-v^{a}=\nabla\varphi-
\Big((Z^\delta-Z^{a})\nabla\eta^{a}
+Z^\delta \nabla\eta
\Big).
$$
Consequently, by  using (\ref{resol_0}), (\ref{resol_1}) and Lemma~\ref{anr}, we
infer 
\begin{equation}\label{resol_2}
\|\langle\partial\rangle^{m+2}(v^\delta-v^{a})\|_{H^{\frac{1}{2}}}
\leq\overline{\omega}_{m,S}\,\|U\|_{X^{m+3}}.
\end{equation}
Finally, we can write  that 
$$
\|\langle\partial\rangle^{m}(\partial^\gamma \, v^\delta \psi)\|_{H^{\frac{1}{2}}}
\leq
\|\langle\partial\rangle^{m}( \partial^\gamma (v^\delta-v^{a})\psi)\|_{H^{\frac{1}{2}}}
+
\|\langle\partial\rangle^{m}(\partial^\gamma v^{a}\psi)\|_{H^{\frac{1}{2}}}
$$
and hence, by  using (\ref{resol_2}) together with Proposition~\ref{anr} and
Proposition~\ref{pak_bis} (to bound $v^a$)  we get  the bound
\begin{equation}\label{resol_3}
\|\langle\partial\rangle^{m}(\partial^\gamma v^\delta \,  \psi)\|_{H^{\frac{1}{2}}}
\leq
\overline{\omega}_{m,S}\,
\|\langle\partial\rangle^{m}\psi\|_{H^{\frac{1}{2}}}\,.
\end{equation}
This ends  the proof of Lemma~\ref{annr}.
\end{proof}
The result of Lemma~\ref{annr} will be one of the main tool when estimating 
${\mathcal R}^{ijk}_2[V^{a}+U]-{\mathcal R}^{ijk}_2[V^{a}]$. We start by the
estimate of the contribution of the following   terms  in the definition \eqref{R2ijk} of  ${\mathcal R}^{ijk}_2$.
\begin{lem}
\label{comh}
For $S \geq 5, $ $m \geq 2 $, we have the following estimates:
\begin{equation}\label{sh}
\big\|\langle\partial\rangle^m\big(
[\partial_{ij},v^\delta ]\cdot \partial_k\nabla\varphi^\delta
-
[\partial_{ij},v^{a}]\cdot \partial_k\nabla\varphi^{a}\big)
\big\|_{H^{\frac{1}{2}}}
\leq
\overline{\omega}_{m, S}\, \|U\|_{X^{m+3}},
\end{equation}
\begin{equation}
\label{ch}
\big\|\langle\partial\rangle^m\big([\partial_{ij},Z^\delta G^\delta](\partial_{k}\varphi^\delta)-
[\partial_{ij},Z^{a}G^{a}](\partial_{k}\varphi^{a})\big)\big\|_{H^{\frac{1}{2}}}
\leq \overline{\omega}_{m, S}\,\|U\|_{X^{m+3}},
\end{equation}
\begin{equation}
\label{lh}
\big\|\langle\partial\rangle^m\big(
[\partial_{ij},Z^\delta v^\delta]\cdot\nabla\partial_{k} \eta^\delta
-[\partial_{ij},Z^{a} v^{a}]\cdot\nabla\partial_{k}\eta^{a}
\big)\big\|_{H^{\frac{1}{2}}}
\leq
\overline{\omega}_{m,S} \|U\|_{X^{m+3}},
\end{equation}
\begin{equation}
\label{dh}
\big\|\langle\partial\rangle^m\big(
\big[\partial_{ij},Z^\delta D_{\eta}G^\delta \varphi^\delta \big]\cdot\partial_k \eta^\delta
-
\big[\partial_{ij},Z^{a}D_{\eta}G^{a}\varphi^{a}\big]\cdot\partial_k\eta^{a}
\big)\Big\|_{H^{\frac{1}{2}}}
\leq
\overline{\omega}_{m,S}\|U\|_{X^{m+3}},
\end{equation}
where we again use the notation
$ \overline{\omega}_{m,S}= \omega\big(\|V^{a}\|_{{\mathcal W}^{m+S}}+\|U\|_{X^{m+3}}\big).$
\end{lem}
Note that $v$ and $Z$ both satisfy the estimates of Lemma \ref{annr} and
that the Dirichlet-Neumann operator  $G$ acts as a first order operator in space.
Consequently, the statements  \eqref{ch} and \eqref{sh} are formally  very close.
Since the Dirichlet-Neumann operator which has a nonlinear dependence
in the surface  is much more difficult to deal with 
than the  simple operator $\nabla$, we shall only focus on the proof of  (\ref{ch}).
The proof of   \eqref{sh} which is simpler  and follows the same lines is left to the reader. 
In the same way, since because of Lemma \ref{annr},   $Zv$ behaves as $v$ and $Z$,    the proof of \eqref{lh} which  is very similar and simpler than the one of \eqref{ch}  is left to the reader.
Note that  the proof of \eqref{lh} is even simpler since $\eta$ is smoother than $\varphi$:
we are allowed  to put it in $H^1$. 
\begin{proof}[Proof of Lemma \ref{comh}]
As already explained, we focus on the proof of \eqref{ch} and \eqref{dh}.
Let us start with the proof of \eqref{ch}.
The commutator $[\partial_{ij}, ZG]\partial_{k}\varphi$ can be expanded as a sum of terms under the form
\beq
\label{comh1} I[U]= \partial^{\gamma_{0}} Z \, D^n_{\eta}G[\eta] \partial_{k} \partial^\beta \varphi \cdot
\big(\partial^{\gamma_{1}} \eta, \cdots, \partial^{\gamma_{n}} \eta \big) \eeq
where the indices satisfy the constraints
\beq
\label{indices2}
\forall i, \, 0 \leq i \leq n, \,    |\gamma_{i}|
\leq  2, \quad | \beta | \leq 1.
\eeq  To get  (\ref{ch}), we thus  need to estimate
$\| \langle \partial \rangle^m \big( I^\delta - I^{a} \big) \|_{H^{1 \over 2 }}.$
Towards this, we expand
\begin{eqnarray}
\nonumber
I^\delta - I^{a} & = &   \big(\partial^{\gamma_{0}} Z^\delta - \partial^{\gamma_{0}} Z^{a}\big) 
D_{\eta}^n G^{a} \partial_{k} \partial^\beta \varphi^{a}\cdot \big( \partial^{\gamma_{1}} \eta^{a}, \cdots,
\partial^{\gamma_{n}} \eta^{a} \big) \\
\nonumber  & & +  \partial^{\gamma_{0}} Z^\delta \big( D_{\eta}^n G^\delta - D_{\eta}^n G^{a}\big)   \partial_{k} \partial^\beta \varphi^{a}\cdot \big( \partial^{\gamma_{1}} \eta^{a}, \cdots,
\partial^{\gamma_{n}} \eta^{a} \big) \\
\nonumber & & + \partial^{\gamma_{0}} Z^\delta D_{\eta}^n G^\delta \partial_{k} \partial^\beta \varphi \cdot
\big( \partial^{\gamma_{1}} \eta^{a}, \cdots,
\partial^{\gamma_{n}} \eta^{a} \big) 
\\
& & \nonumber 
+ \mathcal{J} \\
\label{Id-Ia}   & \equiv & I_{1}+ I_{2} + I_{3} + \mathcal{J}
    \end{eqnarray}
where $\mathcal{J}$ is  a sum of terms under the form
$$ \mathcal{J}_{l}=  \partial^{\gamma_{0}} Z^\delta D_{\eta}^n G^\delta \partial_{k}\varphi^{\delta}
\cdot \big( \partial^{\overline{\gamma}_{1}} \eta^{a}, \cdots, \partial^{\overline{\gamma}_{l} } \eta^{a},
 \partial^{\overline{\gamma}_{l+1}} \eta, \cdots 
\partial^{ \overline{\gamma }_{n} } \eta \big),$$
where the $\overline{\gamma}_{i}$ are obtained from the $\gamma_{i}$ by a permutation
 and $l$ is such that $l \leq n-1$.

By using    \eqref{anr2},   \eqref{maximum_pak}, \eqref{annr1} and the fact that $ |\gamma_{0}|\leq 2$, 
we obtain
\begin{eqnarray*}
\|\langle \partial \rangle^m I_{1} \|_{H^{1 \over 2 } }
&   \leq &   C \|\langle \partial \rangle^{m}   D_{\eta}^n G^{a} \partial_{k} \partial^\beta \varphi^{a}\cdot \big( \partial^{\gamma_{1}} \eta^{a}, \cdots,
\partial^{\gamma_{n}} \eta^{a} \big) \|_{W^{1,\infty}} \|\langle \partial \rangle^{m+ 2 } \big(
Z^\delta - Z^{a} \big) \|_{H^{1 \over 2 } } \\
& & 
 \leq  \overline{\omega}_{m, S} \|U\|_{X^{m+3}}.
\end{eqnarray*}
Next,   since  we have from  \eqref{annr2}  that
$$ 
\| \langle \partial \rangle^m \big(   \partial^{\gamma_{0}} Z^\delta
D_{\eta}^n G^\delta \partial_{k} \partial^\beta \varphi \cdot
\big( \partial^{\gamma_{1}} \eta^{a}, \cdots,
\partial^{\gamma_{n}} \eta^{a} \big) \big) \|_{H^{1 \over 2 } }
 \leq \overline{\omega}_{m, S} \|  \langle \partial \rangle^{m } \big(
  D_{\eta}^n G^\delta \partial_{k} \partial^\beta \varphi \cdot
\big( \partial^{\gamma_{1}} \eta^{a}, \cdots,
\partial^{\gamma_{n}} \eta^{a} \big) \big) \|_{H^{1 \over 2} },$$
 we get  by using \eqref{rennes1,5} with $l=n$ that
$$   \|  \langle \partial \rangle^{m } \big(
  D_{\eta}^n G^\delta \partial_{k} \partial^\beta \varphi \cdot
\big( \partial^{\gamma_{1}} \eta^{a}, \cdots,
\partial^{\gamma_{n}} \eta^{a} \big) \big) \|_{H^{1 \over 2} } \leq 
 \overline{\omega}_{m,S}  \| \langle \partial \rangle^{m+ 1} 
  \partial_{k}\partial^\beta  \varphi \|_{H^{1 \over 2 }}
   \leq   \overline{\omega}_{m,S}  \| \langle \partial \rangle^{m+ 3}
   \varphi \|_{H^{1 \over 2 }}.$$
This yields  by using \eqref{indices2} that 
$$ \| \langle \partial \rangle^m I_{3} \|_{H^{1 \over 2}}
 \leq \overline{\omega}_{m,S}  \| \langle \partial \rangle^{m+ 3} \varphi \|_{H^{1 \over 2 }}
  \leq   \overline{\omega}_{m,S} \|U \|_{X^{m+ 3}}.$$
In a similar way, we can use \eqref{annr2} to get
$$ \| \langle \partial \rangle^m  I_{2} \|_{H^{1 \over 2} }
\leq  \overline{\omega}_{m, S} \| \langle \partial \rangle^m \big(
\big( D_{\eta}^n G^\delta  - D_{\eta}^n G^{a} \big)\partial_{k} \partial^\beta \varphi^{a} \cdot
\big( \partial^{\gamma_{1}} \eta^{a}, \cdots,
\partial^{\gamma_{n}} \eta^{a} \big)  \big) \|_{H^1 }.$$
Therefore,  by a new use of the Taylor formula as in \eqref{Itaylor} and  \eqref{maximum_tris}
 we  infer
$$ 
\| \langle \partial \rangle^m  I_{2} \|_{H^{1 \over 2} }
\leq  \overline{\omega}_{m, S} \|\langle \partial \rangle^{m+1}\eta \|_{H^1}\leq   \overline{\omega}_{m, S}  \|U\|_{X^{m+3}}.
$$
Finally, to estimate $\mathcal{J}_{l}$,  since $\varphi^\delta = \varphi + \varphi^{a}$,  we can  also use successively \eqref{annr2} 
 and  \eqref{rennes1,5} for the part involving $\varphi$  or 
    \eqref{maximum_tris}
 since $l \leq n-1$ for the part involving $\varphi^{a}$ to get
$$ \|\langle \partial \rangle^m \mathcal{J}_{l} \|_{H^{1 \over 2}} \leq   \overline{\omega}_{m, S}  \|U\|_{X^{m+3}}.$$

 Consequently, summarizing the previous estimates and \eqref{Id-Ia}, we have proven that
\beq
\label{eId-Ia}
\| \langle \partial \rangle^m \big( I^\delta - I^{a}\big)  \|_{H^{1 \over 2 } }\leq \overline{\omega}_{m,S}
\| U \|_{X^{m+ 3 }}.\eeq
This ends the proof of  \eqref{ch}.

Let us turn to the proof of \eqref{dh}.  The commutator $ \big[ \partial_{ij}, Z D_{\eta} G \varphi \big]
\cdot \partial_{k} \eta$ can be expanded in a sum of terms under the form
$$ \partial^{\gamma_{0}} Z  D_{\eta}^n G[\eta] \partial^\beta \varphi \cdot \big(
\partial^{\gamma_{1}} \eta , \cdots, \partial^{\gamma_{n-1}}\eta, \partial^{\gamma} \partial_{k} \eta \big)$$
where the indices satisfy the constraints
$$ n \geq 1, \quad  |\gamma_{0}| + |\beta| + |\gamma_{1}|+\cdots +  |\gamma_{n-1}| + |\gamma|=2, \quad
|\gamma | \leq 1.$$
Consequently, in view of \eqref{comh1}, \eqref{indices2}, we see that 
the number of derivatives  on $\varphi$ is the same as in \eqref{comh1} and that  
$1 + |\gamma |    \leq 2 $. Consequently, \eqref{dh} also follows from \eqref{eId-Ia}. 
This ends the proof of Lemma \ref{comh}.
\end{proof}
To prove Proposition \ref{R_2}, we shall also need the following statement.
\begin{lem}\label{nh}
For $m\geq 2$ and $S\geq 5$,
\begin{multline*}
\big\|\langle\partial\rangle^m
\big(
\nabla \cdot  \big( D^2A( \nabla \eta^\delta) \cdot \big(
\nabla \pa_{k} \eta^\delta, \nabla \pa_{j} \eta^\delta , \nabla \pa_{i} \eta^\delta \big)\big)
-
\nabla \cdot  \big( D^2A( \nabla \eta^{a}) \cdot
( \nabla \pa_{k} \eta^{a}, \nabla \pa_{j} \eta^{a}, \nabla \pa_{i}\eta^{a}\big)\big)\big)
\big\|_{H^{\frac{1}{2}}}
\\
\leq
\overline{\omega}_{m,S} \|U\|_{X^{m+3}}.
\end{multline*}
\end{lem}
\begin{proof}
We  first expand
$$
D^2A( \nabla (\eta+\eta^{a})) \cdot( \nabla \pa_{k} (\eta+\eta^{a}), \nabla \pa_{j} (\eta+\eta^{a}), \nabla
\pa_{i}(\eta+\eta^{a})\big)
-
D^2A( \nabla \eta^{a}) \cdot
( \nabla \pa_{k} \eta^{a}, \nabla \pa_{j} \eta^{a}, \nabla \pa_{i}\eta^{a}\big)
$$
as
\begin{eqnarray*}
&& (D^2A( \nabla (\eta+\eta^{a}))-D^2A( \nabla \eta^{a}))
\cdot
( \nabla \pa_{k} \eta^{a}, \nabla \pa_{j} \eta^{a}, \nabla
\pa_{i}\eta^{a}\big)
\\
&+& 
D^2A( \nabla (\eta+\eta^{a}))\cdot
( \nabla \pa_{k} \eta, \nabla \pa_{j} \eta^{a}, \nabla \pa_{i}\eta^{a}\big)
\\
&+&
D^2A( \nabla (\eta+\eta^{a}))\cdot
( \nabla \pa_{k} (\eta+\eta^{a}), \nabla \pa_{j} \eta, \nabla \pa_{i}\eta^{a}\big)
\\
&+& 
D^2A( \nabla (\eta+\eta^{a}))\cdot
( \nabla \pa_{k} (\eta+\eta^{a}), \nabla \pa_{j} (\eta+\eta^{a}), \nabla \pa_{i}\eta\big)\,
\end{eqnarray*}
 and we rewrite the first term  as 
\begin{eqnarray*}
& &(D^2A( \nabla (\eta+\eta^{a}))-D^2A( \nabla \eta^{a}))
\cdot
( \nabla \pa_{k} \eta^{a}, \nabla \pa_{j} \eta^{a}, \nabla
\pa_{i}\eta^{a}\big)
\\
& &= 
\int_{0}^{1}
D^3A( \nabla (s\eta+\eta^{a}))\cdot
( \eta,\nabla \pa_{k} \eta^{a}, \nabla \pa_{j} \eta^{a}, \nabla \pa_{i}\eta^{a}\big)ds.
\end{eqnarray*}
Let us observe that in all the above terms, we have at most two derivatives
of $\eta$ involved. We can then  apply $\nabla$ to these terms  which implies
that in  the end we have  at  most  three derivative of $\eta$ involved. Since 
 we need an  $H^{\frac{1}{2}}$  estimate after
applying $\langle\partial\rangle^m$, we can complete the  proof of Lemma~\ref{nh} by 
coming back to the definition of $A$ and using the classical Moser type estimates.
\end{proof}
\subsubsection*{End of the proof of Proposition \ref{R_2}}
The combination of Lemma~\ref{comh}
and Lemma~\ref{nh} completes the proof of Proposition~\ref{R_2}.

Note that   Proposition~\ref{R_1} and Proposition~\ref{R_2} lead to the
following statement.
\begin{cor}\label{stef1}
For $m\geq 2$ and $S\geq 5$, we have 
$$
\|\mathcal{G}^{ijk}_1[\eta, \varphi]\|_{X^m}
\leq \omega\big(\|V^{a}\|_{{\mathcal W}^{m+S}}+\|U\|_{X^{m+3}}\big)\|U\|_{X^{m+3}}\,.
$$
\end{cor}
To end the proof of Proposition \ref{main_sect}, it remains to estimate $\mathcal{G}^{ijk}_{2}$.
We have the following statement.
\begin{prop}\label{shaud}
For $m\geq 2$ and $S\geq 5$, we have the estimate
$$
\|\mathcal{G}^{ijk}_2[\eta, \varphi]\|_{X^m}\leq  \overline{\omega}_{m,S}\|U\|_{X^{m+3}}\,.
$$
\end{prop}
\begin{proof}
We need to show that
\begin{equation}\label{ll1}
\|\langle\pa\rangle^m(I[U+V^{a}]-I[V^{a}])\|_{H^{1}}\leq  \overline{\omega}_{m,S} \|U\|_{X^{m+3}}
\end{equation}
and
\begin{equation}\label{ll2}
\|\langle\pa\rangle^m(J[U+V^{a}]-J[V^{a}])\|_{H^{\frac{1}{2}}}\leq  \overline{\omega}_{m,S}  \|U\|_{X^{m+3}},
\end{equation}
where
$$
I[\eta,\varphi]=-\nabla\cdot\big(v[\eta,\varphi]\pa_{ijk}\eta^{a}\big)-G[\eta]\big(Z[\eta,\varphi]\pa_{ijk}\eta^{a}\big)+
G[\eta]\pa_{ijk}\varphi^{a}
$$
and
\begin{eqnarray*}
J[\eta,\varphi] & = &\beta\nabla\cdot\big(A(\nabla\eta)\nabla\pa_{ijk}\eta^{a}\big)-
Z[\eta,\varphi]G[\eta]\big(Z[\eta,\varphi]\pa_{ijk}\eta^{a}\big)
\\
& &
-Z[\eta,\varphi]\, \nabla\cdot
v[\eta,\varphi]\,\pa_{ijk}\eta^{a}-v[\eta,\varphi]\cdot\nabla\pa_{ijk}\varphi^{a}+
Z[\eta,\varphi]G[\eta]\pa_{ijk}\varphi^{a}\,.
\end{eqnarray*}
Again, we  observe that to get (\ref{ll1}) and (\ref{ll2}), it suffices
to use again Proposition~\ref{pak}, Proposition~\ref{pak_bis} and Lemma~\ref{annr}.
All the terms that we have to handle are very similar to the ones that  we  have
estimated in the proofs of Lemma \ref{comh} and Lemma \ref{nh}, consequently, we shall not
give more details.
This ends  the proof of Proposition~\ref{shaud}.
\end{proof}
\subsubsection*{End of the proof of Proposition \ref{main_sect} }
It suffices to 
combine the statements of  Corollary~\ref{stef1} and Proposition~\ref{shaud} 
in view of \eqref{defG}.
\subsection{Estimates of the subprincipal terms}
In this subsection, we turn to the estimates on the subprincipal term of 
\eqref{Ueq} namely the term
$(\mathcal{Q}^{ijk})^\delta - (\mathcal{Q}^{ijk})^a$.
\begin{prop} \label{Q}
For $m \geq 2$ and $S\geq 5$, we have 
\beq\label{Q1}
\| \langle \partial \rangle^m \big( (\mathcal{Q}_1^{ijk})^\delta -(\mathcal{Q}_1^{ijk})^{a} \big)\|_{H^{-\frac{1}{2}}}
\leq  \overline{\omega}_{m,S}\| U\|_{X^{m+2}}\,.
\eeq
and
\begin{equation}\label{Q2}
\| \langle \partial  \rangle^m \big((\mathcal{Q}_2^{ijk})^\delta -(\mathcal{Q}_2^{ijk})^{a}\big)\|_{H^{-1}}
\leq  \overline{\omega}_{m ,S} \| U\|_{X^{m+2}}\,.
\end{equation}
\end{prop}
\begin{proof}
We first prove \eqref{Q1}.
Thanks to \eqref{Q1eq}, we need to estimate 
$\| \langle \partial \rangle^m I_{ijk} \|_{H^{-\frac{1}{2}}}$
with 
$$
I_{ijk}= D_\eta G^\delta \pa_{jk}\varphi^\delta \cdot\pa_i \eta^\delta-
D_\eta G^{a}\pa_{jk}\varphi^{a}\cdot\pa_i\eta^{a}
$$
(and similar expression by cycle change of $i,j,k$).  We can expand this expression 
as
$$
(D_\eta G^\delta -D_\eta G^{a})\pa_{jk}\varphi^{a}\cdot\pa_i\eta^{a}
+
D_\eta G^\delta \pa_{jk}\varphi\cdot\pa_i\eta^{a}
+
D_\eta G^\delta\pa_{jk}\varphi^\delta\cdot\pa_i\eta\,.
$$
Now we observe that each term in the above decomposition is in the scope of
applicability of Proposition~\ref{pak} or Proposition~\ref{pak_bis}.
Indeed the first term can be estimated by invoking Proposition~\ref{pak_bis}
after writing
$$
(D_\eta G^\delta -D_\eta
G^{a})\pa_{jk}\varphi^{a}\cdot\pa_i\eta^{a}
=\int_{0}^1D_\eta^2  G[s\eta+\eta^{a}]\pa_{jk}\varphi^{a}\cdot(\pa_i\eta^{a},\eta)ds.
$$
Observe that in this place we can afford to crudely estimate  the
$H^{-\frac{1}{2}}$  norm by the $H^1$ norm  controlled by Proposition~\ref{pak_bis}.
The second term can be estimated in $H^{-\frac{1}{2}}$ by invoking  Proposition~\ref{pak}.
The third term can be written as
\begin{equation}\label{hhrr}
D_\eta G^\delta\pa_{jk}\varphi\cdot\pa_i\eta+D_\eta G^\delta \pa_{jk}\varphi^{a}\cdot\pa_i\eta.
\end{equation}
Then the first term   in the decomposition (\ref{hhrr}) can be estimated in
$H^{-\frac{1}{2}}$ by using Proposition~\ref{pak} while the second term is in the scope of applicability
of Proposition~\ref{pak_bis}, again after crudely   estimating   the 
$H^{-\frac{1}{2}}$ norm by the $H^1$ norm. This ends  the proof of \eqref{Q1}.

Let us turn to the proof of \eqref{Q2}.  By using \eqref{Q2eq}, 
a duality argument and an integration by parts, we obtain that the issue is  to evaluate
\begin{equation}\label{maik}
\| \langle \partial \rangle^m
(DA(\nabla(\eta^\delta)\cdot(\pa_i\nabla\eta^\delta ,\pa_{jk}\nabla \eta^\delta )-
DA(\nabla(\eta^{a}))\cdot(\pa_i\nabla\eta^{a},\pa_{jk}\nabla\eta^{a})))\|_{L^2}
\end{equation}
(and similar expression by cycle change of $i,j,k$).
For that purpose, it suffices to observe that in the definition of the
expression (\ref{maik}) there are at most $m+3$ derivatives of $\eta$ involved and that at least one of them
is a spatial derivative. Then the estimate of the $L^2$ norm of (\ref{maik}) follows from
the Moser type estimates of Proposition \ref{anr}  after some decompositions in the spirit of the proof of
Lemma~\ref{nh}. This completes the proof of  Proposition~\ref{Q}.
\end{proof}

\subsection{The quadratic form associated to $L^\delta$}
The last step before turning to the proof of the energy estimates is  the study of the quadratic form
associated to the operator $L^\delta= L[U^\delta]$ which arises as the main part of \eqref{W}.

\begin{prop}
\label{LdeltaQ}
We have the estimates
\beq
\label{minor}
(L^\delta W,W)+\|W\|_{L^2}^2\geq { 1 \over  \overline{\omega}_{2,5}}\|W\|_{X^0}^2\,
\eeq
\beq
\label{major}
|(L^\delta W,V)|\leq
\overline{\omega}_{2,5}\|W\|_{X^0}\|V\|_{X^0}\,.
\eeq 
Moreover 
for $m\geq 2$ and $S\geq 5$ if $\partial^\alpha$ a space-time derivative such that $|\alpha|\leq m$, we have
\beq\label{major2}
\big|\big(\big[\partial^\alpha,L^\delta \big]W,V\big)\big|\leq
\overline{\omega}_{m,S}\|W\|_{X^{k-1}}\|V\|_{X^0}\,.
\eeq
\end{prop}
\begin{proof}
To prove \eqref{minor}, 
as a preliminary, we shall first check the positivity 
of  the second order  operator $\mathcal{P}= - \nabla \cdot\big( A  \nabla \big)+ \alpha$.
This is an elliptic operator as given by: 
\begin{lem}\label{P}
There exists $c>0$ such that for every $u\in H^1(\R^2)$,
\begin{eqnarray*}
\big( - \nabla\cdot(A(\nabla\eta)\nabla u)+\alpha u,u\big) & \geq & 
(1+\|\nabla \eta\|_{L^{\infty}(\R^2)})^{-3}\|\nabla u\|^2+\alpha\|u\|^2
\\
& \geq & c(1+\|\nabla \eta\|_{L^{\infty}(\R^2)})^{-3}\|u\|_{H^1(\R^2)}^2\,.
\end{eqnarray*}
\end{lem}
\begin{proof}[Proof of Lemma \ref{P}]
Since
$$
A(\nabla\eta)=
(1+|\nabla\eta|^2)^{-\frac{3}{2}}
\left(\begin{array}{cc} 
1+(\partial_y\eta)^2 & -\partial_x\eta\partial_y\eta
\\
-\partial_x\eta\partial_y\eta &  1+(\partial_x\eta)^2 
\end{array}\right),
$$
the statement follows from an integration by parts  and  the inequality
$$
(\partial_y\eta)^2(\partial_x u)^2+(\partial_x\eta)^2(\partial_y u)^2
-2(\partial_x u)(\partial_y u)(\partial_x\eta)(\partial_y\eta)\geq 0.
$$
This completes the proof of Lemma~\ref{P}.
\end{proof}
\bigskip
Note that 
 by Sobolev embedding, we have that
\beq
\label{lsob} \|\nabla \eta^\delta \|_{L^\infty} \lesssim \|\eta \|_{H^4} + \|\nabla \eta^a  \|_{L^\infty}
\leq  \bar{\omega}_{0, 1}.
\eeq
Consequently, 
with the notation $W=(W_1,W_2)$,  we can use  
 Lemma~\ref{G} and  Lemma~\ref{P},  to get 
\begin{multline*}
(L^\delta W,W)
\geq\frac{1}{
\overline{\omega}_{2,5}}
\Big(\|W_1\|_{H^1}^2+\Big\|\frac{| \nabla |}{(1+|\nabla|)^{\frac{1}{2}}}W_2\Big\|_{L^2}^2\Big)
\\
-\overline{\omega}_{2,5 }\|W_1\|_{L^2}^2+2\Big(\partial_x W_1-\nabla\cdot(W_1 v^\delta),W_2\Big).
\end{multline*}
Therefore (using in particular inequality (\ref{Young})) we get 
\begin{equation*}
(L^\delta W,W)\geq
\frac{1}{\overline{\omega}_{2,5}}\Big(\|W_1\|_{H^1}^2+\| |\nabla |^{\frac{1}{2}}W_2\|_{L^2}^2\Big)
-\overline{\omega}_{2,5}\|W\|_{L^2}^2\,.
\end{equation*}
Using that $\|W_2\|_{H^{1/2}}\approx (\|W_2\|_{L^2}+\||\nabla |^{\frac{1}{2}}W_2\|_{L^2})$, we obtain that
\begin{equation*}
(L^\delta W,W)\geq \frac{1}{\overline{\omega}_{2,5}}\Big(\|W_1\|_{H^1}^2+\|W_2\|_{H^{\frac{1}{2}}}^2\Big)
-\overline{\omega}_{2,5}\|W\|_{L^2}^2\,
\end{equation*}
which in turn implies the claimed inequality.

To prove \eqref{major}, let us  write
$W=(W_1,W_2)\in H^1\times H^{\frac{1}{2}}$ and $V=(V_1,V_2)\in H^1\times H^{\frac{1}{2}}$,  then
we have 
\begin{eqnarray*}
(L^\delta W,V) & = & ({\mathcal P}^\delta W_1,V_1)+(G^\delta W_2,V_2)+\big(((v^\delta\cdot \nabla Z^\delta)+(\pa_t-\pa_x)Z^\delta)\,
W_1,V_1\big)
\\
&&
-(\nabla\cdot(v^\delta W_1)-\pa_x W_1,V_2)-(\nabla\cdot(v^\delta V_1)-\pa_x V_1,W_2)\,.
\end{eqnarray*}
We now estimate each term in the above decomposition.
Coming back to the definition of ${\mathcal P}$, using an integration by parts
and  \eqref{lsob} we have 
$$
|({\mathcal P}^\delta W_1,V_1)|\leq 
\overline{\omega}_{2, 5}\|W_1\|_{H^1}\|V_1\|_{H^1}\,.
$$
Next using Lemma~\ref{G}, we also have that 
$$
|(G^\delta W_2,V_2)|
\leq
\overline{\omega}_{2,5}\|W_2\|_{H^\frac{1}{2}}\|V_2\|_{H^\frac{1}{2}}.
$$ 
Next, we can use the Sobolev embedding and Lemma~\ref{annr} to get
$$
\big|\big(((v^\delta \cdot \nabla Z^\delta )+(\pa_t-\pa_x)Z^\delta)\,W_1,V_1\big)\big|\leq
\overline{\omega}_{2, 5} \|W_1\|_{H^1}\|V_1\|_{H^1}.
$$
The last two terms in the decomposition can be estimated by invoking the
inequality
$$
|(\nabla\cdot(v^\delta V_1)-\pa_x V_1,W_2)|
\leq
\overline{\omega}_{2,5}\|V_1\|_{H^1}\|W_2\|_{H^\frac{1}{2}}
$$
which is a direct consequence of the Sobolev embedding and Lemma~\ref{annr}.
This completes the proof of the first inequality in our statement.
Let us now turn to the proof of the second inequality.
We can first write
\begin{eqnarray*}
([\pa^\alpha,L^\delta ]W,V) & \leq & |([\pa^\alpha,{\mathcal
P^\delta}]W_1,V_1)|
+
|([\pa^\alpha, G^\delta]W_2,V_2)|
\\
&&
+|\big([\pa^\alpha,((v^\delta\cdot \nabla Z^\delta)+(\pa_t-\pa_x)Z^\delta)]
W_1,V_1\big)|
\\
&&
+|(\nabla\cdot([\pa^\alpha, v^\delta]W_1),V_2)| + |(\nabla\cdot([\pa^\alpha,v^\delta]V_1),W_2)|\,.
\end{eqnarray*}
Now each term in the above decomposition can be estimated  as in the
proof of the first inequality by taking the advantage of the commutator structure and thus by using
Lemma \ref{comh}.
This completes the proof of  Proposition~\ref{LdeltaQ}.
\end{proof}
\subsection{Energy estimates for the main part of \eqref{Ueq}}\label{energien}
Recall also that $U(0)=0$. The goal of this section is to prove the following statement.
\begin{prop}\label{energy}
For $m\geq 2$, $S \geq 5$,  a smooth solution  of (\ref{lionia}) satisfy the estimate
\begin{multline*}
\| U(t)\|^2_{X^{m+3}} \leq \omega\Big(\|R^{ap}\|_{{X}^{m+3}_t}+\|V^{a}\|_{{\mathcal W}^{m+S}_t}+\|U\|_{X^{m+3}_t}\Big)
\\
\times\Big( \| R^{ap}\|^2_{{X}^{m+3}_t}
+\int_{0}^t\big(\|U(\tau)\|^2_{X^{m+3}}+
\|F(\tau)\|^2_{X^{m+3}}\big)d\tau\Big).
\end{multline*}
\end{prop}
\begin{proof}[Proof of Proposition~\ref{energy}]
It is more convenient to perform the energy estimate on the equation \eqref{W} satisfied
by $W_{ijk}= PU_{ijk}$.  Since thanks to 
Lemma~\ref{annr}, for $S\geq 5$, we have 
\begin{equation}\label{vio_0}
\|W_{ijk}\|_{X^m}\leq  \overline{\omega}_{m,S}\|U\|_{X^{m+3}}
\end{equation}
and
\begin{equation}\label{vio}
\|U\|_{X^{m+3}}\leq  \overline{\omega}_{m,S}
\Big(\sum_{i,j,k}\|W_{ijk}\|_{X^m}+\|U\|_{X^3}\Big),
\end{equation}
it is equivalent to estimate $W_{ijk}$ or $U_{ijk}$.
We thus   consider, for $|\alpha|\leq m$, the equation solved by
$\partial^\alpha W_{ijk}$. Namely
\begin{equation}\label{thoe}
\partial_t\partial^\alpha W_{ijk}=J\Big(L^\delta\partial^\alpha W_{ijk}+[\partial^\alpha,L^\delta]W_{ijk}-\partial^\alpha
JPJ\big((\mathcal{Q}^{ijk})^\delta -(\mathcal{Q}^{ijk})^{a}\big)\Big)+\partial^{\alpha}(PF_{ijk}).
\end{equation}
We take the $L^2$ scalar product of (\ref{thoe}) with
$$
{\mathcal M}\equiv L^\delta \partial^\alpha W_{ijk}+[\partial^\alpha,L^\delta ]W_{ijk}-\partial^\alpha
JPJ\big((\mathcal{Q}^{ijk})^\delta -(\mathcal{Q}^{ijk})^{a} \big).
$$
From the skew symmetry of $J$ and the symmetry of $L^\delta$, we get the identity
\begin{eqnarray}
\label{idenerg}
& & {d \over dt}
\Big( { 1 \over 2} \big( \partial^\alpha W_{ijk}, L^\delta \partial^\alpha W_{ijk}\big)
+ \mathcal{I}^{\alpha}_{ijk} \Big) 
   = \mathcal{J}^{\alpha}_{ijk}
\end{eqnarray}
where 
\begin{eqnarray}
\label{Idef} & &\mathcal{I}^{\alpha}_{ijk}=  \big(\partial^\alpha W_{ijk}, [\partial^\alpha, L^\delta ] W_{ijk}\big)
- \big( \partial^\alpha W_{ijk}, \partial^\alpha \big( JPJ\big( (\mathcal{Q}^{ijk})^\delta - (\mathcal{Q}^{ijk})^a
 \big) \big),\\
\label{Jdef}
& &   \mathcal{J}^{\alpha}_{ijk}  =  
{1 \over 2} \big(  \partial^\alpha W_{ijk} , [\partial_{t}, L^\delta ] \partial^\alpha W_{ijk}   \big)
+ \big( \partial^\alpha W_{ijk}, \partial_{t} [\partial^\alpha, L^\delta ] W_{ijk}\big)  \\
& &  \quad \quad    \nonumber -
 \big( \partial^\alpha W_{ijk}, \partial_{t} \partial^\alpha  \big(
 JPJ\big( (\mathcal{Q}^{ijk})^\delta - 
(\mathcal{Q}^{ijk})^a
 \big) \big)  
+\big( \mathcal{M}, \partial^\alpha(PF_{ijk}) \big).
\end{eqnarray}
For the proof of Proposition \ref{energy}, we shall estimate
separately all the terms arising in the energy identity (\ref{idenerg}) in the following sequence
of Lemmas.

We first  have the following estimate for $ \mathcal{I}^{\alpha}_{ijk}$  and $\mathcal{J}^{\alpha}_{ijk}$:
\begin{lem}\label{IJ}
For $m\geq 2$, $S\geq  5$, we have the estimates
\beq\label{Iest}
| \mathcal{J }^{\alpha}_{ijk} |\leq  \overline{\omega}_{m,S}\big( \| U\|_{X^{m+3}}\|F_{ijk}\|_{X^m}
+ \|U \|_{X^{m+ 3}}^2\big) ,\quad |\alpha|\leq m
\eeq
and
\beq
\label{Jest}
| \mathcal{I}^{\alpha}_{ijk} | \leq   \overline{\omega}_{m,S} \|U\|_{X^{m+ 3}} \|U\|_{X^{m+ 2} }, \quad
|\alpha | \leq m.
\eeq
\end{lem}
\begin{proof}
Let us start with the estimate of $\mathcal{I}^{\alpha}_{ijk}$. In view of \eqref{Idef}, we can write 
$$ \mathcal{I}^{\alpha}_{ijk}= I_{1}- I_{2}.$$
From \eqref{major2}, in Proposition \ref{LdeltaQ} and Cauchy-Schwarz, we immediately get
$$ |I_{1}| \leq  \overline{\omega}_{m,S} \|U\|_{X^{m+3 }}\|U \|_{X^{m+2}}.$$
To estimate the second term, we can expand
\begin{eqnarray*}
I_{2}& = & \big(\partial^\alpha W_{ijk}[1], \partial^{\alpha}
\big( (\mathcal{Q}_2^{ijk})^\delta  - (\mathcal{Q}_2^{ijk} )^{a} \big) \big)
- \big(\partial^\alpha W_{ijk}[1],\partial^\alpha
\big(Z^\delta \big ( ( \mathcal{Q}_1^{ijk})^\delta - ( \mathcal{Q}_1^{ijk})^{a} \big) \big)
\big) \\
&    & +  \Big( \partial^\alpha W_{ijk}[2],  \partial^\alpha\big( (\mathcal{Q}_{1}^{ijk})^\delta
- (\mathcal{Q}_{1}^{ijk})^a \big) \big) \Big) \\
& =  &  I_{2}^1 + I_{2}^2 + I_{2}^3
\end{eqnarray*}
by using the notation $W_{ijk}= (W_{ijk}[1], W_{ijk}[2])^t$.
A direct application of    Proposition  \ref{Q}  yields
$$ | I_{2}^1 | 
\leq  \|\partial^\alpha W^{ijk}[1] \|_{H^1} \| (\mathcal{Q}_2^{ijk})^\delta  - (\mathcal{Q}_2^{ijk} )^{a} 
\|_{H^{-1}}  \leq   \overline{\omega}_{m,S} \|U\|_{X^{m+ 3 }} \|U \|_{X^{m+ 2}}$$
and
$$ | I_{2}^3 | 
\leq  \|\partial^\alpha W^{ijk}[2] \|_{H^{1\over2}} \| (\mathcal{Q}_1^{ijk})^\delta  - (\mathcal{Q}_1^{ijk} )^{a} 
\|_{H^{-{1\over 2}}}  \leq   \overline{\omega}_{m,S} \|U\|_{X^{m+ 3 }} \|U \|_{X^{m+ 2}}.$$
Note that we have used \eqref{vio_0} for the second part of the estimates.   
For the estimate of $I_{2}^2$, we can use the following lemma:
\begin{lem}
\label{lempff}
For $l \geq 2$,  $l \leq m+ 1$,  $m \geq 2$ and  $S \geq 5$,    we have the estimates 
\beq
\label{Zdernier}
 \| \langle  \pa \rangle^{m+2} \big( Z^\delta - Z^a \big) \|_{H^{\sigma  }} \leq   
 \overline{\omega}_{m,S} \| \langle \pa \rangle^{m+2}  U \|_{H^{ \sigma+ 1 } }, 
  \quad \sigma = -{1 \over 2}, \, 0, \, {1 \over 2 }  \eeq
  and 
\beq
\label{Zdernierbis}
\big| \big( \partial^\beta  \big( Z^\delta  \big ( ( \mathcal{Q}_1^{ijk})^\delta - ( \mathcal{Q}_1^{ijk})^{a} \big) \big)
, F \big) \big| 
\leq \overline{\omega}_{m,S} \|U\|_{X^{l+ 2}} \|F \|_{H^{1 }}, \quad
 | \beta | \leq l.
\eeq
\end{lem}
Let us postpone the proof of this last lemma. By a direct application of it with $F= \partial^\alpha
 W^{ijk}[1]$, $\beta= \alpha$, $l=m$ and a new use of \eqref{vio_0}, we also  get 
$$ |I_{2}^2| \leq \overline{\omega}_{m, S} \|U\|_{X^{m+ 3}} \|U\|_{X^{m+ 2 }}.$$
This proves \eqref{Jest}.

Let us turn to  the estimate of $\mathcal{J}^{\alpha}_{ijk}$. In view of \eqref{Jdef}, we set
$$ \mathcal{J}^{\alpha}_{ijk} = J_{1}+ J_{2}-  J_{3}+ J_{4}.$$
By using \eqref{major2} in Proposition  \ref{LdeltaQ} and (\ref{vio_0}), we infer 
\begin{equation}\label{i1}
|J_1|\leq\ \overline{\omega}_{m,S}\|W_{ijk}\|_{X^m}^2
\leq  \overline{\omega}_{m,S}\|U\|_{X^{m+3}}^2\,.
\end{equation}
In order to estimate $J_{2}$, it suffices to write
$
\partial_{t}[\partial^\alpha,L^\delta]=[\pa_t\pa^\alpha,L^\delta ]+[L^\delta ,\pa_t]\pa^\alpha
$
and to apply \eqref{major2}. This yields
$$ |J_{2}| \leq \overline{\omega}_{m,S} \|U\|_{X^{m+3}}^2.$$
To estimate $J_{3}$, we first expand
\begin{eqnarray*}
J_{3} & = &  \big(\partial^\alpha W_{ijk}[1],  \partial_{t}\partial^{\alpha}
\big( (\mathcal{Q}_2^{ijk})^\delta  - (\mathcal{Q}_2^{ijk} )^{a} \big) \big) - 
\big( \partial^\alpha W_{ijk}[1],\partial_t \partial^{\alpha}
\big( Z^\delta \big (\mathcal{Q}_1^{ijk})^\delta -(\mathcal{Q}_1^{ijk})^{a}\big)\big)\big) \\
& & +
\big(\partial^\alpha 
W_{ijk}[2],\partial_t\partial^\alpha\big( (\mathcal{Q}_1^{ijk})^\delta - (\mathcal{Q}_1^{ijk})^{a}\big) \big)\\
&= &J_{3}^1 + J_{3}^2 + J_{3}^3.
\end{eqnarray*}
The estimate   of $J_{3}^1$ and $J_{3}^2$ follows  by a direct application of Proposition~\ref{Q}
(changing $m$ into $m+1$) while the estimate of $J_{3}^2$ is a consequence of 
the estimate \eqref{Zdernierbis}  of Lemma~\ref{lempff}.

It remains to estimate $J_{4}$. Let us write
\begin{eqnarray*}
\big( \mathcal{M}, \partial^\alpha(PF_{ijk}) \big)
& = & 
\big(L^\delta \partial^\alpha W_{ijk},  \partial^\alpha(PF_{ijk}) \big)
 +  \big( [\partial^\alpha, L^\delta ]W_{ijk}, \partial^\alpha(PF_{ijk}) \big) \\
 & &  
-\big( \partial^\alpha\big( JPJ\big( (\mathcal{Q}^{ijk})^\delta -
(\mathcal{Q}^{ijk})^a \big)  \big), \partial^\alpha(PF_{ijk})  \big). 
\end{eqnarray*}
We can estimate the two first terms by using  \eqref{major}, \eqref{major2} in  Proposition \ref{LdeltaQ},
while the estimate of the last term follows by using again Proposition~\ref{Q} and Lemma \ref{lempff}.
 This ends the proof of Lemma \ref{IJ}, it just remains to prove  Lemma \ref{lempff}. 
 \end{proof}
\begin{proof}[Proof of Lemma \ref{lempff}]
We start with the proof of \eqref{Zdernier}.  By using  the  same decomposition of $Z$
as in the proof of Lemma \ref{annr}  and in particular \eqref{Z1dec}, we see that
 the proof follows   from the following estimate on $G$:
 $$   \| \langle  \pa \rangle^{m+2} \big( G^\delta \varphi - G^a \varphi  \big) \|_{H^{\sigma }} \leq  
 \overline{\omega}_{m,S} \| \langle \pa \rangle^{m+2} \varphi\|_{H^{\sigma + 1  } }, \quad
  \sigma = - {1 \over 2}, \, 0, \, {1\over 2 }.$$
  Since 
  $$  G^\delta \varphi - G^a \varphi = \int_{0}^1 D_{\eta } G[\eta^a + s \eta]\varphi \cdot \eta\, ds$$
   the needed estimate
    for $\sigma = -1/2$ and $\sigma= 1/2$ is a consequence of the refined estimate \eqref{San} in Remark \ref{Sanremark}. The estimate for $\sigma =0$ follows by interpolation.

To prove \eqref{Zdernierbis}, we  first write for $\beta_{1}+\beta_{2}=\beta$
\begin{equation*}
|(\pa^{\beta_{1}}
\big(Z^\delta)\pa^{\beta_{2}}\big((\mathcal{Q}^{ijk}_1)^\delta -(\mathcal{Q}^{ijk}_1)^{a}\big),F\big)|
\leq
\|\pa^{\beta_{2}}\big( (\mathcal{Q}^{ijk}_1)^\delta - (\mathcal{Q}^{ijk}_1)^{a} \big)\|_{H^{-{1\over 2 }}}
\|\pa^{\beta_{1}}(Z^\delta)  F\|_{H^{1\over 2}}\,.
\end{equation*}
The first term in the right hand-side of the above inequality can be
estimated by using the estimate \eqref{Q1} of  Proposition~\ref{Q}  while for  the second one,
we use Lemma \ref{LP}
to get
\begin{eqnarray*}
\|\pa^{\beta_{1}}(Z^\delta)  F\|_{H^1} &  \leq &  \|\pa^{\beta_{1}}(Z^\delta - Z^a)  F\|_{H^1}+
\|\pa^{\beta_{1}}(Z^ a)  F\|_{H^1}  \\
&\leq &   \|\pa^{\beta_{1}}(Z^\delta - Z^a)  F\|_{H^1}+
\omega( \| V^a \|_{{\mathcal W}^{m+S}} ) \| F\|_{H^1}.
\end{eqnarray*}
Next,  we  
invoke the  inequality
$$
\|uv\|_{H^{1}(\R^2)}\leq C\|u\|_{H^{3\over 2 }(\R^2)}\|v\|_{H^{1}(\R^2)}
$$
to get
$$ \|\pa^{\beta_{1}}(Z^\delta - Z^a)  F\|_{H^{1}} \leq C \,  \|\pa^{\beta_{1}} (Z^\delta - Z^a) \|_{H^{3\over 2 }} \|F\|_{H^1}
\leq   C \,\| \langle \pa \rangle ^{l+1} (Z^\delta - Z^a) \|_{H^{1\over 2 }} \|F\|_{H^1}$$
and  hence
$$ \|\pa^{\beta_{1}}(Z^\delta - Z^a)  F\|_{H^1} \leq   \overline{\omega}_{m,S} \|U \|_{X^{m+3}} \|F\|_{H^1}
\leq   \overline{\omega}_{m,S}  \|F\|_{H^1} $$
thanks to  \eqref{Zdernier}.
This completes the proof of  Lemma~\ref{lempff}.
\end{proof}
We shall  consider in addition to the energy identity \eqref{idenerg} the following estimate:
\begin{lem}\label{L2}
Let $U$ be a solution of (\ref{Ueqinstab}) then for $m\geq 2$ and $S\geq 5$,
$$
\frac{d}{dt}\|U(t)\|_{X^3}^2
\leq  \omega\Big(\|R^{ap}\|_{X^3}+\|V^{a}\|_{{\mathcal W}^{m+S}_t}+\|U\|_{X^{m+3}_t}\Big)
\|U\|_{X^{m+3}}^2.
$$
Moreover, we also have the estimate
\beq
\label{dtfin}
\| \partial_{t}^{l}  U\|_{L^2}^2 \leq   
\overline{\omega}_{m,S}  \Big(\| \, |\nabla| \partial_{t}^{l-1} U\|_{X^{0}}^2 +  \|U\|_{X^{m+2}}^2  + \|R^{ap} \|_{X^m}^2\Big), \quad  4 \leq l \leq m+3.
\eeq
\end{lem}
This Lemma will be crucial in order to get  the claimed estimate of Proposition \ref{energy}
 from the identity \eqref{idenerg}. The first estimate will be used to compensate the fact
  that  $L^\delta$ does not control the low frequencies or equivalently the  $L^2$ norm (see \eqref{minor}).
  The second estimate will be used to prove that  by a suitable summation of  the identities \eqref{idenerg}
   we get a control of the $X^{m+3}$ norm of $W$.
    Note that this second estimate basically states that thanks to the equation, we can replace a time
     derivative by  space derivatives.
\begin{proof}
For the estimate of the $L^2$ norm of $U$, it suffices to multiply  (\ref{Ueqinstab}) by $U$, integrate over $\R^2$ and
use   very crude estimates of all the terms.  We  then proceed in a similar way by taking at most three
derivatives of the equation to get  the first claimed estimate.

To get the second estimate, we apply $\partial_{t}^{l-1}$  to \eqref{Ueq}.
We get  by taking the $L^2$ norm of  the obtained equation that
$$  \|\partial_{t}^l U\|_{L^2} \leq    \overline{\omega}_{m,S}
\Big( \|\partial_{t}^{l-1}\eta\|_{H^2} +\| \partial_{t}^{l-1} 
\varphi \|_{H^1} +   \|U\|_{X^{m+2}} \Big) +  \|R^{ap}\|_{X^m}.$$
 We  shall not give the  details of this estimate which can be obtained
  by using the same kind of estimates as  in  Propositions \ref{pak} and \ref{pak_bis} as   previously. 
Note that in particular the term involving $\varphi$ arises in this form by using \eqref{Zdernier}
for $\sigma =0$.
 The idea behind this estimate is simple, the equation allows to replace one time derivative
   by  one space derivative for $\varphi$ and two space derivatives for $\eta$.
Next, since we have that 
$$  \| \, | \nabla | \big( |\nabla | \partial_{t}^{l-1} 
\eta  \big)\|_{L^2} \leq C \| | \nabla|  \partial_{t}^{l-1} U \|_{X^0}$$
and  by standard interpolation  that 
\beq
\label{interpolfi} \| \nabla \partial_{t}^{l-1} 
 C \varphi \|_{L^2}^2 \leq \|\, |\nabla|^{1 \over 2 } \big( |\nabla
| \partial_{t}^{l-1} \big)\varphi ) \|_{L^2}  \,  \|\, |\nabla|^{1 \over 2} \partial_{t}^{l-1} \varphi \|_{L^2} 
\leq C  \|U\|_{X^{m+3}} \|U\|_{X^{m+2}},\eeq
the result follows.
This completes the proof of Lemma~\ref{L2}.
\end{proof}
To prove Proposition~\ref{energy}, in view of \eqref{idenerg}, 
we shall use  the energy
$$
E_{\alpha}(t)\equiv\sum_{i,j,k} \Big(\frac{1}{2}\big(\partial^\alpha W_{ijk},L^\delta \partial^\alpha W_{ijk}\big)
+ \mathcal{I}_{ijk}^\alpha \Big).
$$
We also define 
for   $m \geq 2$,  $ 1 \leq l\leq m$,  $\tau \in [0, t]$
$$ E_{l,m}(\tau)=  \sum_{ 1 \leq  |\alpha | \leq l\,\,  \alpha' \neq 0    } E_{\alpha}(\tau)
+  \Gamma \sum_{  1 \leq | \alpha | \leq l, \, \alpha'= 0} E_{\alpha}(\tau)$$
where $\alpha=(\alpha_0,\alpha')$ with $\alpha'= (\alpha_{1},\alpha_{2})$ and 
$\Gamma >0$ (which depends on $m$ and $t$)  will be carefully chosen.
The aim of the following sequence of lemmas is to  obtain the positivity of
the energy  in order to deduce an estimate from \eqref{idenerg}. We first  have the following :
\begin{lem}
\label{Elk}
There exists $\Gamma>0$ such that  for every $\tau \in [0, t]$, we have
$$   E_{l,m}(\tau) \geq  { 1 \over \overline{\omega}_{m,S,t}  } \|
U(\tau)\|_{X^{l+3}}^2  -   \overline{\omega}_{m,S,t}  \|U
(\tau)\|_{X^{ l +2}}^2
$$
where $\overline{\omega}_{m,S,t}= \omega\big(\|V^{a}\|_{{\mathcal W}^{m+S}_t}+\|U\|_{X^{m+3}_t}\big)$.
\end{lem}
\begin{proof}
We first consider the case that $\alpha' \neq 0$, i.e. $\partial^\alpha \neq \partial_{t}^l$.
In this case,  we have from  \eqref{minor} of Proposition  \ref{LdeltaQ}  and Lemma \ref{IJ} that
$$ E_{\alpha}(\tau) \geq { 1 \over \overline{\omega}_{m,S}  }
\| \partial^\alpha W (\tau)  \|_{X^0}^2 - \|\partial^\alpha W(\tau) \|_{L^2}^2 -  \overline{\omega}_{m, S}
\|U(\tau) \|_{X^{l+3}} \, \|U(\tau)\|_{X^{l+2}}.$$
Now, since $\partial^\alpha$ contains at least one space derivative, we can write for $|\beta |= |\alpha |-1$
that 
$$ 
\|\partial^\alpha W(\tau)\|_{L^2}^2 \lesssim  \|\nabla \partial^{\beta} W(\tau) \|_{L^2}^2 \lesssim    \|U(\tau)\|_{X^{l+ 3 } } \|U
(\tau)\|_{X^{l+ 2}} $$
by using again the interpolation inequality \eqref{interpolfi}.
Consequently, by using the Young inequality and  \eqref{vio}, we find that
\beq
\label{Ealpha}
  \sum_{ \alpha' \neq 0} E_{\alpha}(\tau)
\geq { 1 \over \overline{\omega}_{m,S,t}  } \sum_{\alpha' \neq 0 } \| \partial^\alpha  U(\tau)\|_{X^{0}}^2  -   \overline{\omega}_{m,S,t}  \|U
(\tau)\|_{X^{ l +2}} \|U
(\tau)\|_{X^{ l +3}} , \quad \forall \tau \in [0, t].
\eeq
Now, let us consider the case that $\partial^\alpha = \partial_{t}^l$ i.e. $\alpha' = 0$. By the same consideration
as above, we first get
$$  E_{ \alpha}(\tau)   \geq { 1 \over \overline{\omega}_{m,S, t}  }
\| \partial_{t}^l  W(\tau) \|_{X^0}^2 - \|\partial_{t}^l W(\tau)\|_{L^2}^2 -  \overline{\omega}_{m, S, t}
\|U(\tau) \|_{X^{l+3}} \, \|U(\tau)\|_{X^{l+2}}.$$
Next, from \eqref{vio_0} and  \eqref{dtfin}, we get
$$ 
\|\partial_{t}^l W(\tau)\|_{L^2}^2  \leq  \overline{\omega}_{m,S, t} \big( \sum_{\alpha' \neq 0 }
 \| \partial^\alpha  U(\tau)\|_{X^{0}}^2   + \|U(\tau)\|_{X^{m+2}}^2 \big).
$$
This yields 
\beq\label{Ealpha2}
E_{ \alpha}(\tau)   \geq { 1 \over \overline{\omega}_{m,S,t}  }
\| \partial_{t}^m  W \|_{X^0}^2 -   \overline{\omega}_{m, S,t} \big(  \sum_{\alpha' \neq 0 }  \| \partial^\alpha  U(\tau)\|_{X^{0}}^2  + 
\|U\|_{X^{l+2}}^2\big) , \quad \forall \tau \in[0, t].
\eeq
Consequently, we can add \eqref{Ealpha} times $\Gamma$  sufficiently large   to  \eqref{Ealpha2}
 and use the Young inequality  to  get the result.
This completes the proof of Lemma~\ref{Elk}.
\end{proof}
Let us  finally set  for $\tau \in [0, t]$
$$ E_{m}(\tau) = \sum_{ 1  \leq l \leq  m } \Gamma^{m-l} E_{l,m}(\tau) + \Gamma \|U(\tau)\|_{X^3}^2$$
for $\Gamma$ possibly  larger  to be chosen.
We have
\begin{lem}\label{Ek}
For every $t>0$, there exists $\Gamma$ such that for every $\tau\in[0, t]$, we have
$$ E_{m}(\tau) \geq { 1 \over \overline{ \omega}_{m,S,t}}\|U(\tau)\|_{X^{m+ 3}}^2.$$
\end{lem}
\begin{proof}
We get from Lemma \ref{Elk} that
\begin{multline*} \sum_{1 \leq l \leq m } \Gamma^{m-l} E_{l,m}(\tau) \geq   { 1 \over \overline{\omega}_{m, S,t} } \|U
(\tau)\|_{X^{m+3}}^2
+ \sum_{1 \leq l \leq m-1} \Big( { \Gamma^{m-l} \over  \overline{\omega}_{m, S,t} }  - \Gamma^{m-l - 1}
\overline{\omega}_{m, S,t} \Big) \|U(\tau)\|_{X^{l+3}} \\
 + \big(  \Gamma -   \overline{\omega}_{m, S,t} \big) \|U(\tau)\|_{X^3}^2.
 \end{multline*}
 Consequently, for $\Gamma$  so that $\Gamma \geq \overline{\omega}_{m, S,t}$
we get 
$$  E_{m}(\tau)\geq   { 1 \over \overline{\omega}_{m, S,t} } \|U(\tau)\|_{X^{m+3}}^2.$$
This completes the proof of Lemma~\ref{Ek}.    
\end{proof}
We are now in position to end the proof of Proposition~\ref{energy}.
By using the identity (\ref{idenerg}) and  Lemma \ref{IJ}, we get
$$
\Big|\frac{d}{dt}E_{m}(\tau)\Big|\leq \omega\Big(\|R^{ap}\|_{{X}^{m+S}_t}+\|V^{a}\|_{{\mathcal W}^{m+S}_t}+\|U\|_{X^{m+3}_t}\Big)
\Big(\|U(\tau)\|_{X^{m+3}}^2+\|F_{ijk}(\tau)\|_{X^m}^2\Big)\,,\quad 0\leq \tau\leq t\,.
$$
Moreover, from  the equation solved by $U$, we obtain that at the initial time
\begin{equation*}
|E_m(0)| \leq   
\omega\Big(\|R^{ap}\|_{{X}^{m+3}_0}+\|V^{a}\|_{{\mathcal W}^{m+S}_0}+\|U\|_{X^{m+3}_0}\Big)
\| R^{ap}\|^2_{{X}^{m+3}_0} \,.
\end{equation*}
Consequently, we can integrate in time for $\tau \in [0, t]$ and use Lemma \ref{Ek}
to end  the proof of Proposition~\ref{energy}.
\end{proof}
\subsection{Proof of the energy estimate: proof of  Theorem \ref{theoenergie}}
It suffices to combine  Proposition~\ref{energy} and Proposition~\ref{main_sect}.
\subsection{Final argument.  End of the proof of Theorem \ref{main}}\label{finn}
In this section we   complete the
proof of Theorem \ref{main}. 
From  Theorem \ref{theoenergie}, we have 
for the solution of (\ref{Ueq}) with  initial data $V^{a}(0)=Q+\delta U^{a}(0)$ that 
\begin{multline}\label{colect}
\| U(t)\|^2_{X^{m+3}} \leq 
\omega\Big(\|R^{ap}\|_{{X}^{m+3}_t}+\|V^{a}\|_{{\mathcal W}^{m+S}_t}+\|U\|_{X^{m+3}_t}\Big)
\\
\times\Big( \| R^{ap}\|^2_{{X}^{m+3}_t} 
+\int_{0}^t\big(\|U(\tau)\|^2_{X^{m+3}}+\|R^{ap}(\tau)\|_{X^{m+ 3}}^2\big)d\tau\Big).
\end{multline}
Using (\ref{colect}) and some standard arguments (see the next section), we can define local strong
solutions of the water waves equation with
data $Q+\delta U^{a}$. Note that for the argument providing this small time existence, the specific
structure of $R^{ap}$ is not of importance, one only needs to know that it
belongs to Sobolev spaces. We now show that the estimates on $R^{ap}$ provided by Proposition~\ref{Uap}
allow to extend the solution on much longer times (sufficiently long so that
we see the instability). Thanks to Proposition~\ref{Uap} (with $s$ changed into $m$
 and thus $m$ changed into $p$), we have the bounds
$$
\| R^{ap}(t)\|_{X^{m+3}}\leq
C_{M,m}\delta^{M+3}{ e^{(M+3)\sigma_0t} \over (1+ t )^{M+3 \over 2 p } }\,,
$$
provided $0\leq t\leq T^{\delta}$, $0\leq \delta<\delta_0\ll 1$ with
$T^\delta$ such that
$$
\frac{e^{\sigma_0T^\delta}}{(1+T^\delta)^{\frac{1}{2p}}}=\frac{\kappa}{\delta},
$$
where $\kappa\in (0,1)$ is a small number to be chosen later, independantly of
$\delta\in (0,\delta_0)$.
Coming back to (\ref{colect}), we infer that
\begin{equation}\label{emm}
\|U(t)\|^2_{X^{m+3}}\leq 
\omega\big(C+\|U\|_{X^{m+3}_t}+\kappa C_{M,m}  \big)
\big(
\int_{0}^t \|U(\tau)\|^2_{X^{m+3}}d\tau+
\frac{\delta^{2(M+3)}e^{2(M+3)\sigma_0 t}}
{
(1+t)^{\frac{M+3}{p}}
}
\big),
\end{equation}
as far as $0\leq t\leq T^\delta$.
Let us define $T^*$ as
$$ 
T^*= \sup \{ T\,:\,T\leq T^\delta,\,\, {\rm and}\,\,  \forall\, t \in [0,T], \|U\|_{X^{m+3}_T} \leq 1, \, 1 -  \| \eta \|_{L^\infty}
- \| \eta^a \|_{L^\infty}  >0\}.
$$
Observe that $T$ is well-defined, at least for $\delta\ll 1$.
Using (\ref{emm}), we obtain that for $0\leq t<T^{\star}$,
\begin{equation}\label{final1}
\|U(t)\|^2_{X^{m+3}}\leq \omega(C+\kappa C_{M,m})
\Big(\int_{0}^t \|U(\tau)\|^2_{X^{m+3}}d\tau+
\frac{\delta^{2(M+3)}e^{2(M+3)\sigma_0 t}}
{
(1+t)^{\frac{M+3}{p}}
}
\Big).
\end{equation}
We take an integer $M$ large enough so that $2(M+3)\sigma_0-\omega(C)\geq 20$. At this place
we fix the value of $M$. We then choose $\kappa$ small enough so that
$
1>\omega(C+\kappa C_{M,m})-\omega(C)\,.
$
Such a choice of $\kappa$ is possible thanks to the continuity
assumption on $\omega$.  We also observe that for $A\geq 1$ and $\rho\geq 0$,
there exists $C$ such that for every $t\geq 0$, we have the inequality
\begin{equation}\label{final2}
\int_{0}^t\frac{e^{A\tau}}{(1+\tau)^{\rho}}d\tau\leq C \frac{e^{At}}{(1+t)^{\rho}}\,.
\end{equation}
Thanks to (\ref{final1}), (\ref{final2}) and the choice of $M$ and $\kappa$,
we can apply a bootstrap argument and the Gronwall lemma, we infer that $U(t)$ is
defined for $t\in [0,T^\delta]$ and that 
\begin{equation}\label{krai}
\sup_{0\leq t\leq T^\delta}\|U(t)\|_{X^{m+3}}\leq C_{M,m}\kappa^{M+3}\,.
\end{equation}
The bound (\ref{krai}) implies in particular that
$$
\|U(T^{\delta})\|_{L^2(\R^2)}\leq C_{M,m}\kappa^{M+3}\,.
$$
Let $I$ be the time interval involved in the definition of $U^0$ (see (\ref{U0form})).
Let us fix $\theta\in C^\infty_0(\R)$ which equals one on $I$ and which
vanishes near zero. Let $\Pi$ be a Fourier multiplier on $\R^2_{\xi_1,\xi_2}$
with symbol $\theta(\xi_2)$ (i.e. cutting the zero frequency in $y$).
The map $\Pi$ is bounded on $L^2(\R^2)$. We also have that $\Pi(U^0)=U^0$.
Therefore, using Proposition~\ref{U0} and Proposition~\ref{Uap}, we obtain
that for every $t\geq 0$
\begin{eqnarray*}
\|\Pi(\delta U^{a}(t))\|_{L^2(\R^2)}& \geq & c\delta\frac{e^{\sigma_0 t}}{(1+t)^{\frac{1}{2p}}}
-\sum_{j=1}^{M+1}\delta^{j+1}\|\Pi(U^i(t))\|_{L^2(\R^2)}
\\
& \geq &
c\delta\frac{e^{\sigma_0 t}}{(1+t)^{\frac{1}{2p}}}
-C_{M}\sum_{j=1}^{M+1}\delta^{j+1}\frac{e^{(j+1)\sigma_0 t}}{(1+t)^{\frac{j+1}{2p}}}\,.
\end{eqnarray*}
Thus for $\kappa\ll 1$,
$$
\|\Pi\big(\delta U^{a}(T^\delta)\big)\|_{L^2(\R^2)}\geq \frac{c\kappa}{2},\quad \forall\,
t\geq 0.
$$
Observe that for every $a\in\R$, $\Pi(Q(\cdot-a))=0$.
Recall that the true solution $U^\delta$ with data $Q+\delta U^0(0)$ is
decomposed  as 
$
U^\delta=Q+\delta U^{a} +U.
$
In particular at time zero $U^\delta$ is $\delta$ close to $Q$ in any
$H^s(\R^2)$. On the other hand, for every $a\in\R$, we can write
\begin{eqnarray*}
\|U^\delta(T^\delta,\cdot)-Q(\cdot-a)\|_{L^2} &\geq &
c\|\Pi(U^\delta(T^\delta,\cdot)-Q(\cdot-a))\|_{L^2}
\\
& = &
c\|\Pi(U^\delta(T^\delta,\cdot)-Q(\cdot))\|_{L^2}
\\
& = &
c\|\Pi(\delta U^{a}(T^\delta) +U(T^\delta))\|_{L^2}
\\
& \geq &
c\kappa-C_{M}\kappa^{M+3}\geq \frac{c\kappa}{2},
\end{eqnarray*}
provided $\kappa\ll 1$ (independantly of $\delta$). This completes the proof of Theorem~\ref{main}. 
\section{Sketch of the existence proof}
\label{sketch}

In this section, we sketch an existence proof by a vanishing viscosity type method  for
\eqref{Ueqinstab}. The  local existence of a smooth solution for \eqref{ww} in Sobolev spaces
(which corresponds to the study of \eqref{Ueqinstab} when $V^a=0$ was 
already obtained in \cite{Zhang} by using the Nash-Moser iteration  scheme
   or in \cite{Ambrose-Masmoudi} by using a subtle  Lagrangian  type formulation of the problem,
    we also refer to the work \cite{Coutand}, \cite{Shatah} for the case with vorticity.

   The aim of the following is to sketch a simple proof which avoids the use of a complicated iteration 
   scheme. 

   For a  positive number  $\nu$, we consider the system
   \beq
   \label{nu12}
   \partial_{t} U= \mathcal{F}(U^\delta ) - \mathcal{F}(V^a) -  \nu\, \Delta^2 U -R^{ap}.
    \eeq  
    By standard arguments, for every  positive $\nu $, one can
     get a solution $U^{\nu}$ of \eqref{nu12} in $C([0, T^{\nu}], H^s)$
      for some $T^{\nu}>0$ when $s \geq s_{0}$ is sufficiently large ($s>3$).
     Indeed, we can use   Propositions \ref{pak}, \ref{pak_bis} and Remark \ref{San}
      in the classical  Sobolev framework (i.e. without the time derivative,
       since the time is only parameter in  these propositions, their statement
        remain obviously valid if one replace $\pa$ by $\nabla$).
      From these estimates, the local existence  for \eqref{nu12}  follows
      from Duhamel formula and the Banach fixed point Theorem.
      We also get that the solution can be continued as long as  the $H^{s_{0}}$ norm  for some $s_{0}$
       sufficiently large  
       remains bounded.

    The next step is to prove  that the existence time of $U^\nu$ is uniform in $\nu$ i.e.
     we have to prove  that for some $T>0$, we have  $T^\nu \geq T>0$, for every $\nu \in (0,1]$.
      Towards this, it suffices to prove that the $H^s$ norm  of $U^\nu$ cannot blow-up on $[0, T]$
       for some positive $T$ independent of $\nu$. The main idea is that the  estimate
        of Theorem \ref{theoenergie} still holds for \eqref{nu12}  uniformly in $\nu$.
     By using the same transformation as in Section \ref{analysisof}, we find that
      \eqref{W} is changed into 
          \beq
     \label{nutransform}
    \pa_{t} W_{ijk} =  J  \big( L^\delta  W_{ijk} 
    -JPJ\big((\mathcal{Q}^{ijk})^\delta -(\mathcal{Q}^{ijk})^{a}\big)\big)
+PF_{ijk},
 - \nu  \, \Delta^2 W^{ijk} + \mathcal{R}^\nu
\eeq
with $W_{ijk}= P U_{ijk} $ (for the proof of the energy estimate, we shall denote
$U^\nu$ by $U$ for the sake of clarity)
  and  the subprincipal term $\mathcal{Q}^{ijk}$ and $F$ which contains
     the semilinear terms $\mathcal{G}^{ijk}$
    are the same as previously (in particular, $F$ is given by \eqref{lioniaF}). In
     particular they still satisfy with $\omega$ independent of $\nu $ the estimates of Proposition
      \ref{Q} and Proposition \ref{main_sect}.   The only new terms that show up
       are the diffusion term $-\nu \, \Delta^2$ and  
        $\mathcal{R}^\nu$ which is defined as
  \beq
  \label{Rnu}
  \mathcal{R}^\nu =  \left(  \begin{array}{ll} 0 \\   \nu  [Z^\delta  ,
      \Delta^2] U_{ijk}[1] \end{array} \right)
=
\left(  \begin{array}{ll} 0 \\   \nu  [Z^\delta  ,
      \Delta^2] W_{ijk}[1] \end{array} \right).
  \eeq 
Our goal is to perform energy estimates, uniform in $\nu\in (0,1]$ to
\eqref{nutransform}. Such estimates follow the same lines as the energy
estimates we performed for $\nu=0$. One should take care of the terms coming
from the $\nu$ dependence. Our situation is slightly different from the
classical one because of the presence of $\mathcal{R}^\nu$ coming form the
transformation $W_{ijk}= P U_{ijk} $. The point is that, thanks to the
parabolic term $-\nu\Delta^2$, the contributions of
the terms containing $\nu$ can be put in norms containing two more derivatives
with respect the energy level. For instance,  the estimates on $Z^\delta$, we already
established are strong enough to control the contribution of
$\mathcal{R}^\nu$. More precisely, thanks to Lemma~\ref{annr}, we can estimate
$Z^\delta$ in norms containing two more derivatives with respect to the energy level.
This essentially explains the approach to energy estimates for the equation \eqref{nutransform}.

By applying $\partial^\alpha$ to \eqref{nutransform}, we deduce the energy
identity which is the straightforward generalization of \eqref{idenerg}
$$
{d \over dt}
\Big( { 1 \over 2} \big( \partial^\alpha W_{ijk}, L^\delta \partial^\alpha W_{ijk}\big)
 + \mathcal{I}^{\alpha}_{ijk} \Big)  =  \mathcal{J}^{\alpha}_{ijk} + \mathcal{D}^\nu
$$
where $\mathcal{I}^{\alpha}_{ijk}$ and $\mathcal{J}^\alpha_{ijk}$ are still given by \eqref{Idef}, \eqref{Jdef}
and
\begin{eqnarray*}
 \mathcal{D}^\nu  & = &   -\nu \,  \big(\Delta^2 \partial^\alpha W_{ijk}, L^\delta \partial^\alpha W_{ijk} \big)
   +  \big(\partial^\alpha \mathcal{R}^\nu, L^\delta \partial^\alpha W_{ijk}\big)
\\
& &
-\nu\big(\Delta^2 \partial^\alpha W_{ijk},[\partial^\alpha,L^\delta ]W_{ijk}\big)+
\big(\partial^\alpha \mathcal{R}^\nu,[\partial^\alpha,L^\delta ]W_{ijk}\big)
\\
& &
-\nu\big(\Delta^2 \partial^\alpha W_{ijk},\partial^\alpha JPJ\big((\mathcal{Q}^{ijk})^\delta
-(\mathcal{Q}^{ijk})^{a} \big)\big)
\\
& &
+
\big(\partial^\alpha \mathcal{R}^\nu,\partial^\alpha
JPJ\big((\mathcal{Q}^{ijk})^\delta -(\mathcal{Q}^{ijk})^{a} \big)\big).
\end{eqnarray*}
In view of the estimates of Section~\ref{energien}, it suffices to estimate $\mathcal{D}^\nu$ 
uniformly in $\nu$. Let us start with the estimate of the first term which
will give the parabolic regularisation term.
By using an integration by parts, we have
 \beq
 \label{fin0} - \nu \big( \Delta^2\partial^\alpha W_{ijk}, L^\delta\partial^\alpha W_{ijk}\big)
  =- \nu \big( \Delta \partial^\alpha W_{ijk}, L^\delta  \Delta \partial^\alpha W_{ijk} \big) - \nu \big( \Delta 
  \partial^\alpha W_{ijk}, \big[\Delta, L^\delta ]
   \partial^\alpha W_{ijk}\big).\eeq
By using estimates as in Proposition~\ref{LdeltaQ}, we get the estimate
   \begin{eqnarray}
   \label{fin1}
   & &- \nu \big( \Delta^2 \partial^\alpha W_{ijk}, L^\delta  \partial^\alpha W_{ijk}\big) \\
\nonumber   & & 
     \leq - {\nu \over \overline{\omega}_{2,5} } \|\Delta \partial^\alpha W_{ijk}\|_{X^0}^2
      + \nu\|\Delta \partial^\alpha W_{ijk}\|_{L^2}^2  + \nu\,
      \overline{\omega}_{l, S} 
\|
 \langle \nabla \rangle  \partial^\alpha W_{ijk}\|_{X^0} \|\Delta W_{ijk} \|_{X^l},
 \end{eqnarray}
provided $l\geq 2$ and $S\geq 5$. Next, from standard interpolation in Sobolev spaces, we infer  for every $\zeta>0$ 
$$
\nu \|\Delta \partial^\alpha W_{ijk}\|_{L^2}^2  + \nu\, \overline{\omega}_{l, S} \|
 \langle \nabla \rangle  \partial^\alpha W_{ijk}\|_{X^0} \|\Delta  W_{ijk} \|_{X^l} 
  \leq  \zeta  \, \nu \| \Delta W_{ijk} \|_{X^l}^2 + C(\zeta)\, \overline{\omega}_{l, S} \|  W_{ijk} \|_{X^l}^2
$$
for  some $C(\zeta)>0$ independent of $\nu\in(0,1] $   and hence, we can choose $\zeta$ sufficiently small 
to get
\begin{equation}\label{diss1}
\nu \|\Delta \partial^\alpha W_{ijk}\|_{L^2}^2  + \nu\, \overline{\omega}_{l, S} \|
 \langle \nabla \rangle  \partial^\alpha W_{ijk}\|_{X^0} \|\Delta W_{ijk} \|_{X^l} 
\leq  {\nu \over 4 \; \overline{\omega}_{2, 5}} \| \Delta W_{ijk} \|_{X^l}^2
  + \overline{\omega}_{l,S} \|U \|_{l+3}^2.
\end{equation}   
 This yields thanks to \eqref{fin1}
 \beq
 \label{fin2}
- \nu \big( \Delta^2 \partial^\alpha W_{ijk}, L^\delta  \partial^\alpha W_{ijk}\big) 
 \leq -  {3\nu  \over 4 \; \overline{\omega}_{2, 5}} \| \Delta  \partial^\alpha W_{ijk} \|_{X^0}^2
  + \overline{\omega}_{l,S} \|U \|_{l+3}^2.
 \eeq
  Next, we can study the second term in the right hand-side of the expression
  defining $\mathcal{D}^\nu$. From the form \eqref{Rnu} of  $\mathcal{R}^\nu$, we first  get
  \begin{eqnarray} 
\nonumber     \big(\partial^\alpha \mathcal{R}^\nu, L^\delta \partial^\alpha W_{ijk}\big) 
 &  = &  \nu \Big( \partial^\alpha\big( [Z^\delta, \Delta^2 ]
 W_{ijk}[1]\big),    
- \nabla \cdot\big( v^\delta \partial^\alpha W_{ijk} [1]\big) + \partial_{x} \partial^\alpha
   W_{ijk}[1] \Big)    \\
\nonumber    & &+ \nu \Big(\partial^\alpha\big( [Z^\delta, \Delta^2 ] 
W_{ijk}[1]\big), G^\delta \partial^\alpha W_{ijk}[2] \Big) \\
 \label{Rnudef}  & \equiv &  R^\nu_{1}+ R^\nu_{2}.
   \end{eqnarray}
Let us start with the estimate of $R^\nu _{2}$ which is  the most difficult term to handle. 
Expanding the commutator, we need to estimate
$$  
\nu \Big( \partial^\alpha\big( \nabla^\beta Z^\delta  \nabla^\gamma
W_{ijk}[1]\big), 
G^\delta \partial^\alpha W_{ijk}[2] \Big), \quad |\beta | +  | \gamma |  = 4, \quad |\gamma | \leq 3.
$$
Using integrations by parts, we redistribute the $4$ derivatives of
$\nabla^\beta$ and $\nabla^\gamma$ equally to both sides of the above scalar
product. In addition, we invoke Lemma~\ref{annr} to deal  with the $Z^\delta$ contribution.
This leads to
$$   
\nu  \big|  \Big( \partial^\alpha\big
( \nabla^\beta Z^\delta  \nabla^\gamma W_{ijk}[1]\big),
G^\delta \partial^\alpha W_{ijk}[2] \Big) 
\big|  \leq \zeta \,  \nu \| \Delta  W_{ijk}\|_{X^l}^2
  + C(\zeta) \overline{\omega}_{l,S}  \|U\|_{X^{l+ 3}}^2.$$
Coming back to \eqref{Rnudef}, the term $R_{1}^\nu$ can be estimated in a
similar way. Consequently, we find
\beq\label{comRnu}
\big| \big(\partial^\alpha \mathcal{R}^\nu, L^\delta \partial^\alpha W_{ijk}\big) \big| 
\leq  \zeta \,  \nu \,  \| \Delta   W_{ijk}\|_{X^l}^2 +C(\zeta) \overline{\omega}_{l,S}  \|U\|_{X^{l+ 3}}^2\,.
\eeq
The third and the fourth terms in the expression  defining $\mathcal{D}^\nu$
can be handled very similarly. Finally, the last two terms can be estimated by
using estimates like in Proposition~\ref{Q} (here one needs to revisit the
proof of Proposition~\ref{Q} and to follow more carefully the dependence in
$\overline{\omega}_{m,S}$). In summary, we get the following estimate 
\beq
\label{fin3}\mathcal{D}^\nu \leq 
  - {3\nu  \over 4 \; \overline{\omega}_{2, 5}} \| \Delta W_{ijk} \|_{X^l}^2
 +\zeta \,  \nu \,  \| \Delta W_{ijk}\|_{X^l}^2  + C(\zeta) \overline{\omega}_{l,S} \|U \|_{l+3}^2.
\eeq
From the energy identity, we can use Lemma \ref{IJ}, Lemma \ref{Elk} and Lemma \ref{Ek}
as in Section \ref{energien} to get the energy estimate 
\begin{eqnarray*}
& &  E_{l}(t) +  {\nu} \sum_{i,\,j,\,k} \int_{0}^t   \| \Delta W_{ijk}(\tau) \|_{X^l}^2d\tau
  \leq     \omega
 \Big(\|R^{ap}\|_{{X}^{l+3 }_t}+\|V^{a}\|_{{\mathcal{W}}^{l+ S}_t}+\|U\|_{X^{l+3}_t}\Big) \\
 & &  \hspace{5cm} \times  \Big( |E_{l}(0)| +
\int_{0}^t   \Big(  \zeta \,  \nu \, \sum_{i,\,j,\,k}  \| \Delta   W_{ijk}(\tau)\|_{X^l}^2
  +  C(\zeta)  \|U(\tau) \|_{X^{l+3}}^2\Big)d\tau \Big).
  \end{eqnarray*}
  By choosing  $\zeta<1/2$, we thus get the uniform estimate
\begin{multline*}
E_{l}(t) +  {\nu \over 2 }  \sum_{i,\,j,\,k }\int_{0}^t   \| \Delta W_{ijk}(\tau) \|_{X^l}^2d\tau
 \\
\leq    \omega
 \big(\|R^{ap}\|_{{X}^{l+3}_t}+\|V^{a}\|_{{\mathcal{W}}^{l+S}_t}+\|U\|_{X^{l+3}_t}\big) 
\big(|E_{l}(0)| + \int_{0}^t    \|U(\tau) \|_{ X^{l+3}}^2 d\tau \big).
\end{multline*}
In particular
$$
E_{l}(t) \leq    \omega
 \Big(\|R^{ap}\|_{{X}^{l+3}_t}+\|V^{a}\|_{{\mathcal{W}}^{l+S}_t}+\|U\|_{X^{l+3}_t}\Big) 
\Big(|E_{l}(0)| + 
\int_{0}^t    \|U(\tau) \|_{ X^{l+3}}^2 d\tau \Big).
$$

 From Lemma \ref{Ek} and a standard continuation argument, this yields
  that  $ \|U^\nu (t) \|_{X^{l+ 3}}$ is bounded uniformly in $\nu$ on an interval of time $[0, T]$
   independent  of $\nu$.  This allows to use  strong compactness arguments in a classical way
    in order to prove that a subsequence of $U^\nu$  converges locally strongly in $H^{s_{0}}$
     to a solution of \eqref{Ueqinstab}.

\bigskip
\bigskip
{\bf Acknowledgments.} We benefited from discussions with David Lannes and Jean-Claude Saut on 
the water waves equation. We are also grateful to Francis Nier for useful
discussion on the spectral theory issues of this paper.   


\begin{thebibliography}{10}

\bibitem{Alazard-Metivier} T. Alazard and G. Metivier, {\it Paralinearization of the Dirichlet to Neumann operator
and regularity of three-dimensional water waves},  preprint  2009, arXiv:0901.2888v1.

\bibitem{Alinhac} S. Alinhac, {\it Existence d'ondes de rar\'efaction pour des syst\`emes quasilin\'eaires
 hyperboliques multidimensionnels [Existence of rarefaction waves for multidimensional hyperbolic quasilinear systems].} Comm. Partial Differential Equations 14 (1989), 173-230.

\bibitem{Alvarez-Lannes} 
B.~Alvarez-Samaniego, D.~Lannes,
{\it Large time existence for 3D water-waves and asymptotics}, Invent. Math. 171 (2008), 485-541.
%
\bibitem{Ambrose-Masmoudi}
D. Ambrose and N. Masmoudi, {The zero surface tension limit of two-dimensional water-waves.}
Comm. Pure Appl. Math. 58 (2005), 1287-1315.
%
\bibitem{APS} { J.C~Alexander, R.L.~Pego, R.L.~Sachs},
{\it On the transverse instability of solitary waves in the {K}adomtsev-{P}etviashvili equation},
Phys. Lett. A, 226  (1997), 187-192.
%
\bibitem{AK} C.~Amick, K.~Kirchg\"assner, {\it A theory of solitary water-waves in the presence of surface
tension}, Arch. Ration. Mech. Analysis 105 (1989), 1--49. 
%
\bibitem{Bridges} T. J. Bridges, {\it Transverse instability of solitary-wave states of the water-wave
problem,} J. Fluid Mech.  439  (2001), 255--278.
%
\bibitem{Buffoni} B. Buffoni, E.N. Dancer and J. F. Toland, {\it The regularity and local bifurcation
of steady periodic water waves}, Arch. Rational  Mech. Anal.152 (2000), 207-240.
%
\bibitem{Ch-Lind}D. Christodoulou and H. Lindblad,{\it On the motion of the free surface of a liquid},
Comm. Pure  Appl. Math. 53 (2000), 1536-1602.
%
\bibitem{Strauss} A. Constantin and W. Strauss. {\it Exact periodic traveling water-waves with vorticity}
C. R. Math. Acad. Sci. Paris 335 (2002), 797-800.
%
\bibitem{Coutand} D. Coutand, and  S. Shkoller,
{\it Well-posedness of the free-surface incompressible {E}uler equations with or without surface tension}, 
{J. Amer. Math. Soc.}, 20 (2007),829--930.
%
\bibitem{Craig} W. Craig, {\it An existence theory for water waves and the Boussinesq and 
Korteweg-de Vries scaling limits}. Comm Partial Differential Equations 10 (1985), 787-1003.
%
\bibitem{JD} J.~Dieudonn\'e, {\it Calcul infinit\'esimal}, Collection M\'ethodes, Hermann, Paris, 1980.
%
\bibitem{F} M.~Fedoriuk, {\it Metod perevala}, Mir, Moscow, 1977 (in Russian).
%
%
\bibitem{GT} D. Gilbarg, N.S. Trudinger, {\it Elliptic Partial Differential 
Equations of Second Order}, Second Edition, Springer, 1998.
%
\bibitem{Grenier} E.~Grenier, {\it On the nonlinear instability of Euler and Prandtl equations}, 
Comm. Pures Appl. Math. 53 (2000), 1067-1091.
%
\bibitem{GSS}
M. Grillakis, J. Shatah,   and  W. Strauss.
{\it Stability theory of solitary waves in the presence of symmetry II.}
J. Funct. Anal. 94 (1990), 308--348.
%
\bibitem{GHS} M.~Groves, M.~Haragus, S.M.~Sun, {\it Transverse instability of gravity-capillary 
line solitary waves}, C.R. Acad. Sci. Paris 333 (2001), 421-426.
%
\bibitem{Groves-Mielke} M. Groves and A. Mielke, {\it A spatial dynamics approach to three-dimensional
gravity-capillary steady water waves} Proc. Roy. Soc. Edinb. A 131 (2001), 83-136.
%
\bibitem{Henry}  D.  Henry, Geometric theory of semilinear parabolic equations.
Lecture Notes in Mathematics, 840. Springer-Verlag, Berlin-New York,  1981. iv+348 pp. ISBN: 3-540-10557-3 
%
%
\bibitem{Iguchi} T.~Iguchi, {\it A shallow water approximation for water waves}, Preprint~2008.
%
\bibitem{Iooss} G. Iooss and K. Kirchg\"assner. {\it Water waves for small surface tension: an approach
via normal form}. Proc. Roy. Soc. Edinb. A 122 (1992), 267-299.
%
\bibitem{Iooss2} G. Iooss and P. Plotnikov. {\it Small divisor problem in the theory of three-dimensional water
gravity waves.}  Mem. Amer. Math. Soc, to appear.
%
\bibitem{L1} D.~Lannes, {\it Well-posedness of the water-waves equations},
Journal AMS 18 (2005), 605-654.
%
\bibitem{Lind} H. Lindblad, {\it Well-posedness for the motion of an incompressible liquid with the 
free surface boundary}, Ann. of Math. (2) 162 (2005), 109-194.
%
\bibitem{Mielke}
A.~Mielke, {\it On the energetic stability of solitary water waves}, 
Phil. Trans. R. Soc. Lond. A 360 (2002), 2337-2358.
%
\bibitem{Zhang}
Mei-Ming and Zhifei Zhang, {Well-posedness of the water wave problem with surface tension}, 
preprint 2008.
%
\bibitem{PS}
R.~Pego, S.M.~Sun, {\it On the transverse linear instability of solitary water waves with large 
surface tension}, Proc. Royal Soc. Edinburgh 134 (2004), 733-752.
%
\bibitem{PW}  R.~Pego, M.~Weinstein, {\it Eigenvalues, and instabilities of solitary waves}, 
Phil. Trans. R. Soc. London A 340 (1992), 47-97.
%
\bibitem{Plotnikov-Toland}  P. I. Plotnikov and J. F. Toland, {\it Nash Moser theory for standing water waves},
Arch. Rational. Mech. Anal. 159 (2001), 1-83.
%
\bibitem{RT1} F.~Rousset, N.~Tzvetkov, {\it Transverse nonlinear instability for two-dimensional dispersive
models }, Ann. IHP, Analyse Non Lin\'eaire, 26 (2009) 477-496. 
%
\bibitem{RT2} F.~Rousset, N.~Tzvetkov, {\it Transverse nonlinear instability for some Hamiltonian PDE's},
J. Math.Pures Appl. 90 (2008), 550-590.
%
\bibitem{Schneider-Wayne}
G. Schneider and C. E. Wayne, {\it The rigorous approximation of long wavelength capillariy gravity waves},
Arch. Rational. Mech. Anal. 162 (2002), 247-285.
%
\bibitem{Shatah} J. Shatah,  C.  Zeng, {\it Geometry and a priori estimates for free boundary problems of
          the {E}uler equation}, Comm. Pure Appl. Math., 61 (2008),698--744.
%
\bibitem{Taylor} M.~Taylor, {\it Partial Differential Equations II}, Springer 1997.
%
\bibitem{Wu} S. Wu,  {\it Well-posedness in {S}obolev spaces of the full water wave
problem in 3-{D}}, J. Amer. Math. Soc.,12 (1999),445--495.
%
\bibitem{Zakharov} V.~Zakharov, {\it Stability of periodic waves of finite amplitude on the surface of a 
deep fluid}, J. Appl. Mech. Tech. Phys. 9 (1968), 190-194.
%
\bibitem{Z2} V.~Zakharov, {\it Instability and nonlinear oscillations of solitons}, JEPT Lett. 22 (1975), 172-173.

\end{thebibliography}
\end{document}